\documentclass[12pt]{amsart}
\usepackage{mathrsfs}
\usepackage{amsfonts}
\usepackage[centertags]{amsmath}
\usepackage{amssymb}
\usepackage{amsthm}
\usepackage{graphicx}
\usepackage{float}
\usepackage[all]{xy}
\usepackage{tikz-cd}
\usepackage{caption}
\usepackage[colorlinks]{hyperref}
\hypersetup{
bookmarksnumbered,
pdfstartview={FitH},
breaklinks=true,
linkcolor=blue,  
urlcolor=blue,
citecolor=blue,
bookmarksdepth=2
}

\usepackage{enumerate}
\usepackage[a4paper, portrait, margin=0.9in]{geometry}
\usepackage{tikz-cd}
\usepackage{rotating}
\usetikzlibrary{matrix,arrows,decorations.pathmorphing}
\usepackage{listing}
\usepackage{color}
\usepackage{extarrows}

\theoremstyle{plain}
   
   \newtheorem{theorem}{Theorem}[section]
   \newtheorem{proposition}[theorem]{Proposition}
   
   \newtheorem{lemma}[theorem]{Lemma}
   \newtheorem{corollary}[theorem]{Corollary}
   \newtheorem{conjecture}[theorem]{Conjecture}
   
\theoremstyle{definition}
   \newtheorem{definition}[theorem]{Definition}
   
   \newtheorem{example}[theorem]{Example}

   \newtheorem{remark}[theorem]{Remark}

\swapnumbers \theoremstyle{theorems}

\newcommand{\Z}{\mathbb{Z}}
\newcommand{\T}{\mathbb{T}}

\def\TT{\mathbb{T}}
\def\UU{\mathbb{U}}
\def\kk{\Bbbk}

\def\endproof{\hfill$\square$\medskip}

\begin{document}

\title[Noncommutative marked surfaces II]{\sc Noncommutative marked surfaces II: tagged triangulations, clusters, and their symmetries}

\author[Arkady Berenstein]{Arkady Berenstein \hspace{-3pt}}

\author[Min Huang]{Min Huang \hspace{-3pt}}

\author[Vladimir Retakh]{Vladimir Retakh}



\thanks{This work was partially supported by the Simons Foundation Collaboration Grant for Mathematicians no.~636972 (AB), the National Natural Science Foundation
of China (No.12471023) (MH)}

\address{Arkady Berenstein \\ Department of Mathematics, University of Oregon, Eugene, OR 97403, USA}
\email{arkadiy@math.uoregon.edu}

\address{Min Huang \\ School of Mathematics (Zhuhai), Sun Yat-sen University, Zhuhai, China.}
\email{huangm97@mail.sysu.edu.cn}

\address{Vladimir Retakh \\ Department of Mathematics, Rutgers University, Piscataway, NJ 08854, USA}
\email{vretakh@math.rutgers.edu}

\begin{abstract} The aim of the paper is to define noncommutative cluster structure on several algebras ${\mathcal A}$ related to marked surfaces possibly with orbifold points of various orders, which includes noncommutative clusters, i.e.,  embeddings of a given group $G$ into the multiplicative monoid ${\mathcal A}^\times$ and  an action of a certain braid-like group $Br_{\mathcal A}$ by automorphisms of each cluster group in a compatible way. For punctured surfaces we construct new symmetries, noncommutative tagged clusters and establish a noncommutative Laurent Phenomenon.

\end{abstract}

\maketitle

\tableofcontents

\medskip

\section{Introduction and main results}

Noncommutative cluster theory is still in its infancy. The few examples, including Kontsevich rank $2$ (free) cluster algebra (\cite{BR0,Ko,  LS}) and noncommutative marked surfaces (\cite {BR}) suggest the following informal definition. 

\subsection{Axiomatic approach}
A (noncommutative) cluster structure on a given graded algebra ${\mathcal A}$ over  a field $\kk$ is a certain graded group $Br_{\mathcal A}$ (we refer to it as {\it cluster braid group}) together with a collection of (homogeneous) embeddings $\iota$ of a given graded group $G$ (i.e., a group $G$ with a group homomorphism $G\to \mathbb Z$) into the multiplicative monoid ${\mathcal A}^\times$ (these embeddings are referred to as noncommutative clusters)  and  a (usually faithful) homogeneous action $\triangleright_\iota$ of $Br_{\mathcal A}$ on $G$ for any $\iota$ 
such that:

$\bullet$ The extensions $\iota:\kk G\to {\mathcal A}$ are injective, and their images generate ${\mathcal A}$ (and ${\mathcal A}$ is an Ore localization of $\kk G$).

$\bullet$ (monomial mutation) For any $\iota$ and $\iota'$ we expect a (unique) automorphism 
$\mu_{\iota,\iota'}$ which turns noncommutative clusters to a groupoid $\Gamma_{\mathcal A}$ so that the automorphism group $Aut(\iota)$ of any $\iota$ is isomorphic to 
 $Br_{\mathcal A}$ so that $\triangleright_\iota$ is the natural action of $Aut(\iota)$ on $G$.

$\bullet$ For any cluster homomorphism $f:{\mathcal A}\twoheadrightarrow {\mathcal A}'$ we expect a unique subgroupoid $\Gamma_{\mathcal A}^f$ and a functor $f_*:\Gamma_{\mathcal A}^f\to \Gamma_{{\mathcal A}'}$ so that its restriction to the automorphism group of each object is injective.

$\bullet$ In particular, if $\sigma$ is a cluster automorphism, we claim that the quotient homomorphism $\varphi_\sigma({\mathcal A})\twoheadrightarrow {\mathcal A}_\sigma={\mathcal A}/\langle Im(\sigma-1)\rangle$, of the coinvariant algebra of $\sigma$ is a cluster homomorphism, where the clusters on ${\mathcal A}_\sigma$ are those clusters $\iota$ of ${\mathcal A}$ for which $\iota(\kk G)$ is $\sigma$-invariant and $\varphi_\sigma(\iota(\kk G))\cong \kk G_\sigma$ for some other group $G_\sigma$ (which then becomes the cluster group of ${\mathcal A}_\sigma$).

Based on numerous examples, we expect in some cases a (noncommutative) {\it Laurent Phenomenon} as well:

$\bullet$  Given a cluster $\iota:G\hookrightarrow {\mathcal A}^\times$, for any  cluster $\iota':G\hookrightarrow {\mathcal A}^\times$ 
there is a submonoid $M_{\iota'}\subset G$ generating $G$ such that $\iota'(M_{\iota'})$ is in  the semiring $\mathbb{Z}_{\ge 0} \iota(G)$, moreover, 
$$\iota'(m)=\iota(\mu_{\iota,\iota'}(m))+\text{lower ~terms~in}~\iota(G)$$
for any $m\in {M_{\iota'}}$. 
\smallskip

In fact, this axiomatic allows us to define the {\it upper} cluster algebra ${\mathcal U}\subset {\mathcal A}$ to be the intersection of all $\iota(\kk(G))$ in ${\mathcal A}$, which will match its definition in the commutative and quantum situation.

We expect $G$ to be (almost) free, making both ${\mathcal A}$ and ${\mathcal U}$ more interesting. For instance, in the noncommutative rank $2$ case, ${\mathcal A}$ is the localization of the subalgebra ${\mathcal A}_{r_1,r_2}$ of $\kk\langle y_1^{\pm 1},y_2^{\pm 1}\rangle$ generated by $y_k$, $k\in {\mathbb Z}$ and $z$ (in the notation of \cite{BR0}). We expect that the corresponding upper cluster algebra ${\mathcal U}_{r_1,r_2}$ is generated by $y_0,y_1,y_2,y_3$. 

Here $G=\langle y_1,y_2\rangle$ is the free group of rank $2$ with the cluster braid group action given by (in the notation of \cite{BR0}) $T_1,T_2\in Aut(G)$ via 
\begin{equation}
    \label{eq:braid rank 2}
 T_i(y_j)=\begin{cases}
 y_i & \text{if $i=j$}\\
 y_1^{-r_1}y_2 & \text{if $i=1,j=2$}\\
 y_1y_2^{r_2} & \text{if $i=2,j=1$}\\

\end{cases}
\end{equation}
where $r_1,r_2$ are fixed natural numbers. We denote by $Br_{r_1,r_2}$ the subgroup of $Aut(G)$ generated by $T_1$ and $T_2$.
We show in Section \ref{sec:Rank $2$ cluster groups and braid action} that $Br_{r_1,r_2}$ is essentially an Artin braid group, i.e., it satisfies
$$\underbrace{T_1 T_2 T_1\cdots}_{m \text{ letters}} =\underbrace{T_2 T_1 T_2\cdots}_{m \text{ letters}}\ ,$$ where 
$m=\begin{cases} 
3 & \text{if $r_1r_2=1$}\\
4 & \text{if $r_1r_2=2$}\\
6 & \text{if $r_1r_2=3$}\\
\end{cases}$,
which justifies the name. 
We prove (Theorem \ref{th:faithful rank 2}) that \eqref{eq:braid rank 2} is, indeed, the presentation of $Br_{r_1,r_2}$ when $r_1r_2\in \{1,2,3\}$ and $Br_{r_1,r_2}$ is free if $r_1r_2\ge 4$. In particular, $Br_{1,1}$ is the ordinary braid group $Br_3$ on $3$ strands.
 We can also illustrate how the abelianization works here 
by replacing $G$ with $\mathbb{Z}^2$. Namely, define $T_i^{ab}\in Aut(\mathbb{Z}^2)=GL_2(\mathbb{Z})$, 
$i=1,2$ by same formulas \eqref{eq:braid rank 2}, i.e., 
$T_1^{ab}=
\begin{pmatrix}
1 & 0\\
r_2 & 1\\
\end{pmatrix}$, 
$T_2^{ab}=
\begin{pmatrix}
1 & -r_1\\
0 & 1\\
\end{pmatrix}$,
and the abelianization homomorphism $Br_{r_1,r_2}\to GL_2(\mathbb{Z})$ by $T_i\mapsto T_i^{ab}$. It is curious to see that the homomorphism is not injective precisely when $r_1r_2\in \{1,2,3\}$ and $T_1^{ab} T_2^{ab}$ in $GL_2(\mathbb{Z})$ is of finite order (Lemma \ref{lem:finiteorder}).
We expect this phenomenon of non-injectivity of the structural homomorphism $Br_{\mathcal A}\to Br_{{\mathcal A}^{ab}}$ to be non-injective frequently, see examples in Section \ref{subsec:Abelianization and $q$-abelianization of Noncommutative surfaces} (by the way, the abelianization homomorphism ${\mathcal A}\to {\mathcal A}^{ab}$ is expected to be a cluster one). In this case,  
the clusters are labeled by integers ($G_k=\langle y_k,y_{k+1}\rangle \simeq F_2$, $k\in \mathbb{Z}$) and the monomial mutations 
$\mu_{k\ell}$ are isomorphisms $G_\ell\simeq G_k$ determined by $\mu_{km}=\mu_{k\ell}\circ \mu_{\ell m}$ whenever $m$ is in the interval $[k,\ell]$, $\mu_{kk}=Id_{G_k}$ and 
$\mu_{k,k+1}(y_{k+1})=y_{k+1}$, 
$\mu_{k,k+1}(y_{k+2})=\begin{cases} 
y_k^{-1} y_{k+1}^{r_{k+1}} & \text{if $k$ is even}\\
y_k^{-1} & \text{if $k$ is odd}\\
\end{cases}$,
$\mu_{k+1,k}(y_{k+1})=y_{k+1}$, 
$\mu_{k+1,k}(y_{k})=\begin{cases} 
y_{k+2}^{-1} y_{k+1}^{r_{k+1}} & \text{if $k$ is odd}\\
y_{k+2}^{-1} & \text{if $k$ is even}\\
\end{cases}$. 

The corresponding algebra ${\mathcal A}_{r_1,r_2}$ defined in \cite{BR0} exhibits Noncommutative Laurent Phenomenon (see \cite{BR0} and Section \ref{sec:rank2alg}).

In the commutative/quantum setting, we claim that the localization ${\mathcal A}$ of a (quantum) cluster algebra $\underline {\mathcal A}$ by the set $X$ of all of its cluster variables satisfies all of the above requirements with $G\cong {\mathbb Z}^m$ (or its central extension $G_q$ in quantum case) so that $\kk G=\kk[x_1^{\pm 1},\ldots,x_m^{\pm 1}]$ for a given cluster $\{x_1,\ldots,x_n\}$ in ${\mathcal A}$. 
The well-known commutative/quantum Laurent Phenomenon asserts that the set of all (quantum) cluster variables belongs to the group algebra $\kk G$ which is an instance of its 
noncommutative counterpart stated above. In these cases, $Br_{\mathcal A}$ is essentially the group of symplectic transvections introduced in \cite{SSVZ}) and as we prove in Section \ref{subsec:classical monomial mutations and braid}, it is always a quotient of an appropriate Artin braid group (which, is the case for the ``abelianization" of $Br_{r_1,r_2}$
above). In the commutative case (geometric type), each seed ${\bf S}$ is essentially the exchange $m\times n$ matrix $\widetilde B=(b_1,\ldots,b_n)$ (all $G_\Sigma$ are copies of $\mathbb{Z}^m$).
For any elementary mutation ${\bf S} \stackrel{k}{\to} {\bf S}'$ define $\mu_{{\bf S}',{\bf S}}: G_{{\bf S}}\to G_{{\bf S}'}$ by 
$\mu_{{\bf S}',{\bf S}}(e_j)=\begin{cases} 
e_j & \text{if $j\neq k$}\\
-e_k+[b_k]_+ & \text{if $j=k$}\\
\end{cases}$ and extend uniquely by transitivity for any ${\bf S}$, ${\bf S}'$ viewed as vertices of the free $n$-valent tree (quantum case is nearly identical, see Section \ref{subsec:classical monomial mutations and braid} for details). The Laurent Phenomenon is well-known in these cases and $\mu_{{\bf S}',{\bf S}}$ can be viewed as the leading term of the Laurent expansion (Theorems \ref{th:monomial mutation surfaces intro}, \ref{th:leading term}, and \ref{th:leading term rank 2}).

\subsection{Noncommutative type $A$ cluster structure}
Our next, totally noncommutative,  cluster algebra ${\mathcal A}_n$ introduced in \cite{BR} (which is corresponding to the Dynkin type $A_{n-3}$) is generated by $x_{ij}^{\pm 1}$ for distinct $i,j\in [1,n]$ subject to 

$\bullet$ (Triangle relations)
$x_{ij}x_{kj}^{-1}x_{ki}=x_{ik}x_{jk}^{-1}x_{ji}$ for distinct $i,j,k\in [1,n]$;

$\bullet$ (Ptolemy relations)
$x_{ik}=x_{ij}x_{lj}^{-1}x_{lk}+x_{il}x_{jl}^{-1}x_{jk}$ for distinct $i,j,k,l\in [1,n]$ such that $i,j,k,l$ are in clockwise order.

Following \cite{BR}, we construct in Section \ref{Sec:triangleg} noncommutative clusters for ${\mathcal A}_n$ as certain embeddings $\iota_\Delta$ of the free group $F_{3n-4}$ into ${\mathcal A}_n$ labeled by triangulations $\Delta$ of the $n$-gon, so that the image of $\iota_\Delta$ is the subgroup of ${\mathcal A}_n^\times$ generated by $x_{ij}$, $(i,j)\in \Delta$. More precisely, following \cite{BR}, we define the triangle group $\TT_\Delta$ to be generated by $t_{ij}$, $(i,j)\in \Delta$ subject to the above triangle relations and claim that the assignments $t_{ij}\mapsto x_{ij}$, $(i,j)\in \Delta$ define an injective homomorphism of groups $\TT_\Delta\hookrightarrow {\mathcal A}_n$ which will play a role of a noncommutative cluster (with a slight abuse of notation, $\TT_\Delta$ is our noncommutative cluster group). 
The noncommutative Laurent Phenomenon holds for all noncommutative clusters for ${\mathcal A}_n$ (see \cite{BR} and Section \ref{sec:NC Laurent surfaces} for details).

Furthermore, for any triangulation $\Delta$ of the $n$-gon and any internal edge $(i,k)\in \Delta$ we define an automorphism $T_{ik}$ of $\TT_\Delta$
 by  
 \begin{equation}
     \label{eq:Tij}
 T_{ik}(t_{\gamma})=\begin{cases}
     t_{ij}t_{kj}^{-1}t_{kl}t_{il}^{-1}t_{ik} & \text{if $\gamma=(ik)$}\\
     t_{ki}t_{li}^{-1}t_{lk}t_{jk}^{-1}t_{ji} & \text{if $\gamma=(ki)$}\\

          t_\gamma & \text{otherwise}\\

 \end{cases},
 \end{equation}
 where $(i,j,k,l)$ is the unique clockwise quadrilateral in $\Delta$ with the diagonal $\gamma$ (that is, $T_\gamma$ scales the noncommutative diagonal $t_\gamma$ by a noncommutative cross-ratio of its quadrilateral and fixes all other diagonals).
 
 We denote by $\underline{Br}_\Delta^+$ (resp. $\underline{Br}_\Delta$) the submonoid (resp. the subgroup)  of $Aut(\TT_\Delta)$ generated by all $T_{ik}$ (clearly, $\underline{Br}_\Delta^+\subset \underline{Br}_\Delta$ and the former generates the latter).

This notation is justified by the following theorem.

\begin{theorem} [Theorem \ref{th:Br on Sigman1}] 
\label{th:braid triangular polygon}  For any $n\ge 4$ and any triangulation $\Delta$ of the $n$-gon, the group $\underline{Br}_\Delta$ is isomorphic to the braid group $Br_{n-2}$ on $n-2$ strands. Moreover, the monoid $\underline{Br}_{\Delta}^+$ is generated by $T_{ij}=T_{ji}$ for all diagonals $(i,j)\in \Delta$ subject to the following relations: 
$$\begin{cases}
T_{ij}T_{jk}T_{ki}T_{ij}=T_{jk}T_{ki}T_{ij}T_{jk}, & \text{if $(i,j,k)$ is a counter-clockwise  triangle in $\Delta$},\\
T_{ij}T_{k\ell}T_{ij}=T_{k\ell}T_{ij}T_{k\ell}, &  \text{if $(i,j)$ and $(k,\ell)$ are two sides of some triangle in $\Delta$},\\
T_{ij}T_{k\ell}=T_{k\ell}T_{ij}, & \text{if $(i,j)$ and $(k,\ell)$ are not two sides of any triangle in $\Delta$.}
\end{cases}
$$
\end{theorem}

For instance, if  $\Delta$ is a triangulation of the hexagon as in Figure \ref{fig:hexagon},
     \begin{figure}[h]
\includegraphics{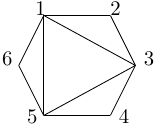}
\caption{Hexagon with a triangulation}
 \label{fig:hexagon}
\end{figure}                                                            
both $\underline{Br}_{\Delta}^+$ and $\underline{Br}_{\Delta}\cong Br_4$ are generated by $T_{13}$, $T_{15}$, and $T_{35}$ subject to $T_{13}T_{35}T_{13}=T_{35}T_{13}T_{35}, T_{35}T_{15}T_{35}=T_{15}T_{35}T_{15}, T_{13}T_{15}T_{13}=T_{15}T_{13}T_{15}$ and 
$T_{31}T_{15}T_{53}T_{31}=T_{15}T_{53}T_{31}T_{15}=T_{53}T_{31}T_{15}T_{53}$.

By definition, the monoid $\underline{Br}^+_\Delta$ maps into the group $\underline{Br}_\Delta$ by the natural (Grothendieck) localization. It follows from Remark \ref{rem:acyclic braid type A} that for any triangulation $\Delta$ of $\Sigma_{n+2}$ with each triangle having a boundary edge,  $\underline{Br}^+_\Delta$ coincides with the standard braid monoid $Br_n^+$. Therefore, results of Brieskorn (see e.g., \cite{P}) imply that $Br_n^+$ naturally embeds into $Br_n$. We conjecture (Conjecture \ref{conj:injective braid monoid group}) that this injectivity holds for any triangulation of any surface.

\subsection{Noncommutative type $B,C,D$ cluster structures} Our next cases of noncommutative cluster algebras, ${\mathcal B}_n$, ${\mathcal C}_n$,
and ${\mathcal D}_n$ (corresponding to Dynkin types $B_{n-1}$, $C_{n-1}$, 
and $D_n$, respectively), are generated by $(x_{ij}^\pm)^{\pm 1}$ for $i,j\in [1,n]$ and $x_{0,i}^{\pm 1}$, $x_{i,0}^{\pm 1}$ (for ${\mathcal B}_n$ and ${\mathcal D}_n$ only), subject to: 

$\bullet$ (Triangle relations) $x_{ij}^+(x_{kj}^+)^{-1}x_{ki}^+=x_{ik}^-(x_{jk}^-)^{-1}x_{ji}^-$ for $i,j,k\in[1,n]$ with $i\neq j,k$ in clockwise order (where $j=k$ is allowed).

$\bullet$  (Additional triangle relations for ${\mathcal B}_n$ and ${\mathcal D}_n$) $x_{0i}(x^-_{ji})^{-1}x_{j0}=x_{0j}(x_{ij}^+)^{-1}x_{i0}$ for $i,j\in [1,n]$ with $i\neq j$;

$\bullet$ (Ptolemy relations) $x_{lj}^-=x_{lk}^+(x_{ik}^+)^{-1}x_{ij}^++x_{li}^+(x_{ki}^-)^{-1}x_{kj}^-$ for $i,j,k,l\in [1,n]$ such that $i,j,k$ are distinct, $i,j\leq l$, and $i,j,k,l$ appear in clockwise order (where $k=l$ is allowed).

$\bullet$ (Additional Ptolemy relations for ${\mathcal B}_n$ and ${\mathcal D}_n$) $x_{ik}^+=x_{ij}^+x_{0j}^{-1}x_{0k}+x_{i0}x_{j0}^{-1}x_{jk}^+$ for $i,j,k\in [1,n]$ with $j\neq i,k$ in clockwise order (we allow $i=k$),

$\bullet$  (Additional relation for $\mathcal B_n$) $x^{+}_{ii}=x^-_{ii}=x_{i0}x_{0i}$ for any $i\in [n]$.

As in the usual Lie-theoretic setting, where $B_{n-1}$ is a folding of $D_n$ and $C_{n-1}$ is a folding of $A_{2n-3}$, we prove the following results (in fact, implicitly we use coinvariant algebra of an automorphism $\sigma$, see Section \ref{subsec:symmetries and orbifolds}).

\begin{theorem} [Corollaries \ref{cor:Chekhov-Shapiro} and \ref{cor:Tp=1}]
\label{th:quotient B_n intro}
For all $n\ge 2$ one has:

$(a)$ For any $d\ge 2$, the quotient 
of ${\mathbb Q}(\cos\frac{\pi}{d})\otimes_\mathbb Q{\mathcal A}_{nd}$  
by relations 
$x_{ij}=x_{i+n, j+n}$ for distinct $i,j=1,\ldots,nd$, (indices taken modulo $nd$), and $x_{i,i+kn}=\frac{\sin(k\pi/d)}{\sin(\pi/d)}x_{i,i+n}=\frac{\sin(k\pi/d)}{\sin(\pi/d)}x_{i+n,i}$, $i=1,\cdots,n, k=1,\cdots,d-1$
is generated by $x_{ij}^+:=x_{ij}, x_{ij}^-:=x_{i,j+(d-1)n}$ for distinct $i,j=1,\ldots,n$
and $x_i:=x_{i,i+n}=x_{i+n,i}$ for $i=1,\ldots,n$ subject to:

$\bullet$ $x_{ij}^+(x_{kj}^+)^{-1}x_{ki}^+=x_{ik}^-(x_{jk}^-)^{-1}x_{ji}^-$ for any distinct $i,j,k$ in clockwise order.
 
$\bullet$  $x_{ij}^+x^{-1}_{j}x_{ji}^+=x_{ij}^-x^{-1}_{j}x_{ji}^-$ for any distinct $i,j$.

$\bullet$ $x_{\ell j}^+=x_{\ell i}^+(x_{ki}^-)^{-1}x_{kj}^-+x_{\ell k}^+(x_{ik}^+)^{-1}x_{ij}^+$ for any distinct $i,j,k,\ell$ in clockwise order.

$\bullet$ $x_j=x_{ji}^- x_i^{-1}x_{ij}^+ + 2\cos\left({\frac{\pi}{d}}\right)x_{ji}^+ x_i^{-1} x_{ij}^+ +x_{ji}^+ x_i^{-1}x_{ij}^-$ for any distinct $i,j$.

(This is a noncommutative version of Chekhov-Shapiro algebra from \cite[Section 2.1]{CS}, see also Definition \ref{def:generalized Chekhov-Shapiro}). In particular, this is ${\mathcal C}_n$ if $d=2$.

$(b)$ ${\mathcal B}_n$ is the 
quotient of ${\mathcal D}_n$ given by relations 
$x_{0i}=x^{-1}_{i0}x_{ii}^{+},x_{i0}=x_{ii}^{-}x_{0i}^{-1}$ for  $i=1,\ldots,n$.
\end{theorem}

We expect that all noncommutative clusters $\iota:G\hookrightarrow {\mathcal X}_n$, where $\mathcal X_n\in \{\mathcal B_n,\mathcal C_n,\mathcal D_n\}$, are in one-to-one correspondence with appropriate triangulations (=the corresponding commutative clusters) $\Delta$ of a once punctured $n$-gon as follows (See Sections \ref{sec:NC Laurent surfaces}). 

$\bullet$ If ${\mathcal X}_n={\mathcal B}_n$ or ${\mathcal C}_n$ then these are triangulations of once punctured $n$-gon with the collapsed triangle around the puncture. There are $\binom{2n-2}{n-1}$ such triangulations.

 \begin{figure}[ht]
\includegraphics{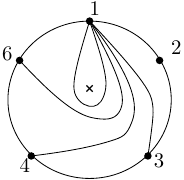}
\caption{}
\end{figure}

$\bullet$ If ${\mathcal X}_n={\mathcal D}_n$ then these are {\it tagged} triangulations of once punctured $n$-gon. There are $\frac{3n-2}{n}\binom{2n-2}{n-1}
=\binom{2n-2}{n}+\binom{2n-1}{n}$ of them, out of which the first summand is the number of  triangulations  with no self-folded triangles (thus approximately $\frac{1}{3}$ of all clusters are unavoidably tagged).

For any such a triangulation $\Delta$, similarly to ${\mathcal A}_n$, we define a triangle group $\TT_\Delta$ with above triangle relations together with a natural inclusion $\iota_\Delta:\TT_\Delta\hookrightarrow {\mathcal D}_n$, $t_\gamma\mapsto x_\gamma$ (all $\TT_\Delta$ are free of same rank, so, with a slight abuse of notation, this is our group $G$ from the axioms in the beginning of the section). This exhausts all noncommutative clusters for ${\mathcal X}_n$ (it follows from \cite{BR} and Theorem \ref{thm:1-relator} that $\TT_\Delta$ is a free group of rank $3n$ for ${\mathcal D}_n$,  $3n-1$ for ${\mathcal B}_n$, and $3n-2$ for ${\mathcal C}_n$).

The noncommutative Laurent Phenomenon holds in ${\mathcal B}_n,{\mathcal C}_n,{\mathcal D}_n$ as well (see Section \ref{sec:NC Laurent surfaces} for details).

Similarly to \eqref{eq:Tij}, for any aforementioned triangulation $\Delta$ of an appropriately punctured $n$-gon, we define an automorphism $T_\gamma$ of $\TT_\Delta$ for all internal edges $\gamma\in \Delta$ (see Section \ref{Sec:triangleg} for details) and denote by $\underline{Br}_\Delta^+$ (resp. $\underline{Br}_\Delta$) the submonoid (reps. the subgroup)  of $Aut(\TT_\Delta)$ generated by all $T_{\gamma}$ (clearly, $Br_\Delta^+\subset Br_\Delta$ and the former generates the latter). The following is an analog of Theorem \ref{th:braid triangular polygon}. 

\begin{theorem} [Corollary \ref{Cor:CD}, Theorem \ref{th:Br on Sigman1}]
\label{th:braid triangular once punctured polygon} 

For any triangulation $\Delta$ as above, the group $\underline{Br}_\Delta$ is isomorphic to a quotient of the Artin braid group $Br_{B_{n-1}}$, $Br_{C_{n-1}}$, and $Br_{D_n}$ respectively for ${\mathcal B}_n$, ${\mathcal C}_n$, and ${\mathcal D}_n$. Moreover, the surjective homomorphism $Br_{C_{n-1}}\twoheadrightarrow \underline{Br}_\Delta$ is an isomorphism.   
\end{theorem}

We expect that the surjective homomorphisms $Br_{B_{n-1}}\twoheadrightarrow \underline{Br}_\Delta$ and $Br_{D_n}\twoheadrightarrow \underline{Br}_\Delta$ are isomorphisms, that is, our actions of $Br_{B_{n-1}}$ and $Br_{D_n}$ on the corresponding free groups $\TT_\Delta$ are faithful. We verified this for ${\mathcal D}_2$, i.e., $Br_{{\mathcal D}_2}\cong \mathbb{Z}^2$ in Example \ref{ex:D_2}.

\subsection{Noncommutative marked surfaces}
The difficulty in proving that these homomorphisms are isomorphisms suggested a more conceptual definition of $Br_\Delta$  as automorphisms groups of objects of a certain groupoid ${\bf Tsurf}_\Sigma^t$ (which is a main example of what we call $\varphi$-groupoids, see Section \ref{subsec:triangulated surfaces and triangle groups} for details). We abbreviate $Br_\Delta:=Aut_{{\bf Tsurf}^t_\Sigma}(\Delta)$, the automorphism group of an object $\Delta$ of the groupoid ${\bf Tsurf}^t_\Sigma$ and refer to it as the {\it braid group} of the triangulation $\Delta$. 
 This is justified by the following

\begin{theorem} [Theorem \ref{th:braid generation}, Theorem \ref{th:Br_n 1}]
\label{th:quiver braid group intro}
$Br_\Delta$ is always generated by elements $T_\gamma$ for all non-pending internal edges $\gamma$ of $\Delta$. Moreover, 

$(a)$ $Br_\Delta\cong {Br}_{n-2}$ for any triangulation $\Delta$ of the unpuncture $n$-gon $\Sigma_{n}$.

$(b)$ $Br_\Delta\cong {Br}_{B_{n-1}}$ for any triangulation $\Delta$ of the $n$-gon with a $0$-puncture. 

$(c)$ 
$Br_\Delta\cong Br_{C_{n-1}}$ for any triangulation $\Delta$ of the once specially puncture $n$-gon.

$(d)$ $Br_\Delta\cong Br_{D_n}$ for any triangulation $\Delta$ of the once punctured $n$-gon. 

$(e)$ $Br_\Delta\cong \hat Br_{p+q-1}$ for any triangulation $\Delta$ of the unpunctured cylinder with $p$ marked points on one boundary and $q$ on another. 

$(f)$ $Br_\Delta\cong Br_{\hat D_n}$ for any triangulation $\Delta$ of the twice punctured $(n-2)$-gon. 

$(g)$ $Br_\Delta\cong Br_{\hat B_n}$ for any triangulation $\Delta$ of the once ordinarily and once specially punctured $(n-1)$-gon. 

$(h)$ $Br_\Delta\cong Br_{\hat C_n}$ for any triangulation $\Delta$ of the twice specially punctured $n$-gon. 
\end{theorem}



Moreover, one of our main results is Theorem \ref{th:brgroup}, in which we give a precise presentation for $Br_\Delta$ for any marked surface $\Sigma$. Rather surprisingly, this generalizes quiver braid groups introduced and studied in  
\cite{BM} and \cite{Q} (Remarks \ref{rem:equi} and \ref{rem:quivers no special}). A Weyl group analogue of this result has also been investigated in \cite{FST24}. In case the surface has no zero or special punctures, the group $Br_\Delta$ is isomorphic to the braid twist group of the decorated version of $\Sigma$ \cite{KQ,Q,Q1,QZ}.

In fact, we can recover both classical cluster structures of the types $A_{n-2}$, $B_{n-1}$, $C_{n-1}$, and $D_n$ as abelianizations of ${\mathcal A}_n$, ${\mathcal B}_n$, ${\mathcal C}_n$, and ${\mathcal D}_n$, respectively, together with their symplectic transvection groups. Similarly, quantum cluster structures of types $A_{n-2}$ and $C_{n-1}$ can be recovered from ${\mathcal A}_n$ and ${\mathcal C}_n$, respectively, by forcing the appropriate generators to $q$-commute ($D_n$ is excluded due to puncture), see Section \ref{subsec:Abelianization and $q$-abelianization of Noncommutative surfaces}.

Generalizing ${\mathcal A}_n$, ${\mathcal B}_n$, ${\mathcal C}_n$, and ${\mathcal D}_n$ and following \cite[Section 3]{BR} we introduce {\it non-commutative surface} ${\mathcal A}_\Sigma$ for any (connected or not) marked surface $\Sigma$ that also may have orbifold points of orders  $\Z_{\geq 2}$ and the order $\frac{1}{2}$ (studied in \cite{FST2}), we refer to them as special punctures and $0$-punctures respectively (Section \ref{subsec:Some notation on surfaces}).

It turns out that the presentation of (generalized) ${\mathcal A}_\Sigma$ can be given only in terms of {\it total angles} $T_i$, $i\in I$ (Section \ref{subsec:From surfaces to their noncommutative versions}). In fact, we need only the following axioms to glue a ``noncommutative surface" out of ``noncommutative triangles".
 
 $\bullet$ If $\Sigma=\Sigma_3$, the unpunctured disk with three marked points $I=\{1,2,3\}$, then ${\mathcal A}_\Sigma$ is generated by $x_{ij}^{\pm 1}$, $i,j\in I$ subject to the triangle relation 
 $$T_1^{23}=T_1^{32},$$ 
 where $T_i^{jk}=x_{ji}^{-1}x_{jk}x_{ik}^{-1}$ is the noncommutative angle at the vertex $i$ of the triangle $\Sigma_3$ (in fact, the above relation is equivalent to $T_2^{13}=T_2^{31}$ or $T_3^{12}=T_3^{21}$, i.e.,  the angles depend only on the vertex. These are noncommutative analogs of Penner's $h$-lengths, see e.g., \cite{BR}).
 \begin{figure}[ht]

\centerline{\includegraphics{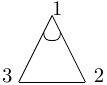}}
\end{figure}

 $\bullet$ If $P$ is a polygon in $\Sigma$, the angle $T_i^P$ is well-defined at every vertex $i$ of $P$ and it is additive in the sense that any subdivision of $P$ by its internal edge at $i$ into two sub-polygons $P'$ and $P''$ results in a relation 
 $$T_i^P=T_i^{P'}+T_i^{P''} \ .$$
 In particular the total angle $T_i\in {\mathcal A}_\Sigma$ is defined for any marked point $i\in I$. 

  \begin{figure}[ht]

\centerline{\includegraphics{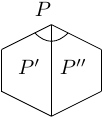}}
\caption{Additivity of angles}
\end{figure}

When $\Sigma$ has ordinary puncture, for any $i\in I_{p,1}(\Sigma)$ and $P\subset I_{p,1}(\Sigma)$, let 
\[\chi_P(i)=
\begin{cases}
1 & \text{if $i\in P$,}\\
0 &\text{otherwise}.
\end{cases}\]
We obtain a surprising generalization of \cite[Proposition 3.15]{MSW}.

\begin{theorem} [Corollary \ref{cor:tagged automorphism}]
\label{th:tagged automorphism intro}        
For any subset $P\subset I_{p,1}(\Sigma)$ the assignments $$x_{\gamma}\mapsto T_{s(\gamma)}^{\chi_P(s(\gamma))}x_{\gamma}T_{t(\gamma)}^{\chi_P(t(\gamma))}
$$ define an involutive automorphism $\varphi_P$
of the algebra ${\mathcal A}_\Sigma$. Moreover, $\varphi_{P\cup P'}=\varphi_P \circ \varphi_{P'}$ if $P\cap P'=\emptyset$.
\end{theorem}

This $\varphi_P$ can be viewed as a noncommutative analog of a green sequence of mutations (see e.g., \cite{Dk} and Remark \ref{rmk:green}). We expect that all cluster automorphisms of ${\mathcal A}_\Sigma$ are compositions of automorphisms of $\Sigma$ and $\varphi_P$
(Conjecture \ref{conj:clusterauto}).

In fact, the elements $x_\gamma^{tag}:=\varphi_P(x_\gamma)$ generalize tagged cluster  coordinates introduced in \cite{FST}. In Section \ref{sec:NC Laurent surfaces} we describe an explicit noncommutative Laurent phenomenon for all (tagged and non-tagged) cluster variables $x_\gamma^{tag}$.

Theorem \ref{th:tagged automorphism intro} implies that the coinvariant algebra of $\varphi_p$ is quotient of ${\mathcal A}_\Sigma$ by the relation $T_p=1$ (unless $\Sigma$ is closed once punctured, see Corollary \ref{cor:Tp=1}) and this is ${\mathcal A}_{\Sigma^{p}}$, where $\Sigma^{p}$ is $\Sigma$ in which $p$ is regarded as a $0$-puncture (Corollary \ref{cor:Tp=1}). Thus, the aforementioned ${\mathcal B}_n$ is a noncommutative disk with a single $0$-puncture. 

Following \cite{BR}, we prove (Corollary \ref{cor:BSigma functorial}) that our noncommutative surfaces ${\mathcal A}_\Sigma$ are topological invariants of $\Sigma$ (possibly with special or $0$-punctures). Specifically, the assignment $\Sigma\mapsto {\mathcal A}_\Sigma$ is a functor from the category of such surfaces to the category of $\mathbb{Q}$-algebras (we expect this functor to be fully faithful).

It turns out that there are even finer invariants, which we refer to as {\it sector subalgebras}.  This is a subalgebra ${\mathcal B}_\Sigma$ of ${\mathcal A}_\Sigma$ generated by {\it noncommutative sectors} $y_{\gamma,\gamma'}:=x_{\overline \gamma}^{-1}x_{\gamma'}$ for all pairs $(\gamma,\gamma')$ of composable (oriented) curves (where $\overline \gamma$ is oppositely oriented $\gamma$), i.e., $\gamma$ and $\gamma'$ form a directed sector in $\Sigma$ 
(these are analogs of $Y$-coordinates on usual/quantum cluster varieties).

For instance, if $\Sigma=\Sigma_n$, is an unpunctured disk with
$n$ boundary points, then ${\mathcal B}_\Sigma$ is generated by all $y_{ij}^k$ for distinct $i,j,k\in [n]$ subject to the relations in \cite[Theorem 2.14]{BR}, see also Theorem \ref{th:sector presentation}.

It is almost immediate (Corollary \ref{cor:BSigma functorial}) that ${\mathcal B}_\Sigma$
is also a topological invariant of $\Sigma$.

Following \cite{BR}, to any triangulation $\Delta$ of any surface $\Sigma$ we assign the {\it triangle group $\TT_\Delta$} (see Definition \ref{def:trig})  generated by $t_\gamma$, $\gamma\in \Delta$, subject to the following relations

$\bullet$ $t_{\overline \gamma}=t_\gamma$ for any special loop $\gamma\in \Delta$.

$\bullet$ $t_{\alpha_1}t^{-1}_{\overline \alpha_2}t_{  \alpha_3}=t_{\overline \alpha_3}t^{-1}_{ \alpha_2}t_{\overline \alpha_1}$
 for any cyclic triangle $(\alpha_1,  \alpha_2,  \alpha_3)$ in $\in \Delta$. 

$\bullet$ $t_\ell=t_\gamma t_{\overline \gamma}$ if $\ell$ is a loop encloses a pending arc $\gamma$ with $s(\gamma)=s(\ell)$. 

By definition, $\T_\Delta$ is naturally graded via $\deg t_\gamma=1$. 

For any oriented marked surface $\Sigma$, the monomial mutations from the beginning of the introduction $\mu_{\Delta',\Delta}:\TT_{\Delta'}\simeq \TT_{\Delta}$ are well-defined (homogeneous) group isomorphisms  viewed as the transitive extensions of ``first halfs" of the Ptolemy relations (Section \ref{subsec:triangle groups}). In fact these monomial mutations are modeled in the aforementioned groupoid ${\bf TSurf}_\Sigma$ as horizontal morphisms $h_{\Delta',\Delta}$ from $\Delta$ to $\Delta'$, under the natural functor from ${\bf TSurf}_\Sigma$ to the groupoid ${\bf Grp}'$ whose objects are groups and morphisms are group isomorphisms (Theorem \ref{th:monomial mutation surfaces} and Remark \ref{rm:monomial mutation surfaces}). 

Generalizing  \cite[Theorem 3.30]{BR}, we prove that for  any triangulation $\Delta$ of $\Sigma$ the assignments $t_\gamma\mapsto x_\gamma$, $\gamma\in \Delta$ define an injective homomorphism of groups $\iota_\Delta: \TT_\Delta\hookrightarrow {\mathcal A}_\Sigma^\times$ which extends to an injective homomorphism of algebras $\kk \TT_\Delta\hookrightarrow {\mathcal A}_\Sigma$ (Theorem \ref{th:iotaDelta} (a)), 
which we view a noncommutative cluster in the sense of the axioms at the beginning of the section and this also gives is a noncommutative Laurent Phenomenon because all $x_\gamma$ belong to the image of $\iota_\Delta$. 
In particular, this recovers the quantum expansion formula from \cite{MSW} for surfaces with no $0$-punctures and no  ordinary punctures (Corollary \ref{cor:expansion21}). 

\begin{theorem} 
[Proposition \ref{expansion1} (1), Corollary \ref{Cor:NCexpan}]
\label{th:monomial mutation surfaces intro} 
Given triangulations $\Delta$ and $\Delta'$ of an oriented surface $\Sigma$, the leading term of the Laurent expansion of any $x_{\gamma'}$, $\gamma'\in \Delta'$ with respect to $x_\gamma,\gamma\in \Delta$ is $\iota_\Delta(\mu_{\Delta,\Delta'}(t_{\gamma'}))$.
\end{theorem}

This monomial mutation is particularly transparent when $\Delta=\Delta_1$ is a star-like triangulation of $\Sigma_n$, i.e., all
diagonals of $\Delta$ start at $1$. In this case, for any diagonal $(ij)\in \Delta'$ with $1<i<j\le n$, the monomial mutation is given by $$\mu_{\Delta,\Delta'}(t_{ij})=t_{i,i+1}t_{1,i+1}^{-1}t_{1j}.$$

\subsection{Noncommutative tagged clusters on surfaces}
For a punctured surface, we can define more such triangulations and groups, which we refer to as {\it tagged}. Following \cite{FST}, we start by selecting a subset $P$ of the set $I_p$ of punctures of $\Sigma$. 
We then create the tagged triangulation $\Delta^{P}$ by replacing all self-folded triangles around points of $P$ in $\Delta$ with tagged bigons, and we tag every remaining point in $P$. 

The {\it tagged} triangle group (see Definition \ref{def:trig}) is generated by $t_\gamma, \gamma\in \Delta^{P}$ subject to the above relations with the following two extra relations.

$\bullet$ $t_{\gamma_1}t_{ \gamma_2}=t_{\overline \gamma_2}t_{\overline \gamma_1}$ for any tagged cyclic bigon $(\gamma_1,\gamma_2)$ in $\Delta$ such that terminal point  $t(\gamma_1)$ lie in $P$ and has valency $2$ {(see the right panel of Figure \ref{Fig:tagbigon}).}

$\bullet$ $t_{\alpha}(t_{\gamma_1}t_{\gamma_2})^{-1}t_{\alpha'}=t_{\overline \alpha'}(t_{\gamma_1}t_{\gamma_2})^{-1}t_{\overline \alpha}$ for any once-punctured cyclic bigon $(\alpha,\alpha')$ enclosing a tagged cyclic bigon $(\gamma_1,\gamma_2)$ in $\Delta$ with $\alpha$ and $\gamma_1$ sharing the same starting point.

 \begin{figure}[ht]
\includegraphics{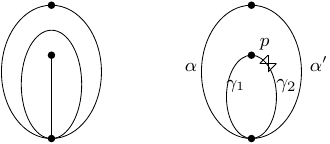}
\caption{Self-folded triangle and tagged cyclic bigon}
\label{Fig:tagbigon}
\end{figure}

In fact, this allows us to extend Theorem \ref{th:monomial mutation surfaces intro} to all tagged and untagged triangulations (Theorem \ref{th:monomial mutation surfaces}, in particular by twisting a cluster $\iota_\Delta$ with our automorphism $\varphi_{I_P}$, we obtain the following result.

\begin{proposition}[Tagging/untagging automorphisms, Proposition \ref{pro:tag/untag}]
Let $\Sigma$ be an oriented punctured surface, $\Delta$ be an ordinary triangulation of $\Sigma$, and $P\subset I_{p,1}(\Sigma)\setminus I_{p,1}(\Delta)$, {where $I_{p,1}(\Delta)$ is the subset of $I_{p,1}(\Sigma)$ consisting of points that are centers of self-folded triangles}. Then the assignments 
$$t_\gamma\mapsto 
\begin{cases} 
t_{\overline\gamma}^{-1}, &\text{if $s(\gamma),t(\gamma)\in P$,}\\
t_{\alpha_4}t^{-1}_{\overline\alpha_3}, &\text{if $s(\gamma)\notin P, t(\gamma)\in P$,}\\
t^{-1}_{\overline\alpha_1}t_{\alpha_2}, &\text{if $t(\gamma)\notin P, s(\gamma)\in P,$}\\
t_{\gamma}, &\text{otherwise},\\
\end{cases}
$$
define an automorphism $\varphi_{P,\Delta}$ of $\TT_\Delta$,
where in the second case, $(\alpha_3,\alpha_4,\gamma)$ is the first cyclic triangle that $\gamma$ passes by rotation counterclockwise along $t(\gamma)$, in the third case, $(\alpha_1,\alpha_2,\overline\gamma)$ is the first cyclic triangle that $\gamma$ passes by rotation counterclockwise along $s(\gamma)$.
\end{proposition}

In particular, if $\Sigma$ is a closed surface, and $\Delta$ is an ordinary triangulation of $\Sigma$, then $\mu_{\Delta,\Delta^{tag}}(t_{\gamma^{tag}})= t_{\overline{\gamma}^{-1}}$ for all $\gamma\in \Delta$.

This will give tagged clusters and tagged Laurent Phenomenon as follows.

For any tagged triangulation $\Delta^{tag}$ of $\Sigma$ let $\Delta$ be the corresponding ordinary triangulation of $\Sigma$
we define an embedding $\iota_{\Delta^{tag}}:\TT _{\Delta^{tag}}\hookrightarrow {\mathcal A}_\Sigma^\times$ 
by $t_{\gamma^{tag}}\mapsto \varphi_P(x_\gamma)$ for all $\gamma\in \Delta$.

We refer to all $\iota_{\Delta^{tag}}$ as the {\it tagged noncommutative clusters}. Following \cite{FST}, together with the ordinary noncommutative clusters $\iota_\Delta$ they complete the cluster structure of ${\mathcal A}_\Sigma$ for any punctured $\Sigma$.  

We prove (Proposition \ref{pr:G-Laurent}, Theorem \ref{thm:laurent}) that noncommutative tagged clusters also give a noncommutative Laurent Phenomenon $\iota_{\Delta^{tag}}:\TT_{\Delta^{tag}}\hookrightarrow {\mathcal A}_\Sigma$ and obtain the corresponding expansion formula for any $x_\gamma$ as sum of elements of $\iota_{\Delta^{tag}}(\TT_{\Delta^{tag}})$. In particular, we write an explicit formula for $x_\gamma^{tag}$ in terms of any (tagged) triangulation $\Delta$ to generalize both classical and quantum cases (\cite[Theorems 4.10, 4.17, 4.20]{MSW}, \cite[Theorem 5.2]{H1}, and \cite{Huang2}).

\subsection{Topological invariance and reductions}
It follows from the discussion of monomial mutations above that  $\TT_\Delta$ is independent of a choice of $\Delta$ (Remark \ref{rm:monomial mutation surfaces}, this e.g., recovers results of \cite{BR}) and therefore we can call it $\TT_\Sigma$. We also show (Remark \ref{rem:canonical triangle group}) that the assignments $\Sigma\mapsto \TT_\Sigma$ define ``almost" a functor from the category of marked surfaces, that is,  $\TT_\Sigma$ is a topological invariant which has a flavor of the fundamental group. However, this invariant is more interesting even for unpunctured disks $\Sigma_n$, for which the fundamental group is trivial (in the forthcoming paper \cite{BHRR} with Eugen Rogozinnikov we explain this in detail). 

By specializing some defining relations of ${\mathcal A}_\Sigma$ to become $q$-commutation relations, we recover quantum cluster algebras of (orientable) surfaces with neither $0$-punctures nor ordinary punctures, as well as an explicit Laurent expansion from \cite{H,H1}.

Furthermore, for any triangulation $\Delta$ of $\Sigma$, we define the {\it sector triangle group} $\UU_\Delta\subset \TT_\Delta$ generated by noncommutative 
sectors $u_{\gamma,\gamma'}:=t_{\overline \gamma}^{-1}t_{\gamma'}$ for any directed 
sector $(\gamma,\gamma')$ in $\Delta$. By definition, we have a commutative diagram 
\begin{equation}
\label{eq:TA-diagram intro}
\xymatrix{
  \kk \UU_\Delta \ar@{^{(}->}[d] \ar@{^{(}->}[r]
                & {\mathcal B}_\Sigma \ar@{^{(}->}[d]  \\
  \kk \TT_\Delta \ar@{^{(}->}[r]
                & {\mathcal A}_\Sigma             }
\end{equation}                
whose vertical arrows are natural inclusions. This diagram, in particular, gives a ``sector" version of the aforementioned Noncommutative Laurent Phenomenon. Similarly to $\TT_\Delta$, the groups $\UU_\Delta$  do not depend on the choice of a triangulation $\Delta$, so there is a canonical group $\UU_\Sigma$ together with almost a functor $\Sigma\mapsto \UU_\Sigma$ refining the aforementioned almost a functor $\Sigma\mapsto \TT_\Sigma$. 
In particular, these groups are also topological invariants of surfaces (see Sections \ref{sec:triangleg1} and \ref{subsec:triangle groups} for details). 
Moreover, the following holds.

\begin{theorem} [Theorem \ref{th:Udelta2}] 
\label{th:Udelta intro}
 Let $\Sigma$ be a marked surface with the Euler characteristic $\chi(\Sigma)$, the set $I=I(\Sigma)\ne \emptyset$ of marked points, the set $I_b\subseteq I$ of marked boundary points, and $h=|\bigsqcup I_{p,\geq 2}|$ special punctures. Assume that $I_{p,0}=\emptyset$. Then for any triangulation $\Delta$ of $\Sigma$ one has:

\noindent (a) If $\Sigma$ has a boundary or special punctures, then $\UU_\Delta$ is:

$\bullet$ A free group in $1$ generator if $\Sigma$ is a disk with $|I_b\sqcup I_{p,1}|+|I_b|=2$, $h=0$.

$\bullet$ Trivial if $\Sigma$ is a disk with $|I_b\sqcup I_{p,1}|=|I_b|=h=1$.

$\bullet$ A free group in $2h-2$ generators if $\Sigma$ is a disk with $|I_b\sqcup I_{p,1}|=|I_b|=1$, $h>1$. 

$\bullet$ A free group in $2h+3|I|-4\chi(\Sigma)-|I_b|$ generators otherwise.

\noindent (b) If $\Sigma$ is a closed surface without special punctures, then $\UU_\Delta$ is:

$\bullet$ Trivial if $\Sigma$ is the sphere with $|I_b\sqcup I_{p,1}|\in \{1,2\}$.

$\bullet$ A free group in $2|I_b\sqcup I_{p,1}|-4$ generators if $\Sigma$ is the sphere with $|I_b\sqcup I_{p,1}|=3$.

$\bullet$ A $1$-relator torsion free group in $3|I_b\sqcup I_{p,1}|-4\chi(\Sigma)+1$ generators otherwise.
\end{theorem}

For instance, if $\Sigma$ is an unpunctured  cylinder with $b_1$ points on one boundary components and $b_2$ on another, then $\UU_\Sigma$ is isomorphic to $\UU_{\Sigma_{b_1+b_2+2}}$.  If $\Sigma$ is 
 a once punctured torus, then  $\UU_\Sigma$ is generated by $a,b,c,d$ subject to $ aba^{-1}b^{-1}=dcd^{-1}c^{-1}$, i.e., it is the fundamental group of a closed genus $2$ surface, and (recall from \cite[Example 3.28]{BR} that in this case $\TT_\Sigma$ is generated by $a,b,c,d,e$ subject to $abcde =cbeda$).


Furthermore, we define the reduced noncommutative surface $\underline {\mathcal A}_\Sigma$ to be the quotient algebra of ${\mathcal A}_\Sigma$ by the relations $x_\gamma=1$ for all boundary curves $\gamma$ (in particular, $\underline {\mathcal A}_\Sigma={\mathcal A}_\Sigma$ for closed surfaces). Similarly, the reduced triangle group $\underline \TT_\Delta$ is the quotient of $\TT_\Delta$ by the relations $t_\gamma=1$ for all boundary edges $\gamma$ in $\Sigma$ and the reduced sector group $\underline \UU_\Delta$ is the image of $\UU_\Delta$
 under the canonical projection $\TT_\Delta\twoheadrightarrow
 \underline \TT_\Delta$.  We show (Proposition \ref{pr:G-Laurent} (b)) that the reduced homomorphisms $\kk \underline \TT_\Delta\to \underline {\mathcal A}_\Sigma$ are injective, therefore, we have a reduced version of the commutative diagram \eqref{eq:TA-diagram intro} verbatim.

Clearly, $\underline {\mathcal B}_\Sigma\subset \underline {\mathcal A}_\Sigma$ and $\underline \UU_\Sigma\subset \underline \TT_\Sigma$. Quite surprisingly, both inclusions become an equality iff $\Sigma$ has neither $0$-punctures nor ordinary punctures (Theorems \ref{th:reduced surfaces} and \ref{th:reduced sector=reduced triangle}). 

We obtain the following surprising 

\begin{theorem} [Theorem \ref{thm:reduced oddgons}] For any $g\ge 0$ the group 
    $\underline \TT_{\Sigma_{2g+3}}=\underline \UU_{\Sigma_{2g+3}}$ is isomorphic to the fundamental group of the closed surface of genus $g$.
\end{theorem}

If $\Sigma$ is an unpunctured cylinder with two marked points, then 
$\underline \TT_\Sigma=\underline \UU_\Sigma$ is generated by $a,b,c$ subject to $cba=abc$, which is not a surface group. More generally, we establish the following

\begin{theorem} [Corollary \ref{cor:Udelta reduced}] 
\label{th:Udelta reduced intro}
In notation of Theorem \ref{th:Udelta intro}, if $\Sigma$ has neither $0$-punctures  nor ordinary punctures, then $\underline \UU_\Sigma=\underline \TT_\Sigma$ is a one-relator torsion free group in $|I_b|+6|I_{p,\geq 2}| +1 -4\chi(\Sigma)$ generators.

\end{theorem}

\subsection{Braid groups and their actions} Returning to the braid group actions, in the context of Theorem \ref{th:braid generation}
we also denote by $Br^+_\Delta$ (see Section \ref{subsec:cluster braids monoids and groups}) the submonoid of the braid group $Br_\Delta$ generated by all $T_\gamma$ 
(see Section \ref{subsec:cluster braids monoids and groups}) and prove that the group $Br_\Delta$  is independent of $\Delta$ (Corollary \ref{cor:brinv}). Unlike $\TT_\Sigma$ or $\UU_\Sigma$, we expect this to be a full invariant with one exception: $Br_{\Sigma_6}\cong Br_{\Sigma_{3,1}}\cong Br_4$ (see Remark \ref{rem:Br_4 exceptional}). The same applies to the image $\underline{Br}_\Delta$ of $Br_\Delta$ in $Aut(\TT_\Delta)$ (we call the latter the {\it (cluster) braid group}\footnote{This agrees with terminology of \cite{HHQ,Q1,KQ}} of $\Delta$) due to the following result.

\begin{corollary} [Corollary \ref{cor:brmuta}]
$\underline{Br}_{\Delta'}=\mu_{\Delta',\Delta} \,\underline{Br}_{\Delta}\,\mu_{\Delta',\Delta}^{-1}$ for any triangulations $\Delta$ and $\Delta'$ of any $\Sigma$, where $\mu_{\Delta',\Delta}:\TT_\Delta\simeq\TT_{\Delta'}$ is the aforementioned monomial mutation.
    
\end{corollary}
Therefore, there are groups $Br_\Sigma$ and $\underline{Br}_\Sigma$ (up to conjugation) isomorphic to all $Br_\Delta$ and $\underline {Br}_\Delta$ for $\Delta\in {\bf TSurf}_\Sigma^t$. In fact, $Br_\Delta$, $\underline{Br}_\Delta$, and $Br_\Sigma$, $\underline{Br}_\Sigma$ can be defined even for non-orientable surfaces; see Section \ref{subsec:non-orientable}. Denote by $\pi_\Delta:Br_\Delta\twoheadrightarrow \underline{Br}_\Delta$ the canonical surjective group homomorphism. 

We show (Proposition \ref{prop:invariant}) that $\UU_\Delta$ is also invariant under each (automatically faithful) $\underline{Br}_\Delta$-action. 
Moreover, this induces a unique (up to conjugation) action of $\underline{Br}_\Sigma$ on both $\TT_\Sigma$ and $\UU_\Sigma$. The former action is faithful by definition and the latter one is faithful when $\Sigma$ is unpunctured (Proposition \ref{prop:faithfulequ}) and conjecture in the punctured case.
Thus, the assignments $\Sigma\mapsto \underline{Br}_\Sigma$ define another topological invariant of marked surfaces.

\begin{example} Let $\Sigma$ be a once-punctured torus. Then $Br_\Sigma$ is a free group of rank $3$ in the $\tau_1,\tau_2,\tau_3$ 
(Corollary \ref{pro:free} (b))
 and we expect that $\pi_\Sigma$ is an isomorphism.
In this case, $\TT_\Sigma$ is generated by $a,b,c,d,e$ subject to $abcde=cbeda$ and
the $Br_\Sigma$-action on $\TT_\Sigma$ (its presentation is in Theorem \ref{th:action}) is given by 
$$\tau_1(x)=\begin{cases}
  b^{-1}c^{-1}eabcde, & \mbox{ if } x=a,\\
  dcbab^{-1}cd, & \mbox{ if } x=d,\\
  x, & \mbox{ otherwise},
\end{cases}\quad
\tau_2(x)=\begin{cases}
  c^{-1}d^{-1}e^{-1}d^{-1}c^{-1}, & \mbox{ if } x=b,\\
  edcbcde, & \mbox{ if } x=e,\\
  x, & \mbox{ otherwise}.
\end{cases}$$
and
$$\tau_3(x)=\begin{cases}
  d^{-1}e^{-1}abc, & \mbox{ if } x=c,\\
  x, & \mbox{ otherwise}.
\end{cases}$$

\end{example}

This example demonstrates that our $Br_\Sigma$ has a flavor of a mapping class group. In the forthcoming work \cite{BHRR}, we will explicitly relate $\TT_\Sigma$, $\UU_\Sigma$, and $Br_\Sigma$ to the corresponding groups on certain ramified double covers of $\Sigma$.

We already established that $\pi_\Delta$ is an isomorphism for  $\Sigma=\Sigma_n$ and the polygon with one special puncture (Theorems \ref{th:Br on Sigman} and \ref{th:Br on Sigman1}) and conjecture it for all $\Sigma$  except for a sphere with $4$ punctures or projective plane with $2$ punctures (Conjecture \ref{conj:faithful}), for which we provide abundant partial evidence (we discuss non-orientable $\Sigma$ in Section \ref{subsec:non-orientable}).


Also, we obtain more surprising braid group homomorphisms based on the following

\begin{theorem} [Proposition \ref{pr:relative braid homomorphism}]
\label{th:boundary gluing}
Let $\Sigma$ be a surface with boundary, and 
let $f:\Sigma\to \Sigma'$ be a surjective morphism of surfaces that only glue boundary arcs of $\Sigma$. Then there is a canonical  homomorphism $f_*:Br_{\Sigma}\to Br_{\Sigma'}$ induced by $f$ (we expect that $f_*$ is always injective, Conjecture \ref{conj:injective relative braid homomorphism}).
\end{theorem}

We obtain a morphism $f:\Sigma_p^2\to \Sigma_{p,2}$, by gluing the two boundary edges of $\Sigma_p^2$ to each other on the boundary component with $2$ point. Since $Br_{\Sigma_p^2}\cong Br_{\hat A_{p+1}}$ and $Br_{\Sigma_{p,2}}\cong Br_{\hat D_{p+2}}$,
we explicitly describe (Corollary \ref{cor:affineatod}) the corresponding homomorphism of braid groups $Br_{\hat A_n}\to Br_{\hat D_{n+1}}$ predicted 
in Theorem \ref{th:boundary gluing} and, of course, expect it to be injective
(alas, we could not find it in the literature).

More generally, for any morphism of marked surfaces $f:\Sigma\to \Sigma'$, we define a subgroup $Br_\Sigma^f$ of $Br_\Sigma$ to be the automorphism group (of any object) of the relative groupoid ${\bf TSurf}_\Sigma^f$ (see Section \ref{sec:taggedtri})
and conjecture (Conjecture \ref{conj:injective relative braid homomorphism}) that the induced homomorphism 
$Br_\Sigma^f\to Br_{\Sigma'}$   
is injective.
In other words, the general noncommutative cluster axiomatic at the beginning of the introduction fully applies to the noncommutative surfaces as well.

\subsection{Noncommutative quotient surfaces} We conclude with the discussion of noncommutative surfaces $\Sigma/\Gamma$ (necessarily with special punctures) where $\Sigma$ is connected and $\Gamma$ is a (necessarily finite) group of automorphisms  $\Sigma$ preserving the set of all marked points (Section \ref{subsec:symmetries and orbifolds}). 
Clearly, $\Gamma$-action on $\Sigma$ lifts to that on  ${\mathcal A}_\Sigma$ by automorphisms via $x_\gamma\mapsto x_{\sigma(\gamma)}$ for all curves $\gamma$ on $\Sigma$ and all $\sigma\in \Gamma$.

It is well-known (\cite[Section 2]{CV}) that $\underline \Sigma:=\Sigma/\Gamma$ is always a surface with an orbifold structure and the canonical projection  $\Sigma\to \underline\Sigma$ is a branched cover (we also write $\Sigma/\sigma$ when $\Gamma$ is the cyclic group generated by $\sigma$). 
One can show (Proposition \ref{prop:symmetries and orbifolds})
that the noncommutative surface  ${\mathcal A}_{\underline \Sigma}$ is isomorphic to some quotient of the coinvariant algebra of $\Gamma$.

 Note that this works also in some non-orientation-preserving situations. For instance, if $\Sigma$ is a sphere with $n$ punctures on the equator of $\sigma$ is the reflection about the equator then $\Sigma/\sigma$ is $\Sigma_n$ and the coinvariant algebra of $\sigma$ in ${\mathcal A}_\Sigma$ is ${\mathcal A}_n$.

In particular, if $\Gamma=\mathbb Z_2$ acting on the isosceles trapezoid $\Sigma_4$, 
\begin{figure}[ht]
\centerline{\includegraphics[width=3.5cm]{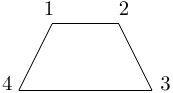}}
\caption{\rm Trapezoid}
\end{figure}
then ${\mathcal A}_{\underline \Sigma}$ is generated by $x_{12}, x_{13},x_{31}, x_{23},x_{32}, x_{34}$ subject to the relations:

$\bullet$ $x_{12}x_{31}^{-1}x_{32}=x_{23}x_{13}^{-1}x_{12}$;

$\bullet$ $x_{32}x_{31}^{-1}x_{34}=x_{34}x_{13}^{-1}x_{23}$;

$\bullet$ 
$x_{13}x_{23}^{-1}x_{13}=x_{12}x_{23}^{-1}x_{34}+x_{23}$

$\bullet$ 
$x_{13}x_{32}^{-1}x_{13}=x_{34}x_{32}^{-1}x_{12}+x_{32}$

We will refer to it as a {\it noncommutative} isosceles trapezoid (its abelianization together with symmetrization $x_{23}=x_{32}$ and $x_{13}=x_{31}$ satisfy the isosceles trapezoid relations). 

Also, if $\Gamma=\mathbb Z_2\times \mathbb Z_2$ acting on the disk $\Sigma_4$, viewed as a rectangle (see the left graph in Figure \ref{Fig:rectangleandtriangle}), 
\begin{figure}[ht]
\centerline{\includegraphics[width=7cm]{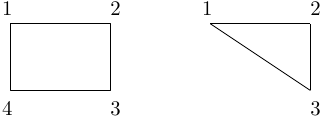}}
\caption{\rm Rectangle and triangle}
\label{Fig:rectangleandtriangle}
\end{figure}
then ${\mathcal A}_{\underline \Sigma}$ is generated by $x_{12}, x_{13}, x_{23}$ subject to the relations:

$\bullet$ $x_{12}x_{13}^{-1}x_{23}=x_{23}x_{13}^{-1}x_{12}$;

$\bullet$ $x_{13}x_{23}^{-1}x_{13}=x_{12}x_{23}^{-1}x_{12}+x_{23}$. 

We will refer to it as a {\it noncommutative} right-angled triangle (see the right graph in Figure \ref{Fig:rectangleandtriangle}), its abelianization is subject to Pythagorean theorem. 

Also, if $\Sigma=\Sigma_n$ is the regular $n$-gon and $\Gamma=I_2(n)$ the dihedral group of symmetries $\Sigma$, then the coinvariant algebra $\underline {\mathcal A}_n$ can be thought of as a noncommutative triangle with one of the remaining angles $\frac{\pi}{n}$ due to the following result.

\begin{theorem} 
\label{th:Chebyshev}
The coinvariant algebra of $I_2(n)$ in ${\mathcal A}_n$, $n\ge 3$, is generated by $a^{\pm 1}, b^{\pm 1}$ subject to the  relation
$p_n(ab^{-1})=0$
where $p_n\in {\mathbb Z}[x]$ is a monic polynomial
given by 
$$p_n(x)=\begin{cases} U_{\frac{n-1}{2}}(\frac{x}{2})-U_{\frac{n-3}{2}}(\frac{x}{2}),
&\text{if $n$ is odd,}\\
2T_{\frac{n}{2}}(\frac{x}{2}), &\text{if $n$ is even},\\
\end{cases}$$
where $T_k$ (resp. $U_k$) is the $k$-th Chebyshev polynomial  of the first (resp. second) kind (the algebraic integer $2\cos(\frac{\pi}{n})$ is a root of $p_n$).

\begin{figure}[ht]

\centerline{\includegraphics{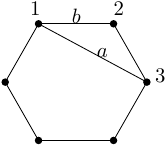}}

\end{figure}
\end{theorem}  

Theorem \ref{th:Chebyshev} is proved in Section \ref{sec:proofofthcheby}.

\begin{remark} In fact, $p_{2n}(x)=\sum \limits_{k=0}^{\left\lfloor {\frac {n}{2}}\right\rfloor }(-1)^k\left({\binom {n-k}{k}}+{\binom {n-k-1}{k-1}}\right)  x^{n-2k}$ and $$p_{2n+3}=\sum_{k=0}^{\left\lfloor {\frac {n}{2}}\right\rfloor }(-1)^k\left(\binom{n+1-k}{k}x^{n+1-2k}-\binom{n-k}{k}x^{n-2k}\right)+((-1)^{\left\lceil {\frac {n}{2}}\right\rceil}-(-1)^{\left\lfloor {\frac {n}{2}}\right\rfloor})/2$$

For instance, $p_3=x-1$, $p_4=x^2-2$, $p_5=x^2-x-1$, $p_6=(x^2-3)x$, $p_7=x^3-x^2-2x+1$, $p_8=x^4-4x^2+2$, $p_9=
(x-1)(x^3-3x-1)$.

\end{remark}

We say that a group $\Gamma$ of automorphisms of $\Sigma$ is {\it admissible} if it preserves a triangulation of $\Sigma$.

Note, however, that in the above examples $\Gamma$ (including those in Theorem \ref{th:Chebyshev}) are not admissible. The following is  immediate from the definition.

\begin{lemma} 
\label{le:invariant triangulations intro}
Let $\Sigma$
be a connected surface and $\gamma$ be an admissible group of automorphisms of $\Sigma$ preserving a given triangulation $\Delta$. Then $\Gamma$ acts by automorphisms of $\TT_\Delta$ and of $Br_\Delta$ in a compatible way.  
\end{lemma}


For instance, $Br_\Delta=Br_{2n-2}$ has an inner automorphism of order $2$ which is induced by the rotation by $\pi$, and $\Delta$ is any centrally symmetric triangulaton of $\Sigma_{2n}$  (Theorem \ref{th:symmetries of n-gon}(a)) and $Br_{3n-2}$ has an inner automorphism of order $3$ which is induced by the rotation by $\frac{2\pi}{3}$ (Theorem \ref{th:symmetries of n-gon}(b)). Also,  $Br_{2n-2}$ has an outer automorphism of order $2$ which is induced by an admissible (i.e., preserving a triangulation) reflection of $\sigma$ (Theorem \ref{th:symmetries of n-gon}(a)). This agrees with the celebrated Dyer-Grossman theorem (\cite{DG}) asserting the only non-trivial outer automorphism $Br_n$, $n\ge 3$  (up to conjugation) is given by $T_i\mapsto T_i^{-1}$.

Similarly, the group $Br_{\Sigma_{n,1}}=Br_{D_n}$ has an inner automorphism $\sigma$ of order $n$
 (Theorem \ref{th:symmetries of n-gon}(c)) in case $n$ is odd. 
Likewise, the group $Br_{\Sigma_{2n,1}}=Br_{D_{2n}}$ has an outer automorphism $\sigma$ of order $2$ induced by an admissible reflection of $\Sigma_{2n,1}$.
 (Theorem \ref{th:symmetries of n-gon}(a)). 

Even though we are unaware of an analog of Dyer-Grossman theorem for $Br_{D_n}$, $n\ge 4$, the above observations would illustrate it as well.

Note that if $\sigma$ is an admissible reflection, then $\Sigma/\sigma$ is an ordinary marked surface $\underline \Sigma$ with boundary. For example, $\underline \Sigma_{2n}=\Sigma_{n+1}$, and $\underline 
 \Sigma_{2n,1}=\Sigma_{n+1}$, and if $\Sigma$ is a sphere with $n$ punctures on the equator and $\sigma$ is the reflection about the equatorial plane, then $\underline \Sigma=\Sigma_n$. Then the coinvariant algebra $({\mathcal A}_\Sigma)_\sigma$ of an admissible reflection $\sigma$ is naturally isomorphic to ${\mathcal A}_{\underline \Sigma}$ (see Remark \ref{rem:coinvariant orbifold} (b)).


We conclude the introduction with a discussion of the behavior of relevant for the quotient map $f_\Gamma:\Sigma\to \underline \Sigma=\Sigma/\Gamma$.

Clearly, by Lemma \ref{le:invariant triangulations intro}, the $\Gamma$-invariant subgroup $Br_\Delta^\Gamma$ of $Br_\Delta$ naturally contains a relative braid group $Br_\Delta^{f_\Gamma}$. We expect that the opposite is also true.

\begin{conjecture} In the assumptions of Lemma \ref{le:invariant triangulations intro},
the relative braid group 
$Br_\Delta^{f_\Gamma}$  is the $\Gamma$-fixed subgroup of $Br_\Delta$, i.e., $Br_\Delta^{f_\Gamma}=Br_\Delta^\Gamma$. 
    
\end{conjecture}

\subsection*{Acknowledgments}
Part of this work was done during visits to Heidelberg University, Max Planck Institute for Mathematics in the Sciences, IHES (AB and VR), and University of Geneva (AB). We thank Anna Wienhard, Eigen Rogozinnikov, Maxim Kontsevich, and Anton Alekseev for fruitful discussions and hospitality. MH would also like to express gratitude to Yu Qiu for insightful discussions. We would like to thank the anonymous referees for their helpful comments and suggestions.


\section{Notation and basic results on noncommutative surfaces}

\subsection{Some notation on surfaces and the category ${\bf Surf}$} 

\label{subsec:Some notation on surfaces}

In this paper, a {\it marked surface} $\Sigma$ 
is an oriented surface (i.e., a smooth not necessarily connected compact $2$-dimensional manifold) with a non-empty finite set $I=I(\Sigma)=I_b\sqcup I_p$ of marked points with 
a subset $I_b=I_b(\Sigma)\subset I$ of marked boundary points, the set $I_p=I_p(\Sigma)$ of internal marked points, which come with the order map $p\mapsto |p|\in \Z_{\ge 0}$. We refer to all such $p$ with $|p|=1$ as {\it ordinary}
punctures, those with $|p|\ge 2$ as {\it special}  punctures and those with $|p|=0$ as {\it zero} punctures. We require that any connected boundary component contains at least one marked point and any closed connected component of $\Sigma$ has at least one ordinary puncture. Sometimes we will use notation $I_{p,k}:=\{p\in I_p:|p|=k\}$
 so that $I_p=\bigsqcup\limits_{k\ge 0} I_{p,k}$.

\begin{remark}\label{rem:orbifold}
{Following the terminology in \cite{FST3}, we assign the points in $I_{p,1}$ as punctures, the points in $I_{p,k}$ ($k\geq 2$) as orbifold points of weight $2$, and the points in $I_{p,0}$ as orbifold points of weight $\frac{1}{2}$. This yields an orbifold with punctures and orbifold points of weights $\frac{1}{2}$ and $2$, which we denote by $\mathcal O(\Sigma)$.}
 \end{remark}

\begin{definition}
 A {\it morphism} $f:\Sigma\to \Sigma'$ of marked surfaces is a smooth map of underlying surfaces with finite fibers such that (we abbreviate $I:=I(\Sigma)$, $I':=I(\Sigma')$): 
 
$\bullet$ 
$f(I_b)\subset I'_b\sqcup I'_{p,1}$, $f(I_{p,1})\subset I'_{p,1}\sqcup I'_{p,0}$,  $f(I_{p,\ge 2})\subset I'_{p,\ge 2}$, $f(I_{p,0})\subset I'_{p,0}$, 
 $f^{-1}(I'_{p,1})\subset  I_{p,1}$.

We abbreviate $I^f:=(f^{-1}(I'_{p,>1})\setminus I_{p,>1})\cup\{p\in I_{p,\ge 2}\mid |p|\neq |f(p)|\}$. 


$\bullet$ For each point $i\in \Sigma\setminus I^f$, there is a neighborhood $\mathcal O_i$ of $i$ in $\Sigma$ such that the restriction of $f$ to $\mathcal O_i$ is injective (if $i\in \partial \Sigma$ is a boundary point, then $\mathcal O_i$ is a “half-neighborhood”).

$\bullet$ For each $p\in I^f$, there is a neighborhood $\mathcal O_p$ of $p$ in $\Sigma$ such that the restriction of $f$ to $\mathcal O_p$ is an $\frac{|f(p)|}{|p|}$-fold cover of $f(\mathcal O_p)$ ramified at $f(p)$.

We denote by ${\bf Surf}$ the category of marked surfaces with the above morphisms.
\end{definition}

 \begin{figure}[ht]

\centerline{\includegraphics[width=12cm]{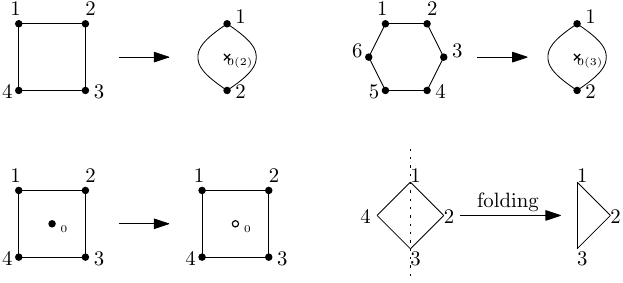}}

\caption{\rm Some examples of morphisms in $\mathbf{Surf}$: {a $2$-to-$1$ ramified covering; a $3$-to-$1$ ramified covering; the identity on the surface, sending the $1$-puncture to the $0$-puncture; and folding along the diagonal $(1,4)$. Here $\circ$ denotes a $0$-puncture, $\bullet$ denotes a $1$-puncture or a boundary marked point, and $\times$ denotes a $\mathbb Z_{\geq 2}$-puncture.}}

\end{figure}


It is well-known that 
for any finite group $\Gamma$ of orientation-preserving  automorphisms 
of an oriented surface $\Sigma$, the quotient space $\Sigma/\Gamma$ is also a surface possibly with an orbifold structure. 

If $\Gamma$ is such a group with the property permuting $I_b(\Sigma)$ and $I_p(\Sigma)$ in an order-preserving way, then $\Sigma/\Gamma$ is an object of ${\bf Surf}$ and the natural quotient map $f_\Gamma:\Sigma\twoheadrightarrow \Sigma/\Gamma$ is a morphism in ${\bf Surf}$. More precisely, 
$I(\Sigma/\Gamma)=f_\Gamma(I(\Sigma))\sqcup \Sigma^{orb,\Gamma}$ where $\Sigma^{orb,\Gamma}\subset \Sigma/\Gamma$ is the set of all orbifold points in $\Sigma/\Gamma\setminus  I(\Sigma/\Gamma)$, so that the order of a point  $p\in \Sigma^{orb,\Gamma}$ is its natural orbifold order, which is the cardinality of the stabilizer of  $p$ in $\Gamma$ ($=|\Gamma|/|\Gamma\cdot p|$). Also, for any special puncture $p$ in $\Sigma$, the order of $f_\Gamma(p)$  is $|p|\cdot |\Gamma\cdot p|$.

In this paper, all curves connect a marked point in $I_b\cup I_{p,1}$ to another marked point in $I_b\cup I_{p,0}\cup I_{p,1}$. They are assumed to be directed, do not intersect the boundary of $\Sigma$ except at their endpoints, {and avoid all marked points in their interiors.} All curves are considered up to isotopy relative to the set of marked points. Denote by $\Gamma(\Sigma)$ the set of all such curves in $\Sigma$.

We denote by $\overline \gamma$ the oppositely directed curve of $\gamma$. Denote by $s(\gamma)$ and $t(\gamma)$ the starting point and ending point, respectively, of $\gamma$. We say that a pair of curves $(\beta,\beta')$ is \emph{composable} if $t(\beta)=s(\beta')$ is not a $0$-puncture.

An \emph{arc} $\gamma$ in $\Sigma$ is a simple curve (up to isotopy with respect to $I$). A \emph{boundary arc} in $\Sigma$ is an arc that lies in the boundary of $\Sigma$. A \emph{special loop} is an arc $\gamma$ that cuts out a monogon around a special puncture. A \emph{pending arc} is an arc incident to a $0$-puncture.

\begin{figure}[ht]

\centerline{\includegraphics[width=3.5cm]{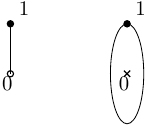}}

\caption{\rm Pending arc and special loop}

\end{figure}

Let $J$ be any non-empty subset of $\bigsqcup I_{p,\geq 2}$. We say that a curve $\gamma$ is $J$-{\it admissible} if after removing any $j\in J\setminus\{s(\gamma),t(\gamma)\}$, the number of self-intersection of $\gamma$ does not change (any curve is $\emptyset$-admissible). Denote by $[\Gamma(\Sigma)]$ the set of $\bigsqcup I_{p,\geq 2}$-admissible curves in $\Sigma$. In particular, we have $\Gamma(\Sigma)=[\Gamma(\Sigma)]$ if and only if $\bigsqcup I_{p,\geq 2}=\emptyset$, i.e., $\Sigma$ contains no special punctures.

For any morphism $f: \Sigma\to \Sigma'$ and $i\in I'(\Sigma)$, we say that $f$ is a \emph{local isomorphism} at $i$ if there exists a local neighborhood $U'$ of $i$ and a dense (not necessarily connected) subset $U$ of $f^{-1}(U')$ such that the restriction of $f$ to $U$ is a bijection $f\mid_{U}:U\simeq U'$.

Below are some examples of local isomorphisms that are, in fact, boundary-gluing maps.

\begin{figure}[ht]
\includegraphics{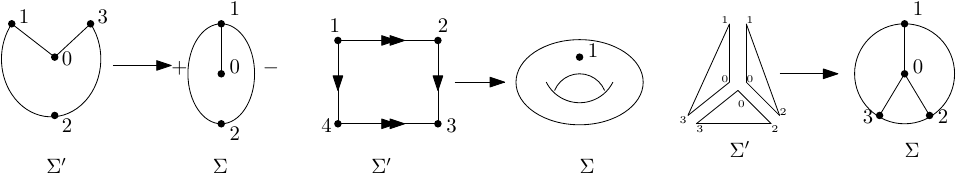}
\caption{Some examples of local isomorphisms}
\label{Fig-e1}
\end{figure}

\begin{figure}[ht]
\includegraphics{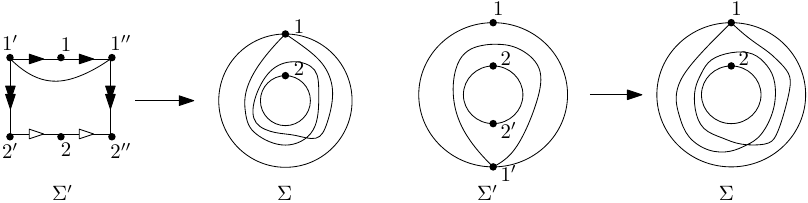}
\caption{Some examples of local isomorphisms}
\end{figure}

Denote by $\Sigma_n$ the disk with $n$ marked points labeled $1,2,\cdots, n$ clockwise and by $(i,j)$ the arc connecting $i$ and $j$. Set $[n]=\{1,2,\cdots,n\}$ and 
$$i^+=\begin{cases}
    i+1, &  \text{ if $i\in [n]\setminus \{n\}$,}\\
    1, &  \text{ if $i=n$,}
\end{cases}
\hspace{5mm}\text{and} \hspace{5mm}
i^-=\begin{cases}
    i-1, & \text{ if $i\in [n]\setminus \{1\}$,}\\
    n, & \text{ if $i=1$}.
\end{cases}$$

\begin{definition} \cite[Definition 3.11]{BR}
We say that a sequence of curves $P=(\gamma_{1}, \dots, \gamma_{n})$ is an $n$-gon in $\Sigma$ if there exists a morphism $f:\Sigma_n\rightarrow \Sigma$ such that $f(i)\in I_b(\Sigma)\cup I_{p,1}(\Sigma)$ and $f(i,i^+) =\gamma_{i}$ for all $i\in [n]$. In particular, we refer to $P$ as a \emph{bigon} if $n=2$, a \emph{triangle} if $n=3$, and a \emph{quadrilateral} if {$n=4$}. 
\end{definition}

\subsection{Noncommutative surfaces and their sector versions}
\label{subsec:From surfaces to their noncommutative versions}

\begin{definition} 
\label{def:ASigma}
For any marked surface $\Sigma\in {\bf Surf}$ define the algebra
$\mathcal A_\Sigma$ over the field $\kk_\Sigma:=\mathbb{Q}(\cos(\frac{\pi}{|p|}),p\in \bigsqcup I_{p,\geq 2})$ to be generated by $x_\gamma$, $\gamma\in \Gamma(\Sigma)$
and $x_\gamma^{-1}$, $\gamma\in [\Gamma(\Sigma)]$
subject to

\begin{enumerate}[$(1)$]
\item (Triangle relations) 
$x_{ \alpha_{ 1}}x^{-1}_{\overline\alpha_{ 2}} x_{ \alpha_{ 3}}=x_{\overline\alpha_{ 3}}x^{-1}_{ \alpha_{ 2}} x_{\overline\alpha_{ 1}}$ for any cyclic triangle $( \alpha_{ 1}, \alpha_{ 2}, \alpha_{ 3})$ in $\Sigma$.

\item (Monogon relations) $x_{\overline\ell}=x_{\ell}$ for each special loop $\ell$.

\item (Zero puncture relations)  $x_{\ell}=x_{\gamma}x_{\overline\gamma}$, if $\ell$ is a loop encloses a pending arc $\gamma$ with $s(\gamma)=s(\ell)$. In particular, $x_{\ell}=x_{\overline\ell}$ for any loop around a $0$-puncture.

\item (Ptolemy relations) $x_{ \alpha'}=x_{\overline\alpha_{ 1}}x^{-1}_{\overline\alpha}x_{ \alpha_{ 3}}+x_{ \alpha_{ 2}}x^{-1}_{ \alpha}x_{\overline\alpha_{ 4}}$
for any cyclic quadrilateral $(\alpha_{1}, \alpha_{2}, \alpha_{3}, \alpha_{4})$
with diagonals $\alpha$ and $\alpha'$ such that $s(\alpha)=s(\alpha_1), s(\alpha')=t(\alpha_1)$. 

\item (Bigon special puncture relations) $x_{\alpha'}=x_{\overline\alpha_{1}}x^{-1}_{\alpha}x_{ \alpha_{ 1}}+2\cos(\frac{\pi}{|p|})x_{\overline\alpha_{ 1}}
x^{-1}_{\alpha}x_{\overline\alpha_2}+x_{\alpha_{ 2}}x^{-1}_{\alpha}x_{\overline\alpha_{2}}$ for any bigon $(\alpha_1,\alpha_2)$ around a special puncture $p$, where $\alpha$ is the loop around $p$ such that $(\alpha_1,\alpha_2,\alpha)$ is a triangle, and $\alpha'$ is the loop around $p$ such that $(\alpha',\alpha_2,\alpha_1)$ is a triangle.

\item (Bigon $0$-puncture relations) $x_{\alpha'}=(x_{\alpha_2}+x_{\overline{\alpha_1}})x_{\overline\alpha}^{-1}$ and $x_{\overline{\alpha'}}=x_{\alpha}^{-1}(x_{\overline{\alpha_2}}+x_{{\alpha_1}})$ for any bigon $(\alpha_1,\alpha_2)$ around a $0$-puncture $p$, where $\alpha$ is the pending arc with $s(\alpha)=s(\alpha_1),t(\alpha)=p$, and $\alpha'$ is the pending arc such that $s(\alpha')=s(\alpha_2),t(\alpha')=p$.
\end{enumerate}
\end{definition}

\begin{remark} This algebra
is a non-commutative version of the generalized cluster algebra defined by Chekhov and Shapiro in \cite[Section 2.1]{ChS}.
\end{remark}  

Following \cite{BR}, we refer to ${\mathcal A}_\Sigma$ as a noncommutative surface.

\begin{theorem} 
\label{th:functoriality nc-surface} For any marked surface $\Sigma$,

$(a)$ 
If $I_{p,0}(\Sigma)=\emptyset$, then the algebra ${\mathcal A}_\Sigma$ is graded by setting $deg~ x_\gamma=1$ {and $deg~ x^{-1}_\gamma=-1$} for any $\gamma$ and admits a unique anti-involution $\overline{\cdot}$ such that $\overline x_\gamma=x_{\overline \gamma}$.

$(b)$ The assignments $\Sigma\mapsto {\mathcal A}_\Sigma$ are almost functorial in the sense that
any morphism $f:\Sigma\to \Sigma'$ in ${\bf Surf}$ induces
a (bar-equivariant) homomorphism of $\mathbb K_{\Sigma'}$-algebras $f_*:\mathbb K_{\Sigma'}\otimes _{\mathbb K_\Sigma}{\mathcal A}_\Sigma^f\to {\mathcal A}_{\Sigma'}$, where ${\mathcal A}_\Sigma^f$ is the subalgebra of  ${\mathcal A}_\Sigma$ generated by $x_\gamma$ for all $\gamma\in \Gamma(\Sigma)$ and $x_\gamma^{-1}$ for all $\gamma$ such that $f(\gamma)$ is  $\bigsqcup I_{p,\geq 2}(\Sigma')$-admissible.
\end{theorem}

We prove Theorem \ref{th:functoriality nc-surface} in Section \ref{subsec:Proof of Theorem functoriality nc-surface}.

\begin{example}
    
\label{ex:A_n} 
In particular, this implies (cf. \cite[Section 3]{BR})

$(a)$ If $\Sigma=\Sigma_3$, the unpunctured disk with three marked points $I=\{i,j,k\}$, then ${\mathcal A}_\Sigma={\mathcal A}_3$ is generated by $x_{ij}$, $i,j\in I$ are distinct subject to the triangle relation 
 $$T_i^{jk}=T_i^{kj}\ ,$$ 
 where $T_i^{jk}=x_{ji}^{-1}x_{jk}x_{ik}^{-1}$ is the {\it noncommutative angle} at the vertex $i$
 (i.e., the noncommutative angles depend only on the vertex)
 \begin{figure}[ht]

\centerline{\includegraphics{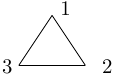}}
\end{figure}

$(b)$ If $\Sigma=\Sigma_4$, the unpunctured disk with $4$ marked points $I=\{1,2,3,4\}$ with diagonals $(13)$ and $(24)$, then ${\mathcal A}_\Sigma={\mathcal A}_4$ 
generated by $x_{ij}$, $i,j\in I$ are distinct subject to the triangle relations 
 $T_i^{jk}=T_i^{kj}$
 for any distinct $i,j,k\in I$ and
$T_1^{24}=T_1^{23}+T_1^{34},~T_2^{13}=T_2^{14}+T_2^{34}$.

\begin{figure}[ht]
\centerline{\includegraphics{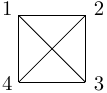}}
\end{figure}

$(c)$ More generally, ${\mathcal A}_n:={\mathcal A}_{\Sigma_n}$ introduced in \cite[Section 3]{BR} is generated by  $x_{ij}^{\pm 1}$ for distinct $i,j\in [1,n]$ subject to 

$\bullet$ (Triangle relations)
$x_{ij}x_{kj}^{-1}x_{ki}=x_{ik}x_{jk}^{-1}x_{ji}$ for distinct $i,j,k\in [1,n]$;

$\bullet$ (Ptolemy relations)
$x_{ik}=x_{ij}x_{lj}^{-1}x_{lk}+x_{ik}x_{jl}^{-1}x_{jk}$ for distinct $i,j,k,l\in [1,n]$ such that $i,j,k,l$ are in clockwise order.

\end{example} 

{We introduce a family of algebras that generalizes the noncommutative algebra $\mathcal C_n$ of type $C_n$. These algebras provide natural candidates for a broader class of noncommutative cluster algebras.}

\begin{definition}
\label{def:generalized Chekhov-Shapiro}
    For any $n\ge 1$ and any $n\times n$ symmetric matrix ${\bf c}$ with entries in some field $\kk$, let ${\mathcal A}_{n,{\bf c}}$ denote the $\kk$-algebra generated by $x_{ij}^{\pm}$ for distinct $i,j=1,\ldots,n$ and $x_i$ for $i=1,\ldots,n$,
subject to the following relations:

 $\bullet$ $x_{ij}^+(x_{kj}^+)^{-1}x_{ki}^+=x_{ik}^-(x_{jk}^-)^{-1}x_{ji}^-$ for any distinct $i,j,k$ in clockwise order,
 
 $\bullet$  $x_{ij}^+x^{-1}_{j}x_{ji}^+=x_{ij}^-x^{-1}_{j}x_{ji}^-$ for any distinct $i,j$;

 $\bullet$ $x_{\ell j}^+=x_{\ell i}^+(x_{ki}^-)^{-1}x_{kj}^-+x_{\ell k}^+(x_{ik}^+)^{-1}x_{ij}^+$ for any distinct $i,j,k,\ell$ in clockwise order,

$\bullet$ $x_j=x_{ji}^- x_i^{-1}x_{ij}^+ + c_{ij}x_{ji}^+ x_i^{-1} x_{ij}^+ +x_{ji}^+ x_i^{-1}x_{ij}^-$ for any distinct $i,j$.
\end{definition}

\begin{remark}
\label{rem:A_{n,d}}  

In particular, we have
${\mathcal C}_n={\mathcal A}_{n,{\bf 0}}$.
More generally, denote by $\underline \Sigma_{n,d}$ the disk with $n$ marked points on the boundary and special puncture of order $d$, i.e., the (orbifold) quotient 
$\Sigma_{nd}/\sigma_d$, where $\sigma_d$ is the rotation of the disk by $\frac{2\pi}{d}$. Then ${\mathcal A}_{\underline \Sigma_{n,d}}={\mathcal A}_{n,{\bf c}}$ where with
$\kk=\mathbb{Q}(\cos(\frac{\pi}{d}))$ and
all $c_{ij}=2\cos(\frac{\pi}{d})$. {In this setting, $\mathcal A_{n,{\bf c}}$ may be viewed as a noncommutative analogue of the coordinate algebra associated with the decorated Teichm\"uller space of the orbifold $\underline{\Sigma}_{n,d}$. The generators $x_{ij}^{\pm}$ play the role of noncommutative $\lambda$-length coordinates associated with arcs connecting the marked points $i$ and $j$.}
\end{remark}

\begin{proposition}\label{pro:total angle}
For any $\Sigma$ and any $i\in {I_b\cup I_{p,1}}$, there exists a unique element $T_i=T_i(\Sigma)\in {\mathcal A}_\Sigma$ satisfying the following conditions: 

$\bullet$ $T_i=T_i^{i^-i^+}$ in case $\Sigma=\Sigma_3$ and $i\in \{1,2,3\}$.
 
$\bullet$ $T_i=\sum\limits_{i'\in f^{-1}(i)} f_*(T_{i'})$  for any morphism $f:\Sigma'\to \Sigma$ that is a local isomorphism at $i$.

we refer to $T_i$ as the \emph{total angle} at $i$.

\end{proposition}

\begin{proof} 
For a boundary marked point $i\in I_b$, suppose that $\gamma^+,\gamma^-$ are two boundary arcs such that $s(\gamma^\pm)=i$. Consider the universal cover or the double ramified cover of the connected component of $\Sigma$ containing $i$, we see that there exists a unique curve $\gamma$ in $\Sigma$ such that $(\gamma^+,\gamma,\overline\gamma^-)$ is a triangle in $\Sigma$. we refer to $(\gamma^+,\gamma,\overline\gamma^-)$ as the canonical triangle with vertex $i$, denoted by $\Delta_i$. Define $T_i:=x^{-1}_{\overline \gamma^+}x_{\gamma}x^{-1}_{\gamma^-}$. The uniqueness follows from the uniqueness of the canonical triangle. 

We now show that $T_i$ satisfies the required properties. First, it is clear that $T_i=T_i^{i^-i^+}$ in case $\Sigma=\Sigma_3$. 

Next, let $f:\Sigma'\to \Sigma$ be a morphism that is a local isomorphism at $i$. For any $i'\in f^{-1}(i)$, we have $i'$ is a boundary marked point in $\Sigma'$. Since $f$ is a local isomorphic at $i$, $f(\cup_{i'\in f^{-1}(i)}\Delta_{i'})$ is a polygon in $\Sigma$. Then $T_i=\sum\limits_{i'\in f^{-1}(i)} f_*(T_{i'})$ follows from the Ptolemy relations.

Now consider a puncture $i\in I_{p,1}$. Choose an arc $\gamma$ with $s(\gamma)=i$, and let $\Sigma_\gamma$ be the surface obtained from $\Sigma$ by cutting along $\gamma$. The canonical morphism $f_\gamma:\Sigma_\gamma\to \Sigma$ is a local isomorphism at $i$, and each $i'\in f^{-1}_\gamma(i)$ is a boundary marked point. Define $T_i=T_i(\gamma):=\sum_{i'\in f^{-1}_\gamma(i)} (f_\gamma)_*(T_{i'})$.

We now prove that $T_i$ is {independent} of the choice of $\gamma$ and satisfies the required conditions.

Assume that $\gamma,\gamma'$ are two arcs with $s(\gamma)=s(\gamma')=i$.

If $\gamma$ and $\gamma'$ do not cross, then we have the following commutative diagram of morphisms:
$$\xymatrix{
& \Sigma_\gamma \ar@<1ex>[dr]^{f_\gamma} \\
\Sigma_{\gamma,\gamma'}\ar@<1ex>[ur]^{f_\gamma'}\ar@<1ex>[dr]_{f_\gamma} &  & \Sigma \\
& \Sigma_{\gamma'}\ar@<1ex>[ur]_{f_{\gamma'}}}$$
Since every $i'\in f^{-1}_\gamma(i)\cup f^{-1}_{\gamma'}(i)$ is a boundary marked point, we have
$$T_i(\gamma)=\sum_{i''\in f^{-1}_{\gamma'}(i')}\sum_{i'\in f^{-1}_{\gamma}(i)} (f_{\gamma'}f_\gamma)_*(T_{i''})=\sum_{i''\in (f_{\gamma'}f_\gamma)^{-1}(i)}(f_{\gamma'}f_\gamma)_*(T_{i''}),$$
$$T_i(\gamma')=\sum_{i''\in f^{-1}_{\gamma}(i')}\sum_{i'\in f^{-1}_{\gamma'}(i)} (f_{\gamma}f_{\gamma'})_*(T_{i''})=\sum_{i''\in (f_{\gamma}f_{\gamma'})^{-1}(i)}(f_{\gamma}f_{\gamma'})_*(T_{i''}).$$
Hence, $T_i(\gamma)=T_i(\gamma')$.

If $\gamma$ and $\gamma'$ cross, then we can resolve their intersection to obtain an arc $\gamma''$ that intersects both 
$\gamma$ and $\gamma'$ fewer times. By induction on the number of crossing points, we have $T_i(\gamma)=T_i(\gamma'')=T_i(\gamma')$. 

Therefore, $T_i$ does not depend on the choice of $\gamma$.

Let $f:\Sigma'\to \Sigma$ be a morphism that is a local isomorphism at $i$.

{\bf Case 1.} Suppose that there exists a puncture $i'\in f^{-1}(i)$. Then $f^{-1}(i)=\{i'\}$. Any arc $\gamma$ with $s(\gamma)=i$ in $\Sigma$ lifts to an arc $\gamma'$ with $s(\gamma')=i'$ in $\Sigma'$.
Thus, 
$$T_i=T_i(\gamma)=\sum_{\hat i\in f^{-1}_\gamma(i)} (f_\gamma)_*(T_{\hat i}),\;\;\; T_{i'}=T_{i'}(\gamma')=\sum_{\hat i'\in f^{-1}_{\gamma'}(i')} (f_{\gamma'})_*(T_{\hat i'}).$$ 
The map $f:\Sigma'\to \Sigma$ induces a morphism $f:\Sigma'_{\gamma'}\to \Sigma_\gamma$, which is local isomorphism at all $\hat i\in f_\gamma^{-1}(i)$, fitting into the following commutative diagram:
$$\centerline{\xymatrix{
& \Sigma'_{\gamma'}\ar[r]^{f}\ar[d]_{f_{\gamma'}} & \Sigma_\gamma \ar[d]^{f_\gamma} \\
& \Sigma'\ar[r]^{f} &   \Sigma }}$$
Since each $\hat i\in f^{-1}_\gamma(i)$ is a boundary marked point, we have 
$$\sum_{\hat i\in f^{-1}_\gamma(i)} T_{\hat i}=\sum_{\hat i'\in f^{-1}_{\gamma'}(i')} f_*(T_{\hat i'}).$$ 
It follows that $T_i=f_*(T_{i'})$.

{\bf Case 2.} Suppose that there are no punctures in $f^{-1}(i)$. Fix $i'\in f^{-1}(i)$ and a boundary arc $\gamma'$ with $s(\gamma')=i'$. Then $\gamma:=f(\gamma')$ is an arc in $\Sigma$ with $s(\gamma)=i$. The map $f:\Sigma'\to \Sigma$ induces a morphism $f:\Sigma'\to \Sigma_\gamma$, which is a local isomorphism at all $\hat i\in f_\gamma^{-1}(i)$, as shown in the following commutative diagram:
$$\xymatrix{
&\Sigma'\ar[r]^{f}\ar[d]_{id} & \Sigma_\gamma \ar[d]^{f_\gamma} \\
&\Sigma'\ar[r]^{f} &   \Sigma }$$
Therefore,
$$T_i=\sum_{\hat i\in f^{-1}_\gamma(i)} T_{\hat i}=\sum_{i'\in f^{-1}(\hat i)} f_*(T_{i'})=\sum\limits_{i'\in f^{-1}(i)} f_*(T_{i'}).$$

The proof is complete.
\end{proof}

{From the proof of Proposition \ref{pro:total angle}, the following is immediate.}

\begin{corollary}\label{cor:laurent}
{ For any $i\in I_b(\Sigma)\cup I_{p,1}(\Sigma)$, $T_i$ is a linear combination of Laurent monomials of the form 
$x_{\overline\gamma_1}^{-1}x_{\gamma_2}\cdots x_{\gamma_{2n}}x_{\overline\gamma_{2n+1}}^{-1},$
where $\gamma_1,\ldots,\gamma_{2n+1}$ is a sequence of composable curves satisfying $s(\gamma_1)=t(\gamma_{2n+1})=i$.}
\end{corollary}

 \begin{example}
In Figure \ref{Fig-e1}, for the once-punctured bigon we have $$T_1(\Sigma)=f(T_1(\Sigma'))+f(T_3(\Sigma'))=((x^+_{21})^{-1}+(x^-_{21})^{-1})x_{20}x^{-1}_{10}.$$

For the once-punctured torus, we have 
 \begin{equation*}
    \begin{array}{rcl}
T_1(\Sigma)
&=&
f(T_1(\Sigma_4))+f(T_3(\Sigma_4))+f(T_2(\Sigma_4))+f(T_4(\Sigma_4))\vspace{2.5pt}\\
&=& f(x_{41}^{-1}x_{42}x_{12}^{-1})+f(x_{12}^{-1}x_{13}x_{23}^{-1})+f(x_{23}^{-1}x_{24}x_{34}^{-1})+f(x_{34}^{-1}x_{31}x_{41}^{-1})\vspace{2.5pt}\\
&=& f(x_{41}^{-1}x_{43}x_{13}^{-1})+f(x_{31}^{-1}x_{32}x_{12}^{-1})+f(x_{12}^{-1}x_{13}x_{23}^{-1})+f(x_{23}^{-1}x_{21}x_{31}^{-1})\vspace{2.5pt}\\
&+& f(x_{43}^{-1}x_{14}x_{31}^{-1})+f(x_{34}^{-1}x_{31}x_{41}^{-1}).
\end{array}
\end{equation*}
 
For the once-punctured triangle, we have 
$$T_0(\Sigma)=x_{10}^{-1}x_{12}x_{02}^{-1}+x_{20}^{-1}x_{23}x_{03}^{-1}+x_{30}^{-1}x_{31}x_{01}^{-1}.$$
\end{example}

\begin{lemma} 
\label{le:angle=ptolemy}

$(a)$  ${\mathcal A}_{\Sigma\sqcup \Sigma'}={\mathcal A}_\Sigma \ast{\mathcal A}_{\Sigma'}$, where $\ast$ denotes the free product of algebras.
 
$(b)$ If $f:\Sigma\sqcup \Sigma\to \Sigma$ is the canonical double cover, then the corresponding morphism ${\mathcal A}_\Sigma \ast{\mathcal A}_\Sigma\to {\mathcal A}_\Sigma$ is the multiplication.

\end{lemma}

\begin{proof} 

For (a), a curve in $\Sigma\sqcup \Sigma'$ is always of the form $\gamma\in \Sigma=\Sigma\sqcup \emptyset$ or $\gamma'\in \Sigma'=\emptyset \sqcup \Sigma'$, as $\Sigma$ and $\Sigma'$ are disconnected in $\Sigma\sqcup \Sigma'$, we thus have ${\mathcal A}_{\Sigma\sqcup \Sigma'}={\mathcal A}_\Sigma \ast{\mathcal A}_{\Sigma'}$.

For (b), a curve in $\Sigma\sqcup \Sigma$ is always of the form $\gamma\in \Sigma=\Sigma\sqcup \emptyset$ or $\gamma'\in \Sigma=\emptyset \sqcup \Sigma$. Then $x_{\gamma,\emptyset}=x_\gamma*1$ and $x_{\emptyset,\gamma'}=1*x_{\gamma'}$ in ${\mathcal A}_\Sigma \ast{\mathcal A}_{\Sigma'}$. Then taking $\Sigma'=\Sigma$ we see that $f_*(x_\gamma*1)=x_\gamma$ and $f_*(1*x_{\gamma'})=x_{\gamma'}$ for any curves $\gamma,\gamma'$ in $\Sigma$. Since  $x*y=(x*1)*(1*y)$, applying $f_*$, we obtain $f_*(x*y)=f_*(x)f_*(y)$ because $f_*$ is an algebra homomorphism. Thus, $f_*$ factors through the multiplication map ${\mathcal A}_\Sigma\otimes {\mathcal A}_\Sigma\to {\mathcal A}_\Sigma$.

The lemma is proved.
\end{proof}

For any pair of (isotopy classes of) curves $(\gamma,\gamma')$ with $t(\gamma)=s(\gamma')$ in $\Sigma$, we abbreviate $y_{\gamma,\gamma'}:=x_{\overline\gamma}^{-1}x_{\gamma'}$ and sometimes refer to it as a \emph{noncommutative sector variable}.

{
We now introduce the sector subalgebra of
$\mathcal A_\Sigma$, which we propose as a natural
candidate for a noncommutative cluster algebra.}

\begin{definition}
{The \emph{sector subalgebra} of $\mathcal A_\Sigma$,
denoted by $\mathcal B_\Sigma$, is the subalgebra
generated by all sector variables $y_{\gamma,\gamma'}$,
where $(\gamma,\gamma')$ ranges over all composable
pairs of curves in $\Sigma$.}
\end{definition}

By definition, ${\mathcal B}_\Sigma$ is subalgebra of the $0$-th graded component of ${\mathcal A}_\Sigma$
if $I_{p,0}(\Sigma)=\emptyset$.

Using almost functoriality (i.e., topological invariance) of ${\mathcal A}_\Sigma$, we obtain the following immediately.

\begin{corollary}
\label{cor:BSigma functorial}
The assignments $\Sigma\mapsto {\mathcal B}_\Sigma$ define almost a functor ${\bf Surf}\to {\bf Alg}_{\mathbb Q}$, {the category of $\mathbb Q$-algebras}, in the same sense as in Theorem \ref{th:functoriality nc-surface}.
    
\end{corollary}

We expect that $\overline  {\mathcal B}_\Sigma\cap  {\mathcal B}_\Sigma=\kk_\Sigma$, moreover, that the subalgebra of ${\mathcal A}_\Sigma$ generated $\overline  {\mathcal B}_\Sigma$ and ${\mathcal B}_\Sigma$ is isomorphic to their free product.

\begin{theorem}
\label{th:sector presentation} 
If $I_{p,0}(\Sigma)=\emptyset$, then the sector subalgebra $\mathcal B_\Sigma$ 
has the following presentation:
\begin{enumerate}[$(1)$]
\item (Triangle relations)
$y_{\alpha_1,\alpha_2}y_{\alpha_3,\alpha_1}y_{\alpha_2,\alpha_3}=1$ for any cyclic triangle $(\alpha_1,\alpha_2,\alpha_3)$.
    
\item (Ptolemy relations) 
$y_{\alpha_1,\alpha'}=
y_{\alpha,\alpha_3}+y_{\alpha_1,\alpha_2}y_{\overline \alpha,\overline\alpha_4}$ for any cyclic quadrilateral $(\alpha_1,\alpha_2,\alpha_3,\alpha_4)$ with diagonals $\alpha$ and $\alpha'$ such that $s(\alpha)=s(\alpha_1)$ and $s(\alpha')=t(\alpha_1)$.

\item (Monogon relations) $y_{\ell,\ell}=1$ for each special loop $\ell$.

\item (Bigon special puncture relations) 
$y_{\overline\alpha',\overline\alpha_1}y_{\overline\alpha,\alpha_1}+2\cos(\frac{\pi}{|p|})y_{\overline\alpha',\overline\alpha_1}y_{\overline\alpha,\overline\alpha_2}+y_{\overline\alpha',\alpha_2}y_{\overline\alpha,\overline\alpha_2}=1$
for any bigon $(\alpha_1,\alpha_2)$ around a special puncture $p$, where $\alpha$ is the loop around $p$ such that $(\alpha_1,\alpha_2,\alpha)$ is a triangle and $\alpha'$ is the loop around $p$ such that $(\alpha',\alpha_2,\alpha_1)$ is a triangle.

\item (Star relations) $y_{\overline \gamma_1,\gamma_2}y_{\overline \gamma_2,\gamma_3}\cdots y_{\overline \gamma_k,\gamma_1}=1$ for any marked point $i$ and a sequence of curves $\gamma_1,\cdots,\gamma_k$ such that $s(\gamma_1)=s(\gamma_2)=\cdots =s(\gamma_k)=i.$  
\end{enumerate}
\end{theorem}

We prove Theorem \ref{th:sector presentation} in Section \ref{subsec:Proof of Theorem sector presentation and reduced surfaces}.

\begin{corollary}  \cite[Theorem 2.14]{BR}
\label{cor:tugged B_n}
  If $\Sigma$ is an unpunctured disk with $I=I_b=[n]=\{1,\ldots,n\}$ then ${\mathcal B}_\Sigma$ is generated by $y_{ij}^k$ for all distinct triples $i,j,k\in I$ subject to  

$\bullet$ (Triangle relations) $y_{ij}^ky_{ji}^k=1, y^k_{ij}y^i_{jk}y^j_{ki}=1$ and $y^l_{ij}y^l_{jk}y^l_{ki}=1$
for  $i,j,k,l\in I$;

$\bullet$ (Ptolemy relations)  
$y_{il}^j=y^k_{ij}y^i_{jl}+y^k_{il}$ for cyclic $(i,j,k,l)$ in $I$.

\end{corollary}

And if $\Sigma=\Sigma_{n,1}$, is a punctured disk with
$n$ boundary points and $I_{p,1}=\{0\}$, then  ${\mathcal A}_\Sigma={\mathcal D}_n$  and Theorem \ref{th:sector presentation} implies the following: 

\begin{corollary}
${\mathcal B}_{\Sigma_{n,1}}$ is generated by $y_{0j}^{i,\pm}=x_{i0}^{-1}x^\pm_{ij}, y_{j0}^{i,\pm}=(x^\pm_{ij})^{-1}x_{i0}$ and $y^0_{ij}=x_{0i}^{-1}x_{0j}$ for distinct $i,j\in [n]$ subject to the relations:

$\bullet$ (Triangle relations) 
$y^{j,\pm}_{0i}y^{0}_{ij}y^{i,\pm}_{j0}=1$
for distinct $i,j\in [n]$;

$\bullet$ (Exchange relations) 
$y_{0k}^{i,-}=y^{j,+}_{0i}y^0_{ik}+y^{j,-}_{0k}$ for all counter-clockwise cyclic $(i,j,k)$ in $[n]$ and $y_{0k}^{i,+}=y^{j,-}_{0i}y^0_{ik}+y^{j,+}_{0k}$ for all clockwise cyclic $(i,j,k)$ in $[n]$.

$\bullet$ (Star relations) ${y_{0j}^{i,\pm}y_{j0}^{i,\pm}=1}$ for all $j\in [n]$ and $y^0_{ij}=y^0_{ji}=1$ for distinct $i,j\in [n]$.
\end{corollary}

Denote by $\underline{\mathcal A}_\Sigma$  the quotient of $\mathcal A_\Sigma$ by the ideal generated by $\{x_{\gamma}-1\mid \gamma \text{ is a boundary arc}\}$. We sometimes refer to 
$\underline{\mathcal A}_\Sigma$ as  \emph{reduced} noncommutative surface.

Likewise, denote by $\underline {\mathcal B}_\Sigma$ the image of ${\mathcal B}_\Sigma$ under the canonical homomorphism ${\mathcal A}_\Sigma\twoheadrightarrow \underline{\mathcal A}_\Sigma$. We sometimes refer to $\underline {\mathcal B}_\Sigma$ as the \emph{reduced sector algebra}. Clearly, $\underline {\mathcal A}_\Sigma={\mathcal A}_\Sigma$ hence $\underline {\mathcal B}_\Sigma={\mathcal B}_\Sigma$ when $\Sigma$ is closed.

\begin{theorem}
\label{th:reduced surfaces} Suppose that $\Sigma$ is not closed with $I_{p,0}(\Sigma)=\emptyset$. Then there exists a projection $\underline \pi$ from
$\underline{\mathcal {A}}_{\Sigma}$ onto $\underline{\mathcal {B}}_{\Sigma}\subset \underline{\mathcal {A}}_{\Sigma}$. Moreover, under this projection, we have 
$\underline{\mathcal {B}}_{\Sigma}=\underline{\mathcal {A}}_{\Sigma}$ if and only if $I_{p,1}(\Sigma)=\emptyset$.
\end{theorem}

We prove Theorem \ref{th:reduced surfaces} in Section \ref{subsec:Proof of Theorem sector presentation and reduced surfaces}.


\subsection{Coinvariants and noncommutative orbifolds}
\label{subsec:Coinvariants and noncommutative orbifolds}

Recall that for any group $\Gamma$ of automorphisms of an algebra ${\mathcal A}$ the coinvariant algebra ${\mathcal A}_\Gamma$ of $\Gamma$ is the quotient of ${\mathcal A}$ by the ideal generated by all $\sigma(a)-a$, $a\in {\mathcal A}$, $\sigma\in \Gamma$.

We simply write ${\mathcal A}_\sigma$ when $\Gamma$ is the cyclic group generated by $\sigma\in Aut({\mathcal A})$.

\begin{proposition}
\label{prop:symmetries and orbifolds} 
Let $\Sigma'\in {\bf Surf}$ and let $f:\Sigma\to \Sigma'$ be the corresponding morphism in ${\bf Surf}$ induced by the quotient $\Sigma'=\Sigma/\Gamma$, where $\Gamma$ is a group of automorphisms of $\Sigma$. Then there is a surjective homomorphism $\kk_{\Sigma'}\otimes_{\kk_\Sigma}({\mathcal A}^f_{\Sigma})_\Gamma\twoheadrightarrow {\mathcal A}_{\Sigma'}$ (in the notation of Theorem \ref{th:functoriality nc-surface}), whose kernel is generated by the following elements:

$\bullet$ $x_{\gamma}-x_{\overline \gamma}$ for all arcs $\gamma$ such that $f(\gamma)$ is a special loop enclosing a special puncture $p$ such that $|p|\neq |f(p)|$;

$\bullet$ $x_{\gamma_k}-{\frac{\sin(k|p|\pi/|f(p)|)}{\sin(|p|\pi/|f(p)|)}}x_\gamma$ for all pairs $(\gamma,\gamma_k)$ such that $f(\gamma)$ is a special loop enclosing a special puncture $p$ such that $|p|\neq |f(p)|$, and $f(\gamma_k)$ is a closed curve with {$k-1$} self-intersection points and enclosing the same special puncture as $f(\gamma)$.
\end{proposition}

We prove Proposition \ref{prop:symmetries and orbifolds} in Section \ref{subsec:Proof of Theorem functoriality nc-surface}.

The following is an immediate consequence of Proposition \ref{prop:symmetries and orbifolds}.

\begin{corollary}\label{cor:Chekhov-Shapiro}  
For any $d\ge 2$, the quotient 
of ${\mathbb Q}(\cos\frac{\pi}{d})\otimes_\mathbb Q{\mathcal A}_{nd}$  
by relations 
$x_{ij}=x_{i+n, j+n}$ for distinct $i,j=1,\ldots,nd$, {(indices taken modulo $nd$)}, and $x_{i,i+kn}={\frac{\sin(k\pi/d)}{\sin(\pi/d)}}x_{i,i+n}={\frac{\sin(k\pi/d)}{\sin(\pi/d)}}x_{i+n,i}$, $i=1,\cdots,n, k=1,\cdots,d-1$
is generated by $x_{ij}^+:=x_{ij}, x_{ij}^-:=x_{i,j+(d-1)n}$ for distinct $i,j=1,\ldots,n$
and $x_i:=x_{i,i+n}=x_{i+n,i}$ for $i=1,\ldots,n$ subject to:

$\bullet$ $x_{ij}^+(x_{kj}^+)^{-1}x_{ki}^+=x_{ik}^-(x_{jk}^-)^{-1}x_{ji}^-$ for any distinct $i,j,k$ in clockwise order.
 
$\bullet$  $x_{ij}^+x^{-1}_{j}x_{ji}^+=x_{ij}^-x^{-1}_{j}x_{ji}^-$ for any distinct $i,j$.

$\bullet$ $x_{\ell j}^+=x_{\ell i}^+(x_{ki}^-)^{-1}x_{kj}^-+x_{\ell k}^+(x_{ik}^+)^{-1}x_{ij}^+$ for any distinct $i,j,k,\ell$ in clockwise order.

$\bullet$ $x_j=x_{ji}^- x_i^{-1}x_{ij}^+ + 2\cos\left({\frac{\pi}{d}}\right)x_{ji}^+ x_i^{-1} x_{ij}^+ +x_{ji}^+ x_i^{-1}x_{ij}^-$ for any distinct $i,j$.
\end{corollary}

\begin{remark} \label{rem:coinvariant orbifold}
$(a)$ Let $\sigma$ be an orientation-preserving automorphism of an oriented surface $\Sigma$. Then ${\mathcal A}_{\Sigma/\sigma}$ is a quotient algebra of $({\mathcal A}_\Sigma)_\sigma$.

$(b)$ Suppose that $\sigma$ is an admissible reflection of $\Sigma$. Then $\mathcal A_{\Sigma/\sigma}\cong {\mathcal A}_{\Sigma_+}\cong {\mathcal A}_{\Sigma_-}$,
where $\Sigma_+$ and $\Sigma_-$ are halves of $\Sigma$ interchanged by $\sigma$ (e.g., $\Sigma_+$ is a fundamental domain of $\Sigma$ and it has a boundary which consists of all curves of $\Sigma$ {fixed by $\sigma$}). This is true because if $\gamma$ is a curve in $\Sigma$ which crosses the reflection line, the image $f(\gamma)=\gamma/\sigma$ is not well-defined in $\Sigma/\sigma$. Because of this we expect that $x_\gamma^{-1}\notin \mathcal A_{\Sigma}^f$. 

In particular,  ${\mathcal A}_{n+1}$ is isomorphic to the coinvariant algebra of the automorphism $\tau$ of ${\mathcal A}_{2n}$ induced by the reflection of $\Sigma_{2n}$ along the diagonal $(1,n+1)$. 

${\mathcal A}_{n+2}$ is the coinvariant algebra of the automorphism $\tau$ of ${\mathcal D}_{2n}$ induced by the reflection of $\Sigma_{2n}$ along the line passing through $1,0$ and $n+1$. 

\end{remark}

\subsection{More automorphisms, tagged curves, and the algebra ${\mathcal B}_n$} 
\label{subsec:symmetries and orbifolds}

Clearly, any automorphism $\sigma$ of $\Sigma$ defines an automorphism of ${\mathcal A}_\Sigma$ via $x_\gamma\mapsto x_{\sigma(\gamma)}$. 

It turns out that there are more automorphisms of ${\mathcal A}_\Sigma$ parametrized by a family ${\bf c}=(c_i,i\in I)$ of invertible elements of ${\mathcal A}_\Sigma$.

\begin{lemma}
\label{le:scaling automorphisms} 
For any family ${\bf c}=(c_i,i\in I)$ as above, the assignments $x_\gamma\mapsto c_{s(\gamma)}x_\gamma c_{t(\gamma)}$ define an endormorphism $\varphi_{\bf c}$  of ${\mathcal A}_\Sigma$.
\end{lemma}

\begin{proof} 
For an odd number $n$ and a sequence of curves $\gamma_1,\gamma_2,\cdots, \gamma_n$ with $s(\gamma_{i+1})=t(\gamma_i)$ for $i=1,\cdots, n-1$, we have 
$$\varphi_{\bf c}(x^{-1}_{\overline\gamma_1}x_{\gamma_2}x^{-1}_{\overline\gamma_3}\cdots x^{-1}_{\overline\gamma_n})=c^{-1}_{s(\gamma_1)}x^{-1}_{\overline\gamma_1}x_{\gamma_2}x^{-1}_{\overline\gamma_3}\cdots x^{-1}_{\overline\gamma_n}c^{-1}_{t(\gamma_n)}.$$ Therefore, $\varphi_{\bf c}$ preserves all relations in Definition \ref{def:ASigma}. Thus $\varphi_{\bf c}$ is an endormorphism.

If  $c_ic'_i=c'_ic_i$ and $\varphi_{\bf c}(c'_i)=c'_i$ for all $i\in I$, then $\varphi_{\bf c}\circ \varphi_{\bf c}'(x_\gamma)=\varphi_{{\bf c}\cdot {\bf c'}}(x_\gamma)$ for any $\gamma$. Thus, $\varphi_{\bf c}\circ \varphi_{\bf c}'=\varphi_{{\bf c}\cdot {\bf c'}}$.

The result follows.
\end{proof}

\begin{remark} 
\label{rem:automorphisms}
We conjecture that the automorphism group of $\mathcal{A}_\Sigma$ is generated by the automorphisms of $\Sigma$ and the invertible scaling endomorphisms described in Lemma~\ref{le:scaling automorphisms}. Furthermore, we expect that the group of units of $\mathcal{A}_\Sigma$ is generated by $\kk^\times$ together with all the variables $x_\gamma$. 
\end{remark}
In particular, when $\Sigma$ is punctured, set $c_i=T_i^{\chi_P(i)}$ for any $i\in I_{p,1}$ and subset $P\subset I_{p,1}$ (here $\chi_P(i)=
\begin{cases}
1 & \text{if $i\in P$}\\
0 &\text{otherwise}
\end{cases}$
is the characteristic function of $P$). Then Lemma \ref{le:scaling automorphisms} implies the following 

\begin{corollary} [Tagging automorphism]
\label{cor:tagged automorphism}    
For any subset $P\subset I_{p,1}(\Sigma)$ the assignments 
$$x_{\gamma}\mapsto T_{s(\gamma)}^{\chi_P(s(\gamma))}x_{\gamma}T_{t(\gamma)}^{\chi_P(t(\gamma))}$$ define an involutive automorphism $\varphi_P$
of the algebra ${\mathcal A}_\Sigma$ such that $\varphi_P(T_i)=T_i$ if $i\notin P$, $\varphi_P(T_i^{\pm 1})=T_i^{\mp 1}$ if $i\in P$. {Moreover, $\varphi_{P\cup P'}=\varphi_P \circ \varphi_{P'}$ if $P\cap P'=\emptyset$.}
\end{corollary} 

\begin{proof}
{By Lemma \ref{le:scaling automorphisms}, these assignments define an endomorphism $\varphi_P$.}
By Corollary \ref{cor:laurent}, $T_i$ is a linear combination of Laurent monomials of the form $x_{\overline{\gamma}_1}^{-1}x_{\gamma_2}\cdots x_{\gamma_{2n}}x_{\overline{\gamma}_{2n+1}}^{-1}$, where $\gamma_1, \dots, \gamma_{2n+1}$ is a sequence of composable curves satisfying $s(\gamma_1) = t(\gamma_{2n+1}) = i$. Consequently, we have
\[
\varphi_P(T_i) = \begin{cases}
    T_i & \text{if } i \notin P, \\
    T_i^{-1} & \text{if } i \in P.
\end{cases}
\]
It follows that $\varphi_P^2(x_\gamma) = x_\gamma$ for any $\gamma$. Moreover, if $P \cap P' = \emptyset$, then
\[
\varphi_P \circ \varphi_{P'}(x_\gamma) = T_{s(\gamma)}^{\chi_{P \cup P'}(s(\gamma))} x_\gamma T_{t(\gamma)}^{\chi_{P \cup P'}(t(\gamma))} = \varphi_{P \cup P'}(x_\gamma).
\]
Thus, $\varphi_P$ is an involutive automorphism, and $\varphi_{P \cup P'} = \varphi_P \circ \varphi_{P'}$ whenever $P \cap P' = \emptyset$.

This completes the proof.
\end{proof}

\begin{remark}\label{rmk:green}
    Corollary \ref{cor:tagged automorphism} can be viewed as a noncommutative version of the cluster transformation defined by a green sequence of mutations. It is closely related to cluster DT-transformations; see \cite{GS, Ko}.
\end{remark}


    
    



\begin{definition}
In view of the above,  we abbreviate $x_{\gamma^P}:=\varphi_P(x_\gamma)$ for any $\gamma$ and any subset $P\subset I_{p,1}$ and sometimes refer to it as a  {\it noncommutative tagged curve} (and  to $\gamma^P$ as the $P$-{\it tagged} curve)
\end{definition}

Clearly, $\gamma^P$ depends only on $\{s(\gamma),t(\gamma)\}\cap P$, e.g.,  $\gamma^\emptyset=\gamma$.

\begin{figure}[ht]

\centerline{\includegraphics {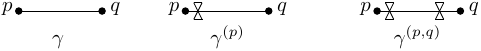}}

\caption{\rm Tagged curves}

\end{figure}

The following is immediate. 

\begin{corollary} 
\label{cor:Tp=1}
In the notation of Corollary \ref{cor:tagged automorphism}, suppose that $|I|\ge 2$ (i.e.,  $\Sigma$ is not a once-punctured closed surface). Then for any subset $P\subset I_{p,1}$, the algebra 
${\mathcal A}_{\Sigma^P}$ is the quotient of ${\mathcal A}_\Sigma$ by the relations $T_p=1$ for all $p\in P$, equivalently $x_\ell=x_\gamma x_{\overline\gamma}$ for all loop encloses an arc $\gamma$ with $s(\ell)=s(\gamma)$ and $t(\gamma)\in P$, where $\Sigma^P$ is obtained from $\Sigma$ by converting the ordinary punctures in $P$ into $0$-punctures.

In particular, when $\Sigma=\Sigma_{n,1}$ is the once-punctured disk and $P$ is the unique puncture, we have ${\mathcal A}_{\Sigma^P}={\mathcal B}_n$.
\end{corollary}


\subsection{Rank $2$ algebras}\label{sec:rank2alg} We recall the definition of Kontsevich's rank $2$ non-commutative cluster algebra, see \cite{Ko,BR}. Given $r_1,r_2\in \mathbb Z_{>0}$ and two variables $x_1,y_1$, for any $k\in \mathbb Z_{>0}$ denote
$r_k=\begin{cases}
    r_1 & \text{if $k$ is odd}\\
    r_2 & \text{if $k$ is even}
\end{cases}$, let
$x_{k+1}=x_ky_kx_k^{-1}$ and $y_{k+1}=(1+y_k^{r_k})x^{-1}_k$ recursively for any $k$.

Denote $z=[x_1,y_1]:=x_1y_1x_1^{-1}y_1^{-1}$. Then $z=[x_k,y_k]$ for all $k$ (see \cite{BR0}). We have 
$$\begin{cases}
x_{k+1}=zy_k \\
y_{k+1}zy_{k-1}=1+y_k^{r_k}\\
y_{k+1}zy_k=y_ky_{k+1}
\end{cases}$$

Let ${\mathcal A}_{r_1,r_2}$ be the subalgebra of $\kk\langle y_1^{\pm 1},y_2^{\pm 1}\rangle$ generated by $y_k$, $k\in {\mathbb Z}$ and $z$.
It follows from \cite{BR0} that ${\mathcal A}_{r_1,r_2}$ is generated by $y_0,y_1,y_2,y_3,z,z^{-1}$. In particular, $y_k$ is a non-commutative Laurent polynomial in $y_1,y_2$ for any $k$.

\section{Triangulations and braid groups}

For two arcs $\gamma, \gamma'\in \Gamma(\Sigma)$, the \emph{crossing number} $n_{\gamma,\gamma'}$ of $\gamma$ and $\gamma'$ is the minimum number of crossings of arcs $\alpha$ and $\alpha'$, where $\alpha$ is isotopic to $\gamma$ and $\alpha'$ is isotopic to $\gamma'$. We call $\gamma$ and $\gamma'$ \emph{compatible} if the crossing number of $\gamma$ and $\gamma'$ is $0$.  

We say that a loop $\gamma$ is {\it around} a point $p\in I_p$ if it only encloses $p$. 

\begin{definition}
A \emph{triangulation} $\Delta$ of $\Sigma$ is a maximal collection of pairwise compatible arcs containing all boundary arcs, such that every $p\in I_{p,0}$ is enclosed by a (necessarily unique) loop in $\Delta$.
\end{definition}

{
In particular, every triangulation contains a loop enclosing each special puncture and a self-folded triangle enclosing each $0$-puncture. As explained in Remark~\ref{rem:orbifold}, a triangulation of $\mathcal O(\Sigma)$ in the sense of \cite{FST} can be obtained from $\Delta$ by removing the loops enclosing $0$-punctures and replacing each loop $\ell$ enclosing a special puncture with a pending arc connecting $s(\ell)$ to the enclosed special puncture. For this reason, we also view our triangulation as a triangulation of $\mathcal O(\Sigma)$.
}

\begin{figure}[ht]
\includegraphics[width=3cm]{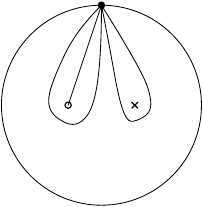}
\caption{\rm An example of triangulation, $\circ$: $0$-puncture, $\times: \mathbb Z_{\geq 2}$ puncture}
\end{figure}

\subsection{Category of triangulated surfaces}

\label{subsec:triangulated surfaces and triangle groups}
\begin{definition}
    
We say that triangulations $\Delta$ and $\Delta'$ are related by a {\it flip} if there are internal arcs $\gamma\in \Delta$ and $\gamma'\in \Delta'$ such that 

$\bullet$ Either $\gamma$ and $\gamma'$ are both loops around a  $0$-puncture $p$ and $\Delta'\setminus \Delta$ is the 
self-folded triangle in $\Delta'$ enclosed by $\gamma'$.

$\bullet$ or $\Delta\setminus \Delta'=\{\gamma,\overline \gamma\}$ and
$\Delta'\setminus \Delta=\{\gamma',\overline \gamma'\}$ otherwise.

In the case of flip, we denote $\Delta'=\mu_\gamma \Delta=\mu_{\overline \gamma}\Delta$.
\end{definition}

\begin{figure}[ht]
\includegraphics{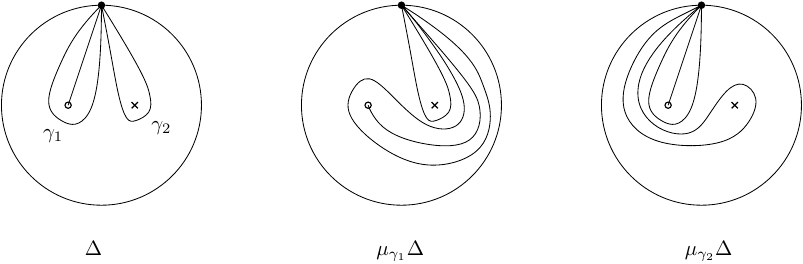}
\caption{\rm Examples of flip}
\end{figure}

The following result was proved by Harer in \cite{Ha} when $\bigsqcup \limits_{k\neq 1}I_{p,k}(\Sigma)=\emptyset$ and by Felikson-Shapiro-Turmarkin in \cite[Theorem 4.2]{FST3} when $\bigsqcup \limits_{k\neq 1}I_{p,k}(\Sigma)\neq\emptyset$.

\begin{theorem} Any triangulations of any $\Sigma\in {\bf Surf}$ are related by a sequence of flip.    
\end{theorem}

{Note that by removing an open neighborhood of each $0$-puncture and special puncture from a marked surface $\Sigma$, we obtain a marked surface $\mathring\Sigma$ with unmarked boundary components (holes) but no $0$-punctures or special punctures. This class of surfaces was discussed in \cite{Hatcher}. The triangulations of $\Sigma$ are in one-to-one correspondence with the generalized ideal triangulations of $\mathring\Sigma$ (see \cite[Definition 9.8]{FST}), and a flip of triangulations of $\Sigma$ corresponds to a generalized flip of generalized triangulations of $\mathring\Sigma$ (see \cite[Section 9, Figure 23]{FST}).}

By \cite[Theorems 1.3 and 2.1]{Ha} and \cite[pp.~190--191]{Hatcher}, we obtain the following analogue of \cite[Theorem 3.10]{FST}.

\begin{theorem}\label{thm:56cycles}
For any marked surface $\Sigma$, the fundamental group of the flip graph is generated by local cycles of lengths $4$, $5$, and $6$. More precisely, these are: cycles of length $4$ corresponding to pairs of commuting flips; cycles of length $5$ obtained by alternately flipping two arcs whose removal creates a pentagonal face; and cycles of length $6$ obtained by alternately flipping two arcs whose removal creates a hexagonal face (see Figure \ref{Fig:56}).
\end{theorem}

\begin{figure}[ht]
\includegraphics[width=12cm]{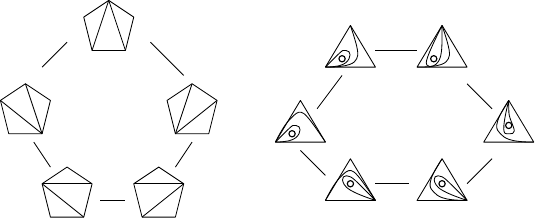}
\caption{Lengths five and six cycles}
\label{Fig:56}
\end{figure}

For any triangulations $\Delta$ and $\Delta'$ of $\Sigma$, we define the distance $dist(\Delta,\Delta')=dist(\Delta',\Delta)$ to be the smallest number of flips from $\Delta$ to $\Delta'$.

Given a morphism $f:\Sigma\to \underline \Sigma$, we say that an arc $\gamma\in \Delta$ is \emph{$f$-admissible} {In particular, if $f$ is a folding map along a line in $\Sigma$, then any curve intersecting that line is not admissible. In contrast, for a ramification covering $f \colon \Sigma \to \Sigma'$, a curve passing through a ramification point is $f$-admissible.} We say that a triangulation $\Delta$ is \emph{$f$-admissible} if every arc in $\Delta$ is $f$-admissible and the collection of arcs in $f(\Delta)$ forms a triangulation of $f(\Sigma)$. In this case, we also denote the resulting triangulation of $f(\Sigma)$ by $f(\Delta)$. {Note that for any arc $\gamma\in \Delta$, we have either $f(\gamma)$ is an arc in $f(\Delta)$ or $f(\gamma)$ is a loop enclosing some special puncture.}

\begin{figure}[ht]
\includegraphics[width=6cm]{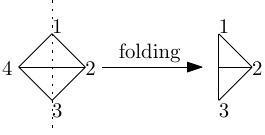}
\caption{}
\label{4to31}
\end{figure}

For example, in Figure \ref{4to31}, the arc $(2,4)$ is not $f$-admissible as $f(2,4)$ is not a curve. In Figure \ref{Fig-morphism}, the triangulations $\Delta$ are $f$-admissible.

\begin{figure}[ht]
\includegraphics[width=15cm]{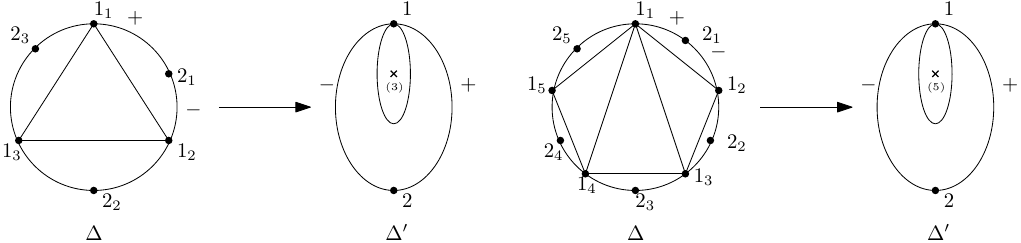}
\caption{}
\label{Fig-morphism}
\end{figure}

For any morphism $f:\Sigma\to \underline\Sigma$ and any curve $\underline \gamma$ in $\underline\Sigma$, the preimage $f^{-1}(\underline\gamma)$ may not consist of curves in $\Sigma$. For example, the loop around $0$ based on $2$ in Figure \ref{Fig-morphism1}.

\begin{figure}[ht]
\includegraphics[width=4cm]{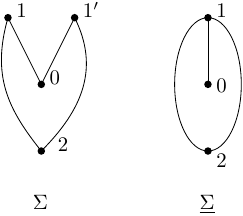}
\caption{}
\label{Fig-morphism1}
\end{figure}

The following is immediate.

\begin{lemma} 
\label{le:f-admissible triangulations}
Let $f:\Sigma\to \underline \Sigma$ be a morphism and $\underline \Delta$ be a triangulation of $f(\Sigma)$. Suppose that the preimage of any $\underline \gamma\in \underline \Delta$ consists of curves in $\Sigma$. Then there exists an $f$-admissible triangulation $\Delta$ of $\Sigma$ such that $f^{-1}(\underline \Delta)\subset \Delta$. Moreover, 

$(a)$ any $f$-admissible triangulation of $\Sigma$ is obtained this way; 

$(b)$ for any non-self-folded arc $\underline\gamma$ in $\underline\Delta$, if $\underline\gamma$ is not a special loop, then any two curves in $f^{-1}(\underline\gamma)$ are not two sides of any triangle in $\Delta$ and $\Delta':=\prod_{\gamma\in f^{-1}(\underline\gamma)}\mu_\gamma(\Delta)$ is $f$-admissible;

$(c)$ for any special loop $\underline\gamma$ around a special puncture  $\underline p$ in $\underline\Delta$, the preimage $f^{-1}(\underline\gamma)$ of $\underline\gamma$ is a $|\underline p|$-gon or an $n$-polygon encloses a special puncture  $p$ such that $|\underline p|=n|p|$;

$(d)$ such $\underline \Delta$ exists for any $f\in {\bf Surf}$.
\end{lemma}

We say that a pair $(\Delta,\underline \Delta)$ is \emph{$f$-compatible} if $\Delta$ is $f$-admissible and $\underline \Delta$ is a triangulation of $\underline \Sigma$ such that $f(\Delta)=f(\Sigma)\cap \underline \Delta$.

For any marked surface $\Sigma$, denote by $\overline \Sigma$ the surface $\Sigma$ with the opposite orientation. For any triangulation $\Delta$ of $\Sigma$, denote by $\overline\Delta$ the same triangulation of $\overline \Sigma$.

{
Recall that for a (commutative) cluster algebra with principal coefficients and initial exchange matrix $B_{t_0}$, the initial extended exchange matrix is given by
$\widetilde{B}_{t_0}=
\begin{pmatrix}
B_{t_0}\\
I_n
\end{pmatrix},
$
where $I_n$ is the $n\times n$ identity matrix.
For any seed $t$ obtained from $t_0$ by a sequence of mutations, the corresponding extended exchange matrix can be written as
$\widetilde{B}_t=
\begin{pmatrix}
B_t\\
C_t^{t_0}
\end{pmatrix}.
$
The matrix $C_t^{t_0}$ is called the \emph{$C$-matrix} of the seed $t$ with respect to the initial seed $t_0$, and its columns are called \emph{$c$-vectors}. In particular, $C_{t_0}^{t_0}=I_n$.
By the sign-coherence theorem for $c$-vectors \cite{GHKK}, each column of a $C$-matrix is sign-coherent, meaning that all its nonzero entries have the same sign.}

{
For the cluster algebra associated with $\mathcal O(\Sigma)$ (see Remark~\ref{rem:orbifold}), each triangulation determines a seed \cite{FST3}. Given two triangulations $\Delta_0$ and $\Delta$, we denote by $C_{\Delta}^{\Delta_0}$ the $C$-matrix of the seed corresponding to $\Delta$ with respect to the initial seed with principal coefficients corresponding to $\Delta_0$.
}

For any triangulations $\Delta_0,\Delta$ of $\Sigma$ and any non-self-folded, non-pending arc $\alpha\in\Delta$, define
\begin{equation*}
\begin{aligned}
\varphi(\Delta_0;\Delta,\mu_\alpha\Delta)
&:=\operatorname{sgn}_\alpha(C_{\Delta}^{\Delta_0}),\\
\phi(\Delta_0;\Delta,\mu_\alpha\Delta)
&:=\operatorname{sgn}_\alpha(C_{\overline{\Delta}}^{\overline{\Delta}_0}),
\end{aligned}
\end{equation*}
where ${\rm sgn}_\alpha(C)$ denotes the sign of the column of $C$ indexed by $\alpha$. Thus, $\varphi$ and $\phi$ record the signs of the corresponding $c$-vectors of the (commutative) seeds associated with $\Delta$ and $\overline{\Delta}$, respectively, relative to the initial seeds associated with $\Delta_0$ and $\overline{\Delta}_0$.

\begin{definition}\label{def:gro}
For any marked surface $\Sigma$ we define the groupoid ${\bf TSurf}_\Sigma$ as the groupoid whose objects are the triangulations of $\Sigma$ and morphisms are generated by $h_{\Delta',\Delta}:\Delta\to \Delta'$ subject to

\begin{enumerate}
 \item[$(1)$]  $h_{\Delta_0,\mu_\alpha\Delta}=h_{\Delta_0,\Delta}h_{\Delta,\mu_\alpha\Delta}^{\varphi(\Delta_0;\Delta,\mu_\alpha\Delta)}$ for any triangulations $\Delta_0,\Delta$ and non-self-folded  and non-pending arc $\alpha\in \Delta$, where for $\varepsilon\in \{\pm\}$
\begin{equation*}
  h_{\Delta',\Delta}^{\varepsilon}=
  \begin{cases}
    h_{\Delta',\Delta}, & \mbox{if $\varepsilon=+$,}  \\
    h_{\Delta,\Delta'}^{-1}, & \mbox{if $\varepsilon=-$}.
  \end{cases}
\end{equation*}
\item [$(2)$] $h_{\mu_\alpha\Delta,\Delta_0}=h_{\mu_\alpha\Delta,\Delta}^{\phi(\Delta_0;\Delta,\mu_\alpha\Delta)}h_{\Delta,\Delta_0}$ 
for any triangulations $\Delta_0,\Delta$ and non-self-folded and non-pending arc $\alpha\in \Delta$ such that $dist(\Delta,\Delta_0)=2$ and $dist(\mu_\alpha\Delta,\Delta_0)=3$.
\item[$(3)$]  (Once punctured bigon relation) \[h_{\Delta,\mu_\alpha\Delta}h_{\mu_\alpha\Delta,\Delta}h_{\Delta,\mu_\beta\Delta}h_{\mu_\beta\Delta,\Delta}=h_{\Delta,\mu_\beta\Delta}h_{\mu_\beta\Delta,\Delta}h_{\Delta,\mu_\alpha\Delta}h_{\mu_\alpha\Delta,\Delta}\]
for any once punctured bigon $(\alpha_1,\alpha_2)$ in $\Delta$ such that $\alpha,\beta\in \Delta$ are the two diagonals connecting the puncture with $\beta\neq \alpha, \overline\alpha$.
\end{enumerate}
\end{definition}

We conjecture $h_{\mu_\alpha\Delta,\Delta_0}=h_{\mu_\alpha\Delta,\Delta}^{\phi(\Delta_0;\Delta,\mu_\alpha\Delta)}h_{\Delta,\Delta_0}$ 
for any triangulations $\Delta_0,\Delta$ and non-self-folded and non-pending arc $\alpha\in \Delta$.

Given a triangulation $\Delta$, we say that $(\gamma,\gamma')$ is directed clockwise in $\Delta$ if there exists $\gamma''\in \Delta$ such that $(\gamma,\gamma',\gamma'')$ or $(\overline \gamma,\gamma',\gamma'')$ is a clockwise cyclic triangle in $\Delta$, 
{see Figure \ref{Fig:clockwisetri}}. 

\begin{figure}[ht]
\includegraphics{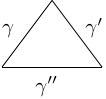}
\caption{An example of clockwise cyclic triangle $(\gamma,\gamma',\gamma'')$}
\label{Fig:clockwisetri}
\end{figure}

\begin{theorem}\label{th:presentation of Tsurf} The category ${\bf TSurf}_\Sigma$ is a groupoid generated by 
$h_{\Delta',\Delta}, dist(\Delta,\Delta')=1$ subject to

$\bullet$ (Diamond/Pentagon/Hexagon relation) For $k\in \{4,5,6\}$ and distinct triangulations $\Delta_i, i=1,\ldots,k$ of $\Sigma$ such that $dist(\Delta_i,\Delta_{i+1\mod k})=1$ for $i=1,\ldots,k$ with $\Delta_2=\mu_{\alpha}(\Delta_1)$ and $\Delta_3=\mu_{\beta}(\Delta_2)$  then
\begin{equation}
\label{eq:defining relations TDelta}
h_{\Delta_1,\Delta_k}h_{\Delta_k,\Delta_{k-1}}=h_{\Delta_1,\Delta_2} h_{\Delta_{2},\Delta_3}\cdots h_{\Delta_{k-2},\Delta_{k-1}}
\end{equation}
whenever $(\alpha,\beta)$ is not directed clockwise in $\Delta_1$. 

$\bullet$ (Horizontal compatibility) For any triangulation $\Delta$, for any non-self-folded and non-pending arcs $\alpha,\beta\in \Delta$ such that $\alpha$ is non-self-folded in $\mu_\beta\Delta$, if $(\beta,\alpha)$ is directed clockwise in $\Delta$, then we have
$$h_{\mu_\alpha\Delta,\Delta}h_{\Delta,\mu_\beta\Delta}h_{\mu_\beta\Delta,\Delta}=h_{\mu_\alpha\Delta,\mu_\beta\mu_\alpha\Delta}h_{\mu_\beta\mu_\alpha\Delta,\mu_\alpha\Delta}h_{\mu_\alpha\Delta,\Delta}.$$

$\bullet$ (Once punctured bigon relation) $h_{\Delta,\mu_\alpha\Delta}h_{\mu_\alpha\Delta,\Delta}h_{\Delta,\mu_\beta\Delta}h_{\mu_\beta\Delta,\Delta}=h_{\Delta,\mu_\beta\Delta}h_{\mu_\beta\Delta,\Delta}h_{\Delta,\mu_\alpha\Delta}h_{\mu_\alpha\Delta,\Delta}$
for any once punctured bigon $(\alpha_1,\alpha_2)$ in $\Delta$ such that $\alpha,\beta\in \Delta$ are the two diagonals connecting the puncture with $\beta\neq \alpha, \overline\alpha$.
\end{theorem}

We prove Theorem \ref{th:presentation of Tsurf} in Section \ref{sec:presentation of Tsurf}.

\begin{remark}
    In case $\Sigma$ is a marked surface without punctures, the opposite groupoid ${\bf TSurf}_\Sigma^{op}$ is isomorphic to the cluster exchange groupoid defined by King-Qiu \cite{KQ}.
\end{remark}

{For any triangulations $\Delta$ and $\Delta'$, assume that $\Delta'=\mu_{\beta_s}\cdots \mu_{\beta_1}(\Delta)$.} By inductively applying the first relation in Definition \ref{def:gro}, we obtain
\[h_{\Delta,\Delta'}=h_{\Delta,\mu_{\beta_1}\Delta}\circ h^{\varepsilon_{2}}_{\mu_{\beta_1}\Delta,\mu_{\beta_2}\mu_{\beta_1}\Delta}\circ \cdots \circ h^{\varepsilon_{s}}_{\mu_{\beta_{s-1}}\cdots \mu_{\beta_1}\Delta,\Delta'},\]
where $\varepsilon_i={\rm sgn}_{\beta_i}(C^{\Delta}_{\mu_{\beta_{i-1}}\cdots \mu_{\beta_1}\Delta})$ for $i\geq 2$.

For any triangulation $\Delta$ of $\Sigma$ we will sometimes use  abbreviation $|\Delta|=\Sigma$.

\begin{definition}\label{def:categoryts}
We define \emph{category of triangulated surfaces} ${\bf TSurf}$ as the category whose objects are triangulations of marked surfaces in ${\bf Surf}$
and the generating morphisms are

$\bullet$ (horizontal) morphism $h_{\Delta', \Delta}:\Delta\to \Delta'$ for any $\Delta$, $\Delta'$ with $|\Delta|=|\Delta'|$,

$\bullet$ (vertical) a unique morphism $v_{f,\Delta, \underline \Delta}:\Delta\to \underline \Delta$ of type $f$, where $f$ is a morphism $|\Delta|\to |\underline \Delta|$ in ${\bf Surf}$ and $(\Delta,\underline \Delta)$ is an $f$-compatible pair

subject to: 

$\bullet$ (Vertical composition relation) $v_{f',\underline \Delta,\underline \Delta'}v_{f,\Delta,\underline \Delta}=v_{f'\circ f,\Delta,\underline \Delta'}$ for any morphisms $f:|\Delta|\to |\underline \Delta|$, $f':|\underline \Delta|\to |\underline \Delta'|$ in ${\bf Surf}$ such that $(\Delta,\underline\Delta)$ is an $f$-compatible pair and $(\underline\Delta,\underline\Delta')$ is an $f'$-compatible pair.

$\bullet$ For any $\Sigma$, the subcategory with objects $\Delta$ for $|\Delta|=\Sigma$ and morphisms generated by $h^{\pm}_{\Delta', \Delta}, |\Delta|=|\Delta'|=\Sigma$ is isomorphic to ${\bf TSurf}_\Sigma$.

\end{definition}

Clearly, the assignments $\Delta\mapsto |\Delta|$ define the forgetful functor ${\bf TSurf}\to {\bf Surf}$ which forgets about triangulation (and all horizontal morphisms collapse to $Id_{|\Delta|}$).

Let $\overline{\cdot}$ be the automorphism of ${\bf Surf}$ which sends $\Sigma$ to $\overline \Sigma$ and identical on morphisms. The following is immediate.

\begin{lemma} 
\label{le:orientation-reversal functor}
$\overline{\cdot}$
extends to an automorphism $\overline{\cdot }$ of ${\bf Tsurf}$ via 
$h_{\Delta',\Delta}\mapsto h_{\overline \Delta,\overline{\Delta'}}^{-1}$, $dist(\Delta,\Delta')=1$,
$v_{f, \Delta,\underline\Delta}\mapsto v_{f,\overline \Delta,\overline{\underline\Delta}}
$.
\end{lemma}


Furthermore, for any morphism $f:\Sigma\to \underline \Sigma$ in ${\bf Surf}$, denote by ${\bf TSurf}_\Sigma^f$ the subcategory of ${\bf TSurf}_\Sigma$ whose objects are $f$-admissible triangulations of $\Sigma$ and morphisms are generated by $h_{\Delta',\Delta}:\Delta\to \Delta'$, where $\Delta,\Delta'$ run over all $f$-admissible triangulations of $\Sigma$.

The following is immediate. 

\begin{lemma} 
\label{le:tsurf^f}
Under the assumptions of Lemma \ref{le:f-admissible triangulations}, fix a triangulation $\underline \Delta_0$ of the closure of the complement $\underline \Sigma\setminus f(\Sigma)$. Then the assignments $\Delta\mapsto f(\Delta)\cup \underline \Delta_0$ define 

(a) a functor $f_*:{\bf TSurf}^f_\Sigma\to {\bf TSurf}_{\underline \Sigma}$, which is covariant if $f$ is orientation preserving and contravariant otherwise.

(b) a natural transformation $\tau$ from the identical inclusion ${\bf TSurf}^f_\Sigma\hookrightarrow \mathcal {\bf TSurf}$ to the composition of $f_*$ with the identical inclusion ${\bf TSurf}_{\underline \Sigma}\hookrightarrow \mathcal {\bf TSurf}_{\underline \Sigma}$.
\end{lemma}

Informally, Lemma \ref{le:tsurf^f} can be thought as a categorical analog of the differential-geometric notion of connection.

We say that a morphism $h_{\Delta',\Delta}$ in ${\bf Tsurf}_\Sigma^f$ is an $f$-flip if either $\Delta,\Delta'$ are related by a flip in ${\bf Tsurf}_\Sigma$ or $f_*(\Delta),f_*(\Delta')$ are related by a flip in ${\bf Tsurf}_{\underline\Sigma}$. For example, let $f:\Sigma\to \underline\Sigma$ be the $4:1$ ramified covering from the octahedron to the bigon with a special puncture of order $4$, then the two morphisms in Figure \ref{Fig:fflip} are $f$-flips. It is immediate that for an $f$-flip $h_{\Delta',\Delta}$ with $\Delta,\Delta'$ are related by a flip, we have either $f_*(\Delta)=f_*(\Delta')$ or $f_*(\Delta),f_*(\Delta')$ are related by a flip in ${\bf Tsurf}_{\underline\Sigma}$.

\begin{figure}[ht]
\includegraphics[width=12cm]{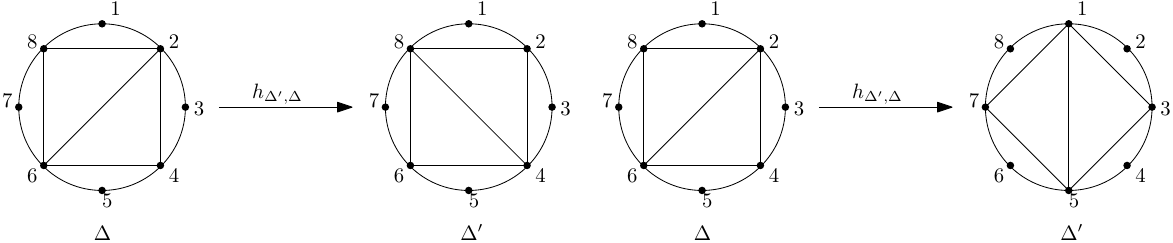}
\caption{}
\label{Fig:fflip}
\end{figure}

In fact, we expect that the groupoid ${\bf Tsurf}_\Sigma^f$ is generated by the $f$-flips.   

Taking into account that ${\bf Tsurf}_\Sigma^f={\bf Tsurf}_\Sigma$ whenever $f$ is an isomorphism, we obtain the following  immediate consequence of Lemma \ref{le:tsurf^f}.

\begin{corollary}

\label{cor:Gamma-action} For any isomorphism $f:\Sigma\simeq \Sigma'$, the assignments $\Delta\mapsto f(\Delta), h_{\Delta',\Delta}\mapsto h^{\varepsilon(f)}_{f(\Delta'),f(\Delta)}$, 
define an isomorphism $f_*$ of groupoids ${\bf Tsurf}_\Sigma\simeq {\bf Tsurf}_{\Sigma‘}$, which is covariant if $f$ is orientation preserving and contravariant otherwise.
\end{corollary}

The following is immediate in view of the behavior of $f$-admissibility under compositions.

\begin{lemma} 
\label{le:f'-action}
In the notation as above, for any morphisms $f:\Sigma\to \underline \Sigma$ and any surjective $f': \Sigma'\to \Sigma$ 
in ${\bf Surf}$ one has 

$(a)$  the restriction of $f'_*:{\bf TSurf}_{\Sigma'}^{f'}\to {\bf TSurf}_\Sigma$ to ${\bf TSurf}_{\Sigma'}^{f\circ f'}$ is a natural full functor 
${\bf Tsurf}_{\Sigma'}^{f\circ f'}\twoheadrightarrow {\bf Tsurf}_\Sigma^f$. 

$(b)$ Any automorphism $\sigma$ of $\Sigma$ defines an isomorphism of groupoids
${\bf Tsurf}_\Sigma^f\to {\bf Tsurf}_ \Sigma^{f\circ \sigma}$.

$(c)$ The group $\Gamma_f:=\{\sigma\in Aut(\Sigma): f\circ \sigma=f\}$ naturally acts on ${\bf Tsurf}_\Sigma^f$ by automorphisms.

\end{lemma}

\begin{remark} Informally, the algebra ${\mathcal A}_\Sigma^f$ in Theorem \ref{th:functoriality nc-surface} is assigned to ${\bf Tsurf}_\Sigma^f$. The homomorphism ${\mathcal A}_\Sigma^f\to {\mathcal A}_{f(\Sigma)}$ from Theorem \ref{th:functoriality nc-surface} was inspired by the functor $f_*$ from Lemma \ref{le:tsurf^f}.
\end{remark}

\subsection{Tagged triangulated surfaces}\label{sec:taggedtri}
{For any subset $P \subset I_{p,1}(\Sigma)$, let $\Delta^P$ denote the corresponding \textit{tagged} triangulation. Specifically, we replace all self-folded triangles around the points of $P$ in $\Delta$ with \textit{tagged} bigons (as shown in the second panel of Figure \ref{Fig:tagged triangulation}) and tag every remaining point in $P$ (where a tagged puncture is depicted as $\bigodot$). Note that this convention differs from that in \cite{FST}, as we tag vertices rather than arcs. See Figure \ref{Fig:tagged triangulation} for an illustrative example.}

\begin{figure}[ht]
\includegraphics{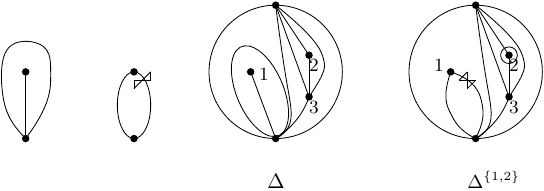}
\caption{Tagged bigon and tagged triangulation}
\label{Fig:tagged triangulation}
\end{figure}

We use the notation $(\gamma,\gamma^{(p)})$ to denote the {\it tagged bigon} corresponding to the self-folded triangle $(\gamma,\gamma,\ell)$ encloses the puncture $p$. 

In particular, if $P$ is empty, then $\Delta=\Delta^P$.

We say that $P$ is the set of tagged vertices of the tagged triangulation $\Delta^P$. Sometimes we denote a tagged triangulation by $\Delta$ and by $tag(\Delta)$ the set $P$ of its tagged vertices and $|\Delta|$ its underlying surface.

Denote by $2^S$ the set of all subsets of $S$. For any set $X$ we define a groupoid $[X]$ whose objects are elements of $X$ with a single arrow between any two elements.

Then for any $\Sigma\in {\bf Surf}$ we abbreviate 
${\bf TSurf}^t_\Sigma:={\bf TSurf}_\Sigma\times [2^{I_{p,1}(\Sigma)}]$, the direct product of categories (see e.g., Appendix \ref{sec:Appendix A}), which is, clearly, a groupoid. That is, the objects of ${\bf TSurf}_\Sigma^t$ are (tagged) triangulations $\Delta$ of $\Sigma$ and the morphisms are generated by $h_{\Delta'^{P'},\Delta^P}:=(h_{P',P},h_{\Delta',\Delta}):\Delta^P\to \Delta'^{P'}, dist(\Delta,\Delta')=1$, where $h_{P',P}:P'\to P$ is the unique morphism in $2^{I_{p,1}(\Sigma)}$.

The following is immediate.

\begin{lemma}
For any $\Sigma\in {\bf Surf}$, the objects of the groupoid ${\bf TSurf}^t_\Sigma$ are tagged triangulations of $\Sigma$ and the morphisms are generated by $h_{\Delta',\Delta}: \Delta\to \Delta'$ for $\Delta,\Delta'\in {\bf TSurf}_\Sigma$ with $dist(\Delta,\Delta')=1$ and $h_{\Delta^{P'},\Delta^P}:\Delta^P\to \Delta^{P'}$ for $\Delta\in {\bf TSurf}_\Sigma, P\subset P'\subset I_{p,1}(\Sigma)$ with $|P'|=|P|+1$, such that
   
$(a)$ the assignments $\Delta\to \Delta$ give a fully faithful functor  $\iota:{\bf TSurf}_\Sigma\to {\bf TSurf}^t_\Sigma$.

$(b)$ For any $P\subset P'\subset I_{p,1}(\Sigma)$ with $|P'|=|P|+1$, we have the following commutative diagram.
    $$\centerline{\xymatrix{
  & \Delta^P \ar[d]_{h_{\Delta'^P,\Delta^P}} \ar[rr]^{h_{\Delta^{P'},\Delta^P}}  &&  \Delta^{P'}\ar[d]
^{h_{\Delta'^{P'},\Delta^{P'}}}           \\
  & \Delta'^P \ar[rr]^{h_{\Delta'^{P'},\Delta'^P}}  && \Delta'^{P'}.}}$$    
\end{lemma}

For any marked surface $\Sigma$ and $P\subset I_{p,1}(\Sigma)$, 
denote by ${{\bf TSurf}_\Sigma^{t}}^{P}$ the full subcategory of ${\bf TSurf}_\Sigma^t$ with objects $\Delta^{P},\Delta\in {\bf TSurf}_\Sigma$. 

The following is immediate as well.

\begin{lemma}
For any subset $P\subset I_{p,1}(\Sigma)$, the assignments
$\Delta^{P'}\mapsto \Delta^{P'\ominus P}$
define an involutive auto-equivalence $F_P$ of ${\bf TSurf}_\Sigma^t$. Moreover, 

(a) the restriction of $F_P$ to the subcategory ${\bf TSurf}_\Sigma$ of ${\bf TSurf}_\Sigma^t$ induces an isomorphism of categories $F_P:{\bf TSurf}_\Sigma\cong {{\bf TSurf}_\Sigma^t}^P$.

(b) For any $P_1,P_2\subset I_{p,1}(\Sigma)$, the set of morphisms $\{h_{\Delta^{P'\ominus P_2},\Delta^{P'\ominus P_1}}\mid \Delta^{P'}\in {\bf TSurf}_\Sigma^t\}$ gives a natural isomorphism from $F_{P_1}$ to $F_{P_2}$.
\end{lemma}

Given a category ${\mathcal C}$ and an object $p$, denote by $Aut_{\mathcal C}(p)$ the group of all automorphisms of $p$ in ${\mathcal C}$. 

\begin{definition}
For any $\Delta\in {\bf TSurf}^t_\Sigma$ we abbreviate $Br_{\Delta}:=Aut_{{\bf TSurf}^t_\Sigma}(\Delta)$ and refer to it as the {\it braid group of $\Delta$}. 
\end{definition}

We clear have $Br_{\Delta}=Aut_{{\bf TSurf}_\Sigma}(\Delta)$ if $\Delta$ is an ordinary triangulation of $\Sigma$.

As ${\bf Tsurf}_\Sigma^t$ is a connected groupoid, the following is immediate.

\begin{corollary}\label{cor:brinv} $Aut_{{\bf TSurf}^t_\Sigma}(\Delta)\cong Aut_{{\bf TSurf}^t_\Sigma}(\Delta')$
for any $\Delta,\Delta'\in {\bf Tsurf}_\Sigma^t$.
    
\end{corollary}

This implies that there is a group $Br_\Sigma$ (up to conjugation) isomorphic to all $Br_\Delta$ for $\Delta\in {\bf TSurf}_\Sigma^t$.

Denote by ${\bf TSurf}^t$ the category whose objects are (tagged)  triangulations $\Delta^P$ of marked surfaces whose morphisms are generated by those of ${\bf TSurf}_{|\Delta|}^t$ as subcategories (we still refer to them as horizontal) together with the vertical morphisms $v_{f,\Delta^P,\Delta'^{f(P)}}: \Delta^P\to \Delta'^{f(P)}$ for any $f:|\Delta|\to |\Delta'|$ in ${\bf Surf}$ such that $f(\Delta)\subset \Delta'$ and $f(P)\subset I'_p(|\Delta'|)$, subject to
$$\centerline{\xymatrix{
  & \Delta \ar[d]_{v_{f,\Delta,\Delta'}} \ar[rr]^{h_{\Delta^{P},\Delta}}  &&  \Delta^{P}\ar[d]
^{v_{f, \Delta^{P},\Delta'^{f(P)}}}           \\
  & \Delta' \ar[rr]^{h_{\Delta'^{f(P)},\Delta'}}  && \Delta'^{f(P)}.}}$$

For any tagged triangulation $\Delta^P\in {\bf TSurf}_\Sigma^t$ and any internal edge $\gamma\in \Delta$, if $\gamma$ is not a side of any self-folded triangle, then denote $\mu_\gamma(\Delta^P)={(\mu_\gamma \Delta)}^{P}$;
if $\gamma$ is a loop of some self-folded triangle in $\Delta$ that surrounds puncture $p\in I_{p,1}(|\Delta|)$, denote $\mu_\gamma(\Delta^P)={(\mu_\gamma \Delta)}^{P\setminus \{p\}}$; if $\gamma$ is a radius of some self-folded triangle encloses puncture $p\in I_{p,1}$ and with loop $\ell$ in $\Delta$, denote $\mu_\gamma(\Delta^P)={(\mu_\ell \Delta)}^{P\cup \{p\}}$. In all cases, we call $\mu_\gamma(\Delta^P)$ the {\it flip} of $\Delta^p$ at $\gamma$.

For any tagged triangulations $\Delta$ and $\Delta'$ of $\Sigma$, we define the distance $dist(\Delta,\Delta')=dist(\Delta',\Delta)$ to be the smallest number of flips from $\Delta$ to $\Delta'$.

More generally, for any morphism $f:\Sigma\to \underline \Sigma$ in ${\bf Surf}$ and any $\Delta\in {\bf TSurf}_\Sigma$ we abbreviate $Br^f_\Delta:=Aut_{{\bf TSurf}^f_\Sigma}(\Delta)$ and refer to it as the {\it relative} braid group of $\Delta$ (with respect to $f$). 

{Denote by ${\bf Grp}'$ the groupoid of groups, where the objects are groups and the morphisms are group isomorphisms.}

\begin{remark}
In view of Lemma \ref{le:groupoid representation}, 

$(a)$ the assignments $\Delta\mapsto Br_\Delta$ define
a functor $Br:{\bf Tsurf}_\Sigma\to  {\bf Grp}'$.

$(b)$ the assignments $\Delta\mapsto Br^f_\Delta$ define
a sub-functor $Br^f:{\bf Tsurf}_\Sigma^f\to  {\bf Grp}'$ of $Br$.
\end{remark}

The following is an immediate consequence of that ${\bf TSurf}_\Sigma$ is a groupoid and of Lemma \ref{le:tsurf^f}.

\begin{lemma}
\label{le:relative homomorphism} Let $\Sigma,\Sigma'\in {\bf Surf}$, $\Delta,\Delta'$ be two triangulations of $\Sigma$ and $f:\Sigma\to \Sigma'$ be a morphism.

$(a)$ The assignments $g\mapsto (f_*(g))^{\varepsilon(f)}$ define a group homomorphism $f_*:Br^f_\Delta\to Br_{f(\Delta)}$ for any $f$-admissible triangulation $\Delta$, {where $\varepsilon(f):=\begin{cases}
+ & \text{if $f$ is orientation-preserving}\\
- & \text{if $f$ is orientation-reversing}.
\end{cases}$}. 
In particular, if $f$ is injective, then $Br^f_\Delta=Br_\Delta$ and $f_*$ is injective.

$(b)$ If $\Delta$ and $\Delta'$ are $f$-admissible, then the restriction of the isomorphism $Br_\Delta\simeq Br_{\Delta'}$ to $Br^f_\Delta$ is an isomorphism $Br^f_\Delta\simeq Br^f_{\Delta'}$.
\end{lemma}

Lemma \ref{le:relative homomorphism} implies that there is a unique subgroup up to conjugation $Br_\Sigma^f$ of $Br_\Sigma$.

\begin{proposition} \label{pr:relative braid homomorphism}
Let $\Sigma$ be a surface with boundary, and
let $f:\Sigma\to \Sigma'$ be a surjective morphism of surfaces that only glue boundary arcs of $\Sigma$. Then there is a canonical homomorphism $f_*:Br_{\Sigma}\to Br_{\Sigma'}$ induced by $f$.
\end{proposition}

\begin{conjecture} 
 \label{conj:injective relative braid homomorphism} In the assumptions of Lemma \ref{le:relative homomorphism}, the homomorphism $f_*$ is injective. In particular, the canonical homomorphism $f_*:Br_{\Sigma}\to Br_{\Sigma'}$ in Proposition \ref{pr:relative braid homomorphism} is injective.
\end{conjecture}

Proposition \ref{prop:BtoA}
below provides some partial evidence of the conjecture.

The following is immediate.

\begin{lemma} 
\label{le:full automorphisms of Delta}
The full automorphism group $Aut_{{\bf TSurf}}(\Delta)$ is isomorphic to the semidirect product $Br_\Delta\rtimes \Gamma_\Delta$, where $\Gamma_\Delta$ is the group of automorphisms of $|\Delta|$ that preserve $\Delta$.
\end{lemma}

It is known that a {torsion free} group $G$ has an inner automorphism of finite order at least $2$, then $G$ has a non-trivial center. The converse for $G=Br_\Delta$ is the following: 

\begin{remark}
\label{rem:inner-outer}
Let $\Sigma\in {\bf Surf}$ be connected and $\Delta\in {\bf TSurf}_\Sigma$. Then, based on abundant evidence (Section \ref{subsec:braid groups of finite types}) we expect that the following are equivalent: 

$\bullet$ $Br_\Delta$ has a non-trivial center; 

$\bullet$  $Br_\Delta$ is of finite Artin type; 

$\bullet$ Either $\Sigma=\Sigma_{n+1}$ or $\Sigma_{n,1}$, $n\ge 2$ or $\Sigma$ is the $n$-gon with a special puncture or a $0$-puncture.
\end{remark}
  
\begin{remark} 
Let $\sigma\in \Gamma_\Delta\setminus \{1\}$ (in notation of Lemma \ref{le:full automorphisms of Delta}). Then we expect  (see Section \ref{subsec:braid groups of finite types}) that $\sigma$ is an inner automorphism of $Br_\Delta$ iff $\Sigma$ is either a disk or a once punctured disk and $\sigma$ a rotation. \end{remark}

For a tagged triangulation $\Delta$ and an internal edge $\gamma\in \Delta$, if $\gamma$ is the radius of some self-folded triangle or a side of some tagged bigon in $\Delta$, denote by $\ell(\gamma)=\ell$ the corresponding loop of the self-folded triangle or the arc enclosing the tagged bigon. Otherwise, set $\ell(\gamma)=\gamma$.

We say that $(\gamma,\gamma')$ is \emph{directed clockwise} in $\Delta$ if $(\ell(\gamma),\ell(\gamma'))$
is directed clockwise in the corresponding ordinary triangulation of $\Delta$. 

As a corollary of Theorem \ref{th:presentation of Tsurf}, we have the following.

\begin{lemma}\label{lem:clock}
 Let $\Delta$ be a tagged triangulation and $\gamma,\gamma'\in \Delta$ be two non-pending internal edges with $\gamma'\neq \gamma,\overline\gamma$. If $(\gamma,\gamma')$ is not directed clockwise in $\Delta$, then  
$$h_{\mu_\gamma\Delta,\Delta}h_{\Delta,\mu_{\gamma'}\Delta}h_{\mu_{\gamma'}\Delta,\Delta}=h_{\mu_\gamma\Delta,\mu_{\gamma'}\mu_\gamma\Delta}h_{\mu_{\gamma'}\mu_\gamma\Delta,\mu_\gamma\Delta}h_{\mu_\gamma\Delta,\Delta}.$$ 
\end{lemma}

\begin{lemma}\label{lem:dumbell}
Let $\Delta$ be a tagged triangulation and $\gamma\in \Delta$ be a non-pending internal edge and let $\Delta'=\mu_\gamma(\Delta)$. For any internal edge  $\gamma'(\neq\gamma,\overline\gamma)\in \Delta'$, we have
$$h_{\Delta',\mu_{\gamma'}\Delta'}h_{\mu_{\gamma'}\Delta',\Delta'}=
\begin{cases}
h_{\Delta',\Delta}
h_{\Delta,\mu_{\gamma'}\Delta}h_{\mu_{\gamma'}\Delta,\Delta}
h^{-1}_{\Delta',\Delta} & \text{if $(\gamma,\gamma')$ is not directed clockwise in $\Delta$}\\
h^{-1}_{\Delta,\Delta'}
h_{\Delta,\mu_{\gamma'}\Delta}h_{\mu_{\gamma'}\Delta,\Delta}
h_{\Delta,\Delta'} &\text{otherwise.}
\end{cases}$$
\end{lemma}

\begin{proof}
If $(\gamma,\gamma')$ is not directed clockwise in $\Delta$, then by Lemma \ref{lem:clock} we have 
$$h_{\Delta',\Delta}h_{\Delta,\mu_{\gamma'}\Delta}h_{\mu_{\gamma'}\Delta,\Delta}=h_{\Delta',\mu_{\gamma'}\Delta'}h_{\mu_{\gamma'}\Delta',\Delta'}h_{\Delta',\Delta}.$$ 
Thus, $$h_{\Delta',\mu_{\gamma'}\Delta'}h_{\mu_{\gamma'}\Delta',\Delta'}=h_{\Delta',\Delta}
h_{\Delta,\mu_{\gamma'}\Delta}h_{\mu_{\gamma'}\Delta,\Delta}
h^{-1}_{\Delta',\Delta}.$$

Otherwise, we have 
$$h_{\Delta,\mu_{\gamma'}\Delta}h_{\mu_{\gamma'}\Delta,\Delta}h_{\Delta,\Delta'}=h_{\Delta,\Delta'}h_{\Delta',\mu_{\gamma'}\Delta'}h_{\mu_{\gamma'}\Delta',\Delta'}.$$ 
Thus, $$h_{\Delta',\mu_{\gamma'}\Delta'}h_{\mu_{\gamma'}\Delta',\Delta'}=h^{-1}_{\Delta,\Delta'}
h_{\Delta,\mu_{\gamma'}\Delta}h_{\mu_{\gamma'}\Delta,\Delta}
h_{\Delta,\Delta'}.$$

The proof is complete.
\end{proof}

For any $\Delta$ and a non-pending internal edge $\gamma\in \Delta$, denote $T_\gamma=T_{\gamma,\Delta}:=h_{\Delta,\mu_\gamma\Delta} h_{\mu_\gamma\Delta,\Delta}\in Br_\Delta.$

\begin{proposition} \label{thm:generate}
Let $\Delta$ be a tagged triangulation and $\gamma\in \Delta$ be a non-pending internal edge and let $\Delta'=\mu_\gamma(\Delta)$. Then for any non-pending internal edge $\gamma'\in \Delta'$, we have 
$$h_{\Delta,\Delta'}T_{\gamma',\Delta'} h^{-1}_{\Delta,\Delta'}=
\begin{cases}
T_{\gamma,\Delta}, &\text{if $\gamma'\notin \Delta$,}\\
T_{\gamma,\Delta} T_{\gamma',\Delta} (T_{\gamma,\Delta})^{-1}, & \text{if $(\gamma,\gamma')$ is not directed clockwise in $\Delta$,}\\
T_{\gamma',\Delta}, &\text{otherwise.}
\end{cases}$$
\end{proposition}

\begin{proof}
In case $\gamma'\notin \Delta$, we have 
$h_{\Delta,\Delta'} T_{\gamma',\Delta'} h^{-1}_{\Delta,\Delta'}=h_{\Delta,\Delta'}h_{\Delta',\Delta}=T_{\gamma,\Delta}.$

In case $\gamma'\in \Delta$, if $(\gamma,\gamma')$ is not directed clockwise in $\Delta$, then by Lemma \ref{lem:dumbell} we have
\begin{equation*}
\begin{array}{rcl}
h_{\Delta,\Delta'} T_{\gamma',\Delta'} h^{-1}_{\Delta,\Delta'}&=&h_{\Delta,\Delta'} T_{\gamma',\Delta'} h^{-1}_{\Delta,\Delta'}=h_{\Delta,\Delta'} (h_{\Delta',\mu_{\gamma'}\Delta'}h_{\mu_{\gamma'}\Delta',\Delta'}) h^{-1}_{\Delta,\Delta'}\vspace{2.5pt}\\
&=&
h_{\Delta,\Delta'}(h_{\Delta',\Delta}h_{\Delta,\mu_{\gamma'}\Delta}h_{\mu_{\gamma'}\Delta,\Delta}h^{-1}_{\Delta',\Delta})h^{-1}_{\Delta,\Delta'}=T_{\gamma,\Delta} T_{\gamma',\Delta} (T_{\gamma,\Delta})^{-1}.
\end{array}
\end{equation*}
Otherwise, by Lemma \ref{lem:dumbell} we have
\begin{equation*}
\begin{array}{rcl}
h_{\Delta,\Delta'} T_{\gamma',\Delta'} h^{-1}_{\Delta,\Delta'}&=&h_{\Delta,\Delta'} T_{\gamma',\Delta'} h^{-1}_{\Delta,\Delta'}=h_{\Delta,\Delta'} (h_{\Delta',\mu_{\gamma'}\Delta'}h_{\mu_{\gamma'}\Delta',\Delta'}) h^{-1}_{\Delta,\Delta'}\vspace{2.5pt}\\
&=&
h_{\Delta,\Delta'}(h_{\Delta,\Delta'}^{-1}
h_{\Delta,\mu_{\gamma'}\Delta}h_{\mu_{\gamma'}\Delta,\Delta}
h_{\Delta,\Delta'})h^{-1}_{\Delta,\Delta'}=T_{\gamma',\Delta}.
\end{array}
\end{equation*}
The proof is complete.
\end{proof}

The following is an analog of \cite[Proposition 2.9]{KQ}.

\begin{theorem} 
\label{th:braid generation} For any $\Delta\in {\bf TSurf}_\Sigma^t$, the group
$Br_\Delta$ is generated by all $T_{\gamma,\Delta}$,  $\gamma$ runs over all non-pending internal edges of $\Delta$. 
\end{theorem}

\begin{proof} 
Denote by $\underline \Gamma$ the directed subgraph of ${\bf TSurf}_\Sigma^t$ so that only the arrows of $\underline \Gamma_\Sigma$ are $h_{\Delta'^{P'},\Delta^P}:\Delta^P\to \Delta'^{P'}$ and $h_{\Delta^P,\Delta'^{P'}}^{-1}$ whenever $dist(\Delta^P,\Delta'^{P'})=1$ or $\Delta=\Delta', P'=P\cup\{p\}$ for some $p\in I_{p,1}(\Sigma)$. Thus, $\underline\Gamma$ generates ${\bf TSurf}_\Sigma^t$.

For any triangulation $\Delta\in {\bf TSurf}_\Sigma^t$, denote by $\widetilde Br_\Delta$ the subgroup of $Br_\Delta$ generated by $T_{\gamma,\Delta}$ for all non-pending internal arcs in $\Delta$. For any non-pending internal edge $\gamma\in \Delta$, by Proposition \ref{thm:generate}, we have $h_{\mu_\gamma\Delta,\Delta}\circ \widetilde Br_{\mu_\gamma\Delta}\circ h_{\mu_\gamma\Delta,\Delta}\subset \widetilde Br_{\Delta}$. By Theorem \ref{th:presentation of Tsurf}, each simple cycle in $\underline\Gamma$ corresponds to a relation in ${\bf TSurf}_\Sigma^t$. Therefore, by Theorem \ref{th:two-cycle generation}, we have $Br_\Delta=\widetilde Br_\Delta$.

The proof is complete.
\end{proof}


\subsection{Presentation of braid groups} \label{subsec:cluster braids monoids and groups}
In this section, we provide presentations of the fundamental groups of ${\bf TSurf}_\Sigma$ and ${\bf TSurf}^t_\Sigma$. 

Recall that for any Coxeter group $W=\langle s_i,i\in I:s_i^2=1,(s_is_j)^{m_{ij}}=1\rangle$ the corresponding braid monoid $Br^+_W$ and the (Artin) braid group $Br_W$ are generated by $T_i$, $i\in I$
subject to:
$$\underbrace{T_iT_jT_i\cdots}_{m_{ij}} =\underbrace{T_jT_iT_j\cdots}_{m_{ij}} \ ,$$
whenever $m_{ij}\ne 0$.

In particular, the (standard) braid group $Br_n=Br_{A_{n-1}}$ on the $n$ strands is generated by $T_1,\ldots,T_{n-1}$ subject to the standard braid relations

$\bullet$ $T_iT_jT_i=T_jT_iT_j$ whenever $|i-j|=1$.

$\bullet$ $T_iT_j=T_jT_i$ otherwise.

$Br_{B_n}=Br_{C_n}$ with the singular node $1$ is generated by
$T_1,\cdots,T_n$ and subject to 

$\bullet$ $T_1T_2T_1T_2=T_2T_1T_2T_1$.

$\bullet$ $T_iT_jT_i=T_jT_iT_j$ whenever $|i-j|=1$ and $i,j\geq 2$.

$\bullet$ $T_iT_j=T_jT_i$ whenever $|i-j|\neq 1$.

$Br_{D_n}$ is generated by
$T_1,\cdots,T_{n}$ and subject to 

$\bullet$ $T_1T_3T_1=T_3T_1T_3$.

$\bullet$ $T_1T_i=T_iT_1$ whenever $i\neq 3$.

$\bullet$ $T_iT_jT_i=T_jT_iT_j$ whenever $|i-j|=1$ and $i,j\geq 2$.

$\bullet$ $T_iT_j=T_jT_i$ whenever $|i-j|\neq 1$ and $i,j\geq 2$.

\medskip

{Let $\Delta$ be an ordinary triangulation $\Delta$.} For any non-pending internal arc $\alpha\in \Delta$, we associate with a word $T_{\alpha}$ with formal inverse $T_\alpha^{-1}$. For a non-self-folded and non-pending internal arc $\alpha\in \Delta$, assume that $\alpha'$ is a non-pending arc in $\mu_\alpha\Delta\setminus \Delta$. For any non-pending internal arc $\beta\in \mu_\alpha \Delta$, denote
$$h^{\mu_\alpha}_{\Delta,\mu_\alpha\Delta}(T_{\beta})=\begin{cases}
    T_{\alpha}, & \text{if $\beta=\alpha'$,}\\
    T_{\alpha} T_\beta T^{-1}_{\alpha}, & \text{{if $(\alpha,\beta)$ is not directed clockwise in $\Delta$},}\\
    T_\beta, & \text{otherwise}.
\end{cases}$$
and
$$h^{\mu_\alpha}_{\Delta,\mu_\alpha\Delta}(T^{-1}_{\beta})=\begin{cases}
    T^{-1}_{\alpha}, & \text{if $\beta=\alpha'$,}\\
    T_{\alpha} T^{-1}_\beta {T^{-1}_{\alpha}}, & \text{{if $(\alpha,\beta)$ is not directed clockwise in $\Delta$},}\\
    T^{-1}_\beta, & \text{otherwise}.
\end{cases}$$

For a sequence of mutations $\mu=\mu_{\alpha_m}\cdots \mu_{\alpha_2}\mu_{\alpha_1}$ and words $T^{\epsilon_1}_{\beta_1}T^{\epsilon_2}_{\beta_2}\cdots T^{\epsilon_n}_{\beta_n}$ with $\epsilon_i\in \{\pm 1\}$ and $\beta_1,\cdots, \beta_n\in \mu \Delta$, denote 
$$h^{\mu_\alpha}_{\Delta,\mu_\alpha\Delta}(T^{\epsilon_1}_{\beta_1}T^{\epsilon_2}_{\beta_2}\cdots T^{\epsilon_n}_{\beta_n})=h^{\mu_\alpha}_{\Delta,\mu_\alpha\Delta}(T^{\epsilon_1}_{\beta_1})h^{\mu_\alpha}_{\Delta,\mu_\alpha\Delta}(T^{\epsilon_2}_{\beta_2})\cdots h^{\mu_\alpha}_{\Delta,\mu_\alpha\Delta}(T^{\epsilon_n}_{\beta_n}),$$
and
$$h^{\mu}_{\Delta,\mu\Delta}(T^{\epsilon_1}_{\beta_1}T^{\epsilon_2}_{\beta_2}\cdots T^{\epsilon_n}_{\beta_n})=h^{\mu_{\alpha_1}}_{\Delta,\mu_{\alpha_1}\Delta}\circ h^{\mu_{\alpha_2}}_{\mu_{\alpha_1}\Delta,\mu_{\alpha_2}\mu_{\alpha_1}\Delta}\circ \cdots \circ h^{\mu_{\alpha_n}}_{\mu_{\alpha_{m-1}}\cdots \mu_{\alpha_1}(\Delta),\mu\Delta}
(T^{\epsilon_1}_{\beta_1}T^{\epsilon_2}_{\beta_2}\cdots T^{\epsilon_n}_{\beta_n}).$$

Recall that for any ordinary triangulation $\Delta$ of $\Sigma$, for any non-pending arc $\gamma$,
denote 
$$\ell(\gamma)=\begin{cases}
    \ell, &\mbox{if $\gamma$ is the radius of some self-folded triangle in $\Delta$ with loop $\ell$,}\\
    \gamma, & \mbox{otherwise}.
\end{cases}$$

For any non-pending arc $\alpha$ in $\Sigma$, define the \emph{weight} of $\alpha$ to be
$$w(\alpha)=
\begin{cases}
1, & \text{if $\alpha$ is not a loop around a $0$-puncture or a special puncture,}\\
|p|, & \text{if $\alpha$ is a special loop around some special puncture $p$,}\\
\frac{1}{2}, & \text{if $\alpha$ is a loop around some $0$-puncture}.
\end{cases}$$

We abbreviate $x^y:=yxy^{-1}$ for any $x,y\in Br_\Delta$.

The following result gives a presentation of all $Br_\Delta$.

\begin{theorem} \label{th:brgroup}
Let $\Sigma$ be a marked surface. 
For any ordinary triangulation $\Delta$ of $\Sigma$, $Br_{\Delta}$ has the following presentation (in the notation of Theorem \ref{th:braid generation}). Generators  $T_\gamma:=T_{\gamma,\Delta}$ are indexed by the non-pending internal edges (up to reversal) of $\Delta$. The {defining} relations are:
\begin{enumerate}[$(R1)$]
    \item $T_{\alpha}T_{\beta}=T_{\beta} T_{\alpha}$
    if either $\ell(\alpha)$ and $\ell(\beta)$ are not two sides of any triangle in $\Delta$; or $\alpha,\beta$ form a self-folded triangle in $\Delta$; or $\alpha,\beta$ are the diagonals of a once-punctured bigon in $\Delta$. 
    
    \item $\begin{cases}
T_{\alpha}T_{\beta}T_\alpha=T_{\beta}T_{\alpha}T_\beta & \text{if $w(\alpha)=w(\beta)=1$}\\
T_{\alpha}T_{\beta}T_\alpha T_\beta=T_{\beta}T_{\alpha}T_\beta T_\alpha & \text{if $w(\alpha)\neq 1=w(\beta)$ or $w(\beta)\neq 1=w(\alpha)$}
\end{cases}$ 
if 
$\ell(\alpha)$ and $\ell(\beta)$ are two sides of exactly one triangle in $\Delta$.

\item $T_\alpha T_\gamma T_\alpha^{-1} T_\beta=T_\beta T_\alpha T_\gamma T_\alpha^{-1}$ 
if $w(\alpha)=1$ and $(\ell(\alpha),\ell(\beta),\ell(\gamma))$ forms a cyclic clockwise triangle in $\Delta$, and any two of these curves are sides of exactly one triangle in $\Delta$; or none of $\alpha, \beta$ and $\gamma$ is a loop, and they form a complete counterclockwise list of the arcs incident to some puncture (see Figure \ref{Fig:R3}). 
\begin{figure}[ht]
\includegraphics{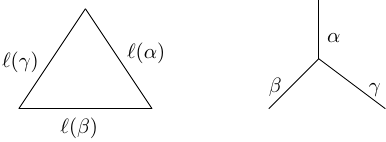}
\caption{Local configuration for relation $R3$}
\label{Fig:R3}
\end{figure}
    \item 
$\begin{cases}
        T_\gamma^{T_\alpha}T_\beta T_\gamma^{T_\alpha}=T_\beta T_\gamma^{T_\alpha} T_\beta,  & \mbox{ if } w(\alpha)=1,\\
        T_\gamma^{T_\alpha}T_\beta =T_\beta T_\gamma^{T_\alpha}, & \mbox{ if } w(\alpha)\neq 1,
    \end{cases}$  
   if there exists $\delta\in \Delta$ such that both $(\ell(\alpha),\beta,\gamma)$ and $(\ell(\delta),\beta,\gamma)$ are cyclic clockwise triangles in $\Delta$ with $\ell(\alpha) \neq \ell(\delta)$ (see Figure \ref{Fig:R4}). 
       \begin{figure}[ht]
\includegraphics[width=8cm]{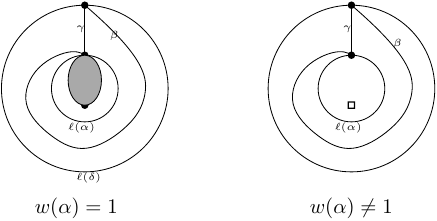}
\caption{Local configuration for relations $R4$ and $R5$, $\Box\in \{\times,\circ\}$}
\label{Fig:R4}
\end{figure}
    \item 
    $T_\delta^{T_\gamma T_\beta}T_\alpha=T_\alpha T_\delta^{T_\gamma T_\beta}$
    if in case $(R4)$, $w(\alpha)=1$ and $\ell(\alpha), \ell(\delta)$ are not two sides of any triangle in $\Delta$ (see the left picture in Figure \ref{Fig:R4}).
    \item 
    $T_\gamma^{T_\alpha T_{\delta}} T_\beta T_\gamma^{T_\alpha T_{\delta}}=T_\beta T_\gamma^{T_\alpha T_{\delta}} T_\beta$ and $T_\gamma^{T_{\delta}T_\alpha} T_\beta T_\gamma^{T_{\delta}T_\alpha}=T_\beta T_\gamma^{T_{\delta}T_\alpha} T_\beta$ if 
    there exists $\zeta\in \Delta$ such that the triples $(\alpha,\beta,\gamma)$, $(\delta,\beta,\gamma)$, and $(\alpha, \delta,\ell(\zeta))$ are three cyclic clockwise triangles in $\Delta$ (see Figure \ref{Fig:R6}).

\begin{figure}[ht]
\includegraphics[width=13cm]{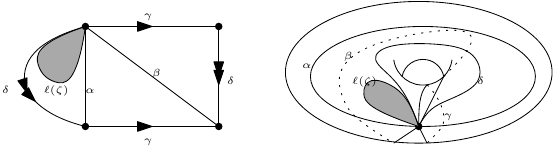}
\caption{Local configuration for relations $R6$ and $R7$}
\label{Fig:R6}
\end{figure}
    
    \item $T_\beta T_{\gamma}^{T_\alpha T_\zeta T_\delta}= T_{\gamma}^{T_\alpha T_\zeta T_\delta} T_\beta$ 
    if in case $(R6)$, $\zeta$ is an internal arc with $w(\zeta)=1$.
   
    \item 
    $T_{\gamma}^{T_\beta^{-1}} T_\delta^{T_\alpha}=T_\delta^{T_\alpha} T_{\gamma}^{T_\beta^{-1}}$ if either none of $\alpha,\beta,\gamma$ and $\delta$ is a loop, and they form a complete counterclockwise list of the arcs incident to some puncture; or $\ell(\alpha)$ and $\ell(\gamma)$ form a once-punctured bigon with diagonals $\beta$ and $\delta$; or $w(\alpha)=w(\beta)=w(\delta)=1$ and $(\ell(\beta),\ell(\delta))$ form a once-punctured bigon with diagonals $\alpha$ and $\gamma$ (see Figure \ref{Fig:R8}).    
\begin{figure}[ht]
\includegraphics{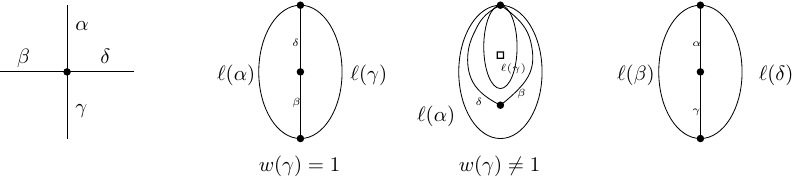}
\caption{Local configuration for relation $R8$, $\Box\in \{\times,\circ\}$}
\label{Fig:R8}
\end{figure}




   \item (Ordinary puncture relations) For any ordinary puncture $p$, let $\alpha_1,\cdots,\alpha_n$ be a complete clockwise list of arcs in $\Delta$ incident to $p$.

   $\bullet$ If there is no loop in $\{\alpha_1,\cdots,\alpha_n\}$, then $Cyl(T_{\alpha_1},\cdots, T_{\alpha_n})$. (We abbreviate the relation $x_1x_2\cdots x_{n}x_1x_2\cdots x_{n-2}=x_2x_3\cdots x_nx_1x_2\cdots x_{n-1}$ by $Cyl(x_1,x_2,\cdots, x_n)$).

   $\bullet$ Otherwise, if there exists a sequence of mutations $\mu$ at some loops in $\{\alpha_1,\cdots,\alpha_n\}$ such that the number of loops incident to $p$ decreases after each step and no loop remain incident to $p$ in $\mu(\Delta)$, assume that $\alpha_1',\cdots,\alpha'_m$ is the complete clockwise list of arcs incident to $p$ in $\mu(\Delta)$, then $Cyl(h^{\mu}_{\Delta,\mu\Delta}(T_{\alpha'_1}),\cdots, h^{\mu}_{\Delta,\mu\Delta}(T_{\alpha'_m}))$.
\end{enumerate}
\end{theorem}

We prove Theorem \ref{th:brgroup} in Section \ref{sec:proofthm:fundegroup}.

\begin{remark}
{$(a)$ Note that, with the exception of relations $(R1)$ and $(R2)$, all remaining relations depend on the orientation of the surface.}

{$(b)$ The second local configuration of Relation $(R3)$ and the first local configuration of Relation $(R8)$ are both special cases of Relation $(R9)$; in each of these cases, the respective relations are equivalent.}

$(c)$ We will see in Lemma \ref{lem:R9unique} that it suffices to choose a single mutation sequence $\mu:\Delta\to \mu\Delta$ to define the relation $R9$ for each puncture $p$ in $Br_\Delta$.

$(d)$ For any tagged triangulation $\Delta$, $Br_{\Delta}$ has the same presentation as that of the corresponding ordinary triangulation.
\end{remark}

\begin{remark}\label{rem:equi}
Let $Q_\Delta$ denote the quiver associated with the triangulation $\Delta$. More precisely, the vertices of $Q_\Delta$ correspond to the non-pending internal arcs in $\Delta$, considered up to reversal. The number of arrows from a vertex $\alpha$ to a vertex $\beta$ is defined as the number of clockwise cyclic triangles of the form $(\ell(\alpha),\ell(\beta),\gamma)$ in $\Delta$, for some $\gamma\in \Delta$. Then

\begin{enumerate}[$(1)$]
    \item the condition for relation $R1$ is equivalent to that there are no arrows between $\alpha$ and $\beta$ in $Q_\Delta$,
    \item  the condition for relation $R2$ is equivalent to that there is exactly one arrow between $\alpha$ and $\beta$ in $Q_\Delta$,
    \item the condition for relation $R3$ is equivalent to that there is a $3$-cycle between $\alpha,\beta$ and $\gamma$ with no double arrows between them in $Q_\Delta$ (see the first quiver in Figure \ref{Fig:subquiver1}),
    \item the condition for relation $R4$ is equivalent to that there is a $3$-cycle between $\alpha,\beta$ and $\gamma$ with a double arrow from $\beta$ to $\gamma$ in $Q_\Delta$ (see the second quiver in Figure \ref{Fig:subquiver1}),
    \item the condition for relation $R5$ is equivalent to that there are $3$-cycles between $\alpha$, $\beta$ and $\gamma$, and between
    $\delta$, $\beta$ and $\gamma$, with no arrows between $\alpha$ and $\delta$ in $Q_\Delta$ (see the third quiver in Figure \ref{Fig:subquiver1}),
    \item the condition for relation $R6$ is equivalent to that there are $3$-cycles between $\alpha$, $\beta$ and $\gamma$, and between
     $\delta$, $\beta$ and $\gamma$, with an arrow from $\alpha$ to $\delta$ in $Q_\Delta$ (see the first quiver in Figure \ref{Fig:subquiver2}),
    \item the condition for relation $R7$ is equivalent to that in case $(R6)$, there is additionally a $3$-cycle between $\alpha$, $\delta$ and $\zeta$ in $Q_\Delta$ with $w(\zeta)=1$ (see the second quiver in Figure \ref{Fig:subquiver2}),
    \item the condition for relation $R8$ is equivalent to that there is a $4$-cycle between $\alpha,\beta,\gamma$ and $\delta$ with no double arrows between them, no arrows between $\alpha$ and $\gamma$, and no arrows between $\beta$ and $\delta$ in $Q_\Delta$ (see the third quiver in Figure \ref{Fig:subquiver2}).
\end{enumerate}

    \begin{figure}[ht]
\begin{center}
\begin{tikzcd}
& \alpha \arrow[dr] & &  & \alpha \arrow[dr] &  & \alpha \arrow[drr] & & \delta\arrow[d] \\
\gamma \arrow[ur] & & \beta \arrow[ll] & \gamma \arrow[ur] & & \beta \arrow[ll, shift left=0.5ex]\arrow[ll, shift right=0.5ex]  & \gamma \arrow[u] \arrow[urr] & & \beta \arrow[ll, shift left=0.5ex]\arrow[ll, shift right=0.5ex] 
\end{tikzcd}  
\caption{Subquivers of $Q_\Delta$}\label{Fig:subquiver1}
\end{center}
\end{figure}

\begin{figure}[ht]
\begin{center}
\begin{tikzcd}
& & & & & \zeta\arrow[dl] & &\\
&   \alpha\arrow[rr] \arrow[drr] & &\delta\arrow[d] &  \alpha\arrow[rr] \arrow[drr] & &\delta\arrow[d]\arrow[lu] & \alpha\arrow[rr] & & \beta \arrow[d] 
  \\
&  \gamma \arrow[u] \arrow[urr] & & \beta \arrow[ll, shift left=0.5ex]\arrow[ll, shift right=0.5ex]& \gamma \arrow[u] \arrow[urr] & & \beta \arrow[ll, shift left=0.5ex]\arrow[ll, shift right=0.5ex]
 & \delta\arrow[u] & & \gamma\arrow[ll]
\end{tikzcd}  
\caption{Subquivers of $Q_\Delta$}\label{Fig:subquiver2}
\end{center}
\end{figure}

\end{remark}

\begin{remark} 
\label{rem:quivers no special}
$(a)$ If $\Sigma$ has no special punctures, then only the relations $R1$-$R8$ hold, with all arcs of weight $1$. Moreover, if $\Sigma$ has no $0$-punctures, then $(Br_\Delta)^{op}$ coincides with the braid group associated with quivers with potentials from \cite{KQ,QZ}. We will explore this remarkable coincidence elsewhere.

$(b)$ If $\Sigma$ is not an annulus with one marked point on each boundary component, then there exists an ordinary triangulation $\Delta$ of $\Sigma$ such that the defining relations for $Br_\Delta$ are given by $R1$ and $R2$ in Theorem \ref{th:brgroup}.
\end{remark}

\begin{example} (\cite[Theorem 2.12]{BM}, \cite[Definition 10.1]{Q}) 
\label{ex:Dn group}
Let $\Delta_0=\{(0,i),(i,0)\mid i=1,\ldots,n\}$ be the central star-like triangulation of $\Sigma_{n,1}$ and $\sigma$ be the rotation of $\Sigma_{n,1}$ by $\frac{2\pi}{n}$. Then $Br_\Delta$ is generated by $T_i:=T_{0i}$ subject to 
$T_iT_jT_i=T_jT_iT_j$
for any adjacent $i,j$ modulo $n$, $T_iT_j=T_jT_i$ for non-adjacent $i,j$ modulo $n$, and
$(T_1\cdots T_n)(T_1\cdots T_{n-2}) =(T_2\cdots T_{n}T_1)(T_2\cdots T_{n-1})$. 

Let $T:=T_1T_2\cdots T_nT_1\cdots T_{n-2}$. Then one can show that $T^n$ for $n$ odd and $T^{\frac{n}{2}}$ for $n$ even is in the center of $Br_\Delta$. 
\end{example}

\begin{example}
    Let $\Sigma$ be the torus with a disk removed and a single marked point on its boundary (see Figure \ref{Fig:R6}). Then $Br_\Sigma$ is generated by $T_{\alpha},T_\beta,T_\gamma$ and $T_\delta$, subject to:

 $\bullet$ $T_\alpha T_\beta T_\alpha=T_\beta T_\alpha T_\beta$, $T_\alpha T_\gamma T_\alpha=T_\gamma T_\alpha T_\gamma$, $T_\alpha T_\delta T_\alpha=T_\delta T_\alpha T_\delta$, $T_\delta T_\beta T_\delta=T_\beta T_\delta T_\beta$, $T_\delta T_\gamma T_\delta=T_\gamma T_\delta T_\gamma$.

 $\bullet$ $(T_\alpha T_\delta T_\gamma T_\delta^{-1} T_\alpha^{-1}) T_\beta (T_\alpha T_\delta T_\gamma T_\delta^{-1} T_\alpha^{-1})=T_\beta (T_\alpha T_\delta T_\gamma T_\delta^{-1} T_\alpha^{-1}) T_\beta$.

 $\bullet$ $(T_\delta T_\alpha  T_\gamma T_\alpha^{-1} T_\delta^{-1} )T_\beta (T_\delta T_\alpha  T_\gamma T_\alpha^{-1} T_\delta^{-1} )= T_\beta (T_\delta T_\alpha  T_\gamma T_\alpha^{-1} T_\delta^{-1} )T_\beta$.

\end{example}

\begin{example} 
$(a)$
Let $\Delta_1=\{(13),(31),(14),(41),(15),(51)\}\cup \{\text{boundary arcs}\}$.
\begin{figure}[ht]
\includegraphics{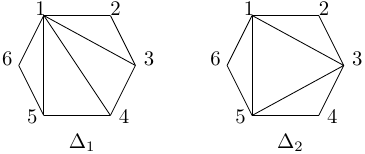}
\caption{Two triangulations of the hexagon}
\end{figure}
In $Br_{\Delta_1}$, we have $T_{13}T_{14}T_{13}=T_{14}T_{13}T_{14}, T_{14}T_{15}T_{14}=T_{15}T_{14}T_{15}$, and $T_{13}T_{15}=T_{15}T_{13}$.

$(b)$ Let $\Delta_2=\{(13),(31),(35),(53),(15),(51)\}\cup \{\text{boundary arcs}\}$. In $Br_{\Delta_2}$, we have $T_{13}T_{35}T_{13}=T_{35}T_{13}T_{35}, T_{35}T_{15}T_{35}=T_{15}T_{35}T_{15}, {T_{13}T_{15}T_{13}=T_{15}T_{13}T_{15}}$, and 
$T_{13}T_{15}T_{35}T_{13}=T_{15}T_{35}T_{13}T_{15}=T_{35}T_{13}T_{15}T_{35}$.
\end{example}

\begin{corollary}\label{pro:free}
$(a)$ $Br_{\Sigma}$ is isomorphic to the free group of rank $2$ for the annulus with one marked point on each boundary component. 

$(b)$ $Br_{\Sigma}$ is isomorphic to the free group of rank $3$ for the once-punctured torus.
\end{corollary}

\begin{conjecture} $Br_\Delta$ is torsion-free for any triangulation of any $\Sigma\in {\bf Surf}$.
\end{conjecture}

\begin{remark}
    \begin{enumerate}
        \item[$(R3')$] The relation $R3$ is equivalent to $Cyl(T_\alpha,T_\beta,T_\gamma)$ if additionally $w(\alpha)=w(\beta)=w(\gamma)=1$.
        \item[$(R4')$] The relation $R4$ is equivalent to 
      $\begin{cases}
        T_\alpha {T_\gamma} T_\beta T_\alpha T_\gamma T_\beta= {T_\gamma} T_\beta T_\alpha T_\gamma T_\beta T_\alpha, & \mbox{ if } w(\alpha)=1,\\
        T_\gamma T_\alpha {T_\gamma} T_\beta T_\alpha =T_\alpha {T_\gamma} T_\beta T_\alpha T_\gamma,  & \mbox{ if } w(\alpha)\neq 1.
    \end{cases}$
     \item[$(R5')$] The relation $R5$ is equivalent to $T_\delta T_\gamma T_\beta T_\delta T_\alpha T_\gamma T_\beta T_\alpha=T_\gamma T_\delta T_\alpha T_\gamma T_\beta T_\delta T_\alpha T_\beta$ if additionally $w(\alpha)=w(\delta)=1$.
     \item[$(R8')$] The relation $R8$ is equivalent to $Cyl(T_\alpha,T_\delta,T_\gamma,T_\beta)$ if additionally $w(\alpha)=w(\beta)=w(\gamma)=w(\delta)=1$.
    \end{enumerate}
\end{remark}

Denote $Br^+_\Delta$ the submonoid of $Br_\Delta$  
generated by all $T_{\gamma,\Delta}$ for every non-pending internal edge $\gamma$ of $\Delta$.

\begin{conjecture} 
\label{conj:injective braid monoid group}  For any oriented marked surface $\Sigma$ and any triangulation $\Delta$ of $\Sigma$, the relations $R1$, $R2$, $R3'$, $R4'$, $R5'$ and $R8'$ give a presentation of ${Br}^+_\Delta$. 
\end{conjecture}

In fact, \cite[Theorem 1.1]{P0} and Theorem \ref{th:Br_n 1} below verify this conjecture for appropriate triangulations of the following surfaces: $\Sigma_n$, $\Sigma_{n,1}$, $\Sigma_{n,2}$, the disk with one special puncture, the disk with one $0$-puncture and any unpunctured cylinder.

\begin{remark} It can happen that a non-free group can contain a free submonoid. For instance, if $G$ is a group generated by $a,b$ subject to $a=ba^{-1}b$ (i.e., $(a^{-1}b)^2=1$), then $G=\langle a,s\,|\,s^2=1\rangle$ is the free product of the infinite cyclic group $\langle a\rangle$ and the $2$-element group $\langle s\rangle$ (here $s=a^{-1}b$).
Consider the submonoid $M$ of $G$ generated by $a$ and $b=as$. Clearly, $M$ is free and freely generated by $a$ and $b$. This explains why we dropped relations $R6$, $R7$, and $R8$ in Conjecture \ref{conj:injective braid monoid group}.
    
\end{remark}

The following is an immediate consequence of  Proposition \ref{pr:relative braid homomorphism}, Theorem \ref{th:brgroup} and Theorem \ref{thm:quiverbraidgroup}, or by direct calculation.

\begin{corollary}\label{cor:affineatod}
    The assignments $\tau_i\mapsto
    \begin{cases}
    \sigma_0, & \text{ if $i=0$},\\
    \sigma_1^{\sigma_2},  &\text{ if $i=1$},\\
    \sigma_{i+1}, & \text{ if $i=2,3,\cdots,n-2$},\\
    \sigma_{n+1}, & \text{ if $i=n-1$},\\
    \sigma_n^{\sigma_2\sigma_3\cdots \sigma_{n-1}}, & \text{ if $i=n$,}
    \end{cases}
    $ define a group homomorphism from the Artin braid group $Br_{\hat A_{n}}$ of type $\hat A_n$ to the Artin braid group $Br_{\hat D_{n+1}}$ of type $\hat D_{n+1}$, where $\sigma_0,\sigma_1,\cdots,\sigma_n$ and $\tau_0,\tau_1,\cdots,\tau_{n+1}$ are the standard generators of $Br_{\hat A_n}$ and $Br_{\hat D_{n+1}}$, respectively, and $x^y:=yxy^{-1}$ for $x,y$ in a group.
\end{corollary}

\begin{figure}[ht]
\includegraphics[width=10cm]{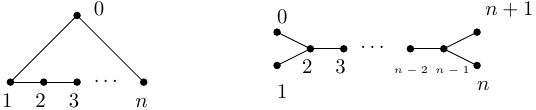}
\caption{Dynkin diagram of type $\hat A_n$ and $\hat D_{n+1}$}
\end{figure}

\subsection{Cluster braid groups of finite types and their symmetries}
\label{subsec:braid groups of finite types}

The following result is an immediate corollary of Theorem \ref{th:brgroup}.

\begin{theorem} 
\label{th:Br_n 1} 
$(a)$ $Br_\Delta\cong {Br}_{n-2}$ for any triangulation $\Delta$ of the unpuncture $n$-gon $\Sigma_{n}$.

$(b)$ $Br_\Delta\cong {Br}_{B_{n-1}}$ for any triangulation $\Delta$ of the $n$-gon with a $0$-puncture. 

$(c)$ 
$Br_\Delta\cong Br_{C_{n-1}}$ for any triangulation $\Delta$ of the once specially puncture $n$-gon.

$(d)$ $Br_\Delta\cong Br_{D_n}$ for any triangulation $\Delta$ of the once punctured $n$-gon. 

$(e)$ $Br_\Delta\cong \hat Br_{p+q-1}$ for any triangulation $\Delta$ of the unpunctured cylinder with $p$ marked points on one boundary and $q$ on another. 

$(f)$ $Br_\Delta\cong Br_{\hat D_n}$ for any triangulation $\Delta$ of the twice punctured $(n-2)$-gon. 

$(g)$ $Br_\Delta\cong Br_{\hat B_n}$ for any triangulation $\Delta$ of the once ordinarily and once specially punctured $(n-1)$-gon. 

$(h)$ $Br_\Delta\cong Br_{\hat C_n}$ for any triangulation $\Delta$ of the twice specially punctured $n$-gon.
\end{theorem} 

\begin{proof}
We prove only (g) and (h), since (a)--(f) are clear. 

For (g), label the boundary marked points by $1,2,\dots,n-1$ in clockwise order, the ordinary puncture by $0$, and the special puncture by $0'$. For each $i=2,\dots,n-1$, let $\gamma_i$ denote the arc from $1$ to $i$ such that $0$ lies to the left of $\gamma_i$ and $0'$ lies to the right. Let $\gamma_1$ be the loop based at $1$ enclosing the puncture $0$, and let $\gamma_1'$ be the loop based at $1$ enclosing the special puncture $0'$. Let $\gamma_1''$ be the arc connecting $1$ and $0$ enclosed by $\gamma_1$. Set
\[
\Delta_0=\{\gamma_i,\overline{\gamma}_i \mid i=2,\dots,n-1\}\cup \{\gamma_1,\gamma_1',\gamma_1'',\overline{\gamma}_1,\overline{\gamma}_1',\overline{\gamma}_1''\}\cup \{\text{boundary segments}\},
\]
see Figure \ref{Fig:affineBC}
Then by Theorem \ref{th:brgroup}, we have $Br_{\Delta_0}\cong Br_{\hat{B}_n}$. Hence $Br_\Delta\cong Br_{\hat{B}_n}$ for any triangulation $\Delta$.

For (h), label the boundary marked points by $1,2,\dots,n-1$ in clockwise order, and the two special punctures by $0$ and $0'$. For each $i=2,\dots,n$, let $\gamma_i$ denote the arc from $1$ to $i$ such that $0$ lies to the left of $\gamma_i$ and $0'$ lies to the right. Let $\gamma_1$ be the loop based at $1$ enclosing the puncture $0$, and let $\gamma_1'$ be the loop based at $1$ enclosing the special puncture $0'$. Set
\[
\Delta_0=\{\gamma_i,\overline{\gamma}_i \mid i=2,\dots,n\}\cup \{\gamma_1,\gamma_1',\overline{\gamma}_1,\overline{\gamma}_1'\}\cup \{\text{boundary segments}\},
\]
see Figure \ref{Fig:affineBC}.
Then by Theorem \ref{th:brgroup}, we have $Br_{\Delta_0}\cong Br_{\hat{C}_n}$. Hence $Br_\Delta\cong Br_{\hat{C}_n}$ for any triangulation $\Delta$.

This completes the proof.
\end{proof}

\begin{figure}[ht]
\includegraphics[width=10cm]{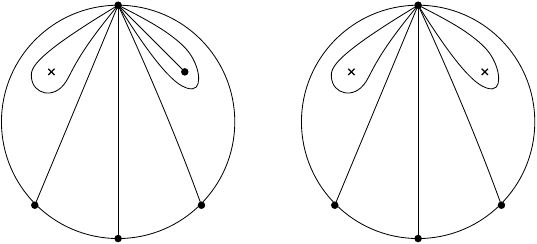}
\caption{Triangulations of affine type $B$ and $C$}\label{Fig:affineBC}
\end{figure}

\begin{remark}
\label{rem:acyclic braid type A}
 We say that a triangulation $\Delta$ of $\Sigma_n$ is acyclic if any triangle of $\Delta$ has a boundary edge.  Then one can show that there is an ordering $\gamma_1,\ldots,\gamma_{n-3}$ of diagonals of $\Delta$ such that the generators $T_i:=T_{\gamma_i}$ of $Br_{\Delta}$ are subject to the standard braid relations.
\end{remark}

\begin{remark} 
\label{rem:Br_4 exceptional}
$Br_{\Sigma_6}\cong Br_{\Sigma_{3,1}}\cong Br_4$, however, we expect that this is the only exceptional isomorphism $Br_\Sigma\cong Br_{\Sigma'}$.
    
\end{remark}

Note that ${\mathcal A}_n={\mathcal A}_{\Sigma_n}$ has a dihedral group $I_2(n)\subset S_n$ of automorphisms so that $\sigma\in I_2(n)$ acts through
$x_{ij}\mapsto 
x_{\sigma(i),\sigma(j)}$.

\begin{theorem}
\label{th:symmetries of n-gon}
$(a)$ Suppose that $\sigma\in \Gamma_\Delta$ reverses the orientation of $|\Delta|$. Then $\sigma$ induces an outer automorphism $T_\gamma\mapsto T_{\sigma(\gamma)}^{-1}$ of $Br_{\Delta}$.

$(b)$ Let $\Delta$ be a triangulation of $\Sigma_{3n}$ which is invariant under rotation $\sigma$ by $\frac{2\pi}{3}$ (for instance, $\Delta=\Delta_0=\{(kn+1,kn+i),(kn+i,kn+1)\mid 0\leq k\leq 2, 3\leq i\leq n+1\}\cup \{\text{boundary arcs}\}$; {see Figure \ref{Fig:rotationinvbasis} for an example}). Then  $\sigma$ induces an inner automorphism $\sigma$ of $Br_\Delta\cong Br_{3n-2}$. Moreover, if $\Delta=\Delta_0$ then $\sigma$ is given by $T_\gamma\mapsto T_{\sigma(\gamma)}=\tau T_\gamma\tau^{-1}$, where 
$$\tau=[(T_{13}T_{2n+1,2n+3}T_{n+1,n+3})(T_{14}T_{2n+1,2n+4}T_{n+1,n+4})\cdots (T_{1,n+1}T_{2n+1,1}T_{n+1,2n+1})]^{n-1}.$$ 
In particular, the center $C(Br_\Delta)$ of $Br_\Delta$ is a cyclic group generated by $\tau^3$.

$(c)$ In the notation of Example \ref{ex:Dn group}, 
let  $\sigma$ be the rotation of $\Sigma_{n,1}$ by $\frac{2\pi}{n}$. If $n$ is odd then 
$\sigma$ induces an inner automorphism of $Br_\Delta$ given by $T_\gamma\mapsto T_{\sigma(\gamma)}=\tau T_\gamma\tau^{-1}$, 
where 
$\tau=T^{\frac{n-1}{2}}.$  

\end{theorem}

We expect $\sigma$ induces an outer automorphism of $Br_\Delta$ in the case $n$ is even in part $(c)$.

We will prove Theorem \ref{th:symmetries of n-gon} in Section \ref{sec:symmetries of n-gon}.

In particular, $Br_4$ has an automorphism $\sigma$ of order $3$ given by $\sigma(T_{ij})=T_{i+3,j+3}$ (both indices are modulo $6$). However, according to Dyer-Grossman theorem (\cite{DG}), all automorphisms of odd order of $Br_n$, $n\ge 3$,  must be inner as in Theorem \ref{th:symmetries of n-gon}(b), which is quite surprising (we could not find this result in the literature and obtained it only by looking at invariant triangulations of the $3n$-gons).

For example, in the hexagon, let $\tau=T_1T_2T_3T_1=\tau_1\tau_2(\tau_2^{-1}\tau_3\tau_2)\tau_1=\tau_1\tau_3\tau_2\tau_1$. Then we have $\tau^3=(\tau_1\tau_2\tau_3)^4$.

\begin{figure}[ht]
\includegraphics{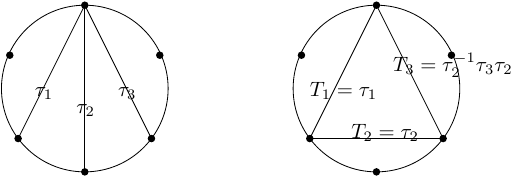}
\caption{}
\label{Fig:rotationinvbasis}
\end{figure}

\begin{proposition}\label{prop:BtoA}
For any $k,n\geq 2$, let $\sigma$ be the clockwise rotation of $\Sigma_{kn}$ by $2\pi/k$, and let $f_{\sigma}:\Sigma_{kn}\to \Sigma_{kn}/\sigma$ be the quotient map. Consider the $\sigma$-invariant triangulation $\Delta$ of $\Sigma_{kn}$ defined as: $\Delta=\{(sn+1,sn+i), (sn+i,sn+1), (1,tn+1), (tn+1,1)\mid 0\leq s\leq k-1, 2\leq i\leq n+1, 1\leq t\leq k\}\cup \{\text{boundary arcs}\}$. Denote 

$\bullet$ $\tau_i=T_{1,i}T_{n+1,n+i}\cdots T_{(k-1)n+1,(k-1)n+i}$ for $i=3,4,\cdots,n$,

$\bullet$
$\rho=T_{1+n,1+2n}T_{1+2n,1+3n}\cdots T_{1+(k-2)n,1+(k-1)n}$,

$\bullet$
$\phi=T_{1,n+1}T_{1,2n+1}\cdots T_{1,(k-1)n+1}$,

$\bullet$ $\tau_{n+1}=\rho\phi\rho$.

Then we have
$$Br_\Delta^{f_\sigma}\supseteq 
     \langle \tau_i\mid i=3,\cdots,n+1\rangle 
     \cong Br_{C_{n-1}}\cong Br_{\Sigma_{kn}}/\sigma.$$
    In particular, there is an embedding $Br_{C_{n-1}}\hookrightarrow Br_{kn-2}$ for any $k,n\geq 2$.
 \end{proposition}

We prove Proposition \ref{prop:BtoA} in Section \ref{sec:proof of prop:BtoA}.

\begin{example}
For a $3n$-gon $\Sigma_{3n}$, let $\sigma$ be the clockwise rotation of $\Sigma_{3n}$ by $2\pi/3$. Then $\Sigma_{3n}/\sigma$ is an $n$-gon with a special puncture of order $3$. 

Let $\Delta=\{(1,i),(i,1), (n+1,n+i),(n+i,n+1),(2n+1,2n+i),(2n+i,2n+1)\mid i=3,4,\cdots, n+1\}\cup\{\text{boundary arcs}\}$. Then $\Delta$ is invariant under $\sigma$. Thus, $\sigma$ induces automorphisms of $T_{\Delta}$ and $Br_{\Delta}$. By abuse of notation, we still denote these induced automorphisms by $\sigma$. For any $i\in \{3,4,\cdots, n\}$, denote $\sigma_i=T_{1,i}T_{2n+1,2n+i}T_{n+1,n+i}$ and
$\sigma_{n+1}=T_{1,n+1}T_{2n+1,1}T_{n+1,2n+1}T_{1,n+1}.$ 

Then $Br_{\Delta}^{\sigma}\supseteq \langle \sigma_i\mid i=3,\cdots,n\rangle\cong Br_{\Sigma/\sigma}$. 
\end{example} 


\begin{theorem}\label{thm:inject}
{The natural homomorphisms
$\iota_D\colon Br_{n}\to Br_{D_{n}}$, $\tau_i\mapsto s_i$ for $i=1,2,\dots,n-1$,
and
$\iota_{\hat A}\colon Br_{n}\to Br_{\hat A_{n-1}}$}, $\tau_i\mapsto \sigma_{i-1}$ for $i=1,\dots,n-1$,
are injective. Here $\tau_1,\dots,\tau_{n-1}$ are the standard generators of $Br_n$; $s_1,\dots,s_n$ are the standard generators of $Br_{D_n}$, where the Dynkin diagram has a trivalent vertex at the node $n-2$; and $\sigma_0,\sigma_1,\dots,\sigma_{n-1}$ are the standard generators of $Br_{\hat A_{n-1}}$.
\end{theorem}

We prove Theorem \ref{thm:inject} in Section \ref{sec:proof of thm:inject}.

\subsection{Braid groups of surfaces with orientation-reversing involutions}
\label{subsec:non-orientable}

Throughout this section, $\sigma$ is an orientation-reversing automorphism of $\Sigma$.

The following is an immediate consequence of Corollary \ref{cor:Gamma-action} (with $f=\sigma$).

\begin{lemma} 
\label{le:anti-action on braid groups}
For any orientation-reversing automorphism $\sigma$ of $\Sigma$ and any $\Delta\in {\bf Tsurf}_\Sigma$ such that $\sigma(\Delta)=\Delta$, the corresponding automorphism of $Br_\Delta$ is given by $T_{\gamma,\Delta}\mapsto T_{\sigma(\gamma),\Delta}^{-1}$ for all non-pending internal edges of $\Delta$.

\end{lemma}

It is well-known that any orientation-reversing automorphism of any oriented surface factors into an orientation-reversing involution and an orientation-preserving automorphism. However,  orientation-reversing involutions are not always conjugate to each other. On the other hand, if such an involution has no fixed points, it is unique up to conjugation (because $\Sigma/\sigma$ is unique up to isomorphisms). If $\Sigma$ is closed, then such an involution $\sigma$ always exists (we sometimes refer to it as the anti-involution of $\Sigma$). 

Denote by $\underline \Sigma$ a non-oriented surface, whose (unramified) double cover is $\Sigma$, i.e., $\underline \Sigma=\Sigma/\sigma$, where $\sigma$ is an anti-involution of $\Sigma$. We denote by ${\bf TSurf}_{\underline \Sigma}$ the subgroupoid of ${\bf TSurf}_\Sigma$ whose objects are $\sigma$-invariant triangulations of $\Sigma$ and morphisms are those morphisms of $h$ in ${\bf TSurf}_\Sigma$ such that $\sigma(h)=h^{-1}$.

Finally, for any $\underline \Delta\in {\bf TSurf}_{\underline \Sigma}$, denote $Br_{\underline \Delta}:=Aut_{{\bf TSurf}_{\underline \Sigma}}(\underline \Delta)$ and refer to it as the braid group of $\underline \Delta$. 

The following is an immediate consequence of Lemma \ref{le:anti-action on braid groups}.

\begin{corollary} 
\label{cor:anti-involution}
In the assumptions as above, one has

$(a)$ The action of $\sigma$ lifts to $Br_\Delta$ via  $\sigma(T_\gamma)= T_{\sigma(\gamma)}^{-1}$ for all non-pending internal edges $\gamma$ of $\Delta$. 

$(b)$ $Br_{\underline \Delta}=(Br_\Delta)^\sigma$,  the $\sigma$-fixed point subgroup of $\sigma$ in $Br_\Delta$. 
    
\end{corollary}

\begin{remark} It is natural to expect that the subgroup 
$ Br_{\underline \Delta}$ from Corollary \ref{cor:anti-involution}(b) is generated by $T_\gamma T_{\sigma(\gamma)}^{-1}= T_{\sigma(\gamma)}^{-1}T_\gamma$, where $\gamma$ runs over all non-pending internal edges of $\Delta$. 
    
\end{remark}

\begin{example}
[Projective plane] Let $\Sigma$ be a sphere with $2n+2$ punctures (which we place uniformly at the equator). Let $\Delta$ be the triangulation of $\Sigma$ as shown below.
Then $Br_\Delta$ is generated by $T_i^+$ and $T_i^-$ for $i=1,\ldots,2n-1$, and $T_j^0$ for $j=1,\ldots, 2n+2$ subject to:

$\bullet$ $T_i^{\pm}T_{i+1}^\pm T_i^{\pm}=T_{i+1}^\pm T_i^{\pm}T_{i+1}^\pm$ for all $i=1,\cdots,2n-2$.

$\bullet$ $T_i^{\pm}T_j^\pm=T_j^{\pm} T_i^\pm$ for all $i,j$ with $|i-j|\neq 1$.

$\bullet$ $T_i^{\pm}T_{i+1}^0 T_i^{\pm}=T_{i+1}^0 T_i^{\pm}T_{i+1}^0 $ and $T_i^{\pm}T_{i+2}^0 T_i^{\pm}=T_{i+2}^0 T_i^{\pm}T_{i+2}^0$ for $i=1,2,\cdots,2n-1$.

$\bullet$ $T_1^{\pm} T_{1}^0 T_1^{\pm}=T_{1}^0 T_1^{\pm}T_{1}^0$ and $T_{2n-1}^{\pm} T_{2n+1}^0 T_{2n-1}^{\pm}=T_{2n+1}^0 T_{2n-1}^{\pm}T_{2n+1}^0$.

$\bullet$ $T_i^{\pm}T_{j}^0=T_{j}^0 T_i^{\pm}$ for all $i=2,\cdots, 2n-2$ and $j\neq i+1,i+2$.

$\bullet$ $T_1^{\pm}T_{j}^0=T_{j}^0 T_1^{\pm}$ for all $j\neq 1,2,3$.

$\bullet$ $T_{2n-1}^{\pm}T_{j}^0=T_{j}^0 T_{2n-1}^{\pm}$ for all $j\neq 2n,2n+1,2n+2$.

$\bullet$ $Cyl(T_i^{\pm},T_{i+1}^{\pm},T_{i+2}^0)$ for all $i=1,2,\cdots, 2n-1$.

$\bullet$ $Cyl(T_1^{\pm},T_{2}^0,T_{1}^0)$ and $Cyl(T_{2n-1}^{\pm},T_{2n+2}^0,T_{2n+1}^0)$.

$\bullet$ $T_i^+T_j^-=T_j^-T_i^+$ for all $i,j$.

\begin{figure}[ht]
\includegraphics{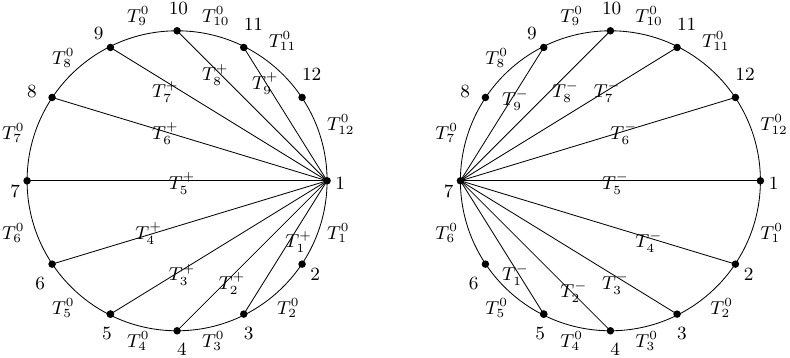}
\caption{The triangulation for the sphere in the case $n=5$}
\label{Fig:trisph}
\end{figure}

Let $\sigma$
be the central symmetry of $\Sigma$, i.e., the only orientation-reversing involution interchanging the punctures (that is, $\underline \Sigma:=\Sigma/\sigma$ is the projective plane with $n$ punctures). Clearly, $\Delta$ is $\sigma$-invariant.

According to Corollary \ref{cor:anti-involution}(a), $\sigma$ acts on $Br_\Delta$ via $\sigma(T_i^\pm )=(T_{2n-i}^{\mp})^{-1}$ and $\sigma(T_j^0)=(T_{n+1+j}^0)^{-1}$ for $j=1,\ldots,2n+2$ (modulo $2n+2$).

Then the $\sigma$-fixed point subgroup $Br_{\underline \Delta}:=(Br_\Delta)^\sigma$ can be viewed as the braid group of the corresponding triangulation $\underline \Delta$ of the projective plane $\underline \Sigma=\Sigma/\sigma$ (we will discuss non-orientable surface elsewhere). 

We expect that  $Br_{\underline \Delta}=(Br_\Delta)^\sigma$ is generated by $T_i:=T_i^+(T_{2n-i}^-)^{-1}$ for $i=1,2,\cdots,2n-1$ and $\tau_j:=T_j^0(T_{n+1+j}^0)^{-1}$ for $j=1,\ldots,n+1$. One can show that the following relations hold (we expect them to be defining): 

$\bullet$ $T_i T_{i+1} T_i=T_{i+1} T_iT_{i+1}$ for all $i=1,2,\cdots, 2n-2$.

$\bullet$ $T_iT_j=T_{j} T_i$ for all $i,j$ with $|i-j|\neq 1$.

$\bullet$ $T_i\tau_j=\tau_j T_i$ for all $i=2,\cdots, 2n-2$ and $j\neq i,i+1 (\text{mod } n+1)$.

$\bullet$ $T_i\tau_j T_i=\tau_j T_i\tau_j$ for all $i=2,\cdots, 2n-2$ and $j=i,i+1 (\text{mod } n+1)$.

$\bullet$ $T_1\tau_j=\tau_j T_1$ for all $j\neq 1,2,3$.

$\bullet$ $T_1\tau_j T_1=\tau_j T_1\tau_j$ for $j= 1,2,3$.

$\bullet$ $T_{2n-1}\tau_j=\tau_j T_{2n-1}$ for all $j\neq n-1,n,n+1$.

$\bullet$ $T_{2n-1}\tau_j T_{2n-1}=\tau_jT_{2n-1} \tau_j$ for $j=n-1,n,n+1$.

$\bullet$ $Cyl(T_i,T_{i+1},\tau_{i+2 (\text{mod n+1})})$ for all $i=1,\cdots,2n-2$.

$\bullet$ $Cyl(T_1,\tau_{2},\tau_{1})$ and $Cyl(T_{2n-1},\tau_{n+1},\tau_n)$.

\end{example}

\subsection{Rank 2 groupoids}
\label{subsec:rank 2 braid}

For any $m\in \mathbb{Z}_{\ge 0}$ let $\Gamma_m$ be the groupoid whose object set is $\Z$ and its set of morphisms is generated by $h_{i+1,i}:i\to i+1$, $h_{i,i+1}:i+1\to i$ and $\sigma_i: i\to i+m+2$ 
subject to: for any $i$,

$\bullet$ $h_{i+2,i+1}h_{i+1,i}=\sigma_{i-m}h_{i-m,i-m+1}\cdots h_{i-2,i-1}h_{i-1,i}.$

$\bullet$ $\sigma_{i+1}h_{i+1,i}=h_{i+1+m,i+m}\sigma_{i}.$ 

$\bullet$ $\sigma_{i}h_{i,i+1}=h_{i+m,i+1+m}\sigma_{i+1}$.

\medskip

Denote by $Br_{\Gamma_m}=Aut_{\Gamma_m}(i)$ the fundamental group of $\Gamma_m$.

In abuse of notation, we denote $x_i=h_{i,i-1}:i-1\to i$ and $y_i=h_{i-1,i}:i\to i-1$ for all $i\in \mathbb Z$. Then $\Gamma_m$ is generated by $x_i,y_i,\sigma_i$ subject to: for any $i$,

$\bullet$ $x_{i+1}x_i=\sigma_{i-m-1} y_{i-m}\cdots y_{i-2}y_{i-1}$.

$\bullet$ $\sigma_{i} x_i=x_{i+m}\sigma_{i-1}, \sigma_{i-1} y_i=y_{i+m}\sigma_i$.

Denote $b_1=y_1x_1,b_0=x_0y_0\in Aut_{\Gamma_m}(0)$. Assume $R$ are the relations that $b_0,b_1$ are satisfied. Set
{\[a_i:=
\begin{cases}
 x_i\cdots x_2x_1: 0\to i & \mbox{if $i>0$},\\
 id:0\to 0 & \mbox{if $i=0$},\\
 y_0y_{-1}\cdots y_{-i+1}:0\to i & \mbox{if $i<0$}.
\end{cases}
\]}

\begin{lemma}\label{lem:xysigma}

(a) For any $i>0$, $y_1y_2\cdots y_i=(\underbrace{b_1b_0b_1\cdots}_i)a_i^{-1}$, 
  $y_1y_2\cdots y_ix_i\cdots x_2x_1=\underbrace{b_1b_0b_1\cdots}_i$.

(b) For any $i\leq 0$, $x_0x_{-1}\cdots x_i=\underbrace{\cdots b_0b_1b_0}_{-i+1}a^{-1}_{i-1}$ and 
  $x_0x_{-1}\cdots x_iy_i\cdots y_{-1}y_0=\underbrace{\cdots b_0b_1b_0}_{-i+1}$.

(c) For any $i>0$, $x_i=a_ia_{i-1}^{-1}$ and 
$y_i=\begin{cases}
    a_{i-1}b_1a_i^{-1}, \text{ if $i$ is odd}\\
    a_{i-1}b_0a_i^{-1}, \text{ if $i$ is even}.  
\end{cases}$

(d) For any $i\leq 0$, $y_i=a_{i-1}a_i^{-1}$ and
$x_i=a_i(\underbrace{\cdots b_0b_1b_0}_{-i})^{-1}(\underbrace{\cdots b_0b_1b_0}_{-i+1})a_{i-1}^{-1}$.

(e)  $\sigma_i=\begin{cases}
  a_{m+2+i}(\underbrace{b_1b_0b_1\cdots }_{m})^{-1}a_i^{-1} & \text{ if $i\geq 0$}\\
  a_{m+2+i}(\underbrace{b_1b_0b_1\cdots}_{m+i+2})^{-1}b_1b_0a_i^{-1} & \text{ if $-m-2\leq i<0$}\\
  a_{m+2+i}b_1b_0a_i^{-1} & \text{ if $i<-m-2$}.
\end{cases}$
\end{lemma}

\begin{proof}

(a) It suffices to show that $y_1y_2\cdots y_ix_i\cdots x_2x_1=\underbrace{b_1b_0b_1\cdots}_i$. It is clear that $i=1$. 
For $i\geq 2$, as $x_{j+1}x_j=\sigma_{j-m-1} y_{j-m}\cdots y_{j-2}y_{j-1}$ for all $j$, we have 
$$y_1y_2\cdots y_ix_i\cdots x_2x_1=
\begin{cases}
 y_1y_2\cdots y_{i-1}x_{i-1}\cdots x_2x_1 x_0y_0 & \text{ if $i$ is even}\\
 y_1y_2\cdots y_{i-1}x_{i-1}\cdots x_2x_1 y_1x_1 & \text{ if $i$ is odd}.\\ 
\end{cases}
$$
Thus we have $y_1y_2\cdots y_ix_i\cdots x_2x_1=\underbrace{b_1b_0b_1\cdots}_i$ by induction.

(b) can by proved similarly to (a).

(c) follows from (a) and (d) follows from (b).

(e) As $x_{m+2}x_{m+1}=\sigma_0y_1y_2\cdots y_{m}$, by (a) we have 
$\sigma_0=a_{m+2}(\underbrace{b_1b_0b_1\cdots }_{m})^{-1}$.

For any $i>0$, as $\sigma_ix_i\cdots x_2x_1=x_{i+m+2}\cdots x_{m+4}x_{m+3}\sigma_0$, we have
$$\sigma_i=a_{m+2+i}a_{m+2}^{-1}\sigma_0a_i^{-1}=a_{m+2+i}(\underbrace{b_1b_0b_1\cdots }_{m})^{-1}a_i^{-1}.$$

If $-m-2\leq i<0$, as $\sigma_iy_{i-1}\cdots y_1y_0=y_{m+1+i}\cdots y_{m+1}y_{m+2}\sigma_0$, we have

\begin{equation*}
\begin{array}{rcl}
\sigma_i &=& \left((\underbrace{b_1b_0b_1\cdots}_{m+2+i})a_{m+2+i}^{-1}\right)^{-1}       (\underbrace{b_1b_0b_1\cdots}_{m+2})a_{m+2}^{-1} \sigma_0 a_i^{-1}\vspace{2.5pt}\\
&=&
\left((\underbrace{b_1b_0b_1\cdots}_{m+2+i})a_{m+2+i}^{-1}\right)^{-1}       (\underbrace{b_1b_0b_1\cdots}_{m+2})(\underbrace{b_1b_0b_1\cdots }_{m})^{-1} a_i^{-1}\vspace{2.5pt}\\
&=& a_{m+2+i}(\underbrace{b_1b_0b_1\cdots}_{m+2+i})^{-1}b_1b_0a_i^{-1}.
\end{array}
\end{equation*}


If $i<-m-2$, as $\sigma_iy_{i-1}\cdots y_1y_0=y_{m+1+i}\cdots y_{m+1}y_{m+2}\sigma_0$, we have 
\begin{equation*}
\begin{array}{rcl}
\sigma_i &=& a_{m+2+i}\underbrace{b_1b_0b_1\cdots }_{m+2}a_{m+2}^{-1}\sigma_0a_i^{-1}\vspace{2.5pt}\\
&=&
a_{m+2+i}\underbrace{b_1b_0b_1\cdots }_{m+2}(\underbrace{b_1b_0b_1\cdots }_{m})^{-1}a_i^{-1}\vspace{2.5pt}\\
&=& a_{m+2+i}b_1b_0a_i^{-1}.
\end{array}
\end{equation*}

The proof is complete.
\end{proof}

The following remark is easy to see.

\begin{remark}\label{rmk:def}
 The objects of $\Gamma_m$ are $i,i\in \mathbb Z$, the morphisms are generated by $x_i: i-1\to i,y_i:i\to i-1,\sigma_i: i\to i+m+2, i\in \mathbb Z$, subject to the following relations:
   \begin{enumerate}
       \item[$(C1)$]: $x_{m+2}x_{m+1}=\sigma_0 y_1\cdots y_{m-1}y_m$;
       \item[$(C2)$]: $\sigma_ix_i\cdots x_2x_1=x_{i+m+2}\cdots x_{m+4}x_{m+3}\sigma_0$ for all $i>0$;
       \item[$(C3)$]: $\sigma_0 y_1y_2\cdots y_i=y_{m+3}y_{m+4}\cdots y_{m+i+2}\sigma_i$ for all $i>0$;
       \item[$(C4)$]:  $\sigma_iy_{i-1}\cdots y_1y_0=y_{m+1+i}\cdots y_{m+1}y_{m+2}\sigma_0$ for all $i<0$;
       \item[$(C5)$]: $\sigma_0 x_{-1}x_{-2}\cdots x_{i-1}=x_{m+2}x_{m+1}\cdots x_{m+i+3}\sigma_i$ for all $i<0$.
   \end{enumerate}
   \end{remark}

\begin{theorem} 
\label{th:rank 2 faithful}
For any $m\geq 0$, $Br_{\Gamma_m}$ is isomorphic to the Artin braid group corresponding to the dihedral group $I_2(m)$. 
\end{theorem}

\begin{proof}

From the proof of Lemma \ref{lem:xysigma}, we see that $\sigma_ix_i\cdots x_2x_1=x_{i+m+2}\cdots x_{m+4}x_{m+3}\sigma_0$ for all $i>0$ and $\sigma_iy_{i-1}\cdots y_1y_0=y_{m+1+i}\cdots y_{m+1}y_{m+2}\sigma_0$ for all $i<0$.

For $i>0$, from the relation $(C3):\sigma_0 y_1y_2\cdots y_i=y_{m+3} y_{m+4}\cdots y_{m+2+i}\sigma_i$ and Lemma \ref{lem:xysigma}, we see that

$(a_{m+2}(\underbrace{b_1b_0b_1\cdots }_{m})^{-1})(\underbrace{b_1b_0b_1\cdots}_i)a_i^{-1}=\left((\underbrace{b_1b_0b_1\cdots}_{m+2})a_{m+2}^{-1}\right)^{-1} (\underbrace{b_1b_0b_1\cdots}_{m+2+i})(\underbrace{b_1b_0b_1\cdots }_{m})^{-1}a_i^{-1},$ equivalently,
$(\underbrace{b_1b_0b_1\cdots}_i)=b_0^{-1}b_1^{-1}(\underbrace{b_1b_0b_1\cdots}_{m+2+i})(\underbrace{b_1b_0b_1\cdots }_{m})^{-1},$ that is,
$$\begin{cases}
 1=(\underbrace{b_1b_0b_1\cdots}_{m})(\underbrace{b_1b_0b_1\cdots }_{m})^{-1} & \text{if $i$ is even}\\
 1=(\underbrace{b_0b_1b_0\cdots}_{m})(\underbrace{b_1b_0b_1\cdots }_{m})^{-1} 
 &\text{if $i$ is odd}.
\end{cases}
$$

Thus, we have $\underbrace{b_0b_1b_0\cdots}_{m}=\underbrace{b_1b_0b_1\cdots}_{m}$.

In case $-m-2\leq i<0$, from the relation $(C5):\sigma_0x_{-1}\cdots x_{i+1}=x_{m+2} x_{m+1}\cdots x_{m+3+i}\sigma_i$ and Lemma \ref{lem:xysigma}, we have 

$(\underbrace{b_1b_0b_1\cdots }_{m})^{-1}(\underbrace{\cdots b_0b_1b_0}_{-i})=(\underbrace{b_1b_0b_1\cdots}_{m+2+i})^{-1}(\underbrace{b_1b_0b_1\cdots}_{m+2})(\underbrace{b_1b_0b_1\cdots }_{m})^{-1}.$

Thus, $(\underbrace{b_1b_0b_1\cdots }_{m})^{-1}(\underbrace{\cdots b_0b_1b_0}_{-i})=(\underbrace{b_1b_0b_1\cdots}_{m+2+i})^{-1}b_1b_0=(\underbrace{b_1b_0b_1\cdots}_{m+i})^{-1},$
that is
$$\begin{cases}
 \underbrace{b_1b_0b_1\cdots}_{m}=\underbrace{b_0b_1b_0\cdots }_{m} & \text{if $i$ is odd}\\
\underbrace{b_1b_0b_1\cdots}_{m}=\underbrace{b_1b_0b_1\cdots }_{m} 
 &\text{if $i$ is even}.
\end{cases}
$$

In case $i<-m-2$, from the relation $(C5):\sigma_0x_{-1}\cdots x_{i+1}=x_{m+2} x_{m+1}\cdots x_{m+3+i}\sigma_i$ and Lemma \ref{lem:xysigma}, we have 
$(\underbrace{b_1b_0b_1\cdots }_{m})^{-1}(\underbrace{\cdots b_0b_1b_0}_{-i})=(\underbrace{\cdots b_0b_1b_0}_{-m-2-i})(\underbrace{b_1b_0b_1\cdots }_{m+2})(\underbrace{b_1b_0b_1\cdots }_{m})^{-1}.$

Thus, $(\underbrace{b_1b_0b_1\cdots }_{m})^{-1}(\underbrace{\cdots b_0b_1b_0}_{-i})=\underbrace{\cdots b_0b_1b_0}_{-m-i},$
that is
$$\begin{cases}
 \underbrace{b_1b_0b_1\cdots}_{m}=\underbrace{b_0b_1b_0\cdots }_{m} & \text{if $i$ is odd}\\
\underbrace{b_1b_0b_1\cdots}_{m}=\underbrace{b_1b_0b_1\cdots }_{m} 
 &\text{if $i$ is even}.
\end{cases}
$$
Therefore, we have $ \underbrace{b_1b_0b_1\cdots}_{m}=\underbrace{b_0b_1b_0\cdots }_{m}$. {We now prove that it is the defining relation for $b_0$ and $b_1$.} 

Let $\widetilde\Gamma_m$ be a groupoid whose objects set is $\mathbb Z$ and its morphisms is generated by $\widetilde a_i:0\to i$, $\widetilde b_0,\widetilde b_1: 0\to 0$ satisfy $\widetilde a_0=id$ and $\underbrace{\widetilde b_1 \widetilde b_0 \widetilde b_1\cdots}_{m}=\underbrace{\widetilde b_0 \widetilde b_1 \widetilde b_0\cdots }_{m}$.

For any $j\in \mathbb Z$, let $\widetilde x_j:j-1\to j,\widetilde y_j:j\to j-1$ and $\widetilde\sigma_j:j\to j+m-2$ be expressed by $\widetilde a_i,\widetilde b_0,\widetilde b_1$ as in Lemma \ref{lem:xysigma} (c) (d) (e). Then by direct computation, we have relations $(C1)$ and $(C_2)$ holds for $\widetilde x_j,\widetilde y_j$, and $\sigma_j$. From the comutation above, the relation $\underbrace{\widetilde b_1 \widetilde b_0 \widetilde b_1\cdots}_{m}=\underbrace{\widetilde b_0 \widetilde b_1 \widetilde b_0\cdots }_{m}$ implies the relations $(C3)$, $(C4)$ and $(C5)$ hold for $\widetilde x_j,\widetilde y_j$, and $\sigma_j$.

Therefore, from Remark \ref{rmk:def}, the assignments $x_i\mapsto \widetilde x_i, y_i\mapsto \widetilde y_i$ and $\sigma_i\mapsto \widetilde \sigma_i$ give an isomorphism of groupoid $\Gamma_m\cong \widetilde \Gamma_m$. It is clear that the isomorphism send $b_0$ and $b_1$ to $\widetilde b_0$ and $\widetilde b_1$, respectively. Thus, $\langle b_0,b_1\rangle \cong \langle \widetilde b_0,\widetilde b_1\rangle$ is isomorphic to the Artin braid group corredsponding to the dihedral group $I_2(m)$.
\end{proof}

We will apply the same idea to prove Theorem \ref{th:brgroup}.

\section{Triangle groups, monomial mutations, and the triangular functor}\label{Sec:triangleg}

\subsection{Triangle groups and their functoriality}\label{sec:triangleg1}

\begin{definition}\label{def:trig}
Generalizing \cite{BR}, for any (tagged) triangulation $\Delta$, we define the {\it triangle group} $\mathbb T_{\Delta}$ to be generated by $t_\gamma,t_{\overline\gamma}, \gamma\in \Delta$ subject to the following relations:

$\bullet$ $t_{\overline \gamma}=t_\gamma$ for any special loop $\gamma\in \Delta$.

$\bullet$ $t_{\alpha_1}t^{-1}_{\overline \alpha_2}t_{  \alpha_3}=t_{\overline \alpha_3}t^{-1}_{ \alpha_2}t_{\overline \alpha_1}$
 for any cyclic triangle $(\alpha_1,  \alpha_2,  \alpha_3)$ in $\in \Delta$. 

$\bullet$ $t_\ell=t_\gamma t_{\overline \gamma}$ if $\ell$ is a loop encloses a pending arc $\gamma$ with $s(\gamma)=s(\ell)$. 

$\bullet$ $t_{\gamma_1}t_{ \gamma_2}=t_{\overline \gamma_2}t_{\overline \gamma_1}$ for any tagged cyclic bigon $(\gamma_1,\gamma_2)$ in $\Delta$ with $t(\gamma_1)=s(\gamma_2)\in tag(\Delta)$.

$\bullet$ $t_{\alpha}(t_{\gamma_1}t_{\gamma_2})^{-1}t_{\alpha'}=t_{\overline \alpha'}(t_{\gamma_1}t_{\gamma_2})^{-1}t_{\overline \alpha}$ for any once-punctured cyclic bigon $(\alpha,\alpha')$ which encloses a tagged cyclic bigon $(\gamma_1,\gamma_2)$ in $\Delta$ with $s(\alpha)=s(\gamma_1)$. 
\end{definition}

The following is immediate.

\begin{lemma}\label{lem:quo}
$(a)$ The assignments $t_{\gamma}\mapsto t_{\overline\gamma}$ give a involutive automorphism $\overline{\cdot}: T_\Delta\to T_\Delta.$

$(b)$ For any surface $\Sigma$ with $I_{p,0}(\Sigma)\neq \emptyset$ and a triangulation $\Delta$, let $\widetilde\Sigma$ denote the surface obtained from $\Sigma$ by converting the points in $I_{p,0}(\Sigma)$ into ordinary punctures, and let $\widetilde\Delta$ be the triangulation of $\widetilde\Sigma$ corresponding to $\Delta$. Then 
$$\mathbb T_\Delta\cong \mathbb T_{\widetilde\Delta}/\langle t_\ell=t_\gamma t_{\overline\gamma} \rangle,$$
where $(\ell,\gamma)$ runs over all pairs such that $(\ell,\gamma,\overline\gamma)$ forms a self-folded triangle enclosing a point in $I_{p,0}(\Sigma)$.
\end{lemma}


Given a marked surface $\Sigma$ and an ordinary triangulation $\Delta$ of $\Sigma$, denote by $I_{p,1}(\Delta)$ the subset of $I_{p,1}(\Sigma)$ consisting of points that are centers of self-folded triangles.

The following is immediate from the definition.
\begin{lemma} 
(Microtagging)\label{lem:micro} 
Let $\Sigma$ be an oriented punctured surface, $\Delta$ be an ordinary triangulation of $\Sigma$. Then for any subset $P\subset I_{p,1}(\Sigma)$ the assignments $$t_{\gamma^{P\setminus I_{p,1}(\Delta)}}\mapsto 
\begin{cases} 
t_{\gamma_1}t_{\gamma_2}, &\text{if $\gamma$ is a loop of a self-folded triangle in $\Delta$ around a puncture in $P$,}\\
t_{\gamma^{P\setminus I_{p,1}(\Delta)}}, &\text{otherwise.}\\
\end{cases}
$$
define an isomorphism $\mu_{\Delta^P}:\TT_{\Delta^{P\setminus I_{p,1}(\Delta)}}\simeq \TT_{\Delta^P}$, 
where in the first case, $(\gamma_1,\gamma_2)$ is a tagged cyclic bigon enclosed by $\gamma$ in $\Delta^P$ with $s(\gamma_1)=s(\gamma)$. 
\end{lemma} 

The following is an immediate refinement of \cite[Theorem 3.26]{BR}, obtained by combining that result with Lemmas \ref{lem:quo} and \ref{lem:micro}.

 \begin{theorem}\label{thm:1-relator}
   Let $\Sigma$ be an oriented marked surface with the Euler characteristic $\chi(\Sigma)$, the set $I=I(\Sigma)\ne \emptyset$ of marked points,  the set $I_b\subseteq I$ of marked boundary points, and $h=| I_{p,\ge 2}|$ special punctures. For any triangulation $\Delta$ of $\Sigma$ one has:

\noindent (a) If $\Sigma$ has a boundary or special punctures, then $\TT_\Delta$ is a free group in:

$\bullet$ $2$ generators if $\Sigma$ is a disk with $|I_b|=1$,  $|I_p|=0$ or a sphere with $|I_{p,0}|=|I_{p,1}|=h=1$, or a sphere with $|I_{p,0}|=h=0, |I_{p,1}|=2$.

$\bullet$ $3$ generators if $\Sigma$ is a sphere with $|I_{p,1}|=2$, $h=1$.

$\bullet$ $2h+3|I_{p,0}|+3|I_b|+4(|I_{p,1}|-\chi(\Sigma))$ generators otherwise.

\noindent (b) If $\Sigma$ is a closed surface with $h=0$, then $\TT_\Delta$ is isomorphic to:

$\bullet$ Trivial if $\Sigma$ is the sphere with $|I_{p,1}|=1$, $|I_{p,0}|=0$.

$\bullet$ A free group in $2|I_{p,0}|+3|I_{p,1}|-4$ generators if $\Sigma$ is the sphere with $|I_{p,0}|+|I_{p,1}|\in \{2,3\}$.

$\bullet$ A $1$-relator torsion free group in $3|I_{p,0}|+4(|I_{p,1}|-\chi(\Sigma))+1$ generators otherwise. 
\end{theorem}

From \cite[Lemma 3.50, Section 3.12]{BR}, we have the following. 

\begin{remark}\label{rmk:generator}
As in Theorem \ref{thm:1-relator}, if $\Delta$ is an ordinary triangulation, then the generators can be chosen to be of two types: either of the form $t_{\gamma}$ for some $\gamma\in \Delta$, or of the form $t_{\overline\gamma_1}^{-1}t_{\gamma_2}t_{\overline \gamma_3}^{-1}$ for some triangles $(\gamma_1,\gamma_2,\gamma_3)$ in $\Delta$. Moreover, for every ordinary puncture $i\in I_{p,1}$, there exists a generator $t_\gamma$ such that $s(\gamma)=i$. Furthermore, in the case $I_{p,0}=\emptyset$, if $\mathbb T_\Delta$ is a 1-relator torsion free group, then the single defining relation is of the form $t_{\overline \gamma_1}^{-1}t_{\gamma_2}t_{\overline\gamma_3}^{-1}t_{\gamma_4}\cdots t_{\overline\gamma_{2n-1}}^{-1}t_{\gamma_{2n}}$ for some composable sequence $(\gamma_1,\gamma_2,\cdots, \gamma_{2n})$ in $\Delta$. 
\end{remark}

\begin{proposition}[Tagging/untagging automorphisms]  \label{pro:tag/untag} 
Let $\Sigma$ be an oriented punctured surface, $\Delta$ be an ordinary triangulation of $\Sigma$, and $P\subset I_{p,1}(\Sigma)\setminus I_{p,1}(\Delta)$. Then the assignments 
$$t_\gamma\mapsto 
\begin{cases} 
t_{\overline\gamma}^{-1}, &\text{if $s(\gamma),t(\gamma)\in P$,}\\
t_{\alpha_4}t^{-1}_{\overline\alpha_3}, &\text{if $s(\gamma)\notin P, t(\gamma)\in P$,}\\
t^{-1}_{\overline\alpha_1}t_{\alpha_2}, &\text{if $t(\gamma)\notin P, s(\gamma)\in P,$}\\
t_{\gamma}, &\text{otherwise},\\
\end{cases}
$$
define an automorphism $\varphi_{P,\Delta}$ of $\TT_\Delta$,
where in the second case, $(\alpha_3,\alpha_4,\gamma)$ is the first cyclic triangle that $\gamma$ passes by rotation counterclockwise along $t(\gamma)$, in the third case, $(\alpha_1,\alpha_2,\overline\gamma)$ is the first cyclic triangle that $\gamma$ passes by rotation counterclockwise along $s(\gamma)$.
\end{proposition}

\begin{proof} 
For any clockwise cyclic triangle $(\gamma_1,\gamma_2,\gamma_3)$,

if $s(\gamma_1),s(\gamma_2),s(\gamma_3)\notin P$, then $$\varphi_{P,\Delta}(t_{\gamma_1}t^{-1}_{\overline\gamma_2}t_{\gamma_3})= t_{\gamma_1}t^{-1}_{\overline\gamma_2}t_{\gamma_3}=t_{\overline\gamma_3}t_{\gamma_2}^{-1}t_{\overline\gamma_1}=\varphi_{P,\Delta}(t_{\overline\gamma_3}t_{\gamma_2}^{-1}t_{\overline\gamma_1}),$$

if $|\{s(\gamma_1),s(\gamma_2),s(\gamma_3)\}\cap P|=1$, we may assume that $s(\gamma_1)\in P, s(\gamma_2), s(\gamma_3)\notin P$, then 
$$\varphi_{P,\Delta}(t_{\gamma_1}t^{-1}_{\overline\gamma_2}t_{\gamma_3})= 
(t_{\beta_2}^{-1}t_{\overline \beta_1})t^{-1}_{\overline\gamma_2}(t_{\overline \gamma_2}t^{-1}_{\gamma_1})
=t_{\beta_2}^{-1}t_{\overline \beta_1}t^{-1}_{\gamma_1}=t_{\overline\gamma_1}^{-1}t_{ \beta_1}t^{-1}_{\overline\beta_2}=\varphi_{P,\Delta}(t_{\overline\gamma_3}t_{\gamma_2}^{-1}t_{\overline\gamma_1}),$$
where $(\gamma_1,\beta_1,\beta_2)$ is the other cyclic triangle in $\Delta$, 

if $|\{s(\gamma_1),s(\gamma_2),s(\gamma_3)\}\cap P|=2$, we may assume that $s(\gamma_1), s(\gamma_2)\in P, s(\gamma_3)\notin P$, then

$$\varphi_{P,\Delta}(t_{\gamma_1}t^{-1}_{\overline\gamma_2}t_{\gamma_3})= t^{-1}_{\overline \gamma_1}
(t_{\beta_1}t_{\overline \beta_2})^{-1}(t_{\overline \gamma_2}t^{-1}_{\gamma_1})
=t^{-1}_{\overline \gamma_1}t_{\overline \beta_2}
t_{\beta_1}^{-1}t_{\overline \gamma_2}t^{-1}_{\gamma_1}=\varphi_{P,\Delta}(t_{\overline\gamma_3}t_{\gamma_2}^{-1}t_{\overline\gamma_1}),$$

if $|\{s(\gamma_1),s(\gamma_2),s(\gamma_3)\}\cap P|=3$, i.e., $s(\gamma_1), s(\gamma_2) s(\gamma_3)\in P$, then 

$$\varphi_{P,\Delta}(t_{\gamma_1}t^{-1}_{\overline\gamma_2}t_{\gamma_3})= t_{\overline\gamma_1}^{-1}t_{\gamma_2}t^{-1}_{\overline\gamma_3}=t^{-1}_{\gamma_3}t_{\overline\gamma_2}t^{-1}_{\gamma_1}=\varphi_{P,\Delta}(t_{\overline\gamma_3}t_{\gamma_2}^{-1}t_{\overline\gamma_1}).$$

The proof is complete.  
\end{proof}

The following is immediate.

\begin{lemma} In the assumptions of Proposition \ref{pro:tag/untag}, the assignments 
$t_{\gamma^P} \mapsto \varphi_{P,\Delta}(t_\gamma)$
define an isomorphism $\mu_{\Delta,\Delta^P}:\TT_{\Delta^P}\simeq \TT_{\Delta}$.
\end{lemma}

Based on this, for any ordinary triangulation $\Delta$ of $\Sigma$ and any $P\subset I_{p,1}(\Sigma)$ define an isomorphism $\mu_{\Delta,\Delta^P}:\TT_{\Delta^P}\simeq \TT_{\Delta}$ by 
$$\mu_{\Delta,\Delta^P}:= \mu_{\Delta,\Delta^{P\setminus I_{p,1}(\Delta)}} \circ (\mu_{\Delta^P})^{-1}.$$

Then for any
(tagged) triangulations $\Delta^P$, ${\Delta}^{P'}$ of $\Sigma$ define the isomorphism $\mu_{{\Delta}^{P'},\Delta^P}:\TT_{\Delta^P}\simeq T_{{\Delta}^{P'}}$ by
$$\mu_{{\Delta}^{P'},\Delta^P}:=(\mu_{{\Delta},\Delta^{P'}})^{-1}\circ \mu_{\Delta,\Delta^P}\ .$$

For any two ordinary triangulations $\Delta,\Delta'$ of $\Sigma$ related by a flip, we assume that $\Delta'=\mu_{\alpha}(\Delta)$ and $\alpha'\in \Delta'$ is not a pending arc.

\begin{lemma}
$(a)$ If $\alpha$ is not a loop around some pending arcs, then the following assignments
$$t_\gamma\mapsto \begin{cases}
    t_{\alpha_1}t^{-1}_{\overline \alpha'}t_{\overline \alpha_3}, & \text{ if } \gamma=\alpha,\\
    t_{\alpha_3}t^{-1}_{ \alpha'}t_{\overline \alpha_1}, & \text{ if } \gamma=\overline\alpha,\\ 
    t_\gamma, & \text{ otherwise}
\end{cases}$$
give an isomorphism 
$\mu_{\Delta',\Delta}:\mathbb T_\Delta\to \mathbb T_{\Delta'}$, where $(\alpha_1,\alpha_2,\alpha_3,\alpha_4)$ is the cyclic quadrilateral in $\Delta$ such that $(\alpha_3,\alpha_4,\alpha)$ is a cyclic triangle in $\Delta$ and $(\alpha_4,\alpha_1,\alpha')$ is a cyclic triangle in $\Delta'$. 

$(b)$ If $\alpha$ is a loop around a pending arc $\beta$ with $s(\alpha)=s(\beta)$, then the following assignments
$$t_\gamma\mapsto \begin{cases}
    t_{\alpha_1}t^{-1}_{\overline \beta'}, & \text{ if } \gamma=\beta,\\
    t^{-1}_{ \beta'}t_{\overline \alpha_1}, & \text{ if } \gamma=\overline\beta,\\
    t_{\alpha_1}t^{-1}_{\alpha'}t_{\overline \alpha_1}, & \text{ if } \gamma=\alpha \text{ or } \overline\alpha,\\
    t_\gamma, & \text{ otherwise}
\end{cases}$$
give an isomorphism 
$\mu_{\Delta',\Delta}:\mathbb T_\Delta\to \mathbb T_{\Delta'}$, where $(\alpha_1,\alpha_2,\beta,\overline\beta)$ is the cyclic quadrilateral in $\Delta$ such that $(\alpha_1,\alpha_2,\overline\alpha)$ is a cyclic triangle in $\Delta$ and $\beta'$ is the pending arc enclosed by $\alpha'$. 
\end{lemma}

\begin{proof}
$(a)$ For triangle $(\alpha_1,\alpha_2,\overline\alpha)$ in $\Delta$, we have

$$\mu_{\Delta',\Delta}(t_{\alpha_1}t^{-1}_{\overline\alpha_2}t_{\overline\alpha})=t_{\alpha_1}t^{-1}_{\overline\alpha_2}t_{\alpha_3}t^{-1}_{ \alpha'}t_{\overline \alpha_1}=t_{\alpha_1}t^{-1}_{\overline\alpha'}t_{\overline\alpha_3}t^{-1}_{ \alpha_2}t_{\overline \alpha_1}=\mu_{\Delta',\Delta}(t_{\alpha}t^{-1}_{\alpha_2}t_{\overline\alpha_1}).$$ 

Similarly, we have $\mu_{\Delta',\Delta}(t_{\alpha_3}t^{-1}_{\overline\alpha_4}t_{\alpha})=\mu_{\Delta',\Delta}(t_{\overline\alpha}t^{-1}_{\alpha_4}t_{\overline\alpha_3})$.
Thus, we have a group homomorphism $\mu_{\Delta',\Delta}:\mathbb T_\Delta\to \mathbb T_{\Delta'}$.

Similarly, assignments
$$t_\gamma\mapsto \begin{cases}
    t_{\overline\alpha_4}t^{-1}_{\alpha'}t_{\alpha_2}, & \text{ if } \gamma=\alpha,\\
    t_{\overline\alpha_2}t^{-1}_{\overline\alpha'}t_{\alpha_4}, & \text{ if } \gamma=\overline\alpha,\\ 
    t_\gamma, & \text{ otherwise}
\end{cases}$$
give a group homomorphism 
$\mu^{-}_{\Delta',\Delta}:\mathbb T_\Delta\to \mathbb T_{\Delta'}$.

Moreover, we have $\mu^-_{\Delta,\Delta'}\circ \mu_{\Delta',\Delta}=id_{\TT_\Delta}$ and $\mu_{\Delta,\Delta'}\circ \mu^{-}_{\Delta',\Delta}=id_{\TT_{\Delta'}}$. Thus, $\mu_{\Delta',\Delta}$ is an isomorphism.

Our proof of the statement $(b)$ is similar to $(a)$, so we omit it.

The proof is complete.
\end{proof}

For any $P\subset I_{p,1}(\Sigma)$, define $\mu_{\Delta'^P,\Delta^{P}}:=\mu_{\Delta'^{P},\Delta'}\mu_{\Delta',\Delta}(\mu_{\Delta^P,\Delta})^{-1}:\TT_{\Delta'^P}\to \mathbb T_{\Delta^{P}}$. It follows that the following diagram commutes
$$\centerline{\xymatrix{
  & \TT_{\Delta} \ar[d]_{\mu_{\Delta',\Delta}} \ar[rr]^{\mu_{\Delta^{P},\Delta}}  &&  \TT_{\Delta^P}\ar[d]^{\mu_{\Delta'^{P},\Delta^{P}}}           \\
  & \TT_{\Delta'} \ar[rr]^{\mu_{\Delta'^{P},\Delta'}}  && \TT_{\Delta'^{P}}.}}$$ 

\begin{proposition}
\label{pro:nu}
$(a)$ For any vertical morphism $v_{f,\Delta,\underline\Delta}$ in ${\bf TSurf}$ the assignments 
\[t_{\gamma}\mapsto 
\begin{cases}
  t_{f(\gamma)}, & \mbox{if $f(\gamma)$ is an arc,}\\
  t_{\ell}, & \text{if $f(\gamma)$ is a loop around a special puncture $\underline p$ with self-crossing,}
\end{cases}
\]
where $\ell$ is the special loop around the special puncture $\underline p$ in $\Delta$,
define a homomorphism of groups
$\nu_{f,\Delta,\underline\Delta}:\TT_{\Delta}\to  \TT_{\underline\Delta}$.

$(b)$ $\nu_{f',\underline \Delta,\underline \Delta'}\nu_{f,\Delta,\underline \Delta}=\nu_{f'\circ f,\Delta,\underline \Delta'}$ for any morphisms $f:|\Delta|\to |\underline \Delta|$, $f':|\underline \Delta|\to |\underline \Delta'|$ in ${\bf Surf}$ such that $(\Delta,\underline\Delta)$ is an $f$-compatible pair and $(\underline\Delta,\underline\Delta')$ is an $f'$-compatible pair.
\end{proposition}

\begin{proof}
We only prove (a), as (b) is clear.

{For any special loop $\gamma \in \Delta$, the image $f(\gamma)$ is also a special loop in $\underline\Delta$, which implies that $t_{f(\gamma)}=t_{\overline {f(\gamma)}}=t_{f(\overline{\gamma})}$.}

For any loop $\ell$ enclosing a pending arc $\gamma$ with $s(\gamma)=s(\ell)$, the image $f(\ell)$ is also a loop enclosing the pending arc $f(\gamma)$; hence $t_{f(\ell)}=t_{f(\gamma)}t_{\overline{f(\gamma)}}=t_{f(\gamma)}t_{f(\overline{\gamma})}$.

For any triangle $(\gamma_1,\gamma_2,\gamma_3)$ in $\Delta$, if $(f(\gamma_1),f(\gamma_2),f(\gamma_3))$ forms a triangle in $\underline\Delta$, then
\[
t_{f(\gamma_1)}t_{\overline{f(\gamma_2)}}^{-1}t_{f(\gamma_3)}=t_{\overline{f(\gamma_3)}}t_{{f(\gamma_2)}}^{-1}t_{\overline{f(\gamma_1)}}.
\]

If one of $f(\gamma_1), f(\gamma_2), f(\gamma_3)$ is a loop around a special puncture with self-crossing, replacing it with the corresponding special loop yields a triangle $(\underline\gamma_1,\underline\gamma_2,\underline\gamma_3)$ in $\underline \Delta$ or a degenerate triangle $(\ell,\ell,\ell)$. In both cases, we have $\nu_{f,\Delta,\underline\Delta}(t_{\gamma_i})=t_{\ell}$ for all $i=1,2,3$. Thus, the relation
\[
t_{f(\gamma_1)}t_{\overline{f(\gamma_2)}}^{-1}t_{f(\gamma_3)}=t_{\overline{f(\gamma_3)}}t_{{f(\gamma_2)}}^{-1}t_{\overline{f(\gamma_1)}}
\]
holds.

Therefore, we obtain a group homomorphism $\nu_{f,\Delta,\underline\Delta}$.

The proof is complete.
\end{proof}

For any $f:|\Delta|\to |\underline \Delta|$ in ${\bf Surf}$ and $P\subset I_{p,1}(|\Delta|)$ such that $f(\Delta)\subset \underline\Delta$ and $f(P)\subset I_{p,1}(|\underline\Delta|)$, define $\nu_{f,\Delta^P,\underline\Delta^{f(P)}}:=\mu_{\underline\Delta^{f(P)},\underline\Delta}\nu_{f,\Delta,\underline\Delta}(\mu_{\Delta^P,\Delta})^{-1}:\TT_{\Delta^P}\to \mathbb T_{\underline\Delta^{f(P)}}$. It follows that the following diagram commutes
$$\centerline{\xymatrix{
  & \TT_{\Delta} \ar[d]_{\nu_{f,\Delta,\underline\Delta}} \ar[rr]^{\mu_{\Delta^{P},\Delta}}  &&  \TT_{\Delta^{P}}\ar[d]^{\nu_{f,\Delta^{P},\underline\Delta^{f(P)}}}           \\
  & \TT_{\underline\Delta} \ar[rr]^{\mu_{\underline\Delta^{f(P)},\underline\Delta}}  && \TT_{\underline\Delta^{f(P)}}.}}$$ 

\begin{theorem} [Triangular functor]
\label{th:monomial mutation surfaces}
The  assignments $\Delta\mapsto \TT_\Delta$, $h_{\Delta',\Delta}\mapsto \mu_{\Delta',\Delta}$ for $\Delta,\Delta'\in {\bf Tsurf}$ with $dist(\Delta,\Delta')=1$, $h_{\Delta,\Delta^P}\mapsto \mu_{\Delta,\Delta^P}$ for all $P\subset I_{p,1}(|\Delta|)$ and  
$v_{f,\Delta^P,\underline\Delta^{f(P)}}\mapsto \nu_{f,\Delta^P,\underline\Delta^{f(P)}}$ for all $f:|\Delta|\to |\underline\Delta|\in {\bf Surf}$ and $P\subset I_{p,1}(|\Delta|)$ such that $f(\Delta)\subset \underline\Delta$ and $f(P)\subset I_{p,1}(|\underline\Delta|)$ define a functor ${\bf F}:{\bf TSurf}^t\to {\bf Grp}$, the category of groups. 
\end{theorem}

We prove Theorem \ref{th:monomial mutation surfaces} in Section \ref{Proof of Theorem th:monomial mutation surfaces}.

\begin{remark} 
\label{rem:canonical triangle group}
In the notation of  Lemma \ref{le:unique up to conjugation homomorphism} (a), we abbreviate $\TT_\Sigma:=G({\bf F}_\Sigma)$ and $\TT_\Sigma:=G({\bf F}_\Sigma)$, where ${\bf F}_\Sigma$ is the restriction of ${\bf F}$ to ${\bf Tsurf}_\Sigma$ and think of it as a canonical triangle group, which is obviously a topological invariant. Thus, Lemma \ref{le:unique up to conjugation homomorphism} (a)
guarantees that the assignments $\Sigma\mapsto \TT_\Sigma$ and $\Sigma\mapsto Br_\Sigma$ are functorial up  to inner automorphisms. 
Similarly, the natural action of $Br_\Sigma$ on $\TT_\Sigma$ is well-defined up to inner automorphisms of $Br_\Sigma$ and automorphisms of $\TT_\Sigma$ by Lemma \ref{le:unique up to conjugation homomorphism} (e.g., the kernel of the corresponding homomorphism $\rho_\Sigma:Br_\Sigma\to Aut(\TT_\Sigma)$ is well-defined).

\end{remark}


\begin{conjecture}
\label{conj:faithful}
Let $\Sigma$ be a connected marked surface different from a sphere with $4$ punctures or projective plane with $2$ punctures. The restriction of {\bf F} to ${\bf Tsurf}_\Sigma^t$ is faithful.
\end{conjecture}

We will see in Example \ref{ex:sphere with 4 points} that the restriction of {\bf F} to ${\bf Tsurf}_\Sigma^t$ is not faithful in case $\Sigma$ is a sphere with $4$ punctures or projective plane with $2$ punctures.

\begin{remark}
 \label{rm:monomial mutation surfaces} 
Given triangulations $\Delta$ and $\Delta'$ of a marked surface $\Sigma$, we denote by $\mu_{\Delta,\Delta'}={\bf F}(h_{\Delta,\Delta'}):\TT_{\Delta'}\simeq \TT_\Delta$ and call it the \emph{monomial mutation} from $\TT_{\Delta'}$ to $\TT_{\Delta}$.
\end{remark}



Thus, we obtain a group homomorphism $\pi_\Delta:Br_\Delta\to Aut(\TT_\Delta)$. Denote its image by $\underline{Br}_\Delta$ and call it \emph{cluster braid group} of $\Delta$.

\begin{corollary} \label{cor:brmuta}
$\underline{Br}_{\Delta'}=\mu_{\Delta',\Delta} \,\underline{Br}_{\Delta}\,\mu_{\Delta',\Delta}^{-1}$ for any triangulations $\Delta$ and $\Delta'$ of any $\Sigma$.
\end{corollary}

Given a morphism $f:\Sigma\to \Sigma'$ in ${\bf Surf}$, for any $f$-admissible $\Delta\in {\bf TSurf}_\Sigma^t$, denote by $\underline Br_\Delta^f$ the image $\pi_\Delta(Br_\Delta^f)$ in $Aut(\TT_\Delta)$ (we sometimes refer to it as the relative cluster braid group of $\Delta$).

Denote by $\underline{\underline {Br}}_\Delta^f$ the set of all $g\in \underline Br_\Delta$ preserving the kernel $K_f$ of the structure homomorphism 
$\TT(f):\TT_\Delta\to \TT_{\Delta'}$. Clearly,  $\underline Br_\Delta^f\subset \underline{\underline {Br}}_\Delta^f$.

We can conjecture that this is an equality. The indirect verification is the following immediate.

\begin{lemma} For any $f:\Sigma\to \Sigma'$, any $f$-admissible triangulation $\Delta$ of $\Sigma$ and any triangulation $\Delta'$ of $\Sigma'$ containing $f(\Delta)$ one has:

$(a)$ A functorial homomorphism of groups $\underline {Br}_\Delta^f\to \underline {Br}_{f(\Delta)}$.

$(b)$ {The assignments $g\mapsto g\cdot K_f$ define a functorial group homomorphism $\underline{\underline{Br}}_\Delta^f \to \operatorname{Aut}(\TT_{f(\Delta)})$. Its restriction to the subgroup $\underline{Br}_\Delta^f \subset \underline{\underline{Br}}_\Delta^f$ coincides with the homomorphism in (a).}
    
\end{lemma}

\begin{theorem} 
\label{th:functor Sigma}
Let $\Sigma=\Sigma_n$ or $\Sigma_n$ with one special puncture. Then the restriction of ${\bf F}$ to ${\bf TSurf}_\Sigma^t$ is a faithful functor of groupoids ${
\bf F}_\Sigma: {\bf TSurf}_\Sigma^t\to {\bf Grp}'$, the groupoid whose objects are groups and arrows are group isomorphisms.

\end{theorem}

\begin{proof} 
    It follows by Theorems \ref{th:Br on Sigman} and \ref{th:Br on Sigman1} in Section \ref{subsec:triangle groups}.
\end{proof}

\subsection{Braid monoid and group actions on triangle groups}
Theorem \ref{th:functor Sigma} implies that the braid group $Br_\Delta$ acts on $\TT_\Delta$, we explicitly compute this action here.

\begin{theorem} \label{th:action}   
For any (tagged) triangulation $\Delta$, $Br_\Delta$ acts on $\TT_\Delta$ as follows.  For any non-pending internal edge $\gamma\in \Delta$,

$(a)$ if $\gamma$ is not a loop around some pending arc, then
\[\begin{array}{ccl} T_{\gamma,\Delta}(t_\beta)= 
\left\{\begin{array}{ll}
t_{\beta}, &\mbox{if $\beta\neq \gamma, \overline \gamma$}, \vspace{1mm}\\
t_{\alpha_1}t_{\overline{\alpha}_2}^{-1}t_{\alpha_3}t_{\overline{\alpha}_4}^{-1}t_{\gamma}, &\mbox{if $\beta=\gamma$},\vspace{1mm}\\
t_{\overline\gamma}t_{\alpha_4}^{-1}t_{\overline\alpha_3}t_{\alpha_2}^{-1}t_{\overline\alpha_1}, &\mbox{if $\beta=\overline\gamma$},
\end{array}\right.
\end{array}\]
where $\gamma\in \Delta$ is a diagonal of some clockwise quadrilateral $(\alpha_1,\alpha_2,\alpha_3,\alpha_4)$ in $\Delta$ such that 
$(\gamma,\alpha_3,\alpha_4)$ is a cyclic triangle in $\Delta$, see Figure \ref{Fig:action}.

$(b)$ if $\gamma$ is a loop around some pending arc $\alpha$ with $s(\gamma)=s(\alpha)$, then
\[\begin{array}{ccl} T_{\gamma,\Delta}(t_\beta)= 
\left\{\begin{array}{ll}
t_{\beta}, &\mbox{if $\beta\neq \alpha,\overline\alpha, \gamma, \overline \gamma$}, \vspace{1mm}\\
t_{\alpha_1}t_{\overline{\alpha}_2}^{-1}t_{\alpha}, &\mbox{if $\beta=\alpha$},\vspace{1mm}\\
t_{\overline\alpha}t_{\alpha_2}^{-1}t_{\overline\alpha_1}, &\mbox{if $\beta=\overline\alpha$},\vspace{1mm}\\
t_{\alpha_1}t_{\overline{\alpha}_2}^{-1}t_{\gamma}t_{\alpha_2}^{-1}t_{\overline\alpha_1}, &\mbox{if $\beta=\gamma$ or $\overline \gamma$},
\end{array}\right.
\end{array}\]
where $\gamma\in \Delta$ is a diagonal of some clockwise quadrilateral $(\alpha_1,\alpha_2,\alpha,\overline\alpha)$ in $\Delta$ such that 
$(\alpha_1,\alpha_2,\overline\gamma)$ is a cyclic triangle in $\Delta$, see Figure \ref{Fig:action}.

\begin{figure}[ht]
\includegraphics{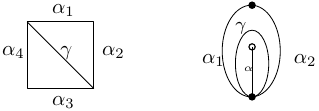}
\caption{}
\label{Fig:action}
\end{figure}
\end{theorem}

\begin{remark} Conjecture \ref{conj:faithful} implies that $Br_\Delta=\underline{Br}_\Delta$, that is, the above action of $Br_\Delta$ on $\TT_\Delta$ is faithful for any triangulation $\Delta$ of $\Sigma$ (with the aforementioned exception).    
\end{remark}

The following is immediate from Theorem \ref{th:Br_n 1}.

\begin{corollary}\label{Cor:CD}
$(a)$ For any triangulation $\Delta$ of the $n$-gon with one $0$-puncture, the group $\underline{Br}_\Delta$ is isomorphic to a quotient of the Artin braid group $Br_{B_{n-1}}$.

$(b)$ 
For any triangulation $\Delta$ of the once-punctured $n$-gon, the group $\underline{Br}_\Delta$ is isomorphic to a quotient of the Artin braid group $Br_{D_{n}}$.
\end{corollary}

\begin{example}
\label{ex:D_2} Let $\Sigma=\Sigma_{2,1}$ be the once-punctured bigon with boundary marked points are labeled $1,2$ and puncture labeled $0$. For triangulation $\Delta=\{(0,1), (0,2), (1,2)^+, (1,2)^-\}$, the triangle group $\mathbb T_\Delta$ is generated by $t^{\pm}_{12}, t^{\pm}_{21}, t_{10}, t_{01}, t_{02}, t_{20}$ subject to $t_{01}(t_{21}^{\pm})^{-1}t_{20}=t_{02}(t_{12}^{\pm})^{-1}t_{10}$. The automorphism $T_{01}, T_{02}\in Aut(\mathbb T_{\Delta})$ are given by
 \begin{equation*}
 T_{01}(t_{\gamma})=\begin{cases}
     t_{12}^{-}(t^{+}_{12})^{-1}t_{10} & \text{if $\gamma=(1,0)$}\\
     t_{01}(t^{+}_{21})^{-1}t^{-}_{21} & \text{if $\gamma=(0,1)$}\\

          t_\gamma & \text{otherwise}\\

 \end{cases},\;\;
 T_{02}(t_{\gamma})=\begin{cases}
     t_{21}^{+}(t^{-}_{21})^{-1}t_{20} & \text{if $\gamma=(2,0)$}\\
     t_{02}(t^{-}_{12})^{-1}t^{+}_{12} & \text{if $\gamma=(0,2)$}\\

          t_\gamma & \text{otherwise}\\

 \end{cases},
 \end{equation*}

The corresponding braid monoid is the monoid generated by $T_{01}, T_{02}$, which is isomorphic to $\mathbb Z_{+}^2$, the braid group $\langle T_{01}, T_{02}\rangle \subset Aut(\mathbb T_\Delta)$ is isomorphic to ${\mathbb Z}^2\cong Br_{D_2}$.

\end{example}

\begin{example} Let $\widetilde\Sigma=\Sigma_2^2$ be the cylinder with $2$ marked points on each boundary and $\Sigma=\Sigma_{2,2}$ be the bigon with $2$-punctures, and let $\pi:\widetilde\Sigma\to \Sigma$ be the map by gluing two boundary segments. Let $\widetilde\Delta$ and $\Delta$ be the triangulations of $\widetilde\Sigma$ and $\Sigma$, respectively, shown in Figure \ref{Fig:glu}.
Then the kernel of $\pi:\TT_{\widetilde \Delta}\to \TT_{\Delta}$ is the normal subgroup of $\TT_{\widetilde\Delta}$ generated by $(t^+_{pq})^{-1}t^{-}_{pq}, (t^+_{qp})^{-1}t^{-}_{qp}$.

\begin{figure}[ht]
\includegraphics{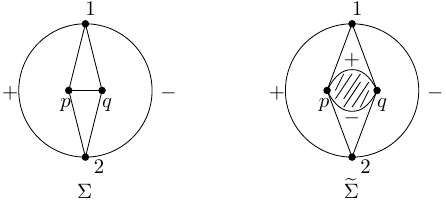}
\caption{}
\label{Fig:glu}
\end{figure}

Since the action of $\underline {Br}_{\widetilde \Delta}$ on $\TT_{\widetilde \Delta}$ fixes $(t^+_{pq})^{-1}t^{-}_{pq}$, $(t^+_{qp})^{-1}t^{-}_{qp}$, it induces an action of $\underline {Br}_\Delta$ on $\TT_{\Delta}$ via a group homomorphism $\underline{Br}_{\widetilde\Delta}\to \underline{Br}_{\Delta}$ given by $T_{\widetilde \gamma}\mapsto T_{\gamma}$ for any $\widetilde \gamma\in \widetilde\Delta$.

\end{example}


\begin{example}\label{ex:sphere with 4 points}
Let $\Sigma$ be the sphere $S^2$ with 4 marked points. Let $\Delta=\{\gamma_0,\gamma'_0,\gamma_i, \overline{\gamma_0},\overline{\gamma'_0},\overline{\gamma_i}\mid i=1,2,3,4\}$ be the triangulation of $\Sigma$, as shown in the picture on the left of Figure \ref{Fig:4sphere}. 

\begin{figure}[ht]
\includegraphics{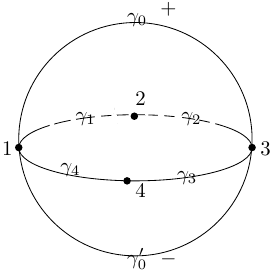}
\caption{4-punctured sphere with triangulation $\Delta$}
\label{Fig:4sphere}
\end{figure}

By calculation, we see that the actions of $T_{\gamma_0,\Delta}(T'_{\gamma_0,\Delta})^{-1}, T_{\gamma_1,\Delta}T_{\gamma_3,\Delta}^{-1}, T_{\gamma_2,\Delta}(T_{\gamma_4,\Delta})^{-1}$ on $\mathbb T_\Delta$ are pairwise commutative. Therefore, $\pi_\Delta: {Br}_{\Delta}\to \underline{Br}_{\Delta}$ is not an isomorphism in this case.

Let $\widetilde \Sigma$ be the twice punctured bigon with triangulation $\widetilde \Delta$, as shown in Figure \ref{Fig:22gon}. Then $\Sigma$ can be obtained from $\widetilde \Sigma$ by gluing $14^+$
and $14^-$.

\begin{figure}[ht]
\includegraphics{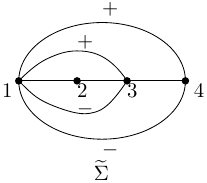}
\caption{Twice punctured bigon}
\label{Fig:22gon}
\end{figure}

By calculation, we have $T_{13^+,\widetilde\Delta}(T_{13^-,\widetilde\Delta})^{-1}T_{12,\widetilde\Delta}T_{34,\widetilde\Delta}^{-1}\neq T_{12,\widetilde\Delta}T_{34,\widetilde\Delta}^{-1}T_{13^+,\widetilde\Delta}(T_{13^-,\widetilde\Delta})^{-1}$. It follows that $\underline{Br}_{\widetilde\Delta}\to \underline{Br}_{\Delta}$ is not injective in this case.

\end{example}

\subsection{Sector groups and their reduced counterparts} 
\label{subsec:triangle groups}

For any pair of curves $\gamma,\gamma'$ in $\Sigma$ with $t(\gamma)=s(\gamma')$, denote $u_{\gamma,\gamma'}=t^{-1}_{\overline\gamma}t_{{\gamma'}}.$

\begin{definition}
\begin{enumerate}
    \item For any ordinary triangulation $\Delta$ of $\Sigma$, define \emph{reduced triangle group} $\underline \TT_\Delta$ as the quotient of $\TT_\Delta$ by relations $t_\gamma=1$ for all boundary arcs $\gamma$. 
    \item We also define \emph{sector group} $\UU_\Delta$ of $\Delta$ to be  subgroup 
\[\mathbb{U}_{\Delta}:=\langle u_{\gamma,\gamma'}\mid t(\gamma)=s(\gamma')\text{ and } \gamma,\gamma' \in\Delta  \rangle.\]
\end{enumerate}
\end{definition}

\emph{Reduced sector group} $\underline{\mathbb{U}}_{\Delta}$ associated with $\Delta$ is defined as the quotient of $\mathbb{U}_{\Delta}$ obtained by specializing $t_{\gamma}$ to $1$ for any boundary segments $\gamma$.

\begin{proposition}\label{prop:invariant}
Assignments $\Delta\mapsto \mathbb U_\Delta$ give a subfunctor of ${\bf F}|_{{\bf TSurf}}$, the restriction of ${\bf F}$ on ${\bf Tsurf}$, where ${\bf F}$ is the functor given in Theorem \ref{th:monomial mutation surfaces}. In particular, 
       $\UU_{\Delta}\cong \UU_{\Delta'}$ for any ordinary triangulations $\Delta,\Delta'$ of $\Sigma$ and $\UU_\Delta$ is invariant under the action of $Br_\Delta$ on $\TT_\Delta$. 
\end{proposition}

\begin{proof}
    For any $\Delta,\Delta'\in {\bf TSurf}$ with $dist(\Delta,\Delta')=1$, we have $\mu_{\Delta',\Delta}(\UU_\Delta)=\UU_{\Delta'}$. For any $f: |\Delta|\to |\underline \Delta|$ with $f(\Delta)\subset \underline\Delta$, we have $\nu_{f,\Delta,\underline\Delta}(\UU_\Delta)\subset \UU_{\underline\Delta}$. Therefore, the assignments $\Delta\mapsto \mathbb U_\Delta$ give a subfunctor of ${\bf F}|_{{\bf TSurf}}$. 

The proof is complete.
\end{proof}

\begin{theorem}
\label{th:Udelta2} Let $\Sigma$ be a marked surface with the Euler characteristic $\chi(\Sigma)$, the set $I=I(\Sigma)\ne \emptyset$ of marked points,  the set $I_b\subseteq I$  of marked boundary points, and $h=|\bigsqcup I_{p,\geq 2}|$ special punctures. Assume that $I_{p,0}=\emptyset$. Then for any triangulation $\Delta$ of $\Sigma$ one has:

\noindent (a) If $\Sigma$ has a boundary or special punctures, then $\UU_\Delta$ is:

$\bullet$ A free group in $1$ generator if $\Sigma$ is a disk with $|I_b\sqcup I_{p,1}|+|I_b|=2$, $h=0$.

$\bullet$ Trivial if $\Sigma$ is a disk with $|I_b\sqcup I_{p,1}|=|I_b|=h=1$.

$\bullet$ A free group in $2h-2$ generators if $\Sigma$ is a disk with $|I_b\sqcup I_{p,1}|=|I_b|=1$, $h>1$. 

$\bullet$ A free group in $2h+3|I|-4\chi(\Sigma)-|I_b|$ generators otherwise.

\noindent (b) If $\Sigma$ is a closed surface without special punctures, then $\UU_\Delta$ is:

$\bullet$ Trivial if $\Sigma$ is the sphere with $|I_b\sqcup I_{p,1}|\in \{1,2\}$.

$\bullet$ A free group in $2|I_b\sqcup I_{p,1}|-4$ generators if $\Sigma$ is the sphere with $|I_b\sqcup I_{p,1}|=3$.

$\bullet$ A $1$-relator torsion free group in $3|I_b\sqcup I_{p,1}|-4\chi(\Sigma)+1$ generators otherwise.
\end{theorem}

The following statements are the main results of this section.

\begin{theorem}\label{thm:sectorgroup} 
For any surface $\Sigma$ with $I_{p,0}=\emptyset$ and triangulation $\Delta$, we have 
\begin{equation}
\label{eq:free sector factorization}
\mathbb T_{\Delta}\cong \mathbb U_{\Delta}* F_{|I_b\cup I_{p,1}|},
\end{equation}
{where $F_{|I_b\cup I_{p,1}|}$ is the free group of rank $|I_b\cup I_{p,1}|$}
Moreover, for any ordinary triangulation $\Delta$, $\mathbb U_\Delta$ is generated by $u_{\overline\gamma,\gamma'}$, where $(\gamma,\gamma')$ runs over all the pair of arcs in $\Delta$ having the same starting point and forming two sides of some triangle in $\Delta$, subject to
\begin{enumerate}[$(1)$]
    \item $u_{\overline \gamma,\gamma'}u_{\overline\gamma',\gamma}=1$.
    
    \item (Triangle relations) {$u_{\gamma_1,\gamma_2}u_{\gamma_3,\gamma_1}u_{\gamma_2,\gamma_3}=1$} for any triangle $(\gamma_1,\gamma_2,\gamma_3)$ in $\Delta$.
    
    \item (Star relations) $u_{\overline \gamma_1,\gamma_2}u_{\overline \gamma_2,\gamma_3}\cdots u_{\overline \gamma_k,\gamma_1}=1$ for any puncture $i$, where $\gamma_1,\gamma_2,\cdots,\gamma_t$ are the arcs in $\Delta$ incident to $i$ in clockwise order with $s(\gamma_1)=s(\gamma_2)=\cdots =s(\gamma_k)=i.$  
\end{enumerate}
    
\end{theorem}

We prove Theorems \ref{th:Udelta2} and \ref{thm:sectorgroup} in Section \ref{sec:proof of thm:sectorgroup}.

\begin{theorem} 
\label{th:Br on Sigman}
For any triangulation $\Delta$ of $\Sigma_n$, 

$(a)$ the action of $Br_\Delta$ on $\UU_\Delta$ is faithful.

$(b)$ The action of $Br_\Delta$ on $\TT_\Delta$ is faithful.

$(c)$ $Br_\Delta$ is isomorphic to $Br_{n-2}$. Moreover, if all internal edges of $\Delta$ are $(1,i)$, $i=3,\ldots,n-1$, then the assignments $T_i\mapsto T_{(1,i+2)}$, $i=1,\ldots,n-3$ define an isomorphism of groups $Br_\Delta\simeq Br_{n-2}$.
    
\end{theorem}

\begin{theorem} 
\label{th:Br on Sigman1}
For any triangulation $\Delta$ of $\Sigma$, the $n$-gon with one special puncture, 

$(a)$ the action of $Br_\Delta$ on $\UU_\Delta$ is faithful.

$(b)$ The action of $Br_\Delta$ on $\TT_\Delta$ is faithful.

$(c)$ $Br_\Delta$ is isomorphic to $Br_{C_{n-1}}$, the Artin braid group of type $C_{n-1}$.
\end{theorem}

We prove Theorems \ref{th:Br on Sigman}, \ref{th:Br on Sigman1} in Sections \ref{subsec:Proof of Theorem th:Br_n 1} and \ref{subsec:proof of Theorem clusterbraidgrouptypeB}, respectively.



\begin{theorem}
\label{th:reduced sector=reduced triangle}
Let $\Sigma$ be a marked surface with $I_{p,0}=\emptyset$ and $\Delta$ be a triangulation. The reduced sector group $\underline{\mathbb{U}}_{\Delta}$ coincides with the reduced triangle group $\underline{\mathbb{T}}_{\Delta}$ if and only if $I_{p,1}=\emptyset$.
\end{theorem}

\begin{proof}
It follows from Theorem \ref{thm:sectorgroup} that $\underline{\mathbb{U}}_{\Delta}=\underline{\mathbb{T}}_{\Delta}$ if $I_{p,1}=\emptyset$.

Suppose that $\Sigma$ is a closed surface. Then $\underline{\mathbb{U}}_{\Delta}=\mathbb{U}_{\Delta}$ and $\underline{\mathbb{T}}_{\Delta}=\mathbb{T}_{\Delta}$. According to the definition, $\mathbb{U}_{\Delta}$ is {included in} the degree $0$ part of $\mathbb{T}_{\Delta}$. It follows that $\mathbb{U}_{\Delta}$ is a proper subgroup of $\mathbb{T}_{\Delta}$.

This completes the proof.
\end{proof}

\begin{proposition}\label{prop:faithfulequ} If $I_{p,0}(\Sigma)\cup I_{p,1}(\Sigma)=\emptyset$ then the following statements  are equivalent.

\noindent $(1)$ $Br_\Delta$-action on $T_\Delta$ is faithful.

\noindent $(2)$ The induced $Br_\Delta$-action on $\UU_\Delta$ is faithful.
\end{proposition}

\begin{proof} We need the following

\begin{lemma}

 Let $\TT$ be a group, $\UU$ be its subgroup such that $\TT=H*\UU$ for some other subgroup $H$. Then for any subgroup $B_H\subset Aut(H)$ and $B_{\UU}\subset Aut(\UU)$ the natural action of $B_\TT=B_H\times B_\UU$ on $\TT$ is faithful.

 \end{lemma}

 \begin{proof} Clearly, any homomorphism $f_H:H\to H$ lifts uniquely to a homomorphism $\hat f_H:\TT\to \TT$ and any homomorphism $f_\UU:\UU\to \UU$ lifts uniquely to a  homomorphism $\hat f_\UU:\TT\to \TT$ and $\hat f_H\circ \hat f_\UU=\hat f_\UU\circ \hat f_H$.
 
 In particular, $\hat f_H=Id_{\TT}$ iff $f_H=Id_H$ and $\hat f_\UU=Id_{\TT}$ iff $f_\UU=Id_\UU$. This implies that homomorphism $Aut(H)\times Aut(\UU)\to Aut(\TT)$ taking $(f_H,f_\UU)$ to $(\hat f_H,\hat f_\UU)$ is injective.

 This complete the proof.
\end{proof}

Applying it to the case $\UU=\UU_\Delta$, $H=F_{I_b\sqcup I_p}$, $B_H=1$, and $B_\UU$ is the restriction of $Br_\Delta$ to $\UU$ (Proposition \ref{prop:invariant}), using \eqref{eq:free sector factorization}, we finish the proof of the proposition.
\end{proof}

We conjecture that Proposition \ref{prop:faithfulequ} holds for all surfaces (which may contain punctures).

\begin{theorem}
\label{thm:reduced oddgons}
For any $g\ge 0$, the group $\underline \TT_{\Sigma_{2g+3}}=\underline \UU_{\Sigma_{2g+3}}$ is isomorphic to the fundamental group of the closed surface of genus $g$.
\end{theorem}

\begin{proof}
Consider the star-like triangulation $\Delta=\{(1,i),(i,1)\mid i=2,\dots,2g+2\}\cup \{\text{boundary arcs}\}$ of $\Sigma_{2g+3}$. Then $\underline \TT_{\Sigma_{2g+3}}$ is generated by $t_{1i}$ and $t_{i1}$ for $i=3,\dots,2g+2$, subject to the following relations:
\begin{itemize}
    \item $t_{13}=t_{31}$ and $t_{1,2g+2}=t_{2g+2,1}$;
    \item $t_{1i}t_{i+1,1}=t_{1,i+1}t_{i1}$ for $i \in \{2,\dots,2g+1\}$.
\end{itemize}
Consequently, $\underline \TT_{\Sigma_{2g+3}}$ is generated by $t_i:=t_{i1}$ for $i=3,\dots,2g+2$, subject to the single relation:
\[
[t_3,t_4][t_4,t_5][t_5,t_6]\cdots [t_{2g+1},t_{2g+2}]=1,
\]
where $[a,b]:=aba^{-1}b^{-1}$.

Now, we change the generators by setting $a_i=t_it_{i+2}^{-1}$ for odd $i \in \{3, 5, \dots, 2g-1\}$, $a_i=t_{i+1}t_{i}t_{i+1}^{-1}$ for even $i \in \{4, 6, \dots, 2g\}$, and $a_{2g+1}=t_{2g+1}, a_{2g+2}=t_{2g+2}$. Using the identity $[a,b][b,c]=[ac^{-1},cbc^{-1}]$, we deduce that $\underline \TT_{\Sigma_{2g+3}}$ is generated by $a_i$ for $i=3,\dots,2g+2$, subject to the standard surface relation:
\[
[a_3,a_4][a_5,a_6]\cdots [a_{2g+1},a_{2g+2}]=1.
\]
This presentation is isomorphic to the fundamental group of the closed surface of genus $g$. The proof is complete.
\end{proof}

\begin{corollary}  
\label{cor:Udelta reduced}
In the notation of Theorem \ref{th:Udelta2}, if $I_{p,0}(\Sigma)\cup I_{p,1}(\Sigma)=\emptyset$, then $\underline \UU_\Sigma=\underline \TT_\Sigma$ is a one-relator torsion free group in  $|I_b|+6|I_{p,\geq 2}| +1 -4\chi(\Sigma)$ generators.
\end{corollary}

\begin{proof}
According to \cite[Theorems 3.49 and 3.50]{BR} (see also \cite[Page 41]{BR}), the group $\TT_\Sigma$ is generated by $4(|I|-\chi(\Sigma))+2|I_{p,\geq 2}|-|I_b| +1$ elements subject to a single relation. Moreover, the elements $t_\gamma$ and $t_{\overline \gamma}$ for all boundary arcs $\gamma$ are included in this generating set. Consequently, $\underline\TT_\Sigma$ is generated by $4(|I|-\chi(\Sigma))+2|I_{p,\geq 2}|-3|I_b| +1$ elements, also subject to one relation.

Since $I_{p,0}(\Sigma)\cup I_{p,1}(\Sigma)=\emptyset$, we have the decomposition $I=I_b\sqcup I_{p,\geq 2}$. Substituting this into the expression yields
\[
4(|I|-\chi(\Sigma))+2|I_{p,\geq 2}|-3|I_b| +1 = |I_b|+6|I_{p,\geq 2}|+1-4\chi(\Sigma).
\]
The result follows.
\end{proof}

\subsection{Rank $2$ cluster groups and braid action}\label{sec:Rank $2$ cluster groups and braid action}

Given $r_1,r_2\in {\mathbb Z}_{\ge 0}$ such that $r_1=0$ if and only if $r_2=0$. Denote $m=\begin{cases}
    2, & \text{ if } r_1r_2=0,\\
    3, & \text{ if } r_1r_2=1,\\
    4, & \text{ if } r_1r_2=2,\\
    6, & \text{ if } r_1r_2=3,\\
    0, & \text{ if } r_1r_2\geq 4.
\end{cases}$

Denote by $\TT_k:=\langle t_k,t_{k+1}\rangle$ the free group freely generated by $t_k,t_{k+1}$, $k\in \Z$.

If $k$ is odd, let 
$$\mu_{k,k+1}: \TT_k\to \TT_{k+1},\quad t_k\mapsto t_{k+2}t_{k+1}^{r_2}, t_{k+1}\mapsto t_{k+1}$$ 
$$\mu_{k+1,k}: \TT_{k+1}\to \TT_{k},\quad t_{k+1}\mapsto t_{k+1}, t_{k+2}\mapsto t_{k+1}^{-r_1}t_{k}$$ 
 be the group isomorphisms.

If $k$ is even, let 
$$\mu_{k,k+1}: \TT_k\to \TT_{k+1},\quad t_k\mapsto t_{k+2}, t_{k+1}\mapsto t_{k+1}$$ 
$$\mu_{k+1,k}: \TT_{k+1}\to \TT_{k},\quad t_{k+1}\mapsto t_{k+1}, t_{k+2}\mapsto t_k$$ 
 be the group isomorphisms.

Let $\sigma_k: \TT_k\to \TT_{k+m+2}$ be the isomorphism given by 
$$t_k\mapsto \begin{cases}
    t_{k+m+2}, & \text{ if } m \text{ is even,}\\
    t_{k+m+3}, & \text{ if } m \text{ is odd,}
\end{cases} \hspace{3mm}\quad t_{k+1}\mapsto \begin{cases}
    t_{k+m+3}, & \text{ if } m \text{ is even,}\\
    t_{k+m+2}, & \text{ if } m \text{ is odd}.
\end{cases}$$

\begin{theorem}
\label{th:faithful rank 2}
In the notation of Section \ref{subsec:rank 2 braid}, {assume that $r_1r_2\neq 0$, the}
assignments $k\mapsto \TT_k$ and $h_{i,i+1}\mapsto \mu_{i,i+1},~h_{i+1,i}\mapsto \mu_{i+1,i}, \sigma_i\mapsto\sigma_i$ define a faithful functor from $\Gamma_{m}$ to ${\bf Grp}'$, the groupoid of groups with isomorphisms. In particular, assignments $T^k_i\mapsto \underline T^k_i$, $i=1,2$ define a faithful action of $Br(I_2(m))$ on each $\TT_k\cong F_2$, where $T^k_1=h_{k,k+1}h_{k+1,k}, T^k_2=h_{k,k-1}h_{k-1,k}$, $\underline T^k_1=\mu_{k,k+1}\mu_{k+1,k}$ and $\underline T^k_2=\mu_{k,k-1}\mu_{k-1,k}$.
\end{theorem}

\begin{proof}
    Denote $\underline T^{r_1,r_2}_1=\mu_{12}\mu_{21}, \underline  T^{r_1,r_2}_2=\mu_{10}\mu_{01}$. Then $\underline  T^{r_1,r_2}_i(t_j)=t_j$ if $i\ne j$ and
$$\underline  T^{r_1,r_2}_1(t_1)=t_1t_2^{r_2},~\underline  T^{r_1,r_2}_2(t_2)=t_1^{-r_1}t_2.$$

In case $r_1r_2=1$, abelianizing $\underline  T^{1,1}_1, \underline  T^{1,1}_2$, we obtain automorphisms $T^{ab}_1=\begin{pmatrix}
     1 & 0\\
     1 & 1
 \end{pmatrix}, T^{ab}_2=\begin{pmatrix}
     1 & -1\\
     0 & 1
 \end{pmatrix}\in Aut(\mathbb Z^2),$ it follows by \cite[Section 1.1.4]{KV} that $$\langle T_1^{ab}, T_2^{ab}\rangle \cong \langle \sigma_1,\sigma_2\mid \sigma_1\sigma_2\sigma_1=\sigma_2\sigma_1\sigma_2, (\sigma_1\sigma_2\sigma_1)^4=1\rangle.$$ It is easy to check directly that $(\underline T^{1,1}_1\underline  T^{1,1}_2\underline  T^{1,1}_1)^n\neq 1$ for any $n\in \mathbb Z_{>0}$. Therefore, $$\langle \underline  T^{1,1}_1, \underline  T^{1,1}_2\rangle \cong \langle \tau_1, \tau_2\mid \tau_1\tau_2\tau_1=\tau_2\tau_1\tau_2\rangle=Br_3.$$

For any $r_1,r_2$, we see that $\langle \underline  T_1^{r_1,r_2}, \underline  T_2^{r_1,r_2}\rangle\cong \langle (\underline  T_1^{1,1})^{r_1}, (\underline  T_2^{1,1})^{r_2}\rangle \subseteq Br_{1,1}$. Then the result follows by \cite[Page 82]{FM}.

The proof is complete.
\end{proof}

\section{Noncommutative Laurent Phenomenon and the expansion formula}\label{sec:NC Laurent surfaces}

\subsection{Laurent phenomenon for noncommutative surfaces}\label{sec:noncomLaurent}

{Recall that $\Gamma(\Sigma)$ denotes the set of curves in $\Sigma$, and $[\Gamma(\Sigma)]$ denotes the set of $\bigsqcup I_{p,\geq 2}$-admissible curves in $\Sigma$, respectively, in Section \ref{subsec:Some notation on surfaces}.} Generalizing \cite[Section 3]{BR}, we establish the following. 

\begin{theorem} 
\label{th:iotaDelta}
Let $\Sigma\in {\bf Surf}$, $\Delta\in {\bf Tsurf}_\Sigma^t$. Then

$(a)$ the assignments $t_\gamma\mapsto x_\gamma$, $\gamma\in \Delta$ define homomorphism $\iota_\Delta:\kk  \TT_\Delta\to {\mathcal A}_\Sigma$.

$(b)$ $\mathcal A_\Sigma=\mathcal A_\Delta[{\bf S}^{-1}]$,  where $\mathcal A_\Delta$ is the subalgebra of $\mathcal A_\Sigma$ generated by $x_\gamma,\gamma\in \Gamma(\Sigma)$ and all $x^{-1}_\alpha, \alpha\in \Delta$, and ${\bf S}$ is the submonoid of $\mathcal A_\Delta$ generated by all $x_\gamma, \gamma\in [\Gamma(\Sigma)]$.

$(c)$ $\iota_\Delta$ is injective.
\end{theorem}

\begin{proof}

Use Theorem \ref{thm:1-relator}, the proof is similar to the proof of  \cite[Theorem 3.36, Corollary 3.37]{BR}.

\end{proof}

Recall that $T_i$ is the total angle at $i$ given by Proposition \ref{pro:total angle} and the following is immediate.

\begin{proposition}\label{prop:total}
    For any ordinary triangulation $\Delta$ of $\Sigma$, we have 
     $$T_i=\sum T_{(\gamma_1,\gamma_2,\gamma_3)}+\sum 2\cos(\frac{\pi}{|p|}) x^{-1}_{\ell_p},$$
where the first summation is over all clockwise triangles 
$(\gamma_1, \gamma_2, \gamma_3)$ in $\Delta$ such that $s(\gamma_1) = i$ and $T_{(\gamma_1,\gamma_2,\gamma_3)}=x_{\overline\gamma_1}^{-1}x_{\gamma_2}x_{\overline\gamma_3}^{-1}$, the second summation is over all clockwise loops $\ell_p$ enclose a special puncture $p$ with $s(\ell_p)=i$. 
\end{proposition}

For any curve $\gamma$ and $P\subset I_{p,1}(\Sigma)$, recall that we have the noncommutative tagged curve $x_{\gamma^P}=\varphi_P(x_\gamma)=T_{s(\gamma)}^{\chi_P(s(\gamma))}x_\gamma T_{t(\gamma)}^{\chi_P(t(\gamma))}$. 

The following is immediate.
\begin{lemma}\label{lem:varphi}
    $\varphi_{P'}(x_{\gamma^P})=x_{\gamma^{P''}}$ where $P''=P\ominus P'$ is the symmetric difference of $P$ and $P'$.
\end{lemma}

For any $P\subset I_{p,1}(\Sigma)$ and any ordinary triangulation $\Delta$, we extend $\iota_\Delta$ to a tagged triangulation $\Delta^P$ of $\Sigma$ by 
$$\iota_{\Delta^P}:=\varphi_P\circ \iota_\Delta\circ \mu_{\Delta,\Delta^P}\ $$
and refer to it as a non-commutative \emph{tagged} cluster.
By definition and Theorem \ref{th:iotaDelta}, $\iota_{\Delta^P}$ is a well-defined injective homomorphism from $\kk \TT_{\Delta^P}$ to ${\mathcal A}_\Sigma$.

In particular, $x_{\gamma}=\iota_{\Delta}(t_{\gamma})$ for all (tagged or ordinary) triangulation $\Delta$ and $\gamma\in \Delta$.

\begin{proposition}
    \label{pr:G-Laurent} For any $\Sigma$ and any $\Delta\in {\bf TSurf}^t_\Sigma$ one has:

$(a)$ The restriction of $\iota_\Delta$ to $\kk\UU_\Delta$ is a well defined injective homomorphism $\kk \UU_\Delta\hookrightarrow {\mathcal B}_\Sigma$, the sector subalgebra of ${\mathcal A}_\Sigma$ defined in Section \ref{subsec:From surfaces to their noncommutative versions}.

$(b)$ $\iota_\Delta$ naturally induces an injective homomorphism  of reduced algebras
$\underline {\iota}_\Delta: \kk \underline{\mathbb T}_\Delta\hookrightarrow \underline{\mathcal A}_\Sigma$ 

In turn, $\underline \iota_\Delta$ restricts to an injective homomorphism  $\kk \underline{\mathbb U}_\Delta\hookrightarrow{} \underline{\mathcal B}_\Sigma$.

\end{proposition}

\begin{proof} 
$(a)$ As $\iota_\Delta(\UU_\Delta)\subset \mathcal B_\Sigma$, we have $\iota_\Delta(\kk \UU_\Delta)\subset \mathcal B_\Sigma$ and
the following commutative diagram 
    $$\centerline{\xymatrix{
   \kk \UU_\Delta \ar@{^{(}->}[d]_{} \ar[r]^{\iota_\Delta}  &  \mathcal B_\Sigma\ar@{^{(}->}[d]
^{}           \\
    \kk \TT_\Delta \ar@{^{(}->}[r]^{\iota_\Delta}  & \mathcal A_\Sigma.}}$$ 
Thus $\iota_\Delta:\kk \UU_\Delta\to \mathcal B_\Sigma$ is injective.

$(b)$ As $\iota_\Delta(t_\gamma)=x_{\gamma}$ for any boundary arc $\gamma$, $\iota_\Delta$ induces an algebra homomorphism $\underline{\iota}_\Delta: \kk \underline \TT_\Delta\to \underline{\mathcal A}_\Sigma$ and the following commutative diagram.
   $$\centerline{\xymatrix{
  \kk \TT_\Delta \ar@{->>}[d]_{} \ar@{^{(}->}[r]^{\iota_\Delta}  &  \mathcal A_\Sigma\ar@{->>}[d]
^{}           \\
    \kk \underline{\TT}_\Delta \ar[r]^{\underline{\iota}_\Delta}  & \underline{\mathcal A}_\Sigma.}}$$

To show that $\underline \iota_\Delta$ is injective, we need the following lemma.

\begin{lemma}\label{lem:ideal}
    Let $G$ be a group, $G_0\subset G$ be a subset and ${\bf S}\subseteq \kk G\setminus \{0\}$ be a submonoid. Denote by $I$ and $J$ the ideal of $\kk G$ and {$(\kk  G)[{\bf S}^{-1}]$}, respectively, generated by $G_0$. Then $I=J\cap \kk  G$.
\end{lemma}

\begin{proof}
It is clear that $I\subset J\cap \kk G$. 

Assume that $x$ is the element in $J\cap \kk G$ such that the number $N$ of $s^{-1}, s\in S\setminus G$ appearing in the expression $x=\sum k_i g_{i,1}s_{i,1}^{-1}g_{i,2}s_{i,2}^{-1}\cdots g_{i,n_i}s_{i,n_i}^{-1}g_{i,n_i+1}\in J$ is minimum, where $k_i\in kk^\times, s_{i,j}\in S\setminus G$ and $g_{i,j}\in G$. To prove $I\subset J\cap \kk G$, it suffices to show that $N=0$. Otherwise, we may assume that $n_1\geq 1$. Then 
\begin{equation}\label{eq:exp}
 g_{1,1}s_{1,1}g_{1,1}^{-1}x=k_1g_{1,1}g_{1,2}s_{1,2}^{-1}\cdots g_{1,n_i}s_{1,n_i}^{-1}g_{1,n_i+1}+\sum_{i\neq 1} k_ig_{1,1}s_{1,1}g_{1,1}^{-1}g_{i,1}s_{i,1}^{-1}g_{i,2}s_{i,2}^{-1}\cdots g_{i,n_i}s_{i,n_i}^{-1}g_{i,n_i+1}.   
\end{equation}

Thus, $g_{1,1}s_{1,1}g_{1,1}^{-1}x\in J\cap \kk G$ has less $s^{-1}, s\in S\setminus G$ in the expression (\ref{eq:exp}), which contradicts the choice of $x$.

The proof is complete.
\end{proof}

Denote by $I$ the ideal of $\kk \TT_\Delta$ and $\mathcal A_\Sigma$, respectively, generated by $t_\gamma$ for all boundary arcs $\gamma$. Denote by $J$ the ideal of $\mathcal A_\Sigma$ generated by $x_\gamma$ for all boundary arcs $\gamma$. Then $Ker(\underline\iota_\Delta)=\kk \TT_\Delta\cap \iota_\Delta^{-1}(J)/I$. By Lemma \ref{lem:ideal}, we have $Ker~\underline\iota_\Delta=\{0\}$. Thus $\underline\iota_\Delta$ is injective.

The proof is complete. 
\end{proof}

The following generalizes \cite[Definition 2.9]{BR}.

\begin{lemma}
\label{le:canonical points}
Given a curve $\gamma$ in $\Sigma$  and a triangulation $\Delta$ of $\Sigma$, there is a unique sequence of   $\vec \gamma^\bullet=(\gamma^1,\ldots,\gamma^r)$  of edges of $\Delta$ (possibly with repetitions) such that
there are exactly $r$ intersection points $p_1,\ldots,p_r$ of $\gamma$ with $\Delta$ so that $p_k\in \gamma\cap \gamma^k$ for $k=1,\ldots, r$ (here $p_1$ is closest to $s(\gamma)$, $p_2$ is next closest to $s(\gamma)$,  etc., $p_r$ is the farthest from $s(\gamma)$, i.e., closest to $t(\gamma)$).
\end{lemma}

Clearly, $\gamma^k$ and $\gamma^{k+1}$ are two edges of a single triangle ${\mathcal T}_k$ in $\Delta$ containing the arc of $\gamma$ from $p_k$ to $p_{k+1}$,  $k=1,\ldots,r-1$. Denote by $\gamma^{[k]}$ the third edge of $\mathcal T_k$. We also denote by ${\mathcal T}_0$ (resp. ${\mathcal T}_r$) the triangle in $\Delta$ containing the arc of $\gamma$ from $s(\gamma)$ to $p_1$ (resp. from $p_r$ to $t(\gamma)$). In fact, if $\gamma^k$ and $\gamma^{k+1}$ are same and comprise a loop $\ell_i$ around $i\in \bigsqcup\limits_{k\neq 1}I_{p,k}(\Sigma)$ then ${\mathcal T}_k=(\ell_i,\ell_i,\ell_i)$ is degenerate, i.e., $\gamma^{[k]}=\ell_i$.

If we glue these triangles $\mathcal T_0,\mathcal T_1,\cdots, \mathcal T_r$, we obtain an $n$-polygon with $O$ special punctures $\Sigma_{\gamma,\Delta}$ and a triangulation $\widetilde\Delta$, where $O$ is the number of degenerated triangles and $n=r+1-2(O-1)$. We call $(\Sigma_{\gamma,\Delta},\widetilde\Delta)$ the \emph{canonical polygon} of $\gamma$ with respect to $\Delta$. Then $\gamma$ lifts uniquely to an arc $\widetilde\gamma$ of $\Sigma_{\gamma,\Delta}$.

If $s(\gamma)=p\in I_{p,1}, t(\gamma)\notin I_{p,1}$, denote by $\mathcal T_1(p),\cdots, \mathcal T_s(p)$ the triangles incident to $p$ in $\Delta$ in clockwise order such that $\mathcal T_1(p)=\mathcal T_0$. We glue $\mathcal T_2(p),\cdots, \mathcal T_s(p)$ to $\Sigma_{\gamma,\Delta}$, we obtain an $(n+s-3)$-polygon $\Sigma_{\gamma^{(p)},\Delta}$ with $1$-puncture, $O$ special punctures and a triangulation $\widetilde\Delta^{(p)}$. We call $(\Sigma_{\gamma^{(p)},\Delta},\widetilde\Delta^{(p)})$ the \emph{canonical once-punctured polygon} of $\gamma^{(p)}$ with respect to $\Delta$.

If $s(\gamma)=p,t(\gamma)=q\in I_{p,1}$, denote by $\mathcal T_1(p),\cdots, \mathcal T_s(p)$ the triangles incident to $p$ in $\Delta$ in clockwise order such that $\mathcal T_1(q)=\mathcal T_0$ and $\mathcal T_1(q),\cdots, \mathcal T_t(q)$ the triangles incident to $q$ in $\Delta$ in clockwise order such that $\mathcal T_1(p)=\mathcal T_r$. We glue $\mathcal T_2(p),\cdots, \mathcal T_s(p), \mathcal T_2(q),\cdots, \mathcal T_t(q)$ to $\Sigma_{\gamma,\Delta}$, we obtain an $(n+s+t-6)$-polygon $\Sigma_{\gamma^{(p,q)},\Delta}$ with $2$-punctures, $O$ special punctures and a triangulation $\widetilde\Delta^{(p,q)}$. We call $(\Sigma_{\gamma^{(p,q)},\Delta},\widetilde\Delta^{(p,q)})$ the \emph{canonical twice-punctured polygon} of $\gamma^{(p,q)}$ with respect to $\Delta$.

\begin{definition} (Admissible sequences) Let $\Delta$ be an ordinary triangulation of $\Sigma$ and $P\subset I$. For a curve $\gamma$ in $\Sigma$, fix the corresponding sequence $\vec \gamma^\bullet=(\gamma^1,\ldots,\gamma^r)$ of edges of $\Delta$. 

(1) If $s(\gamma),t(\gamma)\notin P$, we say that a sequence $\vec \gamma=(\gamma_1,\cdots,\gamma_{2m+1})$ in $\Delta$ (possibly with repetitions) is $(\gamma,\Delta)$-admissible if:
\begin{enumerate}
\item $s(\gamma_1)=s(\gamma)$, $t(\gamma_{2m+1})=t(\gamma)$ and $t(\gamma_k)=s(\gamma_{k+1})$ for $k=1,\ldots,2m$.

\item $(\gamma_2,\gamma_4,\ldots,\gamma_{2m})$ is a subsequence of $(\gamma^1,\ldots,\gamma^r)$. Assume that $\gamma_k=\gamma^{i_k}$ for all $k=2,4,\cdots, 2m$.

\item Each $\gamma_{2k+1}$ belongs to a triangle ${\mathcal T}_\ell$.

\item For each even $k=2,\cdots, 2m$, the arc of $\gamma$ between $p_{i_k}$ and $p_{i_{k+1}}$ is isotopic (up to $\Sigma\setminus I$) to the arc of the path starting at the point $p_{i_k}$, following
first $\gamma_k$, then $\gamma_{k+1}$, and then $\gamma_{k+2}$ until the point $p_{i_{k+1}}$; moreover, the arc of $\gamma$ between $s(\gamma)$ and $p_{i_2}$ (respectively $p_{i_{2m}}$ and $t(\gamma)$) is isotopic to the
arc of the path starting at $s(\gamma)$ (respectively $p_{i_{2m}}$), following first
$\gamma_1$ then $\gamma_2$ (respectively $\gamma_{2m}$ then $\gamma_{2m+1}$) until the point $p_{i_2}$ (respectively $t(\gamma)$).
\end{enumerate}

(2) If $s(\gamma)\in P,t(\gamma)\notin P$ we say that $\vec \gamma=(\gamma_1,\ldots,\gamma_{2m})$ is a $(\gamma^P,\Delta)$-admissible sequence if either $\gamma_1=\ell_p(s(\gamma))$ is a special loop based at $s(\gamma)$ and $(\gamma_2,\ldots,\gamma_{2m})$ is $(\gamma,\Delta)$-admissible or $(\gamma_1,\gamma_2,\gamma_3)$ is a clockwise cyclic triangle with $s(\gamma_1)=s(\gamma)$ and $(\gamma_4,\ldots,\gamma_{2m})$ is $(\gamma,\Delta)$-admissible.

(3) If $s(\gamma)\notin P,t(\gamma)\in P$ we say that $\vec \gamma=(\gamma_1,\ldots,\gamma_{2m})$ is a $(\gamma^P,\Delta)$-admissible sequence if either $\gamma_{2m}=\ell_q(t(\gamma))$ is a special loop based at $t(\gamma)$ and $(\gamma_1,\ldots,\gamma_{2m-1})$ is $(\gamma,\Delta)$-admissible or $(\gamma_{2m-2},\gamma_{2m-1},\gamma_{2m})$ is a clockwise cyclic triangle with $t(\gamma_{2m})=t(\gamma)$ and $(\gamma_1,\ldots,\gamma_{2m-3})$ is $(\gamma,\Delta)$-admissible.

(4) If $s(\gamma),t(\gamma)\in P$ we say that $\vec \gamma=(\gamma_1,\ldots,\gamma_{2m+1})$ is a $(\gamma^P,\Delta)$-admissible sequence if either $\gamma_{1}=\ell_p(s(\gamma))$ is a special loop based at $s(\gamma)$ or $(\gamma_1,\gamma_2,\gamma_3)$ is a clockwise cyclic triangle with $s(\gamma_1)=s(\gamma)$, and either $\gamma_{2m+1}=\ell_q(t(\gamma))$ is a special loop based at $t(\gamma)$ or  $(\gamma_{2m-1},\gamma_{2m},\gamma_{2m+1})$ is a clockwise cyclic triangle with $t(\gamma_{2m+1})=t(\gamma)$, moreover, correspondingly $(\gamma_2,\ldots,\gamma_{2m})$, $(\gamma_4,\ldots,\gamma_{2m})$, $(\gamma_2,\ldots,\gamma_{2m-2})$ or $(\gamma_4,\ldots,\gamma_{2m-2})$ is $(\gamma,\Delta)$-admissible.

\end{definition}

\begin{example}{For the surfaces, triangulations, and arcs $\gamma$ depicted in Figure} \ref{Fig:admseq}, we have $((6,1),(1,5),(5,2),(2,4),(4,3))$ is $((6,3),\Delta)$-admissible; $((0,6),(6,1),(1,0),(0,5),(5,2),(2,3))$ is a $((0,3)^{\{0\}},\Delta)$-admissible sequence; $((0,6),(6,1),(1,0),(0,5),(5,2),(2,0'),(0',3),(3,4),(4,0'))$ is $((0,0')^{\{0,0'\}},\Delta)$-admissible.

    \begin{figure}[ht]
\includegraphics{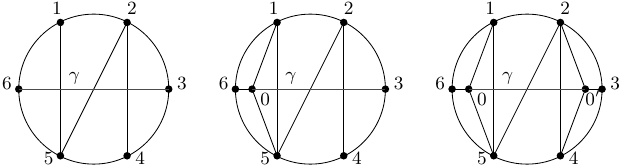}
\caption{}
\label{Fig:admseq}
\end{figure}
\end{example}

We denote by $Adm(\gamma^P,\Delta)$ the set of all $(\gamma^P,\Delta)$-admissible sequences. 

It is clear that $Adm(\gamma^P,\Delta)$ and $Adm(\widetilde\gamma^P,\widetilde\Delta)$ are in one-to-one correspondence under the canonical map from $\Sigma_{\gamma^{(p,q)},\Delta}$ to $\Sigma$.

For any $(\gamma^P,\Delta)$-admissible sequence  $\vec \gamma=(\gamma_1,\ldots,\gamma_r)$  and a monomial $x_{\vec \gamma}\in {\mathcal A}_\Sigma$ 
by
$$x_{\vec\gamma}=x_{\gamma_1,-\varepsilon} x_{\gamma_2,\varepsilon}\cdots x_{\gamma_r,(-1)^r\varepsilon}$$
with $\varepsilon=\begin{cases}
1, & \text{if $s(\gamma_1)\in P$,}\\ 
-1, & \text{otherwise,}\\ 
\end{cases}$
where we abbreviate 
$x_{\gamma,\delta}:=
\begin{cases}
 x_\gamma, & \text{if $\delta=1$,}\\ 
  x_{\overline \gamma}^{-1}, & \text{if $\delta=-1$.}\\ 
\end{cases}$

For any arcs $\gamma,\gamma'\in \Delta$ with $s(\gamma)=s(\gamma')=i$, in the case $i\in I_b(\Sigma)$, if $\gamma'$ is in clockwise direction of $\gamma$ and the boundary curves $\gamma^-$ and $\gamma^+$ originating at $i$ are such that $\gamma^-$ is on the left of $\gamma$ and $\gamma^+$ is on the right of $\gamma'$, we denote by $(\gamma,\gamma')$ (resp. $(\gamma',\gamma)$ ) the sector based at $i$ by traveling from $\gamma$ to $\gamma'$ (resp. $\gamma'$ to $\gamma$) (in a tight neighborhood of $i$) in the clockwise (resp. counter-clockwise) direction. In the case when $i\in I_{p,1}(\Sigma)$, we denote by $(\gamma,\gamma')$ the sector based at $i$ by traveling from $\gamma$ to $\gamma'$  (in a tight neighborhood of $i$) in the clockwise direction. In both cases, we say that $(\gamma,\gamma')$ is a clockwise sector. See Figure \ref{fig:sector}.

\begin{figure}[ht]
\includegraphics{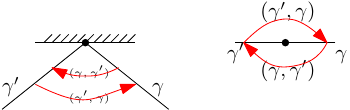}
\caption{}
\label{fig:sector}
\end{figure}

For any $(\gamma^P,\Delta)$-admissible sequence $\vec \gamma=(\gamma_1,\ldots,\gamma_r)$ of curves in $\Sigma$ and a point $p_-$ in $\gamma$,
we say that $k\in [2,r-2]$ is \emph{$p_-$-special} if $\gamma_k$ is a simple loop around some $I_{p_-}\in \bigsqcup\limits_{k\neq 1}I_{p,k}(\Sigma)$ crossing $\gamma$ at $p_-$ as an entrance point.

$\bullet$ either $\gamma_k=\gamma_{k+1}=\gamma_{k+2}$, $\gamma_{k-1}\ne \gamma_k$ 

$\bullet$ or $\gamma_{k+1}=\gamma_k$, and $\gamma_{k-1}\ne \gamma_k\neq \gamma_{k+2}$, and $\widetilde\gamma_k$ is not in the sector  $(\widetilde\gamma_{k-1},\widetilde\gamma_{k+2})$, where $\widetilde\gamma_{i}$ is the preimage of $\gamma_i$ in $\widetilde \Delta^P$ for $i=k-1,k,k+2$.

\medskip

For any triangulation $\Delta$ and any $(\gamma^P,\Delta)$-admissible sequence $\vec \gamma=(\gamma_1,\cdots,\gamma_m)$, we define the weight $c_{\vec \gamma}\in \kk^\times$ by
$c_{\vec\gamma,p_-}=\left ( 
2\cos(\frac{\pi}{|I_{p_-}|})\right )^{N_{p_-}}$
 where $N_{p-}$ is the number of all $p_-$-special $k\in [2,r-2]$ and
 \begin{equation}\label{eq:cgamma}
 c_{\gamma_i}=\begin{cases}
 2\cos(\frac{\pi}{|o|}), & \text{if $s(\gamma)\in P$ and $i=1$ with $\gamma_1$ the loop encloses a special puncture $o$,}\\
  2\cos(\frac{\pi}{|o|}), & \text{if $t(\gamma)\in P$ and $i=m$ with $\gamma_m$ the loop encloses a special puncture $o$,}\\
  1, & \text{otherwise}.
 \end{cases}
 \end{equation}

Then 
\begin{equation}\label{eq:cgammav}
c_{\vec\gamma}:=\prod c_{\vec\gamma,p_-}\prod c_{\gamma_i}
\end{equation}
where the product is over all such special $p_-$ in the canonical sequence $p_1,\ldots,p_r$ attached to $(\gamma,\Delta)$ in lemma \ref{le:canonical points}.

The following is a generalization of \cite[Theorem 3.30]{BR}

\begin{theorem}\label{thm:laurent} Let $\Delta$ be an ordinary triangulation of $\Sigma$. For any $\gamma\in [\Gamma(\Sigma)]$ and $P,P'\subset I_{p,1}(\Sigma)$, we have
\begin{equation*}
   x_{\gamma^P}=\sum\limits_{\vec\gamma\in Adm(\gamma^P,\Delta)} c_{\vec \gamma} x_{\vec \gamma}. 
\end{equation*}
    
\end{theorem}

\begin{proof} We need the following

\begin{proposition}
\label{expansion1}
Let $\Delta$ be a (tagged) triangulation of $\Sigma$. For any $\gamma\in [\Gamma(\Sigma)]$ and  $P\subset I_{p,1}(\Sigma)$, we have

$(1)$ If $\Delta$ is ordinary (i.e., $tag(\Delta)=\emptyset$) then
\begin{equation}
   x_{\gamma^P}=T_{s(\gamma)}^{\chi_{P}(s(\gamma))}\textstyle(\sum\limits_{\vec\gamma\in Adm(\gamma,\Delta)}
  c_{\vec \gamma} x_{\vec \gamma}) T_{t(\gamma)}^{\chi_{P}(t(\gamma))}. 
\end{equation}
To be precise, 

$(1.1)$ If $s(\gamma),t(\gamma)\notin P$, then 
\begin{equation*}
   x_{\gamma}=\sum\limits_{\vec\gamma\in Adm(\gamma,\Delta)} c_{\vec \gamma}x_{\vec \gamma}. 
\end{equation*}

$(1.2)$ If $s(\gamma)\in P,t(\gamma)\notin P$, then 
\begin{equation*}
   x_{\gamma^{P}}=(\sum T_{(\gamma_1,\gamma_2,\gamma_3)}+\sum 2\cos(\frac{\pi}{|p|})x^{-1}_{\ell_p})(\textstyle\sum\limits_{\vec\gamma\in Adm(\gamma,\Delta)}
  c_{\vec \gamma} x_{\vec \gamma}). 
\end{equation*}

$(1.3)$ If $s(\gamma)\notin P,t(\gamma)\in P$, then 
\begin{equation*}
   x_{\gamma^{P}}=(\textstyle\sum\limits_{\vec\gamma\in Adm(\gamma,\Delta)}
  c_{\vec \gamma} x_{\vec \gamma}) (\sum T_{(\gamma'_1,\gamma'_2,\gamma'_3)}+\sum 2\cos(\frac{\pi}{|p|}x^{-1}_{\ell'_p})). 
\end{equation*}

$(1.4)$ If $s(\gamma),t(\gamma)\in P$, then 
\begin{equation*}
   x_{\gamma^{P}}=(\sum T_{(\gamma_1,\gamma_2,\gamma_3)}+\sum 2\cos(\frac{\pi}{|p|}x^{-1}_{\ell_p}))(\textstyle\sum\limits_{\vec\gamma\in Adm(\gamma,\Delta)}
  c_{\vec \gamma} x_{\vec \gamma}) (\sum T_{(\gamma'_1,\gamma'_2,\gamma'_3)}+\sum 2\cos(\frac{\pi}{|p|}x^{-1}_{\ell'_p})).
\end{equation*}

$(2)$ Suppose that $tag(\Delta)=P'$
and $\gamma$ be a curve with $t(\gamma)=j\in I_{p,1}$. Then 

\begin{equation*}
   x_{\gamma^P}=(T^{\Delta}_{s(\gamma)})^{\chi_{P\ominus P'}(s(\gamma))}(\sum\limits_{\vec\gamma\in Adm(\gamma,\Delta)}c_{\vec \gamma} x_{\vec \gamma})(T^\Delta_{t(\gamma)})^{\chi_{P\ominus P'}(s(\gamma))}.
\end{equation*}

To be precise,

$(2.1)$ If $s(\gamma),t(\gamma)\notin P\ominus P'$, then 

\begin{equation*}
   x_{\gamma^P}=\sum\limits_{\vec\gamma\in Adm(\gamma,\Delta)} c_{\vec \gamma} x_{\vec \gamma}.
\end{equation*}

$(2.2)$ If $s(\gamma)\in P\ominus P' ,t(\gamma)\notin P\ominus P'$, then 
\begin{equation*}
   x_{\gamma^P}=(\sum T^\Delta_{(\gamma_1,\gamma_2,\gamma_3)}+\sum 2\cos(\frac{\pi}{|p|}x^{-1}_{\ell_p}))\sum\limits_{\vec\gamma\in Adm(\gamma,\Delta)}c_{\vec \gamma} x_{\vec \gamma}.
\end{equation*}

$(2.3)$ If $s(\gamma)\notin P'\ominus P,t(\gamma)\in P\ominus P'$, then 
\begin{equation*}
   x_{\gamma^P}=\sum\limits_{\vec\gamma\in Adm(\gamma,\Delta)}c_{\vec \gamma} x_{\vec \gamma}(\sum T^\Delta_{(\gamma'_1,\gamma'_2,\gamma'_3)}+\sum 2\cos(\frac{\pi}{|p|}x^{-1}_{\ell'_p})).
\end{equation*}

$(2.4)$ If $s(\gamma),t(\gamma)\in P\ominus P'$, then 
\begin{equation*}
   x_{\gamma^P}=(\sum T^\Delta_{(\gamma_1,\gamma_2,\gamma_3)}+\sum 2\cos(\frac{\pi}{|p|}x^{-1}_{\ell_p}))\sum\limits_{\vec\gamma\in Adm(\gamma,\Delta)}c_{\vec \gamma} x_{\vec \gamma}(\sum T^\Delta_{(\gamma'_1,\gamma'_2,\gamma'_3)}+\sum 2\cos(\frac{\pi}{|p|}x^{-1}_{\ell'_p})), 
\end{equation*}
where in all the cases, $(\gamma_1, \gamma_2, \gamma_3)/(\gamma'_1, \gamma'_2, \gamma'_3)$ runs over all clockwise triangles 
in $\Delta$ such that $s(\gamma_1) =s(\gamma)/t(\gamma)$ and $\ell_p/\ell'_p$ runs over all clockwise special loops enclose a special puncture $p$ with $s(\ell_p)=s(\gamma)/t(\gamma)$. 
\end{proposition}


\begin{proof} 
(2) is followed by (1) and Lemma \ref{lem:varphi}. (1.2) and (1.3) are followed by (1.1) and Proposition \ref{prop:total}. Thus we shall only prove (1.1).

We first assume that $\Sigma$ is an $n$-gon  with $m$ special punctures. 

The case that $m=0$ is proved in \cite[Theorem 3.30]{BR}. For $m>0$, fix an special puncture $o$ with order $|o|$, let $\Sigma'$ be the $n|o|$-gon with $m-1$ special punctures such that orders are the same with the orders of the rest orbifold points in $\Sigma$. Then there is a canonical surjective morphism $f_o:\Sigma'\to \Sigma$. Assume that $\ell$ is the loop enclose $o$ in $\Delta$. Then we can lift $\ell$ to an $|o|$-gon $\Sigma_{|o|}$ of $\widetilde\Sigma$. Lift $\Delta$ to a triangulation $\widetilde\Delta$ of $\widetilde\Sigma$ such that $\widetilde\Delta \cap [\Gamma(\Sigma_{|o|})]$ contains the arcs $(1,3),(3,5),(5,7)\cdots$. Then each $(\gamma,\Delta)$ admissible sequence $\overrightarrow\gamma$ can be lift to a unique   $(\widetilde\gamma,\widetilde\Delta)$ admissible sequence $\overrightarrow{\widetilde\gamma}$. 

Under the surjective morphism $\pi:\mathcal A_{\Sigma'}\to \mathcal A_\Sigma$, we have $\pi(c_{\overrightarrow{\widetilde\gamma}}x_{\overrightarrow{\widetilde\gamma}})=c_{\overrightarrow{\gamma}}x_{\overrightarrow{\gamma}}$ for any $\gamma\in Adm(\gamma,\Delta)$. Therefore, by induction we have
$$x_{\gamma}=\pi(x_{\widetilde\gamma})=\pi(\sum\limits_{\widetilde\gamma\in Adm(\widetilde\gamma,\widetilde\Delta)} c_{\vec {\widetilde\gamma}}x_{\vec {\widetilde\gamma}})=\sum\limits_{\gamma\in Adm(\gamma,\Delta)} c_{\vec \gamma}x_{\vec \gamma}.$$

For a general marked surface $\Sigma$ with an ordinary triangulation $\Delta$ such that $I_{p,0}(\Sigma)=\emptyset$, the result follows by considering the canonical polygon $(\Sigma_{\gamma,\Delta},\widetilde \Delta)$ of $\gamma$ with respect to $\Delta$ and {the natural morphism $f: \mathcal A_{\Sigma_{\gamma,\Delta}}\to \mathcal A_\Sigma$.} 

For general marked surface $\Sigma$ with ordinary triangulation $\Delta$ such that $I_{p,0}(\Sigma)\neq \emptyset$, the result follows by Corollary \ref{cor:Tp=1}.

The proposition is proved.   
\end{proof}

The theorem is proved. 

\end{proof}

The following is immediate from Theorem \ref{thm:laurent}.

\begin{proposition} For any curve $\gamma\in \Gamma(\Sigma)$ and any $P\subset I_{p,1}(\Sigma)$, both $x_\gamma$ and $x_{\gamma^P}$ are in the image of both $\iota_\Delta$ and $\iota_{\Delta^P}$. 
\end{proposition}

Chose a point $p$ on $\gamma$ close to $s(\gamma)$, we say that the curve from $s(\gamma)$ to $p$ along $\gamma$ an \emph{initial segment} of $\gamma$.

For $\alpha,\alpha'\in \Delta$ and curve $\gamma$ with $s(\alpha)=s(\alpha')=s(\gamma)$, we say that $\alpha$ is on the \emph{left} of $\alpha'$ with respect to $\gamma$ if $(\alpha,\alpha')$ is a clockwise sector and the initial segment of $\gamma$ lies $(\alpha,\alpha')$. See Figure \ref{fig:left}.

\begin{figure}[ht]
\includegraphics{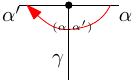}
\caption{}
\label{fig:left}
\end{figure}

Then define a partial order on $Adm(\gamma,\Delta)$ by saying that $\vec \gamma'\prec \vec \gamma$ if

$\bullet$ $\gamma_1\ne \gamma'_1$
and $\gamma'_1$ is on the left of $\gamma_1$ with respect to $\gamma$; or

$\bullet$ if $\gamma_1=\gamma'_1$ and $\gamma'_2\neq \gamma_2$, $\widetilde {p}_2$ is closer to $\widetilde s(\gamma)$ than ${\widetilde p}'_2$, where $\widetilde p_2$ (resp. $\widetilde p'_2$) is the preimage of the crossing point $p_2$ (resp. $p'_2$) of $\gamma$ and $\gamma_2$ (resp. $\gamma'_2$) and $\widetilde s(\gamma)$ is the preimage of $s(\gamma)$ in $\widetilde\Sigma_{\gamma,\Delta}$ ({For example, $((6,5),(5,1),(1,2),(2,4),(4,3))\prec ((6,5),(5,2),(2,3))$. The corresponding surface, triangulation, and arc $\gamma=(6,3)$ are depicted in the first panel of Figure \ref{Fig:admseq}}); or

$\bullet$ if $\gamma_1=\gamma'_1,\gamma_2=\gamma'_2$ and $(\gamma'_3,\ldots,\gamma'_{k'})\prec (\gamma_3,\ldots,\gamma_k)$ in $Adm(\gamma',\Delta)$, where $\gamma'=\gamma\circ \overline\gamma_1\circ \overline \gamma_2$, as shown in Figure \ref{fig:compose}.

\begin{figure}[ht]
\includegraphics{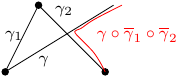}
\caption{}
\label{fig:compose}
\end{figure}

It is immediate the restriction of the above partial order to the (finite) set $Adm(\gamma,\Delta)$ is a total order.

We denote by $\vec\gamma^L$ the largest and $\vec\gamma^R$ the smallest elements of $Adm(\gamma,\Delta)$ and refer to them as the {\it leftmost} and the {\it rightmost}
$(\gamma,\Delta)$-admissible sequences respectively.

\begin{remark}\label{rem:g-vector} In \cite{MS} admissible sequences were called $\Delta$-paths and were identified with perfect matchings. Under this bijection, the leftmost/rightmost admissible sequence corresponds to the {maximal/minimal perfect matching. In particular, the leftmost admissible sequence is the ``non-commutative" extended $g$-vector of $\gamma$ with respective to the (skew-symmetrizable) seed associated with $\Delta$.}
\end{remark}

\begin{corollary}\label{Cor:NCexpan} For any triangulations $\Delta$ and $\Delta'$ of $\Sigma$ one has
$$x_\gamma=\iota_{\Delta'}(\mu_{\Delta',\Delta}(t_\gamma))+\sum_{R(\Delta',\gamma)\prec \vec \gamma'\prec \vec L(\Delta',\gamma)} c_{\vec \gamma'} x_{\vec \gamma'}+\iota_{\Delta'}(\mu_{\Delta',\Delta}^-(t_\gamma))$$
for all $\gamma\notin \Delta$.
\end{corollary}

\begin{proof}
When ${\rm dist}(\Delta',\Delta)=1$, it is evident that the abelianization of $\mu_{\Delta',\Delta}:\mathbb T_\Delta\to \mathbb T_{\Delta'}$ coincides with the extended $g$-matrix of the seed $\Delta$ with respect to $\Delta'$ (where the columns represent the $g$-vectors of non-pending arcs, including boundary arcs). In light of \cite[Proposition 1.3]{NZ}, Relation (1) of Definition \ref{def:gro}, and Theorem \ref{th:monomial mutation surfaces}, this implies that $\mu_{\Delta',\Delta}$ corresponds to the extended $g$-matrix of $\Delta$ with respect to $\Delta'$ for arbitrary $\Delta$ and $\Delta'$. Consequently, the abelianization of $\mu_{\Delta',\Delta}(t_\gamma)$ yields the extended $g$-vector of $x_\gamma$ with respect to the initial seed $\Delta'$. Invoking \cite[Corollary 6.3]{FZ4} and the fact that $F$-polynomials have a constant term of $1$, we deduce from Remark \ref{rem:g-vector} that $\mu_{\Delta',\Delta}(t_\gamma)$ is a monomial corresponding to the leftmost admissible sequence with coefficient $1$. Dually, $\mu^-_{\Delta',\Delta}(t_\gamma)$ is a monomial corresponding to the rightmost admissible sequence with coefficient $1$. The conclusion then follows from Theorem \ref{thm:laurent}.
\end{proof}

In summary, by Theorem \ref{th:monomial mutation surfaces} and Corollary \ref{Cor:NCexpan}, we obtain the following result.

\begin{theorem} For any 
$\Sigma\in {\bf Surf}$ the algebra $\mathcal A_\Sigma$ admits a (generalized) noncommutative cluster structure with group $\mathbb T_\Sigma$.
\end{theorem}

 The untagged version follows from  \cite[Theorem 3.36]{BR}.

\begin{conjecture}\label{conj:clusterauto} In notation of  Remark \ref{rem:automorphisms}, the group of cluster automorphisms of ${\mathcal A}_\Sigma$ is {generated by the subgroup induced by surface automorphisms,} together with $\varphi_p$, $p\in I_{p,1}(\Sigma)$.
    
\end{conjecture}

\begin{example} ${\mathcal D}_2$ is generated by $x_{12}^{\pm}$, $x_{21}^{\pm}$, $x_{10}$, $x_{01}$, $y_{10}$, $y_{01}$, $x_{20}$, $x_{02}$, $y_{20}$, and $y_{02}$ subject to the relations 
$$x_{i0}y_{0i}=y_{i0}x_{0i}, i\in \{1,2\}, x_{21}^+y_{01}^{-1}x_{10}^{-1}x_{12}^-=x_{21}^-x_{01}^{-1}y_{10}^{-1}x_{12}^+, x_{12}^+y_{02}^{-1}x_{20}^{-1}x_{21}^-=x_{12}^-x_{02}^{-1}y_{20}^{-1}x_{21}^+,$$
$$x_{21}^{\pm}x_{01}^{-1}x_{02}=x_{20}x_{10}^{-1}x_{12}^{\pm}, x_{21}^{\pm}y_{01}^{-1}x_{02}=x_{20}y_{10}^{-1}x_{12}^{\pm}$$ and 
$y_{10}=(x_{12}^{+}+x_{12}^-)x_{02}^{-1},~y_{20}=(x_{21}^{+}+x_{21}^-)x_{01}^{-1},
y_{01}=x_{20}^{-1}(x_{21}^{+}+x_{21}^-),
y_{02}=x_{10}^{-1}(x_{12}^{+}+x_{12}^-)$.

This gives four noncommutative clusters for ${\mathcal D}_2$ (each of them also has frozen variables $x_{12}^{\pm},x_{21}^{\pm}$):  
 $\{ x_{10}, x_{01},x_{20}, x_{02}\}$, $\{ x_{10}, x_{01},y_{10}, y_{01}\}$, $\{y_{20}, y_{02},x_{20}, x_{02}\}$, $\{ y_{10}, y_{01},y_{20}, y_{02}\}$. The last one cannot be reduced to the ordinary triangulation similarly to the commutative or quantum case. We expect that this exhausts all noncommutative clusters.
    
\end{example}

\subsection{Noncommutative rank $2$ cluster algebras and their Laurent phenomenon}

Denote by ${\mathcal A}_{r_1,r_2}$ the subalgebra of ${\mathcal Frac}({\mathbb Q}\langle y_1,y_2\rangle$ generated by all $y_k$, $k\in \Z$ given by the recursion
$$y_{k+1}=y_{k-1}^{-1}z^{-1}+y_k^{r_k}y_{k-1}^{-1}z^{-1}$$

It was proved in \cite{BR0} that ${\mathcal A}_{r_1,r_2}$ is generated by any quadruple $y_{k-1},y_k,y_{k+1},y_{k+2}$, in particular, taking $k=0$, we see that ${\mathcal A}_{r_1,r_2}\subset \mathbb{Q}\langle y_1^{\pm 1},y_2^{\pm 1}\rangle$.

For $k=2$ this is the mutation from the cluster $(y_1,y_2)$ to $(y_3,y_2)$.
Set $z:=[y_k^{-1},y_{k-1}]=y_k^{-1}y_{k-1}y_ky_{k-1}^{-1}$.
Then we have isomorphisms $f_1,f_2:\TT_2\to \TT_1$ given by 
$f_i(y_2)=y_2$ and
$$f_1(y_3)=y_1^{-1}z^{-1},~f_2(y_3)=y_2^{r_2}y_1^{-1}z^{-1}$$
$$f_1^{-1}(y_1)=z^{-1}y_3^{-1},~f_2^{-1}(y_1)=z^{-1}y_3^{-1}y_2^{r_2}$$
In this case the $k$-th noncommutative cluster is the free group generated by $t_1,t_2$ and the noncommutative Laurent Phenomenon can be deduced from \cite[Theorem 6]{R}.

  \begin{corollary}
\label{th:leading term rank 2}
For any $k\in \mathbb{Z}$ one has
$$y_k=\iota_k(x^{\mu_{1k}(g_k)})+lower ~terms$$

\end{corollary}

\section{Proofs of main results}

\subsection{Proof of Theorem \ref{th:Chebyshev}}\label{sec:proofofthcheby}

The coinvariant algebra ${\mathcal A}_n/I_2(n)$ is the quotient of ${\mathcal A}_n$ by the ideal $I$ which is generated by $x_{ij}-x_{ji}$ for any $i,j\in [n]$ and $x_{ij}-x_{kl}$ for any $i,j,k,l\in [n]$ with $j-i\equiv l-k\; (mod\; n)$. As ${\mathcal A}_n$ is generated by $x_{1i},x_{i1}, 1<i\leq n$, we have ${\mathcal A}_n/I_2(n)$ is generated by $x_{1i}+I, 1<i\leq n$. The relations for $x_{1i}+I$ are $x_{1i}+I=x_{1,(n+2-i)}+I$. 

For any $i$ with $4\leq i\leq n$, we have $x_{1i}=x_{1,i-1}x_{i-2,i-1}^{-1}x_{i-2,i}-x_{1,i-2}x_{i-1,i-2}^{-1}x_{i-1,i}$. Denote by $a=x_{13}+I$ and $b=x_{12}+I$. Therefore, in ${\mathcal A}_n$ we have 
$$(x_{1,i}+I)b^{-1}=\left((x_{1,i-1}+I)b^{-1}\right)\left(ab^{-1}\right)-(x_{1,i-2}+I)b^{-1}.$$
It follows that $(x_{1,i}+I)b^{-1}=U_{i-2}(\frac{ab^{-1}}{2}), 2\leq i\leq n$, where $U_i$ are Chebyshev polynomials of the second kind.

In case $n$ is odd, $U_{\frac{n-1}{2}}(\frac{ab^{-1}}{2})=U_{\frac{n-1}{2}-1}(\frac{ab^{-1}}{2})$ implies $U_{i}(\frac{ab^{-1}}{2})=U_{n-2-i}(\frac{ab^{-1}}{2})$ for all $1<i\leq n-2$. It shows that ${\mathcal A}_n/I_2(n)$ is generated by $a^{\pm}, b^{\pm}$ subjects to $(U_{\frac{n-1}{2}}-U_{\frac{n-1}{2}-1})(\frac{ab^{-1}}{2})=0$.

In case $n$ is even, $U_{\frac{n}{2}}(\frac{ab^{-1}}{2})=U_{\frac{n}{2}-2}(\frac{ab^{-1}}{2})$ implies $U_{i}(\frac{ab^{-1}}{2})=U_{n-2-i}(\frac{ab^{-1}}{2})$ for all $1<i\leq n-2$. It shows that ${\mathcal A}_n/I_2(n)$ is generated by $a^{\pm}, b^{\pm}$ subjects to $(U_{\frac{n}{2}}-U_{\frac{n}{2}-2})(\frac{ab^{-1}}{2})=0$.

The proof is complete.
\endproof

\subsection{Proofs of Theorem \ref{th:sector presentation} and Theorem \ref{th:reduced surfaces}}
\label{subsec:Proof of Theorem sector presentation and reduced surfaces}
{Recall that $y_{\gamma,\gamma'}:=x_{\overline\gamma}^{-1}x_{\gamma'}$ for any pair of curves $(\gamma,\gamma')$ with $t(\gamma)=s(\gamma')$.}

\begin{proposition}\label{prop:quotient} Let $\Sigma$ be a marked surface
with $I_{p,0}(\Sigma)=\emptyset$.
For any $i\in I_b\cup I_{p,1}$, fix a curve $\gamma_i$ with $s(\gamma_i)=i$ (all these curves are automatically distinct). Then the assignments $x_{\gamma}\mapsto y_{\overline{\gamma_{s(\gamma)}},\gamma}$ (e.g., $x_{\gamma_i}\mapsto 1$) define an algebra homomorphism
$\pi: \mathcal A_\Sigma\rightarrow \mathcal A_\Sigma$ which is a projection onto $\mathcal B_\Sigma$.
\end{proposition}

\begin{proof} 
First, we prove that $\pi$ is a homomorphism.

(Triangle relations) For each cyclic triangle $(\alpha_1,\alpha_2,\alpha_3)$ in $\Sigma$, we have $$\pi(x_{\alpha_1}x_{\overline\alpha_2}^{-1}x_{\alpha_3})=y_{\overline{\gamma_{s(\alpha_1)}},\alpha_1}(y_{\overline{\gamma_{s(\overline\alpha_2)}},\overline\alpha_2})^{-1}y_{\overline{\gamma_{s(\alpha_3)}},\alpha_3}=x_{ \gamma_{s(\alpha_1)}}^{-1}x_{\alpha_1}x_{\overline\alpha_2}^{-1}x_{\alpha_3},$$
$$\pi(x_{\overline \alpha_3}x_{\alpha_2}^{-1}x_{\overline\alpha_3})=x_{\gamma_{s(\overline \alpha_3)}}^{-1}x_{\overline \alpha_3}x_{\alpha_2}^{-1}x_{\overline \alpha_1}.$$
Thus $\pi(x_{\alpha_1}x_{\overline\alpha_2}^{-1}x_{\alpha_3})=\pi(x_{\overline \alpha_3}x_{\alpha_2}^{-1}x_{\overline\alpha_1})$ follows by $s(\alpha_1)=s(\overline\alpha_3)$.

(Ptolemy relations) For each cyclic quadrilateral $( \alpha_{ 1}, \alpha_{ 2}, \alpha_{ 3}, \alpha_{ 4})$ with diagonals $ \alpha$ and $ \alpha'$  such that $s(\alpha)=s(\alpha_1), s(\alpha')=t(\alpha_1)$, we have
$$\pi(x_{\overline\alpha_{1}}x^{-1}_{\overline\alpha}x_{ \alpha_{ 3}}+x_{ \alpha_{ 2}}x^{-1}_{ \alpha}x_{\overline\alpha_{ 4}})=x^{-1}_{\gamma_{s(\overline\alpha_1)}}x_{\overline\alpha_{1}}x^{-1}_{\overline\alpha}x_{ \alpha_{ 3}}+x^{-1}_{\gamma_{s(\alpha_2)}}x_{ \alpha_{ 2}}x^{-1}_{ \alpha}x_{\overline\alpha_{ 4}}=\pi(x_{ \alpha'}).$$

(Monogon relations) For each special loop $\gamma$,
 $\pi(x_{\overline\gamma})=x^{-1}_{ \gamma_{s(\overline\gamma)}}x_{\overline\gamma}=x^{-1}_{ \gamma_{s(\gamma)}}x_{\gamma}=\pi(x_\gamma)$.

(Bigon special puncture relations) For each bigon $(\alpha_1,\alpha_2)$ around a special puncture $p$, assume that $\alpha$ is the loop around $p$ such that $(\alpha_1,\alpha_2,\alpha)$ is a triangle and $\alpha'$ is the loop around $p$ such that $(\alpha',\alpha_2,\alpha_1)$ is a triangle, we have
\begin{equation*}
\begin{array}{rcl}
& &\pi(x_{\overline\alpha_{ 1}}x^{-1}_{\alpha}x_{ \alpha_{ 1}}+2\cos(\frac{\pi}{|p|})x_{\overline\alpha_{ 1}}x^{-1}_{\alpha}x_{\overline\alpha_2}+x_{ \alpha_{ 2}}x^{-1}_{ \alpha}x_{\overline\alpha_{2}})\vspace{2.5pt}\\
&=&
x^{-1}_{\gamma_{t(\alpha_{ 1})}}(x_{\overline\alpha_{ 1}}x^{-1}_{\alpha}x_{ \alpha_{ 1}}+2\cos(\frac{\pi}{|p|})x_{\overline\alpha_{ 1}}x^{-1}_{\alpha}x_{\overline\alpha_2}+x_{ \alpha_{ 2}}x^{-1}_{ \alpha}x_{\overline\alpha_{2}}) \vspace{2.5pt}\\
&=& x^{-1}_{\gamma_{s(\alpha')}}x_{\alpha'}=\pi(x_{\alpha'}).
\end{array}
\end{equation*}

Therefore, we obtain an algebra homomorphism $\pi:\mathcal A_\Sigma\to \mathcal A_\Sigma$.


Next, show that $\pi^2=\pi$. Indeed, 
$$\pi^2(x_\gamma)=\pi(y_{\overline{\gamma_{s(\gamma)}},\gamma})=y_{\overline{\gamma_{s(\gamma)}},\gamma}$$
for any $\gamma$.

Finally, prove that the image of $\pi$ is $\mathcal B_\Sigma$. Indeed,
$$\pi(y_{\gamma,\gamma'})=\pi(x_{\overline\gamma}^{-1}x_{\gamma'})=y^{-1}_{\overline{\gamma_{s(\overline\gamma)}},\overline\gamma} y_{\overline{\gamma_{s(\gamma')}},\gamma'}=(x^{-1}_{\gamma_{s(\overline\gamma)}}x_{\overline\gamma})^{-1}(x^{-1}_{ \gamma_{s(\gamma')}}x_{\gamma'})=x_{\overline\gamma}^{-1}x_{\gamma'}=y_{\gamma,\gamma'}$$
for any $y_{\gamma,\gamma'}\in \mathcal B_\Sigma$. 

The proof is complete.
\end{proof}

The following follows immediately from Proposition \ref{prop:quotient}.

\begin{corollary}\label{cor:relationB}
For any $\Sigma\in {\bf Surf}$
with $I_{p,0}(\Sigma)=\emptyset$ 
and ordinary triangulation $\Delta$ of $\Sigma$, the sector subalgebra $\mathcal B_\Sigma$ 
has the following presentation:
\begin{itemize}
    \item For each cyclic triangle $(\alpha_1,\alpha_2,\alpha_3)$ in $\Sigma$, we have $\pi(x_{\alpha_1}x_{\overline\alpha_2}^{-1}x_{\alpha_3})=\pi(x_{\overline \alpha_3}x_{\alpha_2}^{-1}x_{\overline\alpha_1})$, i.e., $y_{\gamma_{s(\alpha_1)},\alpha_1}(y_{\gamma_{s(\overline\alpha_2)},\overline\alpha_2})^{-1}y_{\gamma_{s(\alpha_3)},\alpha_3}=y_{\gamma_{s(\overline\alpha_3)},\overline\alpha_3}(y_{\gamma_{s(\alpha_2)},\alpha_2})^{-1}y_{\gamma_{s(\overline\alpha_1)},\overline\alpha_1}.$

    \item For each loop $\gamma$ cuts out a monogon which contains only a special puncture, we have $y_{\gamma_{s(\overline \gamma)},\overline\gamma}=y_{\gamma_{s(\gamma)},\gamma}$.

   \item  For each cyclic quadrilateral $( \alpha_{ 1}, \alpha_{ 2}, \alpha_{ 3}, \alpha_{ 4})$ with diagonals $ \alpha$ and $ \alpha'$  such that $s(\alpha)=s(\alpha_1), s(\alpha')=t(\alpha_1)$, we have $\pi(x_{\overline\alpha_{1}}x^{-1}_{\overline\alpha}x_{ \alpha_{ 3}}+x_{ \alpha_{ 2}}x^{-1}_{ \alpha}x_{\overline\alpha_{ 4}})=\pi(x_{ \alpha'}),$ i.e., 
$$y_{\gamma_{s(\overline\alpha_{1})},\overline\alpha_{1}}y^{-1}_{\gamma_{s(\overline\alpha)},\overline\alpha}y_{\gamma_{s( \alpha_{ 3})}, \alpha_{ 3}}+y_{\gamma_{s(\alpha_{ 2})},\alpha_{ 2}}y^{-1}_{\gamma_{s(\alpha)},\alpha}y_{\gamma_{s(\overline\alpha_{4})}, \overline\alpha_{4}}=y_{\gamma_{s(\alpha')},\alpha'}.$$

   \item For each bigon $(\alpha_1,\alpha_2)$ around a special puncture $p$, assume that $\alpha$ is the loop around $p$ such that $(\alpha_1,\alpha_2,\alpha)$ is a triangle and $\alpha'$ is the loop around $p$ such that $(\alpha',\alpha_2,\alpha_1)$ is a triangle, we have
\begin{equation*}
\begin{array}{rcl}
& &y_{\gamma_{s(\alpha')},\alpha'}\vspace{2.5pt}\\
&=&
y_{\gamma_{s(\overline\alpha_1)},\overline\alpha_{1}}y^{-1}_{\gamma_{s(\alpha)},\alpha}y_{\gamma_{s(\alpha_1)},\alpha_{1}}+2\cos(\frac{\pi}{|p|})y_{\gamma_{s(\overline\alpha_1)},\overline\alpha_{1}}y^{-1}_{\gamma_{s(\alpha)},\alpha}y_{\gamma_{s(\overline\alpha_2)},\overline\alpha_2}+y_{\gamma_{s(\alpha_2)},\alpha_{2}}y^{-1}_{\gamma_{s(\alpha)},\alpha}y_{\gamma_{s(\overline\alpha_2)},\overline\alpha_{2}}.
\end{array}
\end{equation*}
\end{itemize}

\end{corollary}

\medskip

{\bf Proof of Theorem \ref{th:sector presentation}}.

We first prove the relations in Theorem \ref{th:sector presentation} hold.

For the Triangle relations,
we have 
$$y_{\alpha_1,\alpha_2}y_{\alpha_3,\alpha_1}y_{\alpha_2,\alpha_3}=x_{\overline\alpha_1}^{-1}x_{\alpha_2}x_{\overline\alpha_3}^{-1}x_{\alpha_1}x^{-1}_{\overline\alpha_2}x_{\alpha_3}=1.$$

For the Ptolemy relations,
we have
$$y_{\alpha_1,\alpha'}=x_{\overline\alpha_1}^{-1}x_{\alpha'}=x_{\overline\alpha_1}^{-1}(x_{\overline\alpha_{1}}x^{-1}_{\overline\alpha}x_{ \alpha_{ 3}}+x_{ \alpha_{ 2}}x^{-1}_{ \alpha}x_{\overline\alpha_{ 4}})=x^{-1}_{\overline\alpha}x_{\alpha_3}+x^{-1}_{\overline\alpha_1}x_{\alpha_2}x^{-1}_{\alpha}x_{\overline\alpha_4}=
y_{\alpha,\alpha_3}+y_{\alpha_1,\alpha_2}y_{\overline \alpha,\overline\alpha_4}.$$
    
For the Monogon relations,
we have 
$y_{\gamma,\gamma}=x_{\overline\gamma}^{-1}x_{\gamma}=1$.

For the Bigon special puncture relations,
as
$x_{\alpha'}=x_{\overline\alpha_{1}}x^{-1}_{\alpha}x_{ \alpha_{ 1}}+2\cos(\frac{\pi}{|p|})x_{\overline\alpha_{ 1}}x^{-1}_{\alpha}x_{\overline\alpha_2}+x_{\alpha_{ 2}}x^{-1}_{\alpha}x_{\overline\alpha_{2}}$, we have
$$1=y_{\overline\alpha',\overline\alpha_1}y_{\overline\alpha,\alpha_1}+2\cos(\frac{\pi}{|p|})y_{\overline\alpha',\overline\alpha_1}y_{\overline\alpha,\overline\alpha_2}+y_{\overline\alpha',\alpha_2}y_{\overline\alpha,\overline\alpha_2}.$$

For the Star relations,
we have 
$$y_{\overline \gamma_1,\gamma_2}y_{\overline \gamma_2,\gamma_3}\cdots y_{\overline \gamma_k,\gamma_1}=x_{\gamma_1}^{-1}x_{\gamma_2}x^{-1}_{\gamma_2}x_{\gamma_3}\cdots x_{\gamma_k}^{-1}x_{\gamma_1}=1.$$ 

Thus the relations in Theorem \ref{th:sector presentation} hold.

We then show these are the defining relations.
It suffices to prove that the relations in Theorem \ref{th:sector presentation} imply the relations in Corollary \ref{cor:relationB}.

For any cyclic triangle $(\alpha_1,\alpha_2,\alpha_3)$ in $\Sigma$, we have 
\begin{equation*}
\begin{array}{rcl}
&& y_{\gamma_{s(\alpha_1)},\alpha_1}(y_{\gamma_{s(\overline\alpha_2)},\overline\alpha_2})^{-1}y_{\gamma_{s(\alpha_3)},\alpha_3}y^{-1}_{\gamma_{s(\overline\alpha_1)},\overline\alpha_1}y_{\gamma_{s(\alpha_2)},\alpha_2}y^{-1}_{\gamma_{s(\overline\alpha_3)},\overline\alpha_3}\vspace{2.5pt}\\
&=& y_{\gamma_{s(\alpha_1)},\alpha_1}y_{\alpha_2, \overline \gamma_{s(\overline\alpha_2)}}y_{\gamma_{s(\alpha_3)},\alpha_3}
y_{\alpha_1,\overline\gamma_{s(\overline\alpha_1)}}y_{\gamma_{s(\alpha_2)},\alpha_2}
y_{\alpha_3,\overline\gamma_{s(\overline\alpha_3)}} \vspace{2.5pt}\\
&=& y_{\gamma_{s(\alpha_1)},\alpha_1}y_{\alpha_2, \alpha_3}
y_{\alpha_1,\alpha_2}
y_{\alpha_3,\overline\gamma_{s(\overline\alpha_3)}}\vspace{2.5pt}\\
&=& y_{\gamma_{s(\alpha_1)},\alpha_1}
y_{\overline\alpha_1,\overline \alpha_3}
y_{\alpha_3,\overline\gamma_{s(\overline\alpha_3)}}
=
1,
\end{array}
\end{equation*}
where 
the last equality is followed by the Star relation. 

For any cyclic quadrilateral $( \alpha_{ 1}, \alpha_{ 2}, \alpha_{ 3}, \alpha_{ 4})$ with diagonals $ \alpha$ and $ \alpha'$  such that $s(\alpha)=s(\alpha_1)$ and 
$s(\alpha')=t(\alpha_1)$, as 
$y_{\alpha_1,\alpha'}=
y_{\alpha,\alpha_3}+y_{\alpha_1,\alpha_2}y_{\overline \alpha,\overline\alpha_4},$
we have
\begin{equation*}
\begin{array}{rcl}
y_{\gamma_{s(\alpha')},\alpha'}&=&y_{\gamma_{s(\alpha')},\overline\alpha_1}y_{\alpha_1,\alpha'}\vspace{2.5pt}\\
&=& y_{\gamma_{s(\alpha')},\overline\alpha_1}(y_{\alpha,\alpha_3}+y_{\alpha_1,\alpha_2}y_{\overline \alpha,\overline\alpha_4}) \vspace{2.5pt}\\
&=& y_{\gamma_{s(\alpha')},\overline\alpha_1}y_{\alpha,\alpha_3}+y_{\gamma_{s(\alpha')},\alpha_2}y_{\overline \alpha,\overline\alpha_4}
\vspace{2.5pt}\\
&=& y_{\gamma_{s(\overline\alpha_1)},\overline\alpha_1}y_{\alpha,\alpha_3}+y_{\gamma_{s(\alpha_2)},\alpha_2}y_{\overline \alpha,\overline\alpha_4}\vspace{2.5pt}\\
&=& 
y_{\gamma_{s(\overline\alpha_1)},\overline\alpha_1}y_{\alpha,\alpha_3}+y_{\gamma_{s(\alpha_2)},\alpha_2}y_{\overline \alpha,\overline\alpha_4},
\end{array}
\end{equation*}
and 
\begin{equation*}
\begin{array}{rcl}
&& y_{\gamma_{s(\overline\alpha_{1})},\overline\alpha_{1}}y^{-1}_{\gamma_{s(\overline\alpha)},\overline\alpha}y_{\gamma_{s( \alpha_{ 3})}, \alpha_{ 3}}+y_{\gamma_{s(\alpha_{ 2})},\alpha_{ 2}}y^{-1}_{\gamma_{s(\alpha)},\alpha}y_{\gamma_{s(\overline\alpha_{4})}, \overline\alpha_{4}}\\
&=& y_{\gamma_{s(\overline\alpha_{1})},\overline\alpha_{1}}y_{\alpha,\overline\gamma_{s(\overline\alpha)}}y_{\gamma_{s( \alpha_{ 3})}, \alpha_{ 3}}+y_{\gamma_{s(\alpha_{ 2})},\alpha_{ 2}}y_{\overline\alpha,\overline\gamma_{s(\alpha)}}y_{\gamma_{s(\overline\alpha_{4})}, \overline\alpha_{4}} \vspace{2.5pt}\\
&=& y_{\gamma_{s(\overline\alpha_{1})},\overline\alpha_{1}}y_{\alpha,\alpha_3}+y_{\gamma_{s(\alpha_{ 2})},\alpha_{ 2}}y_{\overline \alpha,\overline\alpha_4}.
\end{array}
\end{equation*}
Thus, we have
$$y_{\gamma_{s(\overline\alpha_{1})},\overline\alpha_{1}}y^{-1}_{\gamma_{s(\overline\alpha)},\overline\alpha}y_{\gamma_{s( \alpha_{ 3})}, \alpha_{ 3}}+y_{\gamma_{s(\alpha_{ 2})},\alpha_{ 2}}y^{-1}_{\gamma_{s(\alpha)},\alpha}y_{\gamma_{s(\overline\alpha_{4})}, \overline\alpha_{4}}=y_{\gamma_{s(\alpha')},\alpha'}.$$

For each loop $\gamma$ cuts out a monogon which contains only a special puncture, by the Star relations, we have
$y_{\gamma,\overline\gamma_{s(\gamma)}}y_{\gamma_{s(\gamma)},\gamma}=y_{\gamma,\gamma}=1$. Thus $y_{\gamma_{s(\gamma)},\gamma}=y^{-1}_{\gamma,\overline\gamma_{s(\gamma)}}=y_{\gamma_{s(\overline \gamma)},\overline\gamma}$.

For each bigon $(\alpha_1,\alpha_2)$ around a special puncture $p$ of order $3$, assume that $\alpha$ is the loop around $p$ such that $(\alpha_1,\alpha_2,\alpha)$ is a triangle and $\alpha'$ is the loop around $p$ such that $(\alpha',\alpha_2,\alpha_1)$ is a triangle, we have 
\begin{equation*}
\begin{array}{rcl}
 y_{\gamma_{s(\alpha')},\alpha'}
&=& y_{\gamma_{s(\alpha')},\alpha'}
y_{\overline\alpha',\overline\alpha_1}y_{\overline\alpha,\alpha_1}+
2\cos(\frac{\pi}{|p|})y_{\gamma_{s(\alpha')},\alpha'}y_{\overline\alpha',\overline\alpha_1}y_{\overline\alpha,\overline\alpha_2}+
y_{\gamma_{s(\alpha')},\alpha'}y_{\overline\alpha',\alpha_2}y_{\overline\alpha,\overline\alpha_2}
\vspace{2.5pt}\\
&=& y_{\gamma_{s(\overline\alpha_1)},\overline\alpha_{1}}
y_{\overline \alpha,\alpha_1}
+2\cos(\frac{\pi}{|p|})y_{\gamma_{s(\overline\alpha_1)},\overline\alpha_{1}}y_{\overline\alpha,\overline\alpha_2}
+y_{\gamma_{s(\alpha_2)},\alpha_{2}}
y_{\overline\alpha,\overline\alpha_{2}}.
\end{array}
\end{equation*}
\begin{equation*}
\begin{array}{rcl}
&& y_{\gamma_{s(\overline\alpha_1)},\overline\alpha_{1}}y^{-1}_{\gamma_{s(\alpha)},\alpha}y_{\gamma_{s(\alpha_1)},\alpha_{1}}+
2\cos(\frac{\pi}{|p|})y_{\gamma_{s(\overline\alpha_1)},\overline\alpha_{1}}y^{-1}_{\gamma_{s(\alpha)},\alpha}y_{\gamma_{s(\overline\alpha_2)},\overline\alpha_2}
+y_{\gamma_{s(\alpha_2)},\alpha_{2}}y^{-1}_{\gamma_{s(\alpha)},\alpha}y_{\gamma_{s(\overline\alpha_2)},\overline\alpha_{2}}\\
&=& y_{\gamma_{s(\overline\alpha_1)},\overline\alpha_{1}}y_{\overline \alpha,\overline\gamma_{s(\alpha)}}y_{\gamma_{s(\alpha_1)},\alpha_{1}}+
2\cos(\frac{\pi}{|p|})y_{\gamma_{s(\overline\alpha_1)},\overline\alpha_{1}}y_{\overline\alpha,\overline\gamma_{s(\alpha)}}
y_{\gamma_{s(\overline\alpha_2)},\overline\alpha_2}
+
y_{\gamma_{s(\alpha_2)},\alpha_{2}}y_{\overline\alpha,\overline\gamma_{s(\alpha)}}y_{\gamma_{s(\overline\alpha_2)},\overline\alpha_{2}}\vspace{2.5pt}\\
&=& y_{\gamma_{s(\overline\alpha_1)},\overline\alpha_{1}}
y_{\overline \alpha,\alpha_1}
+2\cos(\frac{\pi}{|p|})y_{\gamma_{s(\overline\alpha_1)},\overline\alpha_{1}}y_{\overline\alpha,\overline\alpha_2}
+y_{\gamma_{s(\alpha_2)},\alpha_{2}}
y_{\overline\alpha,\overline\alpha_{2}}.
\end{array}
\end{equation*}
Thus, \begin{equation*}
y_{\gamma_{s(\alpha')},\alpha'}=
y_{\gamma_{s(\overline\alpha_1)},\overline\alpha_{1}}y^{-1}_{\gamma_{s(\alpha)},\alpha}y_{\gamma_{s(\alpha_1)},\alpha_{1}}+2\cos(\frac{\pi}{|p|})y_{\gamma_{s(\overline\alpha_1)},\overline\alpha_{1}}y^{-1}_{\gamma_{s(\alpha)},\alpha}y_{\gamma_{s(\overline\alpha_2)},\overline\alpha_2}+y_{\gamma_{s(\alpha_2)},\alpha_{2}}y^{-1}_{\gamma_{s(\alpha)},\alpha}y_{\gamma_{s(\overline\alpha_2)},\overline\alpha_{2}}.
\end{equation*}

The proof is complete.
\endproof

\medskip

{\bf Proof of Theorem \ref{th:reduced surfaces}.} Denote by $\mathcal I$ the kernel of the canonical homomorphism ${\mathcal A}_\Sigma\twoheadrightarrow \underline {\mathcal A}_\Sigma$ (i.e., the ideal of  generated by $\{ x_\gamma-1 \mid \gamma \text{ is a boundary arc}\}$). 

In Proposition \ref{prop:quotient}, choose $\gamma_i$, $i\in I_b\cup I_{p,1}$ in such a way that $\gamma_i$ is a boundary arc iff $i\in I_b$. 

Since $\pi(\mathcal I)\subset \mathcal I$, there is a unique algebra homomorphism $\underline\pi: \underline {\mathcal A}_\Sigma\to \underline {\mathcal A}_\Sigma$ such that $\underline \pi(\underline x)=\underline{\pi(x)}$ for all $x\in \mathcal A_\Sigma$.
Clearly, $\underline \pi$ is a projection onto $\mathcal B_\Sigma$. It is also clear that $Ker~\underline \pi$ is generated by all $\underline x_{\gamma_i}$, $i\in I_{p,1}$. If $I_{p,1}= \emptyset$, then $Ker~\underline \pi=\{0\}$. Otherwise, 
Lemma \ref{le:abelianization of a surface} implies that $\underline x_{\gamma_i}\ne 1$ because it is a cluster variable in the abelianization/symmetrization of $\mathcal A_\Sigma$. Therefore, $Ker~\underline \pi \ne \{0\}$ if $I_{p,1}\ne \emptyset$.

Thus $\underline \pi$ is an isomorphism $\underline{\mathcal A}_\Sigma\cong \underline {\mathcal B}_\Sigma$ iff $I_{p,1}= \emptyset$.

The proof is complete.
\endproof

\subsection{Proof of Theorem \ref{th:presentation of Tsurf}}\label{sec:presentation of Tsurf}

\begin{lemma}\label{lem:dual}(\cite{FZ4}, \cite[Proposition 1.3]{NZ}) We have    $C^{\Delta}_{\mu_\alpha\Delta_0}=C^{\Delta}_{\Delta_0}(J_\alpha+[-\varepsilon_\alpha B_{\Delta_0}]_+^{\bullet \alpha})$, where $B_{\Delta_0}$ is the exchange matrix associated with $\Delta_0$, the notation $[M]^{\bullet \alpha}$ means all columns of the matrix $M$ are set to zero except the $\alpha$-th column, $\varepsilon_\alpha$ is the sign of the $\alpha$-th column of the $C$-matrix $C_{\Delta_0}^\Delta$.
\end{lemma}

We first prove that ${\bf TSurf}_\Sigma$ satisfies the relations.

For the Diamond/Pentagon/Hexagon relation, assume that $\Delta_{i+1}=\mu_{\alpha_i}(\Delta_i)$, $\Delta_k=\mu_{\beta_1}(\Delta_1)$, and $\Delta_{k-1}=\mu_{\beta_2}(\Delta_k)$. {Then $\alpha=\alpha_1, \beta=\alpha_2$. Since the pair $(\alpha, \beta)$ is not directed clockwise in $\Delta_1$, it follows from Theorem \ref{thm:56cycles} and direct calculation (see Figure \ref{Fig:56}) that}
\[
\mathrm{sgn}_{\alpha_1}(C_{\Delta_1}^{\Delta_1}) = \mathrm{sgn}_{\alpha_2}(C_{\Delta_2}^{\Delta_1}) = \cdots = \mathrm{sgn}_{\alpha_{k-2}}(C_{\Delta_{k-1}}^{\Delta_1}) = \mathrm{sgn}_{\beta_1}(C_{\Delta_1}^{\Delta_1}) = \mathrm{sgn}_{\beta_2}(C_{\Delta_1}^{\Delta_1}) = +.
\]
In fact, the sequence $\alpha_1, \dots, \alpha_k$ is a green sequence.

Thus, $h_{\Delta_1,\Delta_{k-1}}=h_{\Delta_1,\Delta_k}h_{\Delta_k,\Delta_{k-1}}=h_{\Delta_1,\Delta_2} h_{\Delta_{2},\Delta_3}\cdots h_{\Delta_{k-2},\Delta_{k-1}}.$

For horizontal compatibility, suppose $(\beta,\alpha)$ is directed clockwise in $\Delta$. Then we have 
$$sgn_{\beta}(C_{\mu_\beta\mu_\alpha\Delta}^{\mu_\beta\mu_\alpha\Delta})=sgn_{\alpha}(C_{\mu_\beta\Delta}^{\mu_\beta\mu_\alpha\Delta})=sgn_{\beta}(C_{\overline{\mu_\beta\Delta}}^{\overline{\mu_\beta\Delta}})=sgn_{\alpha}(C_{\overline{\Delta}}^{\overline{\mu_\beta\Delta}})=+,$$
$$sgn_{\beta}(C_{\Delta}^{\mu_\beta\mu_\alpha\Delta})=sgn_{\beta}(C_{\overline{\mu_\alpha\Delta}}^{\overline{\mu_\beta\Delta}})=-.$$

Therefore, $h_{\mu_\beta\mu_\alpha\Delta,\mu_\beta\Delta}= h^{-1}_{\mu_\alpha\Delta,\mu_\beta\mu_\alpha\Delta}h_{\mu_\alpha\Delta,\Delta}h_{\Delta,\mu_\beta\Delta}=h_{\mu_\beta\mu_\alpha\Delta,\mu_\alpha\Delta}h_{\mu_\alpha\Delta,\Delta}h^{-1}_{\mu_\beta\Delta,\Delta}.$

It follows that $h_{\mu_\alpha\Delta,\Delta}h_{\Delta,\mu_\beta\Delta}h_{\mu_\beta\Delta,\Delta}=h_{\mu_\alpha\Delta,\mu_\beta\mu_\alpha\Delta}h_{\mu_\beta\mu_\alpha\Delta,\mu_\alpha\Delta}h_{\mu_\alpha\Delta,\Delta}.$

\medskip

Let ${\bf TSurf}'_\Sigma$ be the groupoid defined by the presentation in this theorem. For any ordinary triangulations $\Delta,\Delta'$, we define a morphism $h_{\Delta,\Delta'}$ in ${\bf TSurf}'_\Sigma$ as follows:

$\bullet$ $h_{\Delta,\Delta'}=id_\Delta$ if $dist(\Delta,\Delta')=0$,

$\bullet$ $h_{\Delta,\Delta'}=h_{\Delta,\mu_\alpha\Delta'}h_{\mu_\alpha\Delta',\Delta'}^{sgn_\alpha(C^{\Delta}_{\mu_\alpha\Delta'})}$ for $\alpha$ such that $dist(\Delta,\mu_\alpha\Delta')<dist(\Delta,\Delta')$.

We claim $h_{\Delta,\Delta'}$ is well-defined. Let $\mu^1:\Delta\to \Delta'$ and $\mu^2:\Delta\to \Delta'$ be two shortest mutation sequences from $\Delta$ to $\Delta'$. {By Theorem \ref{thm:56cycles}}, the fundamental group of the graph of flips is generated by cycles of lengths $4$, $5$ and $6$, it suffices to consider the case where $(\mu^2)^{-1}\circ \mu^1:\Delta\to \Delta$ forms a simple cycle. Then $(\mu^2)^{-1}\circ \mu^1$ has length $4$ or length $6$, and the well-defined of $h_{\Delta,\Delta'}$ follows by the diamond and hexagonal relations, and Lemma \ref{lem:dual}.

We now prove that the morphisms $h_{\Delta,\Delta'}$ satisfy the relations in Definition \ref{def:gro}. 

Suppose that $\Delta_0,\Delta$ are two triangulations and $\alpha$ is a non-self-folded internal arc in $\Delta$. 

If $dist(\Delta_0,\Delta)\neq dist(\Delta_0,\mu_\alpha\Delta)$, then $$h_{\Delta_0,\mu_\alpha\Delta}=h_{\Delta_0,\Delta}h_{\Delta,\mu_\alpha\Delta}^{sgn_\alpha(C^{\Delta_0}_{\Delta})}=h_{\Delta_0,\Delta}h_{\Delta,\mu_\alpha\Delta}^{\varphi(\Delta_0;\Delta,\mu_\alpha\Delta)}.$$

If $dist(\Delta_0,\Delta)=dist(\Delta_0,\mu_\alpha\Delta)$. Assume that $dist(\Delta_0,\Delta)=dist(\Delta_0,\mu_\beta\Delta)+1$ for some $\beta$. Then 
$\mu_\alpha\mu_\beta\Delta=\mu_{\alpha'}\mu_\beta\mu_\alpha\Delta$ for $\alpha'\in \mu_\alpha\Delta\setminus \Delta$ and
$\mu_\beta\mu_\alpha\Delta,\mu_\alpha\Delta,\Delta,\mu_\beta\Delta,\mu_\alpha\mu_\beta\Delta$ form a $5$-cycle. We further have
$dist(\Delta_0,\mu_\beta\Delta)=dist(\Delta_0,\mu_\alpha\mu_\beta\Delta)+1$,
$dist(\Delta_0,\mu_\alpha\Delta)=dist(\Delta_0,\mu_\beta\mu_\alpha\Delta)+1$, and
$dist(\Delta_0,\mu_\beta\mu_\alpha\Delta)=dist(\Delta_0,\mu_\alpha\mu_\beta\Delta)+1$.

Thus, we have $$h_{\Delta_0,\Delta}=h_{\Delta_0,\mu_\beta\Delta}h_{\mu_\beta\Delta,\Delta}^{sgn_\beta(C^{\Delta_0}_{\mu_\beta\Delta})}=h_{\Delta_0,\mu_\alpha\mu_\beta\Delta}h_{\mu_\alpha\mu_\beta\Delta,\Delta}^{sgn_\alpha(C^{\Delta_0}_{\mu_\alpha\mu_\beta\Delta})}h_{\mu_\beta\Delta,\Delta}^{sgn_\beta(C^{\Delta_0}_{\mu_\beta\Delta})},$$
and 
\begin{equation*}
\begin{array}{rcl}
h_{\Delta_0,\mu_\alpha\Delta}
&=&h_{\Delta_0,\mu_\beta\mu_\alpha\Delta}h_{\mu_\beta\mu_\alpha\Delta,\Delta}^{sgn_\beta(C^{\Delta_0}_{\mu_\beta\mu_\alpha\Delta})}=h_{\Delta_0,\mu_{\alpha'}\mu_\beta\mu_\alpha\Delta}h_{\mu_{\alpha'}\mu_\beta\mu_\alpha\Delta,\Delta}^{sgn_{\alpha'}(C^{\Delta_0}_{\mu_{\alpha'}\mu_\beta\mu_\alpha\Delta})}
h_{\mu_\beta\mu_\alpha\Delta,\Delta}^{sgn_\beta(C^{\Delta_0}_{\mu_\beta\mu_\alpha\Delta})}\\
&=&
h_{\Delta_0,\mu_\alpha\mu_\beta\Delta}h_{\mu_\alpha\mu_\beta\Delta,\Delta}^{sgn_{\alpha'}(C^{\Delta_0}_{\mu_\alpha\mu_\beta\Delta})}
h_{\mu_\beta\mu_\alpha\Delta,\Delta}^{sgn_\beta(C^{\Delta_0}_{\mu_\beta\mu_\alpha\Delta})}.
\end{array}
\end{equation*}

By Lemma \ref{lem:dual} and the pentagon relation, we have 
$$h_{\Delta_0,\mu_\alpha\Delta}=h_{\Delta_0,\Delta}h_{\Delta,\mu_\alpha\Delta}^{sgn_\alpha(C^{\Delta_0}_{\Delta})}=h_{\Delta_0,\Delta}h_{\Delta,\mu_\alpha\Delta}^{\varphi(\Delta_0;\Delta,\mu_\alpha\Delta)}.$$

For any triangulations $\Delta_0,\Delta$ and non-self-folded arc $\alpha\in \Delta$ such that $dist(\Delta,\Delta_0)=2$ and $dist(\mu_\alpha\Delta,\Delta_0)=3$, assume that $\Delta=\mu_{\beta_2}\mu_{\beta_1}\Delta_0$. 

If $\alpha\in \Delta_0$, then $h_{\Delta,\Delta_0}=h_{\Delta,\mu_{\beta_1}\Delta_0}h_{\mu_{\beta_1}\Delta_0,\Delta_0},$ $h_{\mu_\alpha\Delta,\Delta_0}=h_{\mu_\alpha\Delta,\Delta}h_{\Delta,\mu_{\beta_1}\Delta_0}h_{\mu_{\beta_1}\Delta_0,\Delta_0}$, and $sgn_{\alpha}(C_{\overline\Delta}^{\overline\Delta_0})=+$. Thus, $h_{\mu_\alpha\Delta,\Delta_0}=h_{\mu_\alpha\Delta,\Delta}^{\phi(\Delta_0;\Delta,\mu_\alpha\Delta)}h_{\Delta,\Delta_0}$.

If $\alpha\notin \Delta_0$, then $\alpha\in \mu_{\beta_1}\Delta_0\setminus \Delta_0$. Suppose $(\alpha,\beta_2)$ is not directed clockwise in $\Delta$. Then $h_{\Delta,\Delta_0}=h_{\Delta,\mu_{\beta_1}\Delta_0}h_{\mu_{\beta_1}\Delta_0,\Delta_0},$ $h_{\mu_\alpha\Delta,\Delta_0}=h_{\mu_\alpha\Delta,\Delta}h_{\Delta,\mu_{\beta_1}\Delta_0}h_{\mu_{\beta_1}\Delta_0,\Delta_0}$, and $sgn_{\alpha}(C_{\overline\Delta}^{\overline\Delta_0})=+$. Thus, $h_{\mu_\alpha\Delta,\Delta_0}=h_{\mu_\alpha\Delta,\Delta}^{\phi(\Delta_0;\Delta,\mu_\alpha\Delta)}h_{\Delta,\Delta_0}$.

Suppose $(\alpha,\beta_2)$ is directed clockwise in $\Delta$. Then $h_{\Delta,\Delta_0}=h_{\Delta,\mu_{\beta_1}\Delta_0}h_{\mu_{\beta_1}\Delta_0,\Delta_0},$ $h_{\mu_\alpha\Delta,\Delta_0}=h_{\mu_\alpha\Delta,\mu_{\beta_1}\Delta}h_{\mu_{\beta_1}\Delta_0,\Delta_0}^{-1}=h_{\mu_\alpha\Delta,\Delta}h_{\Delta,\mu_{\beta_1}\Delta}h_{\mu_{\beta_1}\Delta_0,\Delta_0}^{-1}$. Thus, $h_{\mu_\alpha\Delta,\Delta_0}=h_{\mu_\alpha\Delta,\Delta}^{\phi(\Delta_0;\Delta,\mu_\alpha\Delta)}h_{\Delta,\Delta_0}$, and $sgn_\alpha(C_{\Delta}^{\overline{\Delta_0}})=-$.

By horizontal compatibility, we have $h_{\mu_\alpha\Delta,\Delta_0}=h_{\Delta,\mu_\alpha\Delta}^{-1}h_{\Delta,\Delta_0}=h_{\mu_\alpha\Delta,\Delta}^{\phi(\Delta_0;\Delta,\mu_\alpha\Delta)}h_{\Delta,\Delta_0}$.

The proof is complete.
\endproof

\subsection{Proof of Theorem \ref{th:brgroup}}\label{sec:proofthm:fundegroup}

For any ordinary triangulation $\Delta$, denote $\widetilde {Br}_\Delta$ the group generated by $T_\gamma$, $\gamma$ runs over all non-pending internal edges (up to reversal) of $\Delta$, and subject to the relations in Theorem \ref{th:brgroup}.

Given a group $G$, denote $x^y:=yxy^{-1}$ for $x,y\in G$. We use the following notation:

$\bullet$ $Co(x,y)$ if $xy=yx$, 

$\bullet$ $Br_3(x,y)$ if $xyx=yxy$, 

$\bullet$ $Br_4(x,y)$ if $xyxy=yxyx$,

$\bullet$ $Cyl(x_1,x_2,\cdots, x_n)$ if $x_1x_2\cdots x_nx_1\cdots x_{n-2}=x_2x_3\cdots x_nx_1\cdots x_{n-1}$.

It is easy to verify that $Cyl(x_1,x_2,\cdots, x_n)$ holds if and only if $Cyl(x_2,x_3,\cdots, x_n,x_1)$ holds, provided that $Br_{3}(x_{i},x_{i+1})$ for $i=1,2,\cdots,n-1$ and $Br_3(x_1,x_n)$ hold, see \cite{Q}. 
 
Before proceeding, we first establish the following result. {Recall that the quiver $Q_\Delta$ associated with $\Delta$ was defined in Remark \ref{rem:equi}.} Throughout, we will repeatedly appeal to the equivalence given in Remark \ref{rem:equi}.

\begin{theorem}\label{thm:quiverbraidgroup}
    Let $\Delta,\Delta'$ be two ordinary triangulations of $\Sigma$. Assume that $\Delta'=\mu_{\alpha_0}(\Delta)$ and $\alpha_0'\in \Delta'\setminus \Delta$. There are mutually inverse canonical group isomorphisms 
    $$h_{\Delta',\Delta}: \widetilde{Br}_{\Delta}\cong \widetilde{Br}_{\Delta'}, \hspace{5mm} h^{-}_{\Delta',\Delta}: \widetilde{Br}_{\Delta'}\cong \widetilde{Br}_{\Delta}$$
    satisfying
    \begin{equation}\label{eq:h}
       h_{\Delta',\Delta}(T_{\beta})=\begin{cases}
    T_{\alpha'_0} & \text{if $\beta=\alpha_0$}\\
    T_{\alpha'_0} T_\beta {T^{-1}_{\alpha'_0}} & \text{if there is an arrow from $\alpha_0$ to $\beta$ in $Q_{\Delta}$}\\
    T_\beta & \text{otherwise}. 
    \end{cases}
    \end{equation}
\begin{equation}\label{eq:h^-}
       h^-_{\Delta,\Delta'}(T_{\beta})=\begin{cases}
    T_{\alpha_0} & \text{if $\beta=\alpha'_0$}\\
    T^{-1}_{\alpha_0} T_\beta {T_{\alpha_0}} & \text{if there is an arrow from $\alpha_0$ to $\beta$ in $Q_\Delta$}\\
    T_\beta & \text{otherwise}. 
    \end{cases}
    \end{equation}
\end{theorem}

\subsubsection{Proof of Theorem \ref{thm:quiverbraidgroup}}\label{sec:proofofquiverbraidgroup}

For {non-pending internal arcs} $\alpha,\beta\in \Delta$, denote by $Q_\Delta(\alpha,\beta)$ the difference of the number of arrows from $\beta$ to $\alpha$ and the number of arrows from $\alpha$ to $\beta$ in $Q_\Delta$.

Equations (\ref{eq:h}) and (\ref{eq:h^-}) define a pair of mutually inverse isomorphisms between the free groups generated by the sets of {non-pending internal} arcs in $\Delta$ and $\Delta'$, with arcs identified up to their reversed directions. So we only need to prove that the relations $R1$ to $R9$ in Theorem \ref{th:brgroup} are preserved under $h_{\Delta',\Delta}$. 

{To distinguish between the relations in $\widetilde{Br}_\Delta$ and $\widetilde{Br}_{\Delta'}$, we append the subscript $\Delta$ to the relations, writing them as $Co(x,y;\Delta)$, $Br_3(x,y;\Delta)$, $Br_4(x,y;\Delta)$, and $Cyl(x_1,\dots,x_n;\Delta)$ hereafter.}

We may assume that the arcs are non-self-folded, as we can replace self-folded arcs with loops around them otherwise.

\textbf{For $R1$}: If $\alpha_0 = \alpha$ or $\alpha_0 = \beta$, then the image of the relation $Co(T_\alpha, T_\beta; \Delta)$ under $h_{\Delta', \Delta}$ is equivalent to either $Co(T_{\alpha'_0}, T_\beta; \Delta')$ or $Co(T_\alpha, T_{\alpha'_0}; \Delta')$. This equivalence holds by condition $R1$ for $\Delta'$.

We then assume that $\alpha_0\neq \alpha,\beta$. {From \eqref{eq:h}, we need to consider the signs of $Q_\Delta(\alpha_0,\alpha)$ and $Q_\Delta(\alpha_0,\beta)$}.

(Case 1) $Q_\Delta(\alpha_0,\alpha),Q_\Delta(\alpha_0,\beta)\geq 0$. Then $h_{\Delta',\Delta}(Co(T_\alpha,T_\beta;\Delta))$ is equivalence to $Co(T_\alpha,T_\beta;\Delta')$, which holds by $R1$ for $\Delta'$. 

(Case 2)  $Q_\Delta(\alpha_0,\alpha)<Q_\Delta(\alpha_0,\beta)=0$ or $Q_\Delta(\alpha_0,\beta)<Q_\Delta(\alpha_0,\alpha)=0$. We may assume that $Q_\Delta(\alpha_0,\alpha)<Q_\Delta(\alpha_0,\beta)=0$. Then $h_{\Delta',\Delta}(Co(T_\alpha,T_\beta;\Delta))$ follows by $Co(T_{\alpha},T_\beta;\Delta')$ and $Co(T_{\alpha'_0},T_\beta;\Delta')$. 

(Case 3) $Q_\Delta(\alpha_0,\alpha)<0<Q_\Delta(\alpha_0,\beta)$ or $Q_\Delta(\alpha_0,\beta)<0<Q_\Delta(\alpha_0,\alpha)$. We may assume that $Q_\Delta(\alpha_0,\alpha)<0<Q_\Delta(\alpha_0,\beta)$. Then there is a 3-cycle between $\alpha'_0,\beta,\alpha$ in $Q_{\Delta'}$.

(Case 3.1) $w(\alpha_0)\neq 1$. As $Q_\Delta(\alpha,\beta)=0$, we have $w(\alpha)=w(\beta)=1$ and there is a double arrow from $\beta$ to $\alpha$ in $Q_{\Delta'}$. Then $h_{\Delta',\Delta}(Co(\alpha,\beta;\Delta))$ is equivalence to $Co(T_{\alpha}^{T_{\alpha'_0}},T_\beta;\Delta')$, which holds by $R4$ for $\Delta'$.

(Case 3.2) $w(\alpha_0)=1$. Then there are no double arrows between $\alpha'_0,\alpha$, and $\beta$. Thus $h_{\Delta',\Delta}(Co(\alpha,\beta;\Delta))\Leftrightarrow Co(T_{\alpha}^{T_{\alpha'_0}},T_\beta;\Delta')$, which holds by $R3$ for $\Delta'$. 

\medskip

\textbf{For $R2$}: If $\alpha_0=\alpha$ or $\beta$ then $h_{\Delta',\Delta}(Br_*(T_\alpha,T_\beta;\Delta))\Leftrightarrow Br_*(T_{\alpha_0'},T_\beta;\Delta')$ or $Br_*(T_\alpha,T_{\alpha_0'};\Delta')$ for $*\in \{3,4\}$, which holds by $R2$ for $\Delta'$.

We now consider the case $\alpha_0\neq \alpha,\beta$. {From \eqref{eq:h}, we need to consider the signs of $Q_\Delta(\alpha_0,\alpha)$ and $Q_\Delta(\alpha_0,\beta)$}.

(Case 1) $Q_\Delta(\alpha_0,\alpha),Q_\Delta(\alpha_0,\beta)\geq 0$. Then $h_{\Delta',\Delta}(Br_*(\alpha,\beta;\Delta))\Leftrightarrow Br_*(T_{\alpha},T_\beta;\Delta')$, which holds by $R2$ for $\Delta'$.

(Case 2)  $Q_\Delta(\alpha_0,\alpha)<Q_\Delta(\alpha_0,\beta)=0$ or $Q_\Delta(\alpha_0,\beta)<Q_\Delta(\alpha_0,\alpha)=0$. Then $h_{\Delta',\Delta}(Br_*(\alpha,\beta;\Delta))$ follows by $Br_*(T_{\alpha},T_\beta;\Delta')$, and $Co(T_{\alpha'_0},T_\beta;\Delta')$ or $Co(T_{\alpha'_0},T_\alpha;\Delta')$. 

(Case 3) $Q_\Delta(\alpha_0,\alpha)<0<Q_\Delta(\alpha_0,\beta)$. 

We first assume that there is no double arrow between $\alpha_0,\beta$, and $\alpha$ in $Q_\Delta$. If $w(\alpha_0)=1$, then there is no arrow between $\alpha,\beta$ in $Q_{\Delta'}$. Thus, $h_{\Delta',\Delta}(Br_*(\alpha,\beta;\Delta))$ follows by Lemma \ref{lem:r21} below. If $w(\alpha_0)\neq 1$, then there is a three cycle between $\alpha_0',\beta$ and $\alpha$ but there is no double arrow between them. Then $h_{\Delta',\Delta}(Br_*(\alpha,\beta;\Delta))$ follows by Lemma \ref{lem:r23}.

We then assume that there is a double arrow between $\alpha_0,\beta$, and $\alpha$ in $Q_\Delta$. Then there is $3$-cycle between $\alpha'_0,\beta,\alpha$ and a double arrow between $\alpha'_0,\beta$ and $\alpha$ but no double arrow from $\beta$ to $\alpha$ in $Q_{\Delta'}$. Thus, $h_{\Delta',\Delta}(Br_*(\alpha,\beta;\Delta))$ follows by Lemma \ref{lem:R231}.

(Case 4) $Q_\Delta(\alpha_0,\beta)<0<Q_\Delta(\alpha_0,\alpha)$. Then $w(\alpha)=w(\beta)=1$ and there is a $3$-cycle between $\alpha_0',\alpha,\beta$ and a double arrow from $\alpha$ to $\beta$ in $Q_{\Delta'}$. We thus have $w(\alpha_0)=1$ (otherwise, there is no arrow between $\alpha$ and $\beta$ in $Q_\Delta$). Therefore, $h_{\Delta',\Delta}(Br_3(\alpha,\beta;\Delta))\Leftrightarrow Br_3(T_\beta^{T_{\alpha'_0}},T_\beta;\Delta')$, which holds by $R4$ for $\Delta'$.

\medskip

\textbf{For $R3$}: If $\alpha_0=\alpha$, then $h_{\Delta',\Delta}(Co(T_\gamma^{T_\alpha},T_\beta;\Delta))\Leftrightarrow Co(T_\gamma,T_\beta;\Delta')$, which holds by $R1$ for $\Delta'$. 

If $\alpha_0=\beta$, then
$h_{\Delta',\Delta}(Co(T_\gamma^{T_\alpha},T_\beta;\Delta))$ follows by $Br_3(T_\alpha,T_{\alpha'_0};\Delta')$ and $Co(T_\alpha,T_\gamma;\Delta')$ in case $w(\beta)=1$ and $Br_4(T_\alpha,T_{\alpha'_0};\Delta')$ and $Co(T_{\alpha'_0}^{T_{\alpha}},T_\gamma;\Delta')$ in case $w(\beta)\neq 1$. 

If $\alpha_0=\gamma$, then $h_{\Delta',\Delta}(Co(T_\gamma^{T_\alpha},T_\beta;\Delta))$ follows by $Br_3(T_\alpha,T_{\alpha_0'};\Delta')$ and $Co(T_\alpha,T_\beta;\Delta')$ in case $w(\gamma)=1$ and $Br_4(T_\alpha,T_{\alpha'_0};\Delta')$ and $Co(T_{\alpha'_0},T_\beta^{T_{\alpha}};\Delta')$ in case $w(\beta)\neq 1$.

As there is a $3$-cycle between $\alpha,\beta,\gamma$ in $Q_\Delta$ but no double arrow between them, we have $\alpha,\beta,\gamma$ form a triangle in $\Delta$ or $\{\alpha,\beta,\gamma\}$ is a complete counter-clockwise list of the arcs incident to some puncture. If the latter case occurs, then $w(\alpha)=w(\beta)=w(\gamma)=1$ and $Co(T_\gamma^{T_\alpha},T_\beta;\Delta)\xLeftrightarrow{Br_3(T_\alpha,T_\gamma)} Cyl(T_{\alpha},T_\gamma,T_\beta;\Delta)$. We defer the proof of this case to the proof for the relation $R9$.

We now consider the case that $\alpha_0\neq \alpha,\beta,\gamma$ and $\alpha,\beta,\gamma$ form a triangle in $\Delta$. {From \eqref{eq:h}, we need to consider the signs of $Q_\Delta(\alpha_0,\alpha)$, $Q_\Delta(\alpha_0,\beta)$ and $Q_\Delta(\alpha_0,\gamma)$}.

(Case 1) $Q_\Delta(\alpha_0,\alpha),Q_\Delta(\alpha_0,\beta),Q_\Delta(\alpha_0,\gamma)\geq 0$. Then $$h_{\Delta',\Delta}(Co(T_\gamma^{T_\alpha},T_\beta;\Delta))\Leftrightarrow Co(T_\gamma^{T_\alpha},T_\beta;\Delta'),$$ which holds by $R3$ for $\Delta'$. 

(Case 2) $Q_\Delta(\alpha_0,\alpha_1)<0=Q_\Delta(\alpha_0,\alpha_2)=Q_\Delta(\alpha_0,\alpha_3)=0$ for $\{\alpha_1,\alpha_2,\alpha_3\}=\{\alpha,\beta,\gamma\}$. Then 
$h_{\Delta',\Delta}(Co(T_\gamma^{T_\alpha},T_\beta;\Delta))$ follows by 
$Co((T_\gamma)^{T_\alpha},T_\beta;\Delta')$, $Co(T_{\alpha_0'},T_{\alpha_2};\Delta')$ and $Co(T_{\alpha_0'},T_{\alpha_3};\Delta')$.

As there is no double arrow between $\alpha,\beta$, and $\gamma$, we have any two of $\{\alpha,\beta,\gamma\}$ cannot be two sides of two different triangles in $\Delta$. Therefore, $\alpha_0$ connects at most two of $\alpha,\beta,\gamma$ in $Q_\Delta$. We have the remaining cases to be considered:

(Case 3) $Q_\Delta(\alpha_0,\alpha), Q_\Delta(\alpha_0,\beta)\neq 0=Q_\Delta(\alpha_0,\gamma)$. 

(Case 3.1) $Q_\Delta(\alpha_0,\alpha), Q_\Delta(\alpha_0,\beta)<0$. Then $h_{\Delta',\Delta}(Co(T_\gamma^{T_\alpha},T_\beta;\Delta))$ follows by 
$Co(T_\gamma^{T_\alpha},T_\beta;\Delta')$ and $Co(T_{\alpha_0'},T_{\gamma};\Delta')$.

(Case 3.2) $Q_\Delta(\alpha_0,\alpha)<0 <Q_\Delta(\alpha_0,\beta)$. Then $Q_\Delta(\alpha_0,\alpha)=-1, Q_\Delta(\alpha_0,\beta)=1$ and $w(\alpha_0)=w(\alpha)=w(\beta)=1$. Thus, the subquiver of $Q_{\Delta'}$ formed by $\alpha'_0,\beta,\gamma,\alpha$ is isomorphic to the third quiver in Figure \ref{Fig:subquiver2}. Then 
$$
\begin{array}{rcl}
h_{\Delta',\Delta}(Co(T_\gamma^{T_\alpha},T_\beta;\Delta))\Leftrightarrow Co(T_\gamma^{T_{\alpha_0'}T_{\alpha}},T_\beta;\Delta')
&\xLeftrightarrow{Br_3({T_\beta,T_{\alpha'_0}};\Delta')}& Co(T_\gamma^{T_{\alpha}},{T_{\alpha_0'}^{T_\beta}};\Delta')  \\
&\Leftrightarrow&
Co(T_\gamma,{T_{\alpha_0'}^{{T_{\alpha}^{-1}}T_\beta}};\Delta')  \\
&\xLeftrightarrow[Co({T_\beta,T_{\alpha}};\Delta')]{Br_3({T_\alpha,T_{\alpha'_0}};\Delta')}&
Co(T_\gamma,T_{\alpha}^{T_\beta T_{\alpha_0'}};\Delta')  \\
&\Leftrightarrow& 
Co(T_\gamma^{T_\beta^{-1}},T_{\alpha}^{T_{\alpha_0'}};\Delta'),
\end{array}
 $$which holds by $R8$ for $\Delta'$. 
 
(Case 3.3) $Q_\Delta(\alpha_0,\beta)<0 <Q_\Delta(\alpha_0,\alpha)$. Then $Q_\Delta(\alpha_0,\alpha)=1, Q_\Delta(\alpha_0,\beta)=-1$ and $w(\alpha_0)=w(\alpha)=w(\beta)=1$. Thus, the subquiver of $Q_{\Delta'}$ formed by $\alpha'_0,\alpha,\beta,\gamma$ is isomorphic to the third quiver in Figure \ref{Fig:subquiver1}. Then $h_{\Delta',\Delta}(Co(T_\gamma^{T_\alpha},T_\beta;\Delta))$ follows by $Co(T_\gamma^{T_\beta T_{\alpha}},T_{\alpha'_0};\Delta')$ (relation $R5$) and $Br_3(T_{\alpha_0'},T_\beta;\Delta')$. 

(Case 4) $Q_\Delta(\alpha_0,\alpha), Q_\Delta(\alpha_0,\gamma)\neq 0=Q_\Delta(\alpha_0,\beta)$. 

(Case 4.1) $Q_\Delta(\alpha_0,\alpha), Q_\Delta(\alpha_0,\gamma)<0$. Then $h_{\Delta',\Delta}(Co(T_\gamma^{T_\alpha},T_\beta;\Delta))$ follows by 
$Co((T_\gamma)^{T_\alpha},T_\beta;\Delta')$, $Co(T_{\alpha_0'},T_{\beta};\Delta')$.

(Case 4.2) $Q_\Delta(\alpha_0,\alpha)<0 <Q_\Delta(\alpha_0,\gamma)$. Then $Q_\Delta(\alpha_0,\alpha)=-1, Q_\Delta(\alpha_0,\gamma)=1$ and $w(\alpha_0)=w(\alpha)=w(\gamma)=1$. Thus, the subquiver of $Q_{\Delta'}$ formed by $\alpha'_0,\gamma,\alpha,\beta$ is isomorphic to the third quiver in Figure \ref{Fig:subquiver1}. Then $h_{\Delta',\Delta}(Co(T_\gamma^{T_\alpha},T_\beta;\Delta))$ follows by $Br_3(T_\alpha^{T_{\alpha'_0}},T_\gamma;\Delta')$ (relation $R4$), $Co(T_\beta^{T_\alpha T_\gamma}, T_{\alpha'_0};\Delta')$ (relation $R5$) and $Br_3(T_\alpha, T_{\alpha'_0};\Delta')$. 

(Case 4.3) $Q_\Delta(\alpha_0,\gamma)<0 <Q_\Delta(\alpha_0,\alpha)$. Then $Q_\Delta(\alpha_0,\alpha)=1, Q_\Delta(\alpha_0,\gamma)=-1$ and $w(\alpha_0)=w(\alpha)=w(\gamma)=1$. Thus, the subquiver of $Q_{\Delta'}$ formed by $\alpha'_0,\alpha,\beta,\gamma$ is isomorphic to the third quiver in Figure \ref{Fig:subquiver2}. Then $$
\begin{array}{rcl}
h_{\Delta',\Delta}(Co(T_\gamma^{T_\alpha},T_\beta;\Delta)) \Leftrightarrow Co(T_\gamma^{T_\alpha T_{\alpha'_0}}, T_\beta;\Delta') \Leftrightarrow Co(T_\gamma^{T_{\alpha'_0}}, T_\beta^{T_\alpha ^{-1}};\Delta'),
\end{array}
$$which holds by $R8$ for $\Delta'$.

(Case 5) $Q_\Delta(\alpha_0,\beta), Q_\Delta(\alpha_0,\gamma)\neq 0=Q_\Delta(\alpha_0,\alpha)$.

(Case 5.1) $Q_\Delta(\alpha_0,\beta), Q_\Delta(\alpha_0,\gamma)<0$. Then $h_{\Delta',\Delta}(Co(T_\gamma^{T_\alpha},T_\beta;\Delta))$ follows by 
$Co(T_\gamma^{T_\alpha},T_\beta;\Delta')$, $Co(T_{\alpha_0'},T_{\alpha};\Delta')$.

(Case 5.2) $Q_\Delta(\alpha_0,\beta)<0 <Q_\Delta(\alpha_0,\gamma)$. Then $Q_\Delta(\alpha_0,\beta)=-1, Q_\Delta(\alpha_0,\gamma)=1$ and $w(\alpha_0)=w(\beta)=w(\gamma)=1$. Thus, the subquiver of $Q_{\Delta'}$ formed by $\alpha'_0,\gamma,\alpha,\beta$ is isomorphic to the third quiver in Figure \ref{Fig:subquiver2}. Then 
$$
\begin{array}{rcl}
h_{\Delta',\Delta}(Co(T_\gamma^{T_\alpha},T_\beta;\Delta)) \Leftrightarrow Co(T_\gamma^{T_\alpha}, T_\beta^{T_{\alpha'_0}};\Delta') \xLeftrightarrow{Br_3(T_\alpha,T_\gamma;\Delta')} Co(T_\alpha^{T_{\gamma}^{-1}}, T_\beta^{T_{\alpha'_0}};\Delta'),
\end{array}
$$which holds by $R8$ for $\Delta'$.

(Case 5.3) $Q_\Delta(\alpha_0,\gamma)<0 <Q_\Delta(\alpha_0,\beta)$. Then $Q_\Delta(\alpha_0,\beta)=1, Q_\Delta(\alpha_0,\gamma)=-1$ and $w(\alpha_0)=w(\beta)=w(\gamma)=1$. Thus, the subquiver of $Q_{\Delta'}$ formed by $\alpha,\beta,\gamma,\alpha'_0$ is isomorphic to the third quiver in Figure \ref{Fig:subquiver1}. Then $h_{\Delta',\Delta}(Co(T_\gamma^{T_\alpha},T_\beta;\Delta))$ 
follows by $Co(T_\alpha^{T_\gamma T_{\beta}},T_{\alpha'_0};\Delta')$ (relation $R5$) and $Br_3(T_{\alpha},T_\beta;\Delta')$. 

\medskip

\textbf{For $R4$}: If $\alpha_0=\alpha$, then 
$$h_{\Delta',\Delta}(Br_3(T_\gamma^{T_\alpha},T_\beta;\Delta))\Leftrightarrow Br_3(T_\gamma,T_\beta;\Delta')\quad (\text{relation $R2$})$$ and $$h_{\Delta',\Delta}(Co(T_\gamma^{T_\alpha},T_\beta;\Delta))\Leftrightarrow Co(T_\gamma,T_\beta;\Delta')\quad (\text{relation $R1$}).$$ 

If $\alpha_0=\beta$, then
$h_{\Delta',\Delta}(R4)$ follows by $Br_3(T_\alpha,T_{\alpha'_0};\Delta')$ and $Br_3(T_\alpha,T_\gamma;\Delta')$ in the case $w(\alpha)=1$ and $Br_4(T_\alpha,T_{\alpha'_0};\Delta')$ and $Co(T_{\alpha'_0}^{T_{\alpha}},T_\gamma;\Delta')$ in the case $w(\alpha)\neq 1$. 

If $\alpha_0=\gamma$, then $h_{\Delta',\Delta}(R4)$ follows by $Br_3(T_\alpha,T_{\alpha_0'};\Delta')$ and $Br_3(T_\alpha,T_\beta;\Delta')$ in the case $w(\alpha)=1$ and $Br_4(T_\alpha,T_{\alpha'_0};\Delta')$ and $Co(T_{\alpha'_0},T_\beta^{T_{\alpha}};\Delta')$ in the case $w(\alpha)\neq 1$.

We now consider the case $\alpha_0\neq \alpha,\beta,\gamma$.  {From \eqref{eq:h}, we need to consider the signs of $Q_\Delta(\alpha_0,\alpha)$, $Q_\Delta(\alpha_0,\beta)$ and $Q_\Delta(\alpha_0,\gamma)$}. As there is a double arrow from $\beta$ to $\gamma$, we have $Q_{\Delta}(\alpha_0,\beta)\leq 0\leq Q_{\Delta}(\alpha_0,\gamma)$ and $Q_{\Delta}(\alpha_0,\beta)>0$ if and only if $Q_{\Delta}(\alpha_0,\gamma)<0$.

(Case 1) $Q_\Delta(\alpha_0,\alpha),Q_\Delta(\alpha_0,\beta),Q_\Delta(\alpha_0,\gamma)\geq 0$. Then 
$$h_{\Delta',\Delta}(Br_3(T_\gamma^{T_\alpha},T_\beta;\Delta))\Leftrightarrow Br_3(T_\gamma^{T_\alpha},T_\beta;\Delta')\quad (\text{relation $R4$})$$ and 
$$h_{\Delta',\Delta}(Co(T_\gamma^{T_\alpha},T_\beta;\Delta))\Leftrightarrow Co(T_\gamma^{T_\alpha},T_\beta;\Delta')\quad (\text{relation $R4$}).$$

(Case 2) $Q_\Delta(\alpha_0,\alpha)<0$, $Q_\Delta(\alpha_0,\beta)=Q_\Delta(\alpha_0,\gamma)=0$. Then 
$h_{\Delta',\Delta}(R4)$ follows by 
$Br_3(T_\gamma^{T_\alpha},T_\beta;\Delta')$, $Co(T_{\alpha_0'},T_{\beta};\Delta')$ and $Co(T_{\alpha_0'},T_{\gamma};\Delta')$ in the case $w(\alpha)=1$ and $Co(T_\gamma^{T_\alpha},T_\beta;\Delta')$, $Co(T_{\alpha_0'},T_{\beta};\Delta')$ and $Co(T_{\alpha_0'},T_{\gamma};\Delta')$ in the case $w(\alpha)\neq 1$.

(Case 3) $Q_\Delta(\alpha_0,\beta)< 0<Q_\Delta(\alpha_0,\gamma)$. 

(Case 3.1) $Q_\Delta(\alpha_0,\alpha)=0$. If $w(\alpha_0)=w(\alpha)=1$, then 
$$
\begin{array}{rcl}
 h_{\Delta',\Delta}(Br_3(T_\gamma^{T_\alpha},T_\beta;\Delta)) &\Leftrightarrow&  Br_3(T_\gamma^{T_\alpha},T_\beta^{T_{\alpha'_0}};\Delta') \xLeftrightarrow{Br_3(T_{\alpha_0'},T_\beta;\Delta')}  Br_3(T_\gamma^{T_\beta T_\alpha},T_{\alpha'_0};\Delta')    \\
     &\xLeftrightarrow{Co(T_{\gamma}^{T_\alpha},T_\beta;\Delta')}  &  Br_3(T_\gamma^{T_\alpha},T_{\alpha'_0};\Delta') \xLeftrightarrow{Co({T_\alpha},T_{\alpha'_0};\Delta')} Br_3(T_\gamma,T_{\alpha'_0};\Delta'),
\end{array}
$$
which holds by $R2$ for $\Delta'$.

If $w(\alpha_0)=1<w(\alpha)$, then 
$$
\begin{array}{rcl}
h_{\Delta',\Delta}(Co(T_\gamma^{T_\alpha},T_\beta;\Delta)) &\Leftrightarrow&  Co(T_\gamma^{T_\alpha},T_\beta^{T_{\alpha'_0}};\Delta') \xLeftrightarrow{Br_3(T_{\alpha_0'},T_\beta;\Delta')}  Co(T_\gamma^{T_\beta T_\alpha},T_{\alpha'_0};\Delta')
\end{array}
$$
follows by Lemma \ref{lem:R9}.

If $w(\alpha_0)\neq 1$, then $h_{\Delta',\Delta}(R4)$ follows by Lemma \ref{lem:R41}.

(Case 3.2) $Q_\Delta(\alpha_0,\alpha)<0$. Then $w(\alpha_0)=w(\alpha)=1$ and the subquiver of $Q_{\Delta'}$ formed by $\beta,\gamma,\alpha,\alpha'_0$ is isomorphic to the first quiver in Figure \ref{Fig:subquiver2}. Thus,
$$
\begin{array}{rcl}
h_{\Delta',\Delta}(Br_3(T_\gamma^{T_\alpha},T_\beta;\Delta)) &\Leftrightarrow&  Br(T_\gamma^{T_{\alpha'_0}T_\alpha T_{\alpha'_0}^{-1}},T_\beta^{T_{\alpha'_0}};\Delta') \xLeftrightarrow{Br_3(T_{\alpha},T_\beta;\Delta')}  Br_3(T_\gamma^{T_{\alpha'_0}^{-1}},T_{\alpha}^{T_\beta};\Delta')\\
&\Leftrightarrow& 
Br_3(T_\gamma,T_{\alpha}^{T_{\alpha'_0} T_\beta};\Delta')\quad (\text{relation $R6$}).
\end{array}
$$

(Case 3.3) $Q_\Delta(\alpha_0,\alpha)>0$. This case is similar to the Case 3.2.

\medskip

\textbf{For $R5$}: If $\alpha_0=\alpha$, then   
      $$
\begin{array}{rcl}
 h_{\Delta',\Delta}( Co(T_{\delta}^{T_\gamma T_\beta}, T_{\alpha};\Delta))=Co(T_\delta^{T_\gamma T_{\alpha'_0}T_\beta T^{-1}_{\alpha'_0}}, T_{\alpha'_0};\Delta') &\xLeftrightarrow {Co(T_{\delta},T_{\alpha'_0};\Delta')}&  Co(T_\delta^{T_{\alpha'_0}T_\beta},T_{\gamma}^{-1}T_{\alpha_0'}T_{\gamma};\Delta')    \\
     &\xLeftrightarrow{Br_3(T_{\gamma},T_{\alpha'_0};\Delta')}  &  Co(T_\delta^{T_{\alpha'_0}T_\beta},T_{\gamma}^{T_{\alpha_0'}};\Delta') \\
     &\Leftrightarrow& Co(T_\delta^{T_\beta},T_{\gamma};\Delta'),
\end{array}
$$
which holds by $R3$ for $\Delta$.

If $\alpha_0=\beta$ or $\gamma$, then
$h_{\Delta',\Delta}( Co(T_{\delta}^{T_\gamma T_\beta}, T_{\alpha};\Delta))$ follows by $Co(T_\delta^{T_\beta T_\gamma}, T_{\alpha'_0};\Delta')$, which holds by $R5$ for $\Delta$.

If $\alpha_0=\delta$ with $w(\delta)=1$, then 
      $$
\begin{array}{rcl}
 h_{\Delta',\Delta}( Co(T_{\delta}^{T_\gamma T_\beta}, T_{\alpha};\Delta))=Co(T_{\alpha'_0}^{T_\gamma T_{\alpha'_0}T_\beta T^{-1}_{\alpha_0'}}, T_{\alpha};\Delta') &\xLeftrightarrow{Br_3(T_\beta,T_{\alpha'_0};\Delta')} &  Co(T_\beta,T_{\gamma}^{-1}T_{\alpha}T_{\gamma};\Delta'), 
\end{array}$$
which holds by $R3$ for $\Delta$.

If $\alpha_0=\delta$ with $w(\delta)\neq 1$, then the subquiver of $Q_{\Delta'}$ formed by $\alpha,\beta,\alpha'_0,\gamma$ is isomorphic to the third quiver in Figure \ref{Fig:subquiver2}. Then
      $$
\begin{array}{rcl}
 h_{\Delta',\Delta}( Co(T_{\delta}^{T_\gamma T_\beta}, T_{\alpha};\Delta))=Co(T_{\alpha'_0}^{T_\gamma T_{\alpha'_0}T_\beta T^{-1}_{\alpha'_0}}, T_{\alpha};\Delta') &\xLeftrightarrow {Br_4(T_{\beta},T_{\alpha'_0};\Delta')}&  Co(T_{\alpha'_0}^{T_\gamma T^{-1}_\beta }, T_{\alpha};\Delta')    \\
&\Leftrightarrow  &  Co(T_{\alpha'_0}^{T_\beta^{-1}},  T_\gamma^{-1}T_{\alpha} T_\gamma ;\Delta')\\
     &\xLeftrightarrow{Br_3(T_\alpha,T_\gamma;\Delta')}&  Co(T_{\alpha'_0}^{T_\beta^{-1}},  T_\gamma^{ T_\alpha};\Delta'),
\end{array}
$$
which holds by $R8$ for $\Delta$.

We now consider the case that $\alpha_0\neq \alpha,\beta,\gamma,\delta$. Then we have $Q_{\Delta}(\alpha_0,\beta)=Q_{\Delta}(\alpha_0,\gamma)=0$.

(Case 1) $Q_\Delta(\alpha_0,\alpha),Q_\Delta(\alpha_0,\delta)\geq 0$. Then 
$$h_{\Delta',\Delta}( Co(T_{\delta}^{T_\gamma T_\beta}, T_{\alpha};\Delta))\Leftrightarrow Co(T_\delta^{T_\gamma T_\beta},T_\alpha;\Delta'): R5.$$

(Case 2) $Q_\Delta(\alpha_0,\delta)=0> Q_\Delta(\alpha_0,\alpha)$ or $Q_\Delta(\alpha_0,\alpha)=0> Q_\Delta(\alpha_0,\delta)$. Then 
$$h_{\Delta',\Delta}( Co(T_{\delta}^{T_\gamma T_\beta}, T_{\alpha};\Delta))\Leftrightarrow Co(T_{\alpha'_0}^{T_\gamma T_{\alpha'_0}T_\beta T^{-1}_{\alpha_0'}}, T_{\alpha};\Delta')$$ follows by $Co(T_\delta^{T_\gamma T_\beta},T_\alpha;\Delta')$, $Co(T_\beta;T_{\alpha_0'};\Delta')$, $Co(T_\gamma;T_{\alpha_0'};\Delta')$, and $Co(T_\delta;T_{\alpha_0'};\Delta')$ or $Co(T_\alpha;T_{\alpha_0'};\Delta')$.

(Case 3) $0> Q_\Delta(\alpha_0,\alpha),Q_\Delta(\alpha_0,\delta)$. Then $ h_{\Delta',\Delta}( Co(T_{\delta}^{T_\gamma T_\beta}, T_{\alpha};\Delta))$ follows by $Co(T_\delta^{T_\gamma T_\beta},T_\alpha;\Delta')$, $Co(T_\beta;T_{\alpha_0'};\Delta')$, and $Co(T_\gamma;T_{\alpha_0'};\Delta')$.

(Case 4) $Q_\Delta(\alpha_0,\alpha)>0> Q_\Delta(\alpha_0,\delta)$ or $Q_\Delta(\alpha_0,\delta)>0> Q_\Delta(\alpha_0,\alpha)$. Then we have $w(\delta)=w(\alpha_0)=1$. 
As 
$$
\begin{array}{rcl}
 Co(T_{\delta}^{T_\gamma T_\beta}, T_{\alpha};\Delta) \xLeftrightarrow {Br_3(T_{\beta},T_{\delta};\Delta):R2}  Co(T_{\beta}^{T^{-1}_\delta}, T_{\alpha}^{T_{\gamma}^{-1}};\Delta)  
\xLeftrightarrow{Br_3(T_{\alpha},T_{\gamma};\Delta):R2}    Co(T_\beta, T_{\gamma}^{T_\delta T_\alpha} ;\Delta), 
\end{array}
$$ it suffices to prove that $h_{\Delta',\Delta}(Co(T_\beta, T_{\gamma}^{T_\delta T_\alpha} ;\Delta))$ holds.
We may assume that $Q_\Delta(\alpha_0,\alpha)>0> Q_\Delta(\alpha_0,\delta)$, as $Co(T_\beta, T_{\gamma}^{T_\delta T_\alpha} ;\Delta)$ does not depend on the order of $T_\alpha$ and $T_\delta$. Thus, the subquiver of $Q_{\Delta'}$ formed by 
$\alpha,\beta,\gamma,\delta,\alpha'_0$ isomorphic to the second quiver in Figure \ref{Fig:subquiver2}. Then, 
$$
\begin{array}{rcl}
 h_{\Delta',\Delta}(Co(T_\beta, T_{\gamma}^{T_\delta T_\alpha} ;\Delta))\Leftrightarrow Co(T_{\beta}, T_{\gamma}^{T_{\alpha'_0}T_\delta T^{-1}_{\alpha'_0}T_\alpha};\Delta') &\xLeftrightarrow {Co(T_{\delta}^{T_{\alpha'_0}},T_{\alpha};\Delta')}&  Co(T_{\beta}, T_{\gamma}^{T_\alpha T_{\alpha'_0}T_\delta T^{-1}_{\alpha'_0}};\Delta')    \\
&\xLeftrightarrow{Co(T_{\alpha'_0},T_\gamma;\Delta')}  &  Co(T_\beta, T_{\gamma}^{T_\alpha T_{\alpha'_0} T_\delta} ;\Delta') 
\end{array}
$$
which holds by $R7$ for $\Delta$.

\medskip

\textbf{For $R6$:} We have $w(\alpha)=w(\beta)=w(\gamma)=w(\delta)=1$. 

If $\alpha_0=\alpha$, then 
$$
\begin{array}{rcl}
  h_{\Delta',\Delta}(Br_3(T_{\gamma}^{T_\alpha T_\delta },T_\beta ;\Delta)) &\Leftrightarrow& Br_3(T_{\gamma}^{(T_{\alpha'_0}T^{T_{\alpha'_0}}_\delta) }, T_\beta^{T_{\alpha'_0}};\Delta')\\
  &\Leftrightarrow &  Br_3(T_{\gamma}^{(T^{T_{\alpha'_0}}_\delta) }, T_\beta;\Delta')    \\
&\xLeftrightarrow[\text{Lemma \ref{lem:basic}}]{Br_3(T_\delta^{T_{\alpha'_0}},T_\gamma;\Delta')}  &  Br_3({T^{T_{\alpha'_0}}_\delta}, T_\beta^{T_{\gamma}};\Delta') \\
&\xLeftrightarrow{Br_3(T_\beta,T_\gamma;\Delta')} &  
Br_3({T^{T_\beta T_{\alpha'_0}}_\delta}, {T_{\gamma}};\Delta'),
\end{array}
$$
which holds by $R6$ for $\Delta$.

$$
\begin{array}{rcl}
  h_{\Delta',\Delta}(Br_3(T_{\gamma}^{T_\delta T_\alpha},T_\beta ;\Delta))
  \Leftrightarrow Br_3(T_{\gamma}^{T_{\alpha'_0}T_\delta}, T_\beta^{T_{\alpha'_0}};\Delta') &\Leftrightarrow &  Br_3(T_{\gamma}^{T_\delta}, T_\beta;\Delta'),
\end{array}
$$
which holds by $R4$ for $\Delta$.

The case that $\alpha_0=\delta$ is dual to the case that $\alpha_0=\alpha$, so we omit it.

If $\alpha_0=\beta$, then 
$$
\begin{array}{rcl}
  h_{\Delta',\Delta}(Br_3(T_{\gamma}^{T_\alpha T_\delta },T_\beta ;\Delta))
  &\Leftrightarrow& Br_3(T_{\gamma}^{T_{\alpha}T_\delta T_{\alpha'_0}}, T_{\alpha'_0};\Delta')\\
 &\Leftrightarrow &  Br_3(T_{\gamma}^{T_\delta T_{\alpha'_0}}, T_{\alpha}^{-1}T_{\alpha'_0} T_{\alpha};\Delta')    \\
&\xLeftrightarrow{Br_3(T_\alpha,T_{\alpha'_0};\Delta')}  &  Br_3(T_{\gamma}^{T_\delta T_{\alpha'_0}}, T_{\alpha}^{T_{\alpha'_0}};\Delta') \\
&\Leftrightarrow &  
Br_3(T_{\gamma}^{T_{\alpha'_0}^{-1}T_\delta T_{\alpha'_0}}, T_{\alpha};\Delta')\\
&\xLeftrightarrow[\text{Lemma } \ref{lem:basic}]{Br_3(T_\gamma,T_{\alpha'_0}^{-1}T_\delta T_{\alpha'_0};\Delta')}  & 
Br_3(T_{\alpha'_0}^{-1}T_\delta T_{\alpha'_0}, T_{\alpha}^{T_\gamma};\Delta')
\\
&\xLeftrightarrow[Br_3(T_{\alpha'_0},T_\delta;\Delta')]{Br_3(T_\alpha,T_\gamma;\Delta')}  & 
Br_3(T_{\alpha'_0}^{T_\alpha T_\delta}, T_\gamma;\Delta'),
\end{array}
$$
which holds by $R6$ for $\Delta$.

We can similarly prove that $ h_{\Delta',\Delta}(Br_3(T_{\gamma}^{T_\delta T_\alpha },T_\beta ;\Delta))$ holds in $Br_{\Delta'}$.

The case that $\alpha_0=\gamma$ is dual to the case that $\alpha_0=\beta$, so we omit it.

We now consider the case that $\alpha_0\neq \alpha,\beta,\gamma,\delta$. Then we have $Q_{\Delta}(\alpha_0,\beta)=Q_{\Delta}(\alpha_0,\gamma)=0$ and $Q_{\Delta}(\alpha_0,\delta)\geq 0 \geq Q_{\Delta}(\alpha_0,\alpha)$. Moreover, $Q_{\Delta}(\alpha_0,\delta)\neq 0$ if and only if $0\neq Q_{\Delta}(\alpha_0,\alpha)$.

(Case 1) $Q_{\Delta}(\alpha_0,\delta)=0=Q_{\Delta}(\alpha_0,\alpha)$. Then 
$$h_{\Delta',\Delta}(Br_3(T_{\gamma}^{T_\alpha T_\delta },T_\beta ;\Delta))\Leftrightarrow Br_3(T_{\gamma}^{T_{\alpha}T_\delta}, T_{\beta};\Delta')$$ and
$$ h_{\Delta',\Delta}(Br_3(T_{\gamma}^{T_\delta T_\alpha},T_\beta ;\Delta))\Leftrightarrow Br_3(T_{\gamma}^{T_\delta T_{\alpha}}, T_{\beta};\Delta').$$

(Case 2) $Q_{\Delta}(\alpha_0,\alpha)>0>Q_{\Delta}(\alpha_0,\delta)$. Then 
$$
\begin{array}{rcl}
  h_{\Delta',\Delta}(Br_3(T_{\gamma}^{T_\alpha T_\delta },T_\beta ;\Delta))
  &\Leftrightarrow& Br_3(T_{\gamma}^{(T_{\alpha}^{T_{\alpha'_0}}T_\delta)}, T_{\beta};\Delta')\\ &\xLeftrightarrow{{Co(T_{\alpha'_0},T_\beta;\Delta')}} &  Br_3(T_{\gamma}^{T_\alpha T^{-1}_{\alpha'_0} T_\delta}, T_{\beta};\Delta')    \\
&\xLeftrightarrow{Br_3(T_\alpha,T_{\beta};\Delta')}  &  Br_3(T_{\gamma}^{T^{-1}_{\alpha'_0} T_\delta}, T_\alpha^{T_{\beta}};\Delta') \\
&\xLeftrightarrow[Co(T_{\alpha'_0},T_\delta;\Delta')]{Br_3(T_\gamma,T_\delta;\Delta')} &  
Br_3({T_{\alpha'_0}^{-1}T_\delta T_{\alpha'_0}}, T_{\alpha}^{T_\gamma T_\beta};\Delta')\\
&\xLeftrightarrow[Br_3(T_\delta, T_{\alpha'_0};\Delta')]{Co(T_\alpha^{T_\gamma T_\beta},T_\delta;\Delta')}  & 
Br_3(T_{\alpha'_0}, T_{\alpha}^{T_\gamma T_\beta};\Delta')
\\
&\xLeftrightarrow{Co(T_{\alpha'_0}, T_\gamma T_\beta;\Delta')} & 
Br_3(T_{\alpha'_0}, T_{\alpha};\Delta'),
\end{array}
$$
which holds by $R2$ for $\Delta$.

We can similarly prove that $ h_{\Delta',\Delta}(Br_3(T_{\gamma}^{T_\delta T_\alpha },T_\beta ;\Delta))$ holds in $Br_{\Delta'}$.

\medskip

\textbf{For $R7$:} If $\alpha_0=\alpha$, then  
$$
\begin{array}{rcl}
 h_{\Delta',\Delta}(Co(T_\beta, T_{\gamma}^{T_\alpha T_\zeta T_\delta } ;\Delta))
&\Leftrightarrow& 
Co(T_\beta^{T_{\alpha'_0}}, T_{\gamma}^{(T_{\alpha'_0} T_\zeta T_\delta^{T_{\alpha'_0}})};\Delta') \\ &\Leftrightarrow &  Co(T_\beta, T_{\gamma}^{(T_\zeta T_\delta^{T_{\alpha'_0}})};\Delta')    \\
&\xLeftrightarrow{Br_3(T_\beta,T_{\zeta};\Delta')}  &  Co(T_\zeta^{T_\beta}, T_{\gamma}^{T_{\alpha'_0} T_\delta {T^{-1}_{\alpha'_0}}};\Delta') \\
&\Leftrightarrow &  
Co(T_\zeta^{T_{\alpha'_0} T^{-1}_\delta {T^{-1}_{\alpha'_0}} T_\beta}, T_{\gamma};\Delta')\\
&\xLeftrightarrow{Co(T_\zeta^{T_\beta},T_{\alpha'_0};\Delta')}  & 
Co(T_\zeta^{T_{\alpha'_0} T^{-1}_\delta T_\beta}, T_{\gamma};\Delta')
\\
&\xLeftrightarrow[Br_3(T_{\beta},T_\zeta;\Delta')]{Co(T_\delta,T_\zeta;\Delta')}  & 
Co(T_\delta^{T_{\alpha'_0} T^{-1}_\zeta T_\beta}, T_{\gamma};\Delta')\\
&\xLeftrightarrow{Co(T_\gamma,T_\zeta;\Delta')}  & 
Co(T_\delta^{T_\zeta T_{\alpha'_0} T^{-1}_\zeta T_\beta}, T_{\gamma};\Delta')
\\
&\xLeftrightarrow{Co(T_\beta,T_{\alpha_0'}^{T_\zeta};\Delta')}  & 
Co(T_\delta^{T_\beta T_\zeta T_{\alpha'_0} T^{-1}_\zeta }, T_{\gamma};\Delta')\\
&\xLeftrightarrow{Co(T_\delta, T_{\delta};\Delta')}  & 
Co(T_\delta^{T_\beta T_\zeta T_{\alpha'_0}}, T_{\gamma};\Delta'),
\end{array}
$$
which holds by $R7$ for $\Delta$.

The case that $\alpha_0=\delta$ is dual to the case that $\alpha_0=\alpha$, so we omit it.

If $\alpha_0=\beta$, as $$
\begin{array}{rcl}
 Co(T_\beta, T_{\gamma}^{T_\alpha T_\zeta T_\delta } ;\Delta)    &\xLeftrightarrow[Br_3(T_\delta,T_\gamma,\Delta)\quad(\text{relation $R2$})]{Br_3(T_\alpha,T_\beta,\Delta)\quad(\text{relation $R2$})} & Co(T_\zeta^{-1} T_\alpha^{T_\beta} T_\zeta, T_{\delta}^{T_\gamma^{-1}} ;\Delta) \\
   &\xLeftrightarrow[Co(T_\zeta,T_\beta,\Delta)\quad(\text{relation $R2$})]{Br_3(T_\alpha,T_\beta,\Delta)\quad(\text{relation $R2$})} & 
   Co(T_\zeta^{T_\beta T_\alpha}, T_{\delta}^{T_\gamma^{-1}} ;\Delta) \\
   &\Leftrightarrow&
   Co(T_\zeta^{T_\gamma T_\beta T_\alpha}, T_{\delta};\Delta)
\end{array}
$$ and $$h_{\Delta',\Delta}(Co(T_\zeta^{T_\gamma T_\beta T_\alpha}, T_{\delta};\Delta))\Leftrightarrow Co(T_\zeta^{T_{\alpha'_0} T_\gamma T_\alpha}, T_{\delta};\Delta'),$$ which holds by $R7$ for $\Delta'$. Thus, we have $h_{\Delta',\Delta}(Co(T_\beta, T_{\gamma}^{T_\alpha T_\zeta T_\delta } ;\Delta))$ holds in $Br_{\Delta'}$.

The case that $\alpha_0=\gamma$ is dual to the case that $\alpha_0=\beta$, so we omit it.

We now consider the case that $\alpha_0\neq \alpha,\beta,\gamma,\delta,\zeta$. Then $Q_{\Delta}(\alpha_0,\alpha)=Q_{\Delta}(\alpha_0,\beta)=Q_{\Delta}(\alpha_0,\gamma)=Q_{\Delta}(\alpha_0,\delta)=0$.

(Case 1) $Q_{\Delta}(\alpha_0,\zeta)\geq 0$. Then $ h_{\Delta',\Delta}(Co(T_\beta, T_{\gamma}^{T_\alpha T_\zeta T_\delta } ;\Delta))\Leftrightarrow  Co(T_\beta, T_{\gamma}^{T_\alpha T_\zeta T_\delta } ;\Delta'),$ which holds by $R7$ for $\Delta'$.

(Case 2) $Q_{\Delta}(\alpha_0,\zeta)<0$. Then $ h_{\Delta',\Delta}(Co(T_\beta, T_{\gamma}^{T_\alpha T_\zeta T_\delta } ;\Delta))\Leftrightarrow  Co(T_\beta, T_{\gamma}^{(T_\alpha T_\zeta^{T_{\alpha'_0}} T_\delta) } ;\Delta')$ follows by $Co(T_\beta, T_{\gamma}^{T_\alpha T_\zeta T_\delta } ;\Delta')$, $Co(T_{\alpha'_0},T_\alpha;\Delta')$, $Co(T_{\alpha'_0},T_\beta;\Delta')$ and $Co(T_{\alpha'_0},T_\gamma;\Delta')$.

\textbf{For $R8$:} If $\alpha_0=\alpha$, then 
$$
\begin{array}{rcl}
   h_{\Delta',\Delta}(Co(T_\gamma^{T_\beta^{-1}},T_\delta^{T_\alpha};\Delta))  &\Leftrightarrow & 
  Co(T_\gamma^{T_{\alpha'_0}T_\beta^{-1} T^{-1}_{\alpha'_0}},T_\delta^{T_{\alpha'_0}};\Delta')  \xLeftrightarrow{Co(T_{\alpha'_0},T_\gamma;\Delta')}  Co(T_\gamma^{T_\beta^{-1}},T_\delta;\Delta') \\
     &\Leftrightarrow & 
Co(T_\gamma,T_\delta^{T_\beta};\Delta'), 
\end{array}
$$
which holds by $R3$ for $\Delta'$.

If $\alpha_0=\beta$, then 
$
\begin{array}{rcl}
   h_{\Delta',\Delta}(Co(T_\gamma^{T_\beta^{-1}},T_\delta^{T_\alpha};\Delta))  \Leftrightarrow  
  Co(T_\gamma,T_\delta^{T_{\alpha}};\Delta'),   
\end{array}$
which holds by $R3$ for $\Delta'$.

If $\alpha_0=\gamma$ and $w(\gamma)=1$, then 
$$
\begin{array}{rcl}
  h_{\Delta',\Delta}( Co(T_\gamma^{T_\beta^{-1}},T_\delta^{T_\alpha};\Delta) ) \Leftrightarrow  
  Co(T_{\alpha'_0}^{T_\beta^{-1}},T_\delta^{T_\alpha T_{\alpha'_0}};\Delta')  &\xLeftrightarrow[Co(T_\alpha,T_{\alpha'_0};\Delta')]{Br_3(T_{\alpha'_0},T_\beta;\Delta')} &  Co(T_{\beta}^{T_{\alpha'_0}},T_\delta^{T_{\alpha'_0}T_\alpha };\Delta')  \\
     &\Leftrightarrow & 
Co(T_\beta,T_\delta^{T_\alpha};\Delta'),    
\end{array}
$$
which holds by $R3$ for $\Delta'$.

If $\alpha_0=\gamma$ and $w(\gamma)\neq 1$, then 
$$
\begin{array}{rcl}
  h_{\Delta',\Delta}( Co(T_\gamma^{T_\beta^{-1}},T_\delta^{T_\alpha};\Delta) ) \Leftrightarrow  
  Co(T_{\alpha'_0}^{T_\beta^{-1}},T_\delta^{T_\alpha T_{\alpha'_0}};\Delta')  &\xLeftrightarrow{Co(T_\alpha,T_{\alpha'_0};\Delta')} &  Co(T_{\alpha'_0}^{T_\beta^{-1}},T_\delta^{T_{\alpha'_0}T_\alpha };\Delta')  \\
     &\Leftrightarrow & 
Co(T_{\alpha'_0}^{T_{\alpha'_0}^{-1}T_\beta^{-1}},T_\delta^{T_\alpha };\Delta')\\
&\xLeftrightarrow[Br_3(T_\alpha,T_\delta;\Delta')]{Br_4(T_\beta,T_{\alpha'_0};\Delta')}&
Co(T_{\alpha'_0}^{T_\beta},T_\alpha^{T_\delta^{-1} };\Delta')\\
&\Leftrightarrow&
Co(T_{\alpha'_0}^{T_\delta T_\beta},T_\alpha;\Delta'),
\end{array}
$$
which holds by $R5$ for $\Delta'$.

If $\alpha_0=\delta$, then 

$$
\begin{array}{rcl}
  h_{\Delta',\Delta}( Co(T_\gamma^{T_\beta^{-1}},T_\delta^{T_\alpha};\Delta))  \Leftrightarrow  
Co(T_\gamma,T_{\alpha'_0}^{(T_{\alpha}^{T_{\alpha'_0}})};\Delta')\xLeftrightarrow{Br_3(T_\alpha,T_{\alpha'_0};\Delta')} Co(T_\gamma,T_{\alpha};\Delta'),    
\end{array}$$
which holds by $R1$ for $\Delta'$.

We now consider the case $\alpha_0\neq \alpha,\beta,\gamma,\delta$. 

If $\alpha$ and $\gamma$ are not two sides of any triangle in $\Delta$, and $\beta$ and $\delta$ are not two sides of any triangle in $\Delta$, then $\alpha,\beta,\gamma,\delta$ form a complete counter-clockwise list of the arcs incident to some puncture $p$. In this case, we have $ Co(T_\gamma^{T_\beta^{-1}},T_\delta^{T_\alpha};\Delta)  \Leftrightarrow  
Cyl(T_\alpha, T_\delta, T_\gamma,T_\beta;\Delta)$. We defer the proof of this case to the proof for the relation $R9$. 

Note that $ h_{\Delta',\Delta}(Co(T_\gamma^{T_\beta^{-1}},T_\delta^{T_\alpha};\Delta))  \Leftrightarrow  
Co(T_\gamma^{T_\beta^{-1}},T_\delta^{T_\alpha};\Delta')$ if $Q_{\Delta}(\alpha_0,\zeta)\geq 0$ for any $\zeta\in \{\alpha,\beta,\gamma,\delta\}$. Therefore, we can exclude this case in the subsequent discussion.

(Case 1) $\alpha$ and $\gamma$ are two sides of some triangle in $\Delta$. Then $(\beta,\delta)$ forms a once-punctured bigon with diagonals $\alpha,\gamma$ and $Q_{\Delta}(\alpha_0,\alpha)=Q_{\Delta}(\alpha_0,\gamma)=0$. Then $h_{\Delta',\Delta}(Co(T_\gamma^{T_\beta^{-1}},T_\delta^{T_\alpha};\Delta))$ follows by $Co(T_\gamma^{T_\beta^{-1}},T_\delta^{T_\alpha};\Delta') \quad(\text{relation $R8$})$, $Co(T_{\alpha'_0},T_\alpha)$, and $Co(T_{\alpha'_0},T_\gamma)$.

(Case 2) $\beta$ and $\delta$ are two sides of some triangle in $\Delta$. Then $(\alpha,\gamma)$ forms a once-punctured bigon with diagonals $\beta,\delta$ and $Q_{\Delta}(\alpha_0,\beta)=Q_{\Delta}(\alpha_0,\delta)=0$. Then $h_{\Delta',\Delta}(Co(T_\gamma^{T_\beta^{-1}},T_\delta^{T_\alpha};\Delta))$ follows by $Co(T_\gamma^{T_\beta^{-1}},T_\delta^{T_\alpha};\Delta')\quad(\text{relation $R8$})$, $Co(T_{\alpha'_0},T_\beta)$, and $Co(T_{\alpha'_0},T_\delta)$.

\textbf{For $R9$:} Assume that $\alpha$ is not a self-folded arc and a diagonal of some clockwise cyclic quadrilateral $(\alpha_1,\alpha_2,\alpha_3,\alpha_4)$ in $\Delta$ such that $(\alpha_1,\alpha_2,\overline\alpha)$ forms a triangle. 

If none of $\alpha,\alpha_1,\alpha_2,\alpha_3,\alpha_4$ is incident to the ordinary puncture $p$, then the relation $R9$ is clearly preserved by the map $h_{\Delta',\Delta}$. 

If the number of arcs incident to $p$ in $\Delta$ differs from that in $\Delta'$, then the result follows by Lemma \ref{lem:R9a1}.

Thus, we may assume that $\alpha$ incident to $p$, without loss of generality, assume $s(\alpha)=p$.

{\bf Case 1}: Suppose $s(\alpha)=s(\alpha_4)=p\neq t(\alpha),t(\alpha_1)$. Let $\mu$ be a mutation sequence at loops incident to $p$ such that the number of loops incident to $p$ decreases at each step, and $\alpha_4$ is the only loop incident to $p$ in $\mu\Delta$. Then we have $\mu \mu_\alpha=\mu_\alpha \mu$ and 
$$R9(\Delta)=h^{\mu_{\alpha_4}\mu}_{\Delta,\mu_{\alpha_4}\mu\Delta}(R9(\mu_{\alpha_4}\mu \Delta)), \quad R9(\Delta')=h^{\mu_{\alpha_4}\mu}_{\Delta',\mu_{\alpha_4}\mu\Delta'}(R9(\mu_{\alpha_4}\mu \Delta')),$$
$$R9(\mu\Delta)=h^{\mu_{\alpha_4}}_{\mu\Delta,\mu_{\alpha_4}\mu\Delta}(R9(\mu_{\alpha_4}\mu \Delta)), \quad R9(\mu\Delta')=h^{\mu_{\alpha_4}}_{\mu\Delta',\mu_{\alpha_4}\mu\Delta'}(R9(\mu_{\alpha_4}\mu \Delta')).$$ 
Therefore, $$
\begin{array}{rcl}
   h_{\Delta',\Delta}(R9(\Delta))  & =& h_{\Delta',\Delta}h^{\mu_{\alpha_4}\mu}_{\Delta,\mu_{\alpha_4}\mu\Delta}(R9(\mu_{\alpha_4}\mu \Delta))=h^{\mu}_{\mu_{\alpha}\Delta,\mu_\alpha\mu\Delta}h_{\mu_\alpha\mu\Delta,\mu\Delta}h^{\mu_{\alpha_4}}_{\mu\Delta,\mu_{\alpha_4}\mu\Delta}(R9(\mu_{\alpha_4}\mu \Delta))\\
   & = & h^{\mu}_{\mu_{\alpha}\Delta,\mu_\alpha\mu\Delta}h_{\mu_\alpha\mu\Delta,\mu\Delta}(R9(\mu\Delta)).
\end{array}
$$

By Lemma \ref{lem:R9a}, $h_{\mu_\alpha\mu\Delta,\mu\Delta}(R9(\mu\Delta))$ holds in $Br_{\mu_\alpha\mu\Delta}$. Applying Lemma \ref{lem:R9a1}, it follows that $h^{\mu}_{\mu_{\alpha}\Delta,\mu_\alpha\mu\Delta}h_{\mu_\alpha\mu\Delta,\mu\Delta}(R9(\mu\Delta))$ holds in $Br_{\mu_\alpha\Delta}$.

{\bf Case 2}: Suppose that $s(\alpha)=s(\alpha_2)=p\neq t(\alpha),t(\alpha_3)$. The result follows similarly by applying Lemmas \ref{lem:R9a1} and \ref{lem:R9b}.

{\bf Case 3}:
Suppose that $s(\alpha_1)=s(\alpha_2)=s(\alpha_3)=s(\alpha_4)=p$. The result can also be established using Lemmas \ref{lem:R9a1} and \ref{lem:R9c} in an analogous way.

\begin{lemma}\label{lem:R9}
 Let $\alpha,\beta,\gamma,\delta\in \Delta$. Suppose that there is a $4$-cycle among $\alpha,\beta,\gamma$, and $\delta$, with an arrow from $\beta$ to $\delta$, no double arrows between any of these vertices, and no arrow between $\alpha$ and $\gamma$; see the quiver in Figure \ref{Fig:subquiver3}. If $w(\alpha)\neq 1$, then the relation $Co(T_\delta^{T_\beta T_\alpha},T_\gamma)$ holds. 
\end{lemma}

\begin{figure}[ht]
\begin{center}
\begin{tikzcd}
 \alpha\arrow[rr] & & \beta \arrow[d]\arrow[dll]  
  \\
 \delta\arrow[u] & & \gamma\arrow[ll]
\end{tikzcd}  
\caption{}
\label{Fig:subquiver3}
\end{center}
\end{figure}

\begin{proof}
    We have $w(\beta)=w(\gamma)=w(\delta)=1$, and the arcs $\alpha, \alpha,\beta,\gamma,\delta$ form a complete counterclockwise cyclic list of the arcs incident to some puncture $p$ in $\Delta$. In $\mu_\alpha(\Delta)$, the arcs $\beta,\gamma,\delta$ form a complete counterclockwise cyclic list of the arcs incident to $p$. By $R9$, we see that the relation $Cyl(T_\delta^{T_\alpha},T_\beta,T_\gamma)$ holds. Furthermore, applying the braid relation $Br_3(T_\delta^{T_\alpha},T_\beta)$, it follows that $Co(T_\delta^{T_\beta T_\alpha},T_\gamma)$ holds.

    The proof is complete.
\end{proof}

\begin{lemma}\label{lem:basic}
In a group $G$, if $Br_3(y,z)$ then $Br_3(x,y^z)\Leftrightarrow Br_3(x^y,z)$.
\end{lemma}

\begin{proof}
 As $Br_3(y,z)$, we have both $Br_3(x,y^z): xzyz^{-1}x=zyz^{-1}xzyz^{-1}$ and $Br_3(x^y,z): yxy^{-1}z yxy^{-1}=z yxy^{-1}z$ are equivalent to   
  $zyxzy=yxzyz^{-1}xz$. 

  The proof is complete.
\end{proof}

\begin{lemma}\label{lem:r21}
    Assume that $Q_{\Delta'}(\alpha'_0,\beta)=Q_{\Delta'}(\alpha,\alpha_0')=-1$ and $Q_{\Delta'}(\alpha,\beta)=0$. If $w(\alpha'_0)=1$, then 
    $\begin{cases}
        Br_3(T_\alpha^{T_{\alpha'_0}},T_\beta;\Delta'), & \text{if $w(\alpha)=w(\beta)=1$,}\\
        Br_4(T_\alpha^{T_{\alpha'_0}},T_\beta;\Delta'), & \text{if $w(\alpha)\neq 1=w(\beta)$ or $w(\beta)\neq 1=w(\alpha)$}.
    \end{cases}$
\end{lemma}

\begin{proof}
    We abbreviate $T_1=T_\beta, T_2=T_{\alpha'_0}$ and $T_3=T_{\alpha}$. Then $Co(T_1,T_3)$.

 We first assume that $w(\alpha)=w(\beta)=1$, then we have $Br_3(T_1,T_2), Br_3(T_2,T_3)$. Thus, 
 $$
 \begin{array}{rcl}
(T_2T_3T_2^{-1})T_1(T_2T_3T_2^{-1}) &=& T_2T_3(T_1T_2T_1^{-1})T_3T_2^{-1}\\
&=& T_2T_1T_3T_2T_3T_1^{-1}T_2^{-1}\\
&=& T_2T_1T_2T_3T_2T_1^{-1}T_2^{-1}\\
&=& T_1T_2T_1T_3T_2T_1^{-1}T_2^{-1}\\
&=& T_1(T_2T_3T_2^{-1})T_1.
 \end{array}
 $$ That is $Br_3((T_\alpha)^{T_{\alpha'_0}},T_\beta;\Delta')$ holds.

We then assume that $w(\alpha)\neq 1=w(\beta)$, then we have $Br_3(T_1,T_2), Br_4(T_2,T_3)$.
Thus, 
 $$
 \begin{array}{rcl}
(T_2T_3T_2^{-1})T_1(T_2T_3T_2^{-1})T_1 &=& T_2T_3(T_1T_2T_1^{-1})T_3T_2^{-1}T_1\\
&=& T_2T_1T_3T_2T_3T_1^{-1}T_2^{-1}T_1\\
&=& T_2T_1T_2T_3T_2T_3T_2^{-1}T_1^{-1}T_2^{-1}T_1\\
&=& T_1T_2T_1T_3T_2T_3T_2^{-1}T_1^{-1}T_2^{-1}T_1\\
&=& T_1T_2T_3T_1T_2T_3T_2^{-1}T_1^{-1}T_2^{-1}T_1\\
&=& T_1T_2T_3(T_2^{-1}T_1T_2T_1)T_3T_2^{-1}T_1^{-1}T_2^{-1}T_1\\
&=& T_1T_2T_3T_2^{-1}T_1T_2T_3T_1T_2^{-1}T_1^{-1}T_2^{-1}T_1\\
&=& T_1T_2T_3T_2^{-1}T_1T_2T_3T_2^{-1}\\
&=& T_1(T_2T_3T_2^{-1})T_1(T_2T_3T_2^{-1}).
 \end{array}
 $$ That is $Br_3((T_\alpha)^{T_{\alpha'_0}},T_\beta;\Delta')$ holds.

We can prove similarly that $Br_3((T_\alpha)^{T_{\alpha'_0}},T_\beta;\Delta')$ holds in case $w(\beta)\neq 1=w(\alpha)$.

The proof is complete.
\end{proof}

\begin{lemma}\label{lem:r23}
Assume that there is a $3$-cycle between $\alpha'_0,\beta,\alpha$ but there is no double arrow among them.

$(a)$ If $w(\alpha_0)\neq 1=w(\alpha)=w(\beta)$, then $Br_3((T_\alpha)^{T_{\alpha_0'}},T_\beta)$ holds in $Br_{\Delta'}$.

$(b)$ If $w(\alpha'_0),w(\beta)\neq 1=w(\alpha)$ or $w(\alpha'_0),w(\alpha)\neq 1=w(\beta)$, then $Br_4((T_\alpha)^{T_{\alpha_0'}},T_\beta)$ holds in $Br_{\Delta'}$.
\end{lemma}

\begin{proof}
As $w(\alpha)=1$, we have $Co((T_\beta)^{T_{\alpha}},T_{\alpha'_0})$ by $(R3)$. We abbreviate $T_1=(T_\beta)^{T_{\alpha}}, T_2=T_{\alpha'_0}$ and $T_3=T_{\alpha}$. Then $T_{\beta}=T_3^{-1}T_1T_3$.

$(a)$ Then we have $Co(T_1,T_2)$, $Br_3(T_1,T_3)$ and $Br_4(T_2,T_3)$. Therefore, $Br_3((T_\alpha)^{T_{\alpha_0'}},T_\beta)$ is equivalent to 
\begin{equation}\label{eq:long'}
    (T_2T_3T_2^{-1})(T_3^{-1}T_1T_3)(T_2T_3T_2^{-1})=(T_3^{-1}T_1T_3)(T_2T_3T_2^{-1})(T_3^{-1}T_1T_3).
\end{equation}

By $Br_4(T_2,T_3)$, we have $T_3T_2T_3T_2^{-1}T_3^{-1}=T_2^{-1}T_3T_2$. Thus, (\ref{eq:long'}) is equivalent to 
\begin{equation}\label{eq:long1'}
T_2^{-1}T_3T_2T_1T_2^{-1}T_3T_2=T_1T_2^{-1}T_3T_2T_1.
\end{equation}

It is easy to see that (\ref{eq:long1'}) follows by $Br_3(T_1,T_3)$ and $Co(T_1,T_2)$.

$(b)$ We may assume that $w(\alpha'_0),w(\beta)\neq w(\alpha)=1$.
 Then $Co(T_1,T_2)$, $Br_4(T_1,T_3)$ and $Br_4(T_2,T_3)$. Therefore, $Br_4((T_\alpha)^{T_{\alpha_0'}},T_\beta)$ is equivalent to 
\begin{equation}\label{eq:long}
    (T_2T_3T_2^{-1})(T_3^{-1}T_1T_3)(T_2T_3T_2^{-1})(T_3^{-1}T_1T_3)=(T_3^{-1}T_1T_3)(T_2T_3T_2^{-1})(T_3^{-1}T_1T_3)(T_2T_3T_2^{-1}).
\end{equation}

By $Br_4(T_2T_3)$, we have $T_3T_2T_3T_2^{-1}T_3^{-1}=T_2^{-1}T_3T_2$. Thus, (\ref{eq:long}) is equivalent to 
\begin{equation}\label{eq:long1}
T_2^{-1}T_3T_2T_1T_2^{-1}T_3T_2T_1=T_1T_2^{-1}T_3T_2T_1T_2^{-1}T_3T_2.
\end{equation}

It is easy to see that (\ref{eq:long1}) follows by $Br_4(T_1,T_3)$ and $Co(T_1,T_2)$. 

The proof is complete.
\end{proof}

\begin{lemma}\label{lem:R231}
Assume that there is a $3$-cycle between $\alpha'_0,\beta,\alpha$ and there is no double arrow from $\beta$ to $\alpha$ in $Q_{\Delta'}$. 

$(1)$ If there is a double arrow from $\alpha$ to $\alpha'_0$, then $Br_3((T_\alpha)^{T_{\alpha_0'}},T_\beta)$ holds in $Br_{\Delta'}$ in case $w(\beta)=1$ and $Br_4((T_\alpha)^{T_{\alpha_0'}},T_\beta)$ holds in $Br_{\Delta'}$ in case $w(\beta)\neq 1$.

$(2)$ If there is a double arrow from $\alpha'_0$ to $\beta$ in $Q_{\Delta'}$, then $Br_3((T_\alpha)^{T_{\alpha_0'}},T_\beta)$ holds in $Br_{\Delta'}$ in case $w(\alpha)=1$ and $Br_4((T_\alpha)^{T_{\alpha_0'}},T_\beta)$ holds in $Br_{\Delta'}$ in case $w(\alpha)\neq 1$.
\end{lemma}

\begin{proof}
    We only give the proof of $(1)$, as $(2)$ can be proved similarly. Since there is a double arrow from $\alpha$ to $\alpha'_0$, we see that $w(\alpha)=w(\alpha'_0)=1$.

    If $w(\beta)=1$, then $Br_3((T_\alpha)^{T_{\alpha_0'}},T_\beta)$ follows by $Br_3((T_{\alpha'_0})^{T_{\beta}},T_\alpha)$ and $Br_3(T_{\alpha'_0},T_\beta)$.

    If $w(\beta)\neq 1$, then $Co((T_{\alpha'_0})^{T_\beta},T_\alpha)$ by $(R3)$. We abbreviate $T_1=(T_{\alpha'_0})^{T_\beta}, T_2=T_\beta$ and $T_3=T_\alpha$. Then $T_{\alpha_0'}=T_2^{-1}T_1T_2$, $Co(T_1,T_3)$, $Br_4(T_1,T_2)$ and $Br_4(T_2,T_3)$. Thus, $(T_\alpha)^{T_{\alpha_0'}}=T_{\alpha_0'}T_\alpha(T_{\alpha_0'})^{-1}=T_2^{-1}T_1T_2T_3T_2^{-1}T^{-1}_1T_2$.
Therefore, 
$$\begin{array}{rcl}(T_\alpha)^{T_{\alpha_0'}}T_\beta(T_\alpha)^{T_{\alpha_0'}}T_\beta &=& (T_2^{-1}T_1T_2T_3T_2^{-1}T^{-1}_1T_2)T_2(T_2^{-1}T_1T_2T_3T_2^{-1}T^{-1}_1T_2)T_2\vspace{2.5pt}\\ &=&T_2^{-1}T_1T_2T_3
T_1T_2T^{-1}_1
T_3T_2^{-1}T^{-1}_1T_2T_2\vspace{2.5pt}\\ 
&=& T_2^{-1}T_1T_2T_1T_3
T_2T_3T^{-1}_1
T_2^{-1}T^{-1}_1T_2T_2\vspace{2.5pt}\\
&=& T_1T_2T_1T_2^{-1}
T_3
T_2T_3T_2T^{-1}_1
T_2^{-1}T^{-1}_1T_2\vspace{2.5pt}\\ 
&=& T_1T_2T_1
T_3
T_2T_3T^{-1}_1
T_2^{-1}T^{-1}_1T_2.
\end{array}$$

$$\begin{array}{rcl}T_\beta(T_\alpha)^{T_{\alpha_0'}}T_\beta(T_\alpha)^{T_{\alpha_0'}} &=& T_2(T_2^{-1}T_1T_2T_3T_2^{-1}T^{-1}_1T_2)T_2(T_2^{-1}T_1T_2T_3T_2^{-1}T^{-1}_1T_2)\vspace{2.5pt}\\ &=&
T_1T_2T_3T_2^{-1}T^{-1}_1T_2T_1T_2T_3T_2^{-1}T^{-1}_1T_2\vspace{2.5pt}\\ 
&=& T_1T_2T_3T_1T_2T^{-1}_1T_3T_2^{-1}T^{-1}_1T_2\vspace{2.5pt}\\ 
&=& T_1T_2T_1
T_3
T_2T_3T^{-1}_1
T_2^{-1}T^{-1}_1T_2.
\end{array}$$

Thus, $Br_4((T_\alpha)^{T_{\alpha_0'}},T_\beta)$ holds in $Br_{\Delta'}$.

The proof is complete.
\end{proof}

\begin{lemma}\label{lem:R41}
    Assume that the subquiver of $Q_{\Delta'}$ formed by $\alpha,\beta,\alpha'_0,\gamma$ is isomorphic to the third quiver in Figure \ref{Fig:subquiver2}. If $w(\alpha'_0)\neq 1$, then 
    $\begin{cases}
        Br(T_\gamma^{T_\alpha},T_\beta^{T_{\alpha'_0}};\Delta'), & \text{if $w(\alpha)=1$},\\
        Co(T_\gamma^{T_\alpha},T_\beta^{T_{\alpha'_0}};\Delta'), & \text{if $w(\alpha)\neq 1$}.
    \end{cases}$
\end{lemma}

\begin{proof}
  If $w(\alpha)=1$, 
  then 
  $$
\begin{array}{rcl}
 Br_3(T_\gamma^{T_\alpha},T_\beta^{T_{\alpha'_0}};\Delta') 
 &\xLeftrightarrow{Br_4(T_\beta,T_{\alpha'_0})}&  Br_3(T_\gamma^{T_\alpha},T_{\beta}^{T_\beta^{-1}T_{\alpha'_0}^{-1}};\Delta')    \\
 &\xLeftrightarrow{\text{Lemma \ref{lem:R9}: } Co(T_{\gamma}^{T_\beta T_\alpha },T_{\alpha'_0};\Delta')}&  Br_3(T_\gamma^{T_\alpha},T_{\beta};\Delta')    \\
     &\xLeftrightarrow{Br_3(T_{\gamma},T_\alpha;\Delta')}  &  Br_3(T^{-1}_\gamma T_\alpha T_\gamma,T_{\beta};\Delta') \\
     &\xLeftrightarrow{Co({T_\beta},T_{\gamma};\Delta')}& Br_3(T_\alpha,T_{\beta};\Delta'): R2.
\end{array}
$$  

If $w(\alpha)\neq 1$, then $Co(T_\gamma^{T_\alpha},T_\beta^{T_{\alpha'_0}};\Delta')$ follows by the relation $R8$.

The proof is complete.
\end{proof}

The following lemma is important for us to prove that $h_{\Delta,\mu_\alpha\Delta}$ preserves the relations $R9$.

\begin{lemma}\label{lem:R90}
    Assume that $x_1,x_2\cdots x_n, y$ and $z$ satisfy the following relations:

$\bullet$  $Br_3(x_i,x_{i+1}) \;\; \text{mod } n$ and $Co(x_i,x_j)$ for $i-j\neq \pm 1 (\text{mod } n)$, 

$\bullet$ $x_k=z^y$ for some $k>3$,

$\bullet$ $Br_{3}(y,z), Br_{3}(x_{k-1},y), Br_3(z,x_{k+1})$, $Co(y,x_i)$ for $i\neq k-1,k$ and $Co(z,x_i)$ for $i\neq k,k+1$. 

Then $Cyl(x_1,\cdots,x_n)$ holds if and only if $Cyl(x_2^{x_1}, x_3,\cdots,x_{k-1},y,z,x_{k+1}\cdots, x_n)$.
\end{lemma}

\begin{proof}
It suffices to prove that the first relation implies the second one. Recall we have 
$$Cyl(x_1,\cdots,x_n)\sim Cyl(x_3,\cdots, x_n,x_1,x_2)$$ and $$Cyl(x_2^{x_1}, x_3,\cdots,x_{k-1},y,z,x_{k+1}\cdots, x_n)\sim Cyl(x_3,\cdots,x_{k-1},y,z,x_{k+1}\cdots, x_n,x_2^{x_1}),$$
we shall prove that 
$$
\begin{array}{rcl}
     & & (x_3\cdots x_{k-1}yzx_{k+1}\cdots x_n x_1x_2x_1^{-1})  x_3\cdots x_{k-1}yzx_{k+1}\cdots x_{n-2}x_{n-1}  \\
     &=& (x_4\cdots x_{k-1}yzx_{k+1}\cdots x_n x_1x_2x_1^{-1}x_3)\cdots x_{k-1}yzx_{k+1}\cdots x_{n-1}x_n\\
     &\Leftrightarrow& 
     x_3\cdots x_{k-1}x_{k}yx_{k+1}\cdots x_n x_1x_2x_1^{-1}x_3\cdots x_{k-1}x_{k}yx_{k+1}\cdots x_{n-1}\\
     &=& x_4\cdots x_{k-1}x_k y x_{k+1}\cdots x_n x_1x_2 x_1^{-1}x_3\cdots x_{k-1}x_ky x_{k+1}\cdots x_{n}\\
     &\xLeftrightarrow[Co(x_1,x_3\cdots x_{k-2})]{Co(y,x_{k+1}\cdots x_n x_2^{x_1} x_{k-2})}& 
      x_3\cdots x_{k-1}x_{k}x_{k+1}\cdots x_n x_1x_2x_3\cdots x_1^{-1}y x_{k-1}x_{k}yx_{k+1}\cdots x_{n-1}\\
     &=& x_4\cdots x_{k-1}x_k x_{k+1}\cdots x_n x_1x_2 x_3\cdots  x_1^{-1}y x_{k-1}x_ky x_{k+1}\cdots x_{n}\\
     &\xLeftrightarrow{Cyl(x_3,\cdots, x_n,x_1,x_2)}&
     x_{n}^{-1}\cdots x_{k-1}^{-1}x_1^{-1}y x_{k-1}x_{k}yx_{k+1}\cdots x_{n-1}\\
     &=& x_1^{-1}x_{n}^{-1}\cdots x_{k-1}^{-1}x_1^{-1}y x_{k-1}x_{k}yx_{k+1}\cdots x_{n}\\
     &\xLeftrightarrow[Co(x_1,x_{k-1}\cdots x_{n-1})]{Br_3(x_1,x_{n}), Co(y,x_{k+1}\cdots x_n)}&  
     x_{n-1}^{-1}\cdots x_{k-1}^{-1} y x_{k-1}x_{k}x_{k+1}\cdots x_{n-1}\\
     &=& x_{n}^{-1}\cdots x_{k-1}^{-1} y x_{k-1}x_{k}x_{k+1}\cdots x_{n}\\
     &\Leftrightarrow& Co(x_{n-1}^{-1}\cdots x_{k-1}^{-1} y x_{k-1}x_{k}x_{k+1}\cdots x_{n-1},x_n).
\end{array}
$$
As $x_{k}^{-1}x_{k-1}^{-1}yx_{k-1}x_k=(yz^{-1}y^{-1})yx_{k-1}y^{-1}(yzy^{-1})=yx_{k-1}y^{-1}$, we have 
$$x_{n-1}^{-1}\cdots x_{k-1}^{-1} y x_{k-1}x_{k}x_{k+1}\cdots x_{n-1}=yx_{k-1}y^{-1}.$$ Thus, $Co(x_{n-1}^{-1}\cdots x_{k-1}^{-1} y x_{k-1}x_{k}x_{k+1}\cdots x_{n-1},x_n)$ follows.

The proof is complete.
\end{proof}

Let $\Delta$ be ordinary triangulations of $\Sigma$, and let $p$ be a puncture. 
For any sequence of mutations $\mu:\Delta\to \mu\Delta$ that satisfies the requirement for the relation $R9$, denote by $\mathcal R(\mu)$ the corresponding instance of $R9$ for $p$ in $\Delta$ under $\mu$.

\begin{lemma}\label{lem:R9unique}
For any two mutation sequences $\mu:\Delta\to \mu\Delta$ and $\mu':\Delta\to \mu'\Delta$ satisfying the condition for relation $R9$, we have that $\mathcal R(\mu)$ holds if and only if $\mathcal R(\mu')$ holds (denoted $\mathcal R(\mu)\sim \mathcal R(\mu')$), provided that the relations $R1$ through $R8$ are satisfied. Consequently, it suffices to choose a single mutation sequence $\mu:\Delta\to \mu\Delta$ to define the relation $R9$ for each puncture $p$ in $Br_\Delta$.  
\end{lemma}

\begin{proof}
    We proceed by induction on $n_p(\Delta)$, the number of loops incident to $p$ in $\Delta$. If $n_p(\Delta)=0$, then the result is trivially true. Now assume that the result holds for all triangulations where $n_p(\Delta)<k$, and consider the case where $n_p(\Delta)=k$. 

Since $\mu\neq \mu'$, we may write $\mu=\cdots\mu_{\beta_1}{\vec\mu'}\vec \mu$ and $\mu'=\cdots\mu_{\beta_2}{\vec \mu''}\vec \mu$, where $\beta_1$ does not appear in $\vec \mu''$, $\beta_2$ does not appear in $\vec \mu'$, ${\vec \mu'}$ commutes with both $\mu_{\beta_2}$ and ${\vec \mu''}$, ${\vec \mu''}$ commutes with $\mu_{\beta_1}$, and $\mu_{\beta_1}\mu_{\beta_2}\neq \mu_{\beta_2}\mu_{\beta_1}$. 
  
We may further assume $\vec \mu={\vec \mu'}={\vec \mu''}=\emptyset$, since otherwise, we have
    $$
    \begin{array}{rcl}
       \mathcal R(\mu) &\sim& \mathcal R(\cdots{\vec \mu''}\mu_{\beta_1}{\vec \mu'}\vec \mu)\sim \mathcal R(\cdots\mu_{\beta_1}{\vec \mu''}{\vec \mu'}\vec \mu)\sim \mathcal R(\cdots\mu_{\beta_1}{\vec \mu'}{\vec \mu''}\vec \mu)\\
        & \sim & \mathcal R(\cdots\mu_{\beta_2}{\vec \mu'}{\vec \mu''}\vec \mu)\sim \mathcal R(\cdots{\vec \mu'}\mu_{\beta_2}{\vec \mu''}\vec \mu)\sim \mathcal R(\mu'),
    \end{array}
    $$
where the second, third and fifth equivalences follow by the fact that the operations $h^{\mu_\alpha}_{\Delta,\mu_\alpha\Delta}$ and $h^{\mu_\beta}_{\Delta,\mu_\beta\Delta}$ commute whenever $\mu_{\alpha}\mu_{\beta}=\mu_\beta \mu_\alpha$, and the first, fourth and sixth equivalences follow by induction hypothesis. 

Thus, $\beta_1,\beta_2$ are two sides of some triangle in $\Delta$. Denote the third side by $\beta_3$. As $\Sigma$ is not a once-punctured torus, we can define $S_i$ as the component of $\Sigma\setminus \beta_i$ that does not contain the triangle $(\beta_1,\beta_2,\beta_3)$, for $i=1,2,3$. For each $i=1,2,3$, fix a sequence of mutations $\vec \mu_i$ at the loops of $\Delta$ within $S_i$, chosen so that the number of loops incident to $p$ decreases after each step. These sequences $\vec\mu_i$ commute with both $\mu_{\beta_1}$ and $\mu_{\beta_2}$ for $i=1,2,3$. By induction hypothesis, we may assume that $\vec \mu_i=0$ for $i=1,2,3$.

Now we consider two cases:

\textbf{(Case 1)} If $\beta_3$ is a special loop, then the result follows by Lemma \ref{lem:R9u1}.

\textbf{(Case 2)} If $\beta_3$ is not a special loop, then the result follows by Lemma \ref{lem:R9u2}.
    
The proof is complete.
\end{proof}

\begin{lemma}\label{lem:R9u1}
    $\mathcal R(\mu)\sim \mathcal R(\mu')$ if $\beta_3$ is a special loop.
\end{lemma}

\begin{proof}
We may assume that there is an arrow from $\beta_1$ to $\beta_2$ in the quiver $Q_\Delta$. We have the loops incident to $p$ in $\Delta$ are $\beta_1,\beta_2,\beta_3$. Thus, we may assume that $\mu=\mu_{\beta_3}\mu_{\beta_2}\mu_{\beta_1}$ and $\mu'=\mu_{\beta_3}\mu_{\beta_1}\mu_{\beta_2}$.

\begin{figure}[h]
    \centering
    \includegraphics[width=8cm]{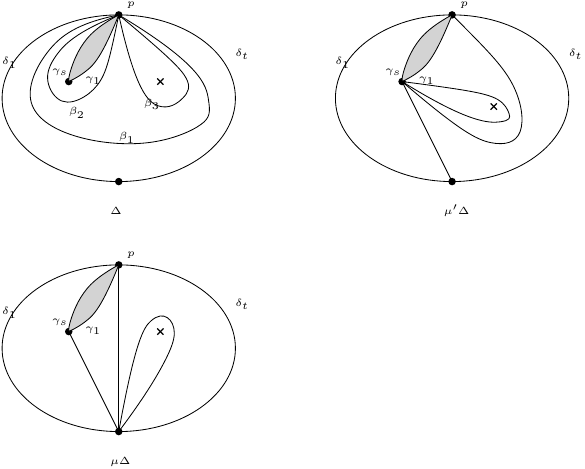}
    \caption{The case $\beta_3$ is a special loop}
\end{figure}

We only consider the case where $\beta_1$ and $\beta_2$ are not the loops in any self-folded triangles, as the other cases can be proved similarly. Suppose that the arcs incident to $p$ in $\Delta$ are $\beta_1,\beta_3 (\text{twice}), \beta_2,\gamma_1,\cdots, \gamma_s,\beta_2,\beta_1,\delta_1,\cdots,\delta_t$.

Thus, by $R2: Br_4(T_{\beta_1},T_{\beta_3})$, we obtain 
$\mathcal R(\mu)=Cyl(T_{\beta_1}^{T_{\beta_3}^{-1}}, T_{\gamma_1}^{T_{\beta_2}},\cdots, T_{\gamma_s},T_{\delta_1}^{T_{\beta_2}T_{\beta_1}},\cdots, T_{\delta_t})$
and $\mathcal R(\mu')=Cyl(T_{\beta_1}, T_{\gamma_1}^{T_{\beta_3}T_{\beta_2}},\cdots, T_{\gamma_s},T_{\delta_1}^{T_{\beta_2}T_{\beta_1}},\cdots, T_{\delta_t})$. 

Then the equivalence $\mathcal R(\mu)\sim \mathcal R(\mu')$ follows from the relations $Co(T_{\gamma_3},T_{\gamma_i}), Co(T_{\gamma_3},T_{\delta_i})$ for all $i\geq 2$, and $Co(T_{\gamma_3},T_{\delta_1}^{T_{\beta_2}T_{\beta_1}})$. The relation $Co(T_{\gamma_3},T_{\delta_1}^{T_{\beta_2}T_{\beta_1}})$ itself follows from the relations $Co(T_{\beta_3},T_{\delta_1})$, $Co(T_{\beta_2},T_{\delta_1})$ and $R3:Co(T_{\beta_3},T_{\beta_1}^{T_{\beta_2}})$.

The proof is complete.
\end{proof}

\begin{lemma}\label{lem:R9u2}
    $\mathcal R(\mu)\sim \mathcal R(\mu')$ if $\beta_3$ is not a special loop.
\end{lemma}

\begin{proof}
We may assume that there is an arrow from $\beta_1$ to $\beta_2$ in the quiver $Q_\Delta$. We have the loops incident to $p$ in $\Delta$ are $\beta_1,\beta_2,\beta_3$. Thus, we may assume that $\mu=\mu_{\beta_3}\mu_{\beta_2}\mu_{\beta_1}$ and $\mu'=\mu_{\beta_3}\mu_{\beta_1}\mu_{\beta_2}$.

\begin{figure}[h]
    \centering
    \includegraphics[width=8cm]{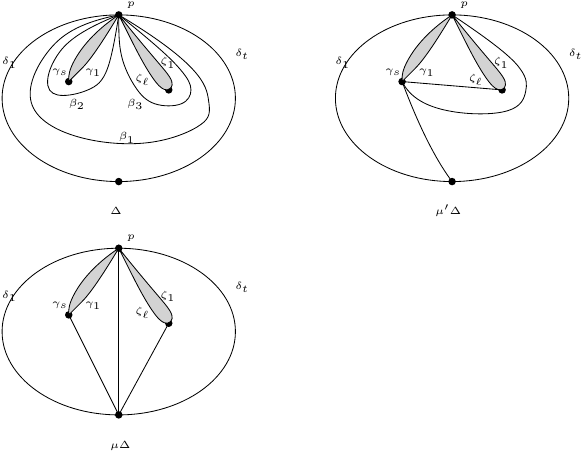}
    \caption{The case $\beta_3$ is not a special loop}
\end{figure}

We only consider the case where $\beta_1$, $\beta_2$ and $\beta_3$ are not the loops in any self-folded triangles, as the other cases can be proved similarly. Suppose that the arcs incident to $p$ in $\Delta$ are $\beta_1,\beta_3,\zeta_1,\cdots, \zeta_\ell,\beta_3, \beta_2,\gamma_1,\cdots, \gamma_s,\beta_2,\beta_1,\delta_1,\cdots,\delta_t$. 

By calculation, we obtain 
$$
\begin{cases}
 \mathcal R(\mu)=Cyl(T_{\zeta_2},\cdots, T_{\zeta_\ell},T_{\beta_3}, T_{\gamma_1}^{T_{\beta_2}},\cdots, T_{\gamma_s},T_{\delta_1}^{T_{\beta_2}T_{\beta_1}},\cdots, T_{\delta_t}, T_{\zeta_1}^{T_{\beta_1}T_{\beta_3}}),\\
 \mathcal R(\mu')=Cyl(T_{\zeta_2},\cdots, T_{\zeta_\ell}, T_{\gamma_1}^{T_{\beta_3}T_{\beta_2}},\cdots, T_{\gamma_s},T_{\delta_1}^{T_{\beta_2}T_{\beta_1}},\cdots, T_{\delta_t}, T_{\beta_1}, T_{\zeta_1}^{T_{\beta_3}}).
\end{cases}
$$

Then $\mathcal R(\mu)\sim \mathcal R(\mu')$ follows by Lemma \ref{lem:R90}.

The proof is complete.
\end{proof}

\begin{lemma}\label{lem:R9a1}
    Let $\Delta$ be a triangulation of $\Sigma$ and $\alpha\in \Delta$ be an internal arc. For any puncture $p$, if the number of arcs incident to $p$ in $\Delta$ differs from that in $\Delta'=\mu_\alpha(\Delta)$, 
    then the relation $R9$ for $p$ in $Br_{\Delta}$ holds in $Br_{\mu_\alpha\Delta}$ under $h_{\mu_\alpha\Delta,\Delta}$. 
\end{lemma}

\begin{proof}
If the number of arcs incident to $p$ in $\Delta$ is less than that in $\Delta'$, then the result follows from Lemma \ref{lem:R9unique}.

We now consider the case that the number of arcs incident to $p$ in $\Delta$ is greater than that in $\Delta'$. Thus, at least one of $s(\alpha_1)$ and $s(\alpha_3)$ is $p$. We may assume that $s(\alpha_1)=p$.

{\bf Case 1}: Suppose $s(\alpha_3)\neq p$. Then $s(\alpha_2),s(\alpha_4)\neq p$. 

Let $\mu:\Delta\to \mu\Delta$ be a mutation sequence that satisfies the requirements for the relation $R9$. Then the sequence $\mu:\Delta'\to \mu\Delta'$ also satisfies the requirements for the relation $R9$. Assume the relations $R9$ in $Br_\Delta$ and $Br_{\Delta'}$ under $\mu$ are of form 
$$R9\Delta:Cyl(T_{\alpha_4},X_1,\cdots,X_n,T_{\alpha_1},T_{\alpha})$$
and 
$$R9\Delta':Cyl(T_{\alpha_4},X_1,\cdots,X_n,T_{\alpha_1})$$
for some Laurent monomials $X_1,\cdots,X_n$ in $T_{\beta},\beta\in \Delta\setminus \{\alpha_1,\alpha_2,\alpha_3,\alpha_4,\alpha\}$.

Then we have $h_{\Delta',\Delta}(R9\Delta)=Cyl(T_{\alpha_4},X_1,\cdots,X_n,T_{\alpha_1}^{T_{\alpha'}},T_{\alpha'})$, which follows from $R9\Delta'$, $Co(T_{\alpha'},X_i)$ for all $i=1,\cdots,n$ and $Br_3(T_{\alpha'},T_{\alpha_1})$.

{\bf Case 2}: Suppose $s(\alpha_3)=p$ and $s(\alpha_2),s(\alpha_4)\neq p$.

Let $\mu\mu_\alpha:\Delta\to \mu\mu_\alpha\Delta$ be a mutation sequence satisfying the requirements for the relation $R9$. Then the sequence $\mu:\Delta'\to \mu\Delta'$  satisfies the requirements for the relation $R9$. 

{\bf Case 2.1}: $\alpha$ is not a special loop.

We may assume the relations $R9$ in $Br_\Delta$ and $Br_{\Delta'}$ under $\mu$ are of form 
$$R9\Delta:Cyl(T_{\alpha_1},T^{T_{\alpha}}_{\alpha_4},X_1,\cdots,X_n,T_{\alpha_3},T_{\alpha_2}^{T_{\alpha}})$$
and 
$$R9\Delta':Cyl(T_{\alpha_1},T_{\alpha_4},X_1,\cdots,X_n,T_{\alpha_3},T_{\alpha_2})$$
for some Laurent monomials $X_1,\cdots,X_n$ in $T_{\beta},\beta\in \Delta\setminus \{\alpha_1,\alpha_2,\alpha_3,\alpha_4,\alpha\}$.

Then we have $h_{\Delta',\Delta}(R9\Delta)=Cyl(T^{T_{\alpha'}}_{\alpha_1},T^{T_{\alpha'}}_{\alpha_4},X_1,\cdots,X_n,T^{T_{\alpha'}}_{\alpha_3},T_{\alpha_2}^{T_{\alpha'}})$, which follows from $R9\Delta'$, $Co(T_{\alpha'},X_i)$ for all $i=1,\cdots,n$.

{\bf Case 2.2}: $\alpha$ is a special loop. 

We may assume the relations $R9$ in $Br_\Delta$ and $Br_{\Delta'}$ under $\mu$ are of form 
$$R9\Delta:Cyl(T_{\alpha_1},T^{T_{\alpha}}_{\alpha_4},X_1,\cdots,X_n)$$
and 
$$R9\Delta':Cyl(T_{\alpha_1},T_{\alpha_4},X_1,\cdots,X_n)$$
for some Laurent monomials $X_1,\cdots,X_n$ in $T_{\beta},\beta\in \Delta\setminus \{\alpha_1,\alpha_2,\alpha_3,\alpha_4,\alpha\}$.

Then we have $h_{\Delta',\Delta}(R9\Delta)=Cyl(T^{T_{\alpha'}}_{\alpha_1},T^{T_{\alpha'}}_{\alpha_4},X_1,\cdots,X_n)$, which follows from $R9\Delta'$, $Co(T_{\alpha'},X_i)$ for all $i=1,\cdots,n$.

{\bf Case 3}: Suppose $s(\alpha_3)=p$ and exactly that only one of $s(\alpha_2),s(\alpha_4)$ equals $p$. We may assume $s(\alpha_2)\neq p= s(\alpha_4)$. 

We prove this case by induction on the loops $n_p(\Delta)$ incident to $p$ in $\Delta$. We have $n_p(\Delta)\geq 3$.

For $n_p(\Delta)=3$, the loops incident to $p$ are $\alpha_3,\alpha_4$ and $\alpha$.

If $\alpha_3$ and $\alpha_4$ are not special loops, then we have the relations $R9$ in $Br_\Delta$ and $Br_{\Delta'}$ are
$$R9(\Delta): Cyl(T_{\alpha_1},T_{\beta_{1}}^{T_{\alpha}T_{\alpha_4}},T_{\beta_{2}},\cdots, T_{\beta_{s_1}},T_{\alpha}^{T_{\alpha}T_{\alpha_4}},T^{T_{\alpha_3}}_{\gamma_{1}},\cdots,T_{\gamma_{s_2}},T_{\alpha_2}^{T_{\alpha_3}T_{\alpha}})$$
and
$$R9(\Delta'):Cyl(T_{\alpha_1},T_{\beta_{1}}^{T_{\alpha_4}},T_{\beta_{2}},\cdots, T_{\beta_{s_1}},T_{\alpha'}^{T_{\alpha_4}},T^{T_{\alpha_3}}_{\gamma_{1}},\cdots,T_{\gamma_{s_2}},T_{\alpha_2}^{T_{\alpha_3}}).$$
Thus, we have
$$h_{\Delta',\Delta}(R9\Delta)= Cyl(T_{\alpha_1}^{T_{\alpha'}},T_{\beta_{1}}^{T_{\alpha'}T_{\alpha_4}},T_{\beta_{2}},\cdots, T_{\beta_{s_1}},T_{\alpha'}^{T_{\alpha'}T_{\alpha_4}},T^{T_{\alpha'}T_{\alpha_3}}_{\gamma_{1}},\cdots,T_{\gamma_{s_2}},T_{\alpha_2}^{T_{\alpha'}T_{\alpha_3}}),$$
which follows from $R9\Delta'$, $Co(T_{\alpha'},T_{\beta_i})$ for $i=2,\cdots,s_1$ and $Co(T_{\alpha'},T_{\gamma_i})$ for $i=2,\cdots,s_2$.

\begin{figure}[h]
    \centering
    \includegraphics{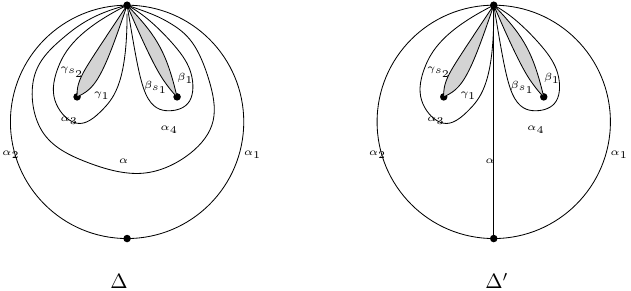}
    \caption{$\alpha_3$ and $\alpha_4$ are not special loops}
\end{figure}

If there is a special loop in $\{\alpha_3,\alpha_4\}$, we may assume that $\alpha_3$ is a special loop as the other cases can be proved similarly, then we have the relations $R9$ in $Br_\Delta$ and $Br_{\Delta'}$ are
$$R9\Delta: Cyl(T_{\alpha_1},T_{\beta_{1}}^{T_{\alpha}T_{\alpha_4}},T_{\beta_{2}},\cdots, T_{\beta_{s_1}},T_{\alpha}^{T_{\alpha}T_{\alpha_4}},T_{\alpha_2}^{T_{\alpha_3}T_{\alpha}})$$
and
$$R9\Delta':Cyl(T_{\alpha_1},T_{\beta_{1}}^{T_{\alpha_4}},T_{\beta_{2}},\cdots, T_{\beta_{s_1}},T_{\alpha'}^{T_{\alpha_4}},T_{\alpha_2}^{T_{\alpha_3}}).$$
Thus, we have
$$h_{\Delta',\Delta}(R9\Delta)= Cyl(T_{\alpha_1}^{T_{\alpha'}},T_{\beta_{1}}^{T_{\alpha'}T_{\alpha_4}},T_{\beta_{2}},\cdots, T_{\beta_{s_1}},T_{\alpha'}^{T_{\alpha'}T_{\alpha_4}},T_{\alpha_2}^{T_{\alpha'}T_{\alpha_3}}),$$
which follows from $R9\Delta'$ and $Co(T_{\alpha'},T_{\beta_i})$ for $i=2,\cdots,s_1$.

\begin{figure}[h]
    \centering
    \includegraphics{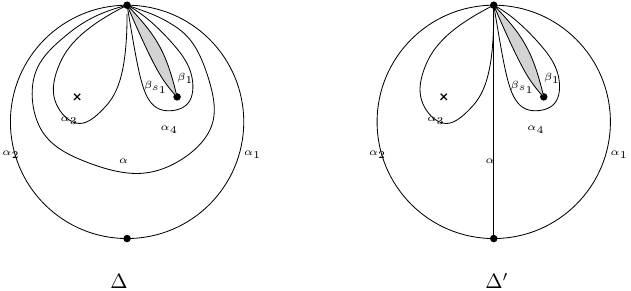}
    \caption{$\alpha_3$ is a special loop}
\end{figure}

For $n_p(\Delta)>3$, let $\mu$ be a mutation sequence at loops incident to $p$ such that the number of loops incident to $p$ decreases at each step, and $\alpha_3,\alpha_4,\alpha$ are the only loops incident to $p$ in $\mu\Delta$. 
Then $\mu$ commutes with $\mu_{\alpha}, \mu_{\alpha_3}$ and $\mu_{\alpha_4}$, and 
$$R9(\Delta)=h^{\mu_{\alpha_4}\mu_{\alpha_3}\mu_{\alpha}\mu}_{\Delta,\mu_{\alpha_4}\mu_{\alpha_3}\mu_{\alpha}\mu\Delta}(R9(\mu_{\alpha_4}\mu_{\alpha_3}\mu_{\alpha}\mu \Delta)), \quad R9(\Delta')=h^{\mu_{\alpha_4}\mu_{\alpha_3}\mu}_{\Delta',\mu_{\alpha_4}\mu_{\alpha_3}\mu\Delta'}(R9(\mu_{\alpha_4}\mu_{\alpha_3}\mu \Delta')),$$

$$R9(\mu\Delta)=h^{\mu_{\alpha_4}\mu_{\alpha_3}\mu_{\alpha}}_{\mu\Delta,\mu_{\alpha_4}\mu_{\alpha_3}\mu_{\alpha}\mu\Delta}(R9(\mu_{\alpha_4}\mu_{\alpha_3}\mu_{\alpha}\mu \Delta)), \quad R9(\mu\Delta')=h^{\mu_{\alpha_4}\mu_{\alpha_3}}_{\mu\Delta',\mu_{\alpha_4}\mu_{\alpha_3}\mu\Delta'}(R9(\mu_{\alpha_4}\mu_{\alpha_3}\mu \Delta')).$$ 
Therefore, $$
\begin{array}{rcl}
   h_{\Delta',\Delta}(R9(\Delta))  & =& h_{\Delta',\Delta}h^{\mu_{\alpha_4}\mu_{\alpha_3}\mu_{\alpha}\mu}_{\Delta,\mu_{\alpha_4}\mu_{\alpha_3}\mu_{\alpha}\mu\Delta}(R9(\mu_{\alpha_4}\mu_{\alpha_3}\mu_{\alpha}\mu \Delta))\\
&=&h^{\mu}_{\mu_{\alpha}\Delta,\mu_\alpha\mu\Delta}h_{\mu_\alpha\mu\Delta,\mu\Delta}h^{\mu_{\alpha_4}\mu_{\alpha_3}\mu_{\alpha}}_{\mu\Delta,\mu_{\alpha_4}\mu_{\alpha_3}\mu_{\alpha}\mu\Delta}(R9(\mu_{\alpha_4}\mu_{\alpha_3}\mu_{\alpha}\mu \Delta))\\
   & = & h^{\mu}_{\mu_{\alpha}\Delta,\mu_\alpha\mu\Delta}h_{\mu_\alpha\mu\Delta,\mu\Delta}(R9(\mu\Delta)).
\end{array}
$$

As $n_p(\mu\Delta)=3$, we have $h_{\mu_\alpha\mu\Delta,\mu\Delta}(R9(\mu\Delta))$ holds in $Br_{\mu_\alpha\mu\Delta}$. By induction hypothesis, we have $h^{\mu}_{\mu_{\alpha}\Delta,\mu_\alpha\mu\Delta}h_{\mu_\alpha\mu\Delta,\mu\Delta}(R9(\mu\Delta))$ holds in $Br_{\mu_\alpha\Delta}$.

The proof is complete.
\end{proof}

Assume that $\alpha$ is not a self-folded arc and a diagonal of some clockwise cyclic quadrilateral $(\alpha_1,\alpha_2,\alpha_3,\alpha_4)$ in $\Delta$ such that $(\alpha_1,\alpha_2,\overline\alpha)$ is a triangle. 

\begin{lemma}\label{lem:R9a}
    Assume that $s(\alpha)=s(\alpha_4)=p\neq t(\alpha),t(\alpha_1)$. If $\alpha_4$ is the unique loop incident to $p$ in $\Delta$, then the relation $R9$ for $p$ in $Br_{\Delta}$ holds in $Br_{\mu_\alpha\Delta}$ under $h_{\mu_\alpha\Delta,\Delta}$.
\end{lemma}

\begin{proof} 
Suppose that $\alpha_1,\alpha,\alpha_4,\beta_1,\cdots,\beta_s,\alpha_4,\alpha_3$, and $\gamma_1,\cdots,\gamma_t$ form a complete clockwise list of the loops incident to $p$ in $\Delta$ for some $s\geq 1$ and $t\geq 0$ ($\alpha_3$ may equal $\alpha_1$, in which case $t=0$).
    
Since ${\alpha_4}$ is the unique loop incident to $p$ in $\Delta$, we have that the relation $R9$ for $p$ in $Br_{\Delta}$ is 
$$R9(\Delta):Cyl(T_{\alpha_1},T_{\alpha},T_{\beta_1}^{T_{\alpha_4}},T_{\beta_2},\cdots, T_{\beta_s},T_{\alpha_3}^{T_{\alpha_4}},T_{\gamma_1},\cdots,T_{\gamma_t})$$
    and the relation $R9$ for $p$ in $Br_{\mu_{\alpha}\Delta}$ is 
$$R9(\mu_\alpha\Delta):Cyl(T_{\alpha_1},T_{\beta_1}^{T_{\alpha_4}},T_{\beta_2},\cdots, T_{\beta_s},T_{\alpha'}^{T_{\alpha_4}},T_{\alpha_3},T_{\gamma_1},\cdots,T_{\gamma_t}).$$
Thus, the relation $R9$ for $p$ in $Br_{\Delta}$ under $h_{\mu_\alpha\Delta,\Delta}$ is 
$$h_{\mu_\alpha\Delta,\Delta}(R9(\Delta)): Cyl(T_{\alpha_1}^{T_{\alpha'}},T_{\alpha'},T_{\beta_1}^{T_{\alpha_4}},T_{\beta_2},\cdots, T_{\beta_s},T_{\alpha_3}^{T_{\alpha_4}T_{\alpha'}},T_{\gamma_1},\cdots,T_{\gamma_t}).$$ 

Then the result follows by Lemma \ref{lem:R90}.

The proof is complete.
\end{proof}

\begin{lemma}\label{lem:R9b}
    Assume that $s(\alpha)=s(\alpha_2)=p\neq t(\alpha),t(\alpha_3)$. If $\alpha_1$ is the unique loop incident to $p$ in $\Delta$, then the relation $R9$ for $p$ in $Br_{\Delta}$ holds in $Br_{\mu_\alpha\Delta}$ under $h_{\mu_\alpha\Delta,\Delta}$.
\end{lemma}

The proof is similar to Lemma \ref{lem:R9a}, so we omit it.

\begin{lemma}\label{lem:R9c}
    Assume that $s(\alpha_1)=s(\alpha_2)=s(\alpha_3)=s(\alpha_4)=p$. If $\alpha_1,\alpha_2,\alpha_3,\alpha_4,\alpha$ form a complete list of the loops incident to $p$ in $\Delta$, then the relation $R9$ for $p$ in $Br_{\Delta}$ holds in $Br_{\mu_\alpha\Delta}$ under $h_{\mu_\alpha\Delta,\Delta}$.
\end{lemma}

\begin{proof}
    Let $\mu:\Delta\to \mu\Delta$ be a sequence of mutations that satisfy the requirements for relation $R9$. 

(Case 1) There is no special loops in $\{\alpha_1, \alpha_2, \alpha_3,\alpha_4,\alpha\}$. We may assume that $\mu=\mu_{\alpha_4}\mu_{\alpha_3}\mu_{\alpha}\mu_{\alpha_2}\mu_{\alpha_1}$. Then $\mu'=\mu_{\alpha_4}\mu_{\alpha_3}\mu_{\alpha_2}\mu_{\alpha'}\mu_{\alpha_1}$ satisfies the requirement for the relation $R9$ for $\mu_\alpha\Delta$.

\begin{figure}[h]
    \centering
    \includegraphics{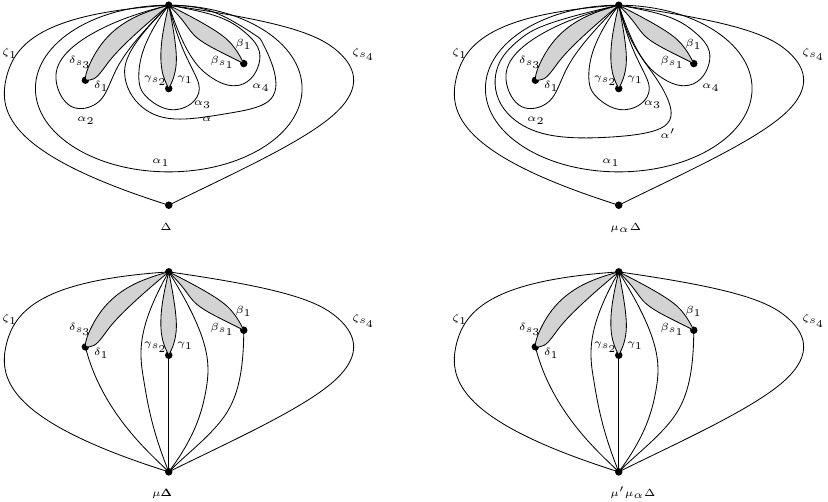}
\end{figure}

Thus, the relations $R9$ for $p$ in $Br_\Delta$ and $Br_{\mu_\alpha\Delta}$ under $\mu$ and $\mu'$, respectively, are 
$$Cyl(T_{\beta_{1}}^{T_{\alpha_1}T_{\alpha}T_{\alpha_4}},T_{\beta_{2}},\cdots, T_{\beta_{s_1}},T_{\alpha_4},T^{T_{\alpha_3}}_{\gamma_{1}},\cdots, T_{\gamma_{s_2}},T_{\alpha}^{T_{\alpha_3}},T^{T_{\alpha_2}}_{\delta_{1}},\cdots, T_{\delta_{s_3}},T^{T_{\alpha_2}T_{\alpha_1}}_{\zeta_{1}},\cdots, T_{\zeta_{s_4}}),$$
$$Cyl(T_{\beta_{1}}^{T_{\alpha_1}T_{\alpha_4}},T_{\beta_{2}},\cdots, T_{\beta_{s_1}},T_{\alpha_4},T^{T_{\alpha'}T_{\alpha_3}}_{\gamma_{1}},\cdots, T_{\gamma_{s_2}},T_{\alpha_3},T^{T_{\alpha_2}}_{\delta_{1}},\cdots, T_{\delta_{s_3}},T^{T_{\alpha_2}T_{\alpha'}T_{\alpha_1}}_{\zeta_{1}},\cdots, T_{\zeta_{s_4}}),$$
and the relation $h_{\mu_\alpha\Delta,\Delta}(R9\Delta)$ is 
$$Cyl(T_{\beta_{1}}^{T_{\alpha'}T_{\alpha_1}T_{\alpha_4}},T_{\beta_{2}},\cdots, T_{\beta_{s_1}},T_{\alpha_4},T^{T_{\alpha'}T_{\alpha_3}}_{\gamma_{1}},\cdots, T_{\gamma_{s_2}},T_{\alpha_3},T^{T_{\alpha_2}}_{\delta_{1}},\cdots, T_{\delta_{s_3}},T^{T_{\alpha_2}T_{\alpha'}T_{\alpha_1}}_{\zeta_{1}},\cdots, T_{\zeta_{s_4}}).$$

As $Cyl(T_{\alpha_1},T_{\alpha_4},T_{\alpha'})$, $Co(T_{\beta_1},T_{\alpha_1})$, $Co(T_{\beta_1},T_{\alpha'})$ and $Br_3(T_{\alpha_4},T_{\beta_1})$ hold in $Br_{\mu_\alpha\Delta}$, we have $Co(T_{\alpha'},T_{\beta_1}^{T_{\alpha_1}T_{\alpha_4}})$ holds and thus $T_{\beta_{1}}^{T_{\alpha'}T_{\alpha_1}T_{\alpha_4}}=T_{\beta_{1}}^{T_{\alpha_1}T_{\alpha_4}}$. Therefore, $h_{\mu_\alpha\Delta,\Delta}(R9\Delta)$ holds.

(Case 2) There are some special loops in $\{\alpha_1, \alpha_2, \alpha_3,\alpha_4,\alpha\}$. We may assume that $\alpha_2$ is a special loop, as the other cases can be proved similarly. We may further assume that $\mu=\mu_{\alpha_4}\mu_{\alpha_3}\mu_{\alpha}\mu_{\alpha_2}\mu_{\alpha_1}$. Then $\mu'=\mu_{\alpha_4}\mu_{\alpha_3}\mu_{\alpha_2}\mu_{\alpha'}\mu_{\alpha_1}$ satisfies the requirement for the relation $R9$ for $\mu_\alpha\Delta$.

\begin{figure}[h]
    \centering
    \includegraphics{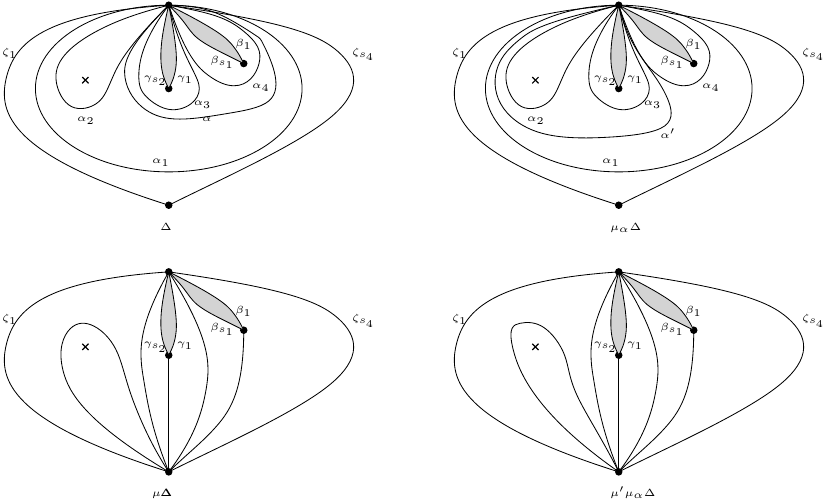}
\end{figure}

Thus, the relations $R9$ for $p$ in $Br_\Delta$ and $Br_{\mu_\alpha\Delta}$ under $\mu$ and $\mu'$, respectively, are 
$$R9(\Delta):Cyl(T_{\beta_{1}}^{T_{\alpha_1}T_{\alpha}T_{\alpha_4}},T_{\beta_{2}},\cdots, T_{\beta_{s_1}},T_{\alpha_4},T^{T_{\alpha_3}}_{\gamma_{1}},\cdots, T_{\gamma_{s_2}},T_\alpha^{T_{\alpha_3}},T^{T_{\alpha_2}T_{\alpha_1}}_{\zeta_{1}},\cdots, T_{\zeta_{s_4}}),$$
$$R9(\mu_\alpha\Delta):Cyl(T_{\beta_{1}}^{T_{\alpha_1}T_{\alpha_4}},T_{\beta_{2}},\cdots, T_{\beta_{s_1}},T_{\alpha_4},T^{T_{\alpha'}T_{\alpha_3}}_{\gamma_{1}},\cdots, T_{\gamma_{s_2}},T_{\alpha_3},T^{T_{\alpha_2}T_{\alpha'}T_{\alpha_1}}_{\zeta_{1}},\cdots, T_{\zeta_{s_4}}),$$
and the relation $h_{\mu_\alpha\Delta,\Delta}(R9\Delta)$ is 
$$Cyl(T_{\beta_{1}}^{T_{\alpha'}T_{\alpha_1}T_{\alpha_4}},T_{\beta_{2}},\cdots, T_{\beta_{s_1}},T_{\alpha_4},T^{T_{\alpha'}T_{\alpha_3}}_{\gamma_{1}},\cdots, T_{\gamma_{s_2}},T_{\alpha_3},T^{T_{\alpha_2}T_{\alpha'}T_{\alpha_1}}_{\zeta_{1}},\cdots, T_{\zeta_{s_4}}).$$

Similarly, as $Cyl(T_{\alpha_1},T_{\alpha_4},T_{\alpha'})$, $Co(T_{\beta_1},T_{\alpha_1})$, $Co(T_{\beta_1},T_{\alpha'})$ and $Br_3(T_{\alpha_4},T_{\beta_1})$ hold in $Br_{\mu_\alpha\Delta}$, we have $Co(T_{\alpha'},T_{\beta_1}^{T_{\alpha_1}T_{\alpha_4}})$ holds and thus $T_{\beta_{1}}^{T_{\alpha'}T_{\alpha_1}T_{\alpha_4}}=T_{\beta_{1}}^{T_{\alpha_1}T_{\alpha_4}}$. Therefore, $h_{\mu_\alpha\Delta,\Delta}(R9\Delta)$ holds.

The proof is complete.
\end{proof}

\subsubsection{Proof of Theorem \ref{th:brgroup}}\label{sec:proofofthmbrg}

Fix an (ordinary) triangulation $\Delta_0$ of $\Sigma$, we construct a groupoid $\hat\Gamma_{\Delta_0}$ as follows: The objects are the same as ${\bf TSurf}_\Sigma$. The morphisms are generated by $\hat h_{\Delta',\Delta_0}:\Delta_0\to \Delta', \Delta'\in {\bf TSurf}_\Sigma$ and $T^{\Delta_0}_\alpha:\Delta_0\to \Delta_0$, $\alpha$ running over all internal edges of $\Delta_0$ such that $\langle T^{\Delta_0}_\alpha\mid \alpha \text{ is an internal edge of }\Delta_0\rangle=Br_{\Delta_0}$.

For any non-self-folded internal arc $\alpha\in \Delta_0$, let 
$$\hat h_{\Delta_0,\mu_\alpha \Delta_0}=T^{\Delta_0}_\alpha \hat h^{-1}_{\mu_\alpha\Delta_0,\Delta_0}, \hspace{5mm} \hat h_{\Delta',\mu_\alpha\Delta_0}:=\hat h_{\Delta',\Delta_0}\hat h^{sgn_\alpha(C^{\Delta'}_{\Delta_0})}_{\Delta_0,\mu_\alpha\Delta_0}$$ and 
$$\begin{array}{rcl}
T^{\mu_\alpha\Delta_0}_{\beta}:&=&
\begin{cases}
  \hat h_{\mu_\alpha\Delta_0,\Delta_0} T^{\Delta_0}_{\alpha} \hat h^{-1}_{\mu_\alpha\Delta_0,\Delta_0}, & \text{ if } \beta\in \mu_\alpha\Delta_0\setminus \Delta_0,\\ 
   \hat h_{\mu_\alpha\Delta_0,\Delta_0}(T^{\Delta_0}_{\alpha})^{-1} T^{\Delta_0}_{\beta}T^{\Delta_0}_{\alpha} \hat h^{-1}_{\mu_\alpha\Delta_0,\Delta_0}, & \text{if there is an arrow from $\alpha$ to $\beta$ in $Q_{\Delta_0}$,} \\ 
   \hat h_{\mu_\alpha\Delta_0,\Delta_0} T^{\Delta_0}_{\beta} \hat h^{-1}_{\mu_\alpha\Delta_0,\Delta_0}, & \text{ otherwise,} 
\end{cases}\vspace{3pt}\\
&=&
\begin{cases}
  \hat h_{\mu_\alpha\Delta_0,\Delta_0} T^{\Delta_0}_{\alpha} \hat h^{-1}_{\mu_\alpha\Delta_0,\Delta_0}, & \text{ if } \beta\in \mu_\alpha\Delta_0\setminus \Delta_0,\\ 
   \hat h^{-1}_{\Delta_0,\mu_\alpha\Delta_0}T^{\Delta_0}_{\beta}\hat h_{\Delta_0,\mu_\alpha\Delta_0}, & \text{if there is an arrow from $\alpha$ to $\beta$ in $Q_{\Delta_0}$,} \\ 
   \hat h_{\mu_\alpha\Delta_0,\Delta_0} T^{\Delta_0}_{\beta} \hat h^{-1}_{\mu_\alpha\Delta_0,\Delta_0}, & \text{ otherwise,} 
\end{cases}
\end{array}
$$ 
where
$\hat h^{\varepsilon}_{\Delta_0,\mu_\alpha\Delta_0}=
\begin{cases}
    \hat h_{\Delta_0,\mu_\alpha\Delta_0}, & \text{if $\varepsilon=+$},\\
    \hat h^{-1}_{\mu_\alpha\Delta_0,\Delta_0}, & \text{if $\varepsilon=-$}.
\end{cases}$

Inductively, we can construct a morphism $\hat h_{\Delta',\Delta}:\Delta\to \Delta'$ for any $\Delta,\Delta'\in \hat\Gamma_{\Delta_0}$ and $T_\beta^\Delta: \Delta\to \Delta$ for any internal arc $\beta\in \Delta$ using a sequence of flips from $\Delta_0$ to $\Delta$.

\begin{proposition}\label{prop:welldef}
The morphisms $\hat h_{\Delta',\Delta}:\Delta\to \Delta'$
and $T_{\beta}^{\Delta}$ 
are well-defined for any $\Delta, \Delta'$, i.e., they do not depend on the flips $\mu$ from $\Delta_0$ to $\Delta$.
\end{proposition}

\begin{proof}
In case $\Delta=\Delta_0$ and $\mu=\mu_{\alpha'}\circ \mu_{\alpha}$ for $\alpha'\in \mu_\alpha(\Delta_0)\setminus \Delta_0$.

Following the mutation $\mu_\alpha$, we obtain $T_{\alpha'}^{\mu_\alpha(\Delta_0)}\hat h^{-1}_{\Delta_0,\mu_\alpha\Delta_0}:\Delta_0\to \mu_\alpha\Delta_0$. We have $$T_{\alpha'}^{\mu_\alpha(\Delta_0)}\hat h^{-1}_{\Delta_0,\mu_\alpha\Delta_0}=(\hat h_{\mu_\alpha\Delta_0,\Delta_0} T^{\Delta_0}_{\alpha} \hat h^{-1}_{\mu_\alpha\Delta_0,\Delta_0})\hat h^{-1}_{\Delta_0,\mu_\alpha\Delta_0}=\hat h_{\mu_\alpha\Delta_0,\Delta_0}.$$

For any $\Delta'$, we obtain the morphism $\hat h_{\Delta',\Delta_0}\hat h^{sgn_\alpha(C^{\Delta'}_{\Delta_0})}_{\Delta_0,\mu_\alpha\Delta_0}\hat h^{sgn_\alpha(C^{\Delta'}_{\mu_\alpha\Delta_0})}_{\mu_\alpha\Delta_0,\Delta_0}:\Delta_0\to \Delta$ following $\mu=\mu_{\alpha'}\circ \mu_{\alpha}$.

Since $sgn_\alpha(C_{\Delta_0}^{\Delta'})\neq sgn_\alpha(C_{\mu_\alpha\Delta_0}^{\Delta'})$, we have $\hat h^{sgn_\alpha(C^{\Delta'}_{\Delta_0})}_{\Delta_0,\mu_\alpha\Delta_0}\hat h^{sgn_\alpha(C^{\Delta'}_{\mu_\alpha\Delta_0})}_{\mu_\alpha\Delta_0,\Delta_0}=1$. Thus, $\hat h_{\Delta',\Delta_0}$ is stable under the flips $\mu_{\alpha'}\circ \mu_\alpha$.

For any internal arc $\beta$ of $\Delta_0$, under the sequence of flips $\mu_{\alpha'}\circ \mu_\alpha$, we obtain
$$
\begin{cases}
  \hat h_{\Delta_0,\mu_\alpha\Delta_0} T^{\mu_\alpha\Delta_0}_{\alpha'} \hat h^{-1}_{\Delta_0,\mu_\alpha\Delta_0}, & \text{ if } \beta=\alpha,\\ 
   \hat h_{\Delta_0,\mu_\alpha\Delta_0}(T^{\mu_\alpha\Delta_0}_{\alpha'})^{-1} T^{\mu_\alpha\Delta_0}_{\beta}T^{\mu_\alpha\Delta_0}_{\alpha'} \hat h^{-1}_{\Delta_0,\mu_\alpha\Delta_0}, & \text{if there is an arrow from $\alpha'$ to $\beta$ in $Q_{\mu_\alpha\Delta_0}$,} \\ 
   \hat h_{\Delta_0,\mu_\alpha\Delta_0} T^{\mu_\alpha\Delta_0}_{\beta} \hat h^{-1}_{\Delta_0,\mu_\alpha\Delta_0}, & \text{ otherwise.} 
\end{cases}$$ 
It is equal to $T^{\Delta_0}_{\beta}$ in all the cases. It implies that $T^{\Delta_0}_{\beta}$ is stable under the sequence of flips $\mu_{\alpha'}\circ \mu_\alpha$.

Therefore, the result is true in case $\Delta=\Delta_0$ and $\mu=\mu_{\alpha'}\circ \mu_{\alpha}$.

To prove the remaining cases, it suffices to prove the cases that $\mu$ is a simple cycle in the graph of flips. We have $\Delta=\Delta_0$ in these cases. Since the fundamental group of the graph of flips is generated by cycles of lengths $4$, $5$ and $6$, to complete the proof, we may assume that $\mu$ is a cycle of length $4$, $5$ or $6$. 

Assume that $\mu=\mu_{\alpha_{k-1}}\circ \cdots \circ \mu_{\alpha_1}\circ \mu_{\alpha_0}$ for $k=4,5$ or $6$. Then $\alpha_{k-1}=\alpha_1$. Denote $\Delta_i=\mu_{\alpha_i}\circ \cdots \circ \mu_{\alpha_2}\circ \mu_{\alpha_1}$ for all $i<k$.

Following the mutations $\mu$, we obtain the morphisms
\begin{equation*}
   \hat h_{\Delta',\Delta_0}\hat h^{sgn_{\alpha_0}(C^{\Delta'}_{\Delta_0})}_{\Delta_0,\Delta_1}\hat h^{sgn_{\alpha_1}(C^{\Delta'}_{\Delta_1})}_{\Delta_1,\Delta_2}\cdots \hat h^{sgn_{\alpha_{k-1}}(C^{\Delta'}_{\Delta_{k-1}})}_{\Delta_{k-1},\Delta_0},
\end{equation*}

\begin{equation*}
   (\hat h^{\varepsilon_{k-1}}_{\Delta_0,\Delta_{k-1}}\cdots \hat h^{\varepsilon_{1}}_{\Delta_2,\Delta_{1}} \hat h^{\varepsilon_{0}}_{\Delta_1,\Delta_{0}})T_\beta^{\Delta_0}(\hat h^{\varepsilon_{k-1}}_{\Delta_0,\Delta_{k-1}}\cdots \hat h^{\varepsilon_{1}}_{\Delta_2,\Delta_{1}} \hat h^{\varepsilon_{0}}_{\Delta_1,\Delta_{0}})^{-1}:\Delta_0\to \Delta_0,
\end{equation*}
where $\varepsilon_i=-$ only if there is an arrow from $\alpha_i$ to $\beta$ in $Q_{\Delta_i}$ for any $0\leq i\leq k-1$.

To show that $\hat h_{\Delta',\Delta_0}: \Delta_0\to \Delta'$ and $T^{\Delta_0}_\beta: \Delta_0\to \Delta'$
do not depend on the mutations $\mu$, we shall prove that  
\begin{equation}\label{eq:inv1}
   1=\hat h^{sgn_{\alpha_0}(C^{\Delta'}_{\Delta_0})}_{\Delta_0,\Delta_1}\hat h^{sgn_{\alpha_1}(C^{\Delta'}_{\Delta_1})}_{\Delta_1,\Delta_2}\cdots \hat h^{sgn_{\alpha_{k-1}}(C^{\Delta'}_{\Delta_{k-1}})}_{\Delta_{k-1},\Delta_0}.
\end{equation}
  
\begin{equation}\label{eq:inv2}
(\hat h^{\varepsilon_{k-1}}_{\Delta_0,\Delta_{k-1}}\cdots \hat h^{\varepsilon_{1}}_{\Delta_2,\Delta_{1}} \hat h^{\varepsilon_{0}}_{\Delta_1,\Delta_{0}})T_\beta^{\Delta_0}(\hat h^{\varepsilon_{k-1}}_{\Delta_0,\Delta_{k-1}}\cdots \hat h^{\varepsilon_{1}}_{\Delta_2,\Delta_{1}} \hat h^{\varepsilon_{0}}_{\Delta_1,\Delta_{0}})^{-1}=T_\beta^{\Delta_0}.
\end{equation}

{\bf Case 1}. $k=4$. Then $\alpha_2=\alpha_0$, $\alpha_3=\alpha_1$ and there is no arrow between $\alpha_0$ and $\alpha_1$ in $Q_{\Delta_0}$.

Following the sequence of mutations $\mu$, we have 
\begin{equation}\label{eq:a1}
\hat h_{\Delta_0,\Delta_1}=T_{\alpha_0}^{\Delta_0}\hat h^{-1}_{\Delta_1,\Delta_0},
\end{equation}
\begin{equation}\label{eq:a2}
\hat h_{\Delta_2,\Delta_1}=\hat h_{\Delta_2,\Delta_0}\hat h^{-1}_{\Delta_1,\Delta_0}, \hspace{5mm} \hat h_{\Delta_1,\Delta_2}=T_{\alpha_1}^{\Delta_1}\hat h^{-1}_{\Delta_2,\Delta_1}=\hat h_{\Delta_1,\Delta_0}T_{\alpha_1}^{\Delta_0}
\hat h^{-1}_{\Delta_2,\Delta_0},
\end{equation}
\begin{equation}\label{eq:a3}\hat h_{\Delta_3,\Delta_2}=\hat h_{\Delta_3,\Delta_1}\hat h^{-1}_{\Delta_2,\Delta_1}=\hat h_{\Delta_3,\Delta_0}\hat h_{\Delta_0,\Delta_1}\hat h^{-1}_{\Delta_2,\Delta_1}=\hat h_{\Delta_3,\Delta_0}T_{\alpha_0}^{\Delta_0}\hat h^{-1}_{\Delta_2,\Delta_0},
\end{equation}
\begin{equation}\label{eq:a4}
\begin{array}{rcl}
\hat h_{\Delta_2,\Delta_3}&=&T_{\alpha_2}^{\Delta_2}\hat h^{-1}_{\Delta_3,\Delta_2}=\hat h_{\Delta_2,\Delta_1}T_{\alpha_0}^{\Delta_1}\hat h^{-1}_{\Delta_2,\Delta_1}\hat h^{-1}_{\Delta_3,\Delta_2}\\
&=&\hat h_{\Delta_2,\Delta_1}\hat h_{\Delta_1,\Delta_0}T_{\alpha_0}^{\Delta_0}\hat h^{-1}_{\Delta_1,\Delta_0}\hat h^{-1}_{\Delta_2,\Delta_1}\hat h^{-1}_{\Delta_3,\Delta_2}\\
&=& \hat h_{\Delta_2,\Delta_0}\hat h^{-1}_{\Delta_3,\Delta_0},
\end{array}
\end{equation}
\begin{equation}\label{eq:a5}\hat h_{\Delta_0,\Delta_3}=\hat h_{\Delta_0,\Delta_2}\hat h^{-1}_{\Delta_3,\Delta_2}=\hat h_{\Delta_0,\Delta_1}\hat h_{\Delta_1,\Delta_2}\hat h^{-1}_{\Delta_3,\Delta_2}=T_{\alpha_0}^{\Delta_0}T_{\alpha_1}^{\Delta_0}(T_{\alpha_0}^{\Delta_0})^{-1}\hat h^{-1}_{\Delta_3,\Delta_0}.
\end{equation}
\begin{equation}\label{eq:a6}
\begin{array}{rcl}
\hat h_{\Delta_3,\Delta_0}&=&T_{\alpha_3}^{\Delta_3}\hat h^{-1}_{\Delta_0,\Delta_3}\\
&=&\hat h_{\Delta_3,\Delta_2}\hat h_{\Delta_2,\Delta_1}\hat h_{\Delta_1,\Delta_0}T_{\alpha_1}^{\Delta_0}(\hat h_{\Delta_3,\Delta_2}\hat h_{\Delta_2,\Delta_1}\hat h_{\Delta_1,\Delta_0})^{-1}\hat h^{-1}_{\Delta_0,\Delta_3}\\
&=& \hat h_{\Delta_3,\Delta_0}T_{\alpha_0}^{\Delta_0}T_{\alpha_1}^{\Delta_0}(T_{\alpha_0}^{\Delta_0})^{-1}\hat h^{-1}_{\Delta_3,\Delta_0}\hat h^{-1}_{\Delta_0,\Delta_3}\\
&=& \hat h_{\Delta_3,\Delta_0}.
\end{array}.
\end{equation}
This in particular implies that $\hat h_{\Delta_3,\Delta_0}$ is stable under the sequence of flips $\mu$.

As $sgn_{\alpha_1}(C^{\Delta'}_{\Delta_1})=-sgn_{\alpha_3}(C^{\Delta'}_{\Delta_3})$ and $sgn_{\alpha_2}(C^{\Delta'}_{\Delta_2})=-sgn_{\alpha_0}(C^{\Delta'}_{\Delta_0})$, by (\ref{eq:a1}) (\ref{eq:a2}) (\ref{eq:a3}) (\ref{eq:a4}),  (\ref{eq:a5}), and the fact that $T_{\alpha_0}^{\Delta_0}T_{\alpha_1}^{\Delta_0}=T_{\alpha_1}^{\Delta_0}T_{\alpha_0}^{\Delta_0}$, we have
$$\hat h^{sgn_{\alpha_0}(C^{\Delta'}_{\Delta_0})}_{\Delta_0,\Delta_1}\hat h^{sgn_{\alpha_1}(C^{\Delta'}_{\Delta_1})}_{\Delta_1,\Delta_2}\hat h^{sgn_{\alpha_2}(C^{\Delta'}_{\Delta_2})}_{\Delta_2,\Delta_3}\hat h^{sgn_{\alpha_3}(C^{\Delta'}_{\Delta_3})}_{\Delta_3,\Delta_0}=1.$$ Thus, (\ref{eq:inv1}) holds. It should be noted that we need the condition that $T_{\alpha_0}^{\Delta_0}T_{\alpha_1}^{\Delta_0}=T_{\alpha_1}^{\Delta_0}T_{\alpha_0}^{\Delta_0}$ only in the case $sgn_{\alpha_0}(C^{\Delta'}_{\Delta_1})=-sgn_{\alpha_1}(C^{\Delta'}_{\Delta_2})=-$.

If there are no arrows between $\beta$ and $\alpha_0$, and no arrows between $\beta$ and $\alpha_1$ in $Q_{\Delta_0}$, then $\varepsilon_i=1$ for all $0\leq i\leq 3$, and $T^{\Delta_0}_\beta$ commutes with $T^{\Delta_0}_{\alpha_0}$ and $T^{\Delta_0}_{\alpha_1}$. By (\ref{eq:a1}) (\ref{eq:a2}) (\ref{eq:a3}) (\ref{eq:a4}) and (\ref{eq:a5}), we have
$$\hat h_{\Delta_0,\Delta_3}\hat h_{\Delta_3,\Delta_2}\hat h_{\Delta_2,\Delta_1}\hat h_{\Delta_1,\Delta_0}=T_{\alpha_1}^{\Delta_0}T_{\alpha_0}^{\Delta_0}.$$
Thus, (\ref{eq:inv2}) holds.

If there is an arrow between $\beta$ and $\alpha_0$, but there are no arrows between $\beta$ and $\alpha_1$ in $Q_{\Delta_0}$, then $\varepsilon_i=1$ for $i=1,3$, and $T^{\Delta_0}_\beta$ commutes with $T^{\Delta_0}_{\alpha_1}$. We may assume that there is an arrow from  $\alpha_0$ to $\beta$ in $Q_{\Delta_0}$ since the other case can be proved similarly. Then $\varepsilon_0=-,\varepsilon_2=+$. By (\ref{eq:a1}) (\ref{eq:a2}) (\ref{eq:a3}) (\ref{eq:a4}) and (\ref{eq:a5}), we have
$$\hat h_{\Delta_0,\Delta_3}\hat h_{\Delta_3,\Delta_2}\hat h_{\Delta_2,\Delta_1}\hat h_{\Delta_0,\Delta_1}^{-1}=T_{\alpha_0}^{\Delta_0}T_{\alpha_1}^{\Delta_0}(T_{\alpha_0}^{\Delta_0})^{-1}=T_{\alpha_1}^{\Delta_0}.$$
Thus, (\ref{eq:inv2}) holds.

If there is an arrow between $\beta$ and $\alpha_0$, and an arrow between $\beta$ and $\alpha_1$ in $Q_{\Delta_0}$, we may assume that there are arrows from $\alpha_0$ and $\alpha_1$ to $\beta$ in $Q_{\Delta_0}$ as the other case can be proved similarly. Then $\varepsilon_0=\varepsilon_1=-,\varepsilon_2=\varepsilon_3=+$. By (\ref{eq:a1}) (\ref{eq:a2}) (\ref{eq:a3}) (\ref{eq:a4}) and (\ref{eq:a5}), we have
$$\hat h_{\Delta_0,\Delta_3}\hat h_{\Delta_3,\Delta_2}\hat h^{-1}_{\Delta_1,\Delta_2}\hat h^{-1}_{\Delta_0,\Delta_1}=1.$$
Thus, (\ref{eq:inv2}) holds.

{\bf Case 2}. $k=5$. In this case there is an arrow between $\alpha_0$ and $\alpha_1$ in $Q_{\Delta_0}$, $\alpha_i\in \Delta_{i-1}\setminus \Delta_{i-2}$ for $2\leq i\leq 4$, and $w(\alpha_0)=w(\alpha_1)=1$. We may assume that there is an arrow from $\alpha_1$ to $\alpha_0$ in $Q_{\Delta_0}$, since otherwise we can consider the mutation sequence $\mu^-=\mu_{\alpha_0}\circ \mu_{\alpha_1}\circ \cdots \circ\mu_{\alpha_4}$ instead.

Following the sequence of mutations $\mu$, using the braid relation $T_{\alpha_0}^{\Delta_0}T_{\alpha_1}^{\Delta_0}T_{\alpha_0}^{\Delta_0}=T_{\alpha_1}^{\Delta_0}T_{\alpha_0}^{\Delta_0}T_{\alpha_1}^{\Delta_0}$ and by calculation, we have
\begin{equation}\label{eq:b1}
       \begin{aligned}
     \hat h_{\Delta_0,\Delta_1} &=  T_{\alpha_0}^{\Delta_0} \hat h^{-1}_{\Delta_1,\Delta_0},\\
     \hat h_{\Delta_2,\Delta_1} &=  \hat h_{\Delta_2,\Delta_0} \hat h^{-1}_{\Delta_1,\Delta_0},\\
     \hat h_{\Delta_1,\Delta_2} &=  \hat h_{\Delta_1,\Delta_0} T_{\alpha_1}^{\Delta_0} \hat h^{-1}_{\Delta_2,\Delta_0},\\
     \hat h_{\Delta_3,\Delta_2} &=  \hat h_{\Delta_3,\Delta_0}T_{\alpha_0}^{\Delta_0} \hat h^{-1}_{\Delta_2,\Delta_0},\\
     \hat h_{\Delta_2,\Delta_3} &=  \hat h_{\Delta_2,\Delta_0} \hat h^{-1}_{\Delta_3,\Delta_0},\\
     \hat h_{\Delta_4,\Delta_3} &=  \hat h_{\Delta_4,\Delta_0} T_{\alpha_0}^{\Delta_0}T_{\alpha_1}^{\Delta_0}(T_{\alpha_0}^{\Delta_0})^{-1} \hat h^{-1}_{\Delta_3,\Delta_0},\\
     \hat h_{\Delta_3,\Delta_4} &=  \hat h_{\Delta_3,\Delta_0} \hat h^{-1}_{\Delta_4,\Delta_0},\\
     \hat h_{\Delta_0,\Delta_4} &=  T_{\alpha_1}^{\Delta_0} \hat h^{-1}_{\Delta_4,\Delta_0}.
    \end{aligned}
\end{equation}

By Lemma \ref{lem:dual},
$(sgn_{\alpha_0}(C^{\Delta'}_{\Delta_0}),sgn_{\alpha_1}(C^{\Delta'}_{\Delta_1}),sgn_{\alpha_2}(C^{\Delta'}_{\Delta_2}),sgn_{\alpha_3}(C^{\Delta'}_{\Delta_3}),sgn_{\alpha_4}(C^{\Delta'}_{\Delta_4}))$ has the following possibilities: $(+,+,+,-,-)$, $(-,+,+,+,-)$, $(-,-,+,+,+)$, $(+,-,-,+,+)$, and $(+,+,-,-,+)$. By (\ref{eq:b1}) and the fact that $T_{\alpha_0}^{\Delta_0}T_{\alpha_1}^{\Delta_0}T_{\alpha_0}^{\Delta_0}=T_{\alpha_1}^{\Delta_0}T_{\alpha_0}^{\Delta_0}T_{\alpha_1}^{\Delta_0}$, we have
$$\hat h^{sgn_{\alpha_0}(C^{\Delta'}_{\Delta_0})}_{\Delta_0,\Delta_1}\hat h^{sgn_{\alpha_1}(C^{\Delta'}_{\Delta_1})}_{\Delta_1,\Delta_2}\hat h^{sgn_{\alpha_2}(C^{\Delta'}_{\Delta_2})}_{\Delta_2,\Delta_3}\hat h^{sgn_{\alpha_3}(C^{\Delta'}_{\Delta_3})}_{\Delta_3,\Delta_4}\hat h^{sgn_{\alpha_4}(C^{\Delta'}_{\Delta_4})}_{\Delta_4,\Delta_0}=1.$$

Thus, (\ref{eq:inv1}) holds.

If there are no arrows between $\beta$ and $\alpha_0$, and no arrows between $\beta$ and $\alpha_1$ in $Q_{\Delta_0}$, then $\varepsilon_i=1$ for all $0\leq i\leq 4$, and $T^{\Delta_0}_\beta$ commutes with $T^{\Delta_0}_{\alpha_0}$ and $T^{\Delta_0}_{\alpha_1}$. By (\ref{eq:b1}), we have
$$\hat h_{\Delta_0,\Delta_4}\hat h_{\Delta_4,\Delta_3}\hat h_{\Delta_3,\Delta_2}\hat h_{\Delta_2,\Delta_1}\hat h_{\Delta_1,\Delta_0}=T_{\alpha_1}^{\Delta_0}T_{\alpha_0}^{\Delta_0}T_{\alpha_1}^{\Delta_0}.$$
Thus, (\ref{eq:inv2}) holds.

If there is an arrow between $\alpha_0$ and $\beta$, but there are no arrows between $\beta$ and $\alpha_1$ in $Q_{\Delta_0}$, then $T_{\beta}^{\Delta_0}$ commutes with $T_{\alpha_1}^{\Delta_0}$.
We may assume that there is an arrow from $\alpha_0$ to $\beta$ since the other case can be proved similarly, then 
$\varepsilon_0=\varepsilon_1=-,\varepsilon_2=\varepsilon_3=\varepsilon_4=+$. By (\ref{eq:b1}), we have
$$\hat h_{\Delta_0,\Delta_4}\hat h_{\Delta_4,\Delta_3}\hat h_{\Delta_3,\Delta_2}\hat h^{-1}_{\Delta_1,\Delta_2}\hat h^{-1}_{\Delta_0,\Delta_1}=T_{\alpha_1}^{\Delta_0}.$$
Thus, (\ref{eq:inv2}) holds.

If there is an arrow between $\alpha_1$ and $\beta$, but there are no arrows between $\beta$ and $\alpha_0$ in $Q_{\Delta_0}$, then $T_{\beta}^{\Delta_0}$ commutes with $T_{\alpha_0}^{\Delta_0}$.
We may assume that there is an arrow from $\alpha_1$ to $\beta$ since the other case can be proved similarly, then 
$\varepsilon_1=\varepsilon_2=-,\varepsilon_0=\varepsilon_3=\varepsilon_4=+$. By (\ref{eq:b1}), we have
$$\hat h_{\Delta_0,\Delta_4}\hat h_{\Delta_4,\Delta_3}\hat h^{-1}_{\Delta_2,\Delta_3}\hat h^{-1}_{\Delta_1,\Delta_2}\hat h_{\Delta_1,\Delta_0}=T_{\alpha_0}^{\Delta_0}.$$
Thus, (\ref{eq:inv2}) holds.

If there is an arrow between $\beta$ and $\alpha_0$, and an arrow between $\beta$ and $\alpha_1$ in $Q_{\Delta_0}$, we may assume that there are arrows from $\alpha_0$ and $\alpha_1$ to $\beta$ in $Q_{\Delta_0}$ as the other case can be proved similarly. Then $\varepsilon_0=\varepsilon_1=\varepsilon_2=-,\varepsilon_3=\varepsilon_4=+$. By (\ref{eq:b1}), we have
$$\hat h_{\Delta_0,\Delta_4}\hat h_{\Delta_4,\Delta_3}\hat h^{-1}_{\Delta_2,\Delta_3}\hat h^{-1}_{\Delta_1,\Delta_2}\hat h^{-1}_{\Delta_0,\Delta_1}=1.$$
Thus, (\ref{eq:inv2}) holds.

{\bf Case 3}. $k=6$. Then there is an arrow between $\alpha_0$ and $\alpha_1$ in $Q_{\Delta_0}$, $\alpha_i\in \Delta_{i-1}\setminus \Delta_{i-2}$ for $2\leq i\leq 5$, and $w(\alpha_0)\neq w(\alpha_1)=1$ or $w(\alpha_1)\neq w(\alpha_0)=1$. We may assume that there is an arrow from $\alpha_1$ to $\alpha_0$ in $Q_{\Delta_0}$, since otherwise we can consider the mutation sequence $\mu^-=\mu_{\alpha_0}\circ \mu_{\alpha_1}\circ \cdots \circ\mu_{\alpha_5}$ instead. 

Following the sequence of mutations $\mu$, using the braid relation $T_{\alpha_0}^{\Delta_0}T_{\alpha_1}^{\Delta_0}T_{\alpha_0}^{\Delta_0}T_{\alpha_1}^{\Delta_0}=T_{\alpha_1}^{\Delta_0}T_{\alpha_0}^{\Delta_0}T_{\alpha_1}^{\Delta_0}T_{\alpha_0}^{\Delta_0}$ and by calculation, we have
\begin{equation}\label{eq:b11}
       \begin{aligned}
     \hat h_{\Delta_0,\Delta_1} &=  T_{\alpha_0}^{\Delta_0} \hat h^{-1}_{\Delta_1,\Delta_0},\\
     \hat h_{\Delta_2,\Delta_1} &=  \hat h_{\Delta_2,\Delta_0} \hat h^{-1}_{\Delta_1,\Delta_0},\\
     \hat h_{\Delta_1,\Delta_2} &=  \hat h_{\Delta_1,\Delta_0} T_{\alpha_1}^{\Delta_0} \hat h^{-1}_{\Delta_2,\Delta_0},\\
     \hat h_{\Delta_3,\Delta_2} &=  \hat h_{\Delta_3,\Delta_0}T_{\alpha_0}^{\Delta_0} \hat h^{-1}_{\Delta_2,\Delta_0},\\
     \hat h_{\Delta_2,\Delta_3} &=  \hat h_{\Delta_2,\Delta_0} \hat h^{-1}_{\Delta_3,\Delta_0},\\
     \hat h_{\Delta_4,\Delta_3} &=  \hat h_{\Delta_4,\Delta_0} T_{\alpha_0}^{\Delta_0}T_{\alpha_1}^{\Delta_0}(T_{\alpha_0}^{\Delta_0})^{-1} \hat h^{-1}_{\Delta_3,\Delta_0},\\
     \hat h_{\Delta_3,\Delta_4} &=  \hat h_{\Delta_3,\Delta_0} \hat h^{-1}_{\Delta_4,\Delta_0},\\
     \hat h_{\Delta_5,\Delta_4} &=  \hat h_{\Delta_5,\Delta_0} T_{\alpha_0}^{\Delta_0}T_{\alpha_1}^{\Delta_0}T_{\alpha_0}^{\Delta_0}(T_{\alpha_1}^{\Delta_0})^{-1}(T_{\alpha_0}^{\Delta_0})^{-1}\hat h^{-1}_{\Delta_4,\Delta_0},\\
     \hat h_{\Delta_4,\Delta_5} &=  \hat h_{\Delta_4,\Delta_0} \hat h^{-1}_{\Delta_5,\Delta_0},\\
      \hat h_{\Delta_0,\Delta_5} &= T_{\alpha_1}^{\Delta_0} \hat h^{-1}_{\Delta_5,\Delta_0},\\
    \end{aligned}
\end{equation}

$(sgn_{\alpha_0}(C^{\Delta'}_{\Delta_0}),sgn_{\alpha_1}(C^{\Delta'}_{\Delta_1}),sgn_{\alpha_2}(C^{\Delta'}_{\Delta_2}),sgn_{\alpha_3}(C^{\Delta'}_{\Delta_3}),sgn_{\alpha_4}(C^{\Delta'}_{\Delta_4}),sgn_{\alpha_5}(C^{\Delta'}_{\Delta_5}))$ has the following possibilities: $(+,+,+,+,-,-)$, $(-,+,+,+,+,-)$, $(-,-,+,+,+,+)$, $(+,-,-,+,+,+)$, $(+,+,-,-,+,+)$ and $(+,+,+,-,-,+)$ by Lemma \ref{lem:dual}. By (\ref{eq:b11}) and the braid relation $T_{\alpha_0}^{\Delta_0}T_{\alpha_1}^{\Delta_0}T_{\alpha_0}^{\Delta_0}T_{\alpha_1}^{\Delta_0}=T_{\alpha_1}^{\Delta_0}T_{\alpha_0}^{\Delta_0}T_{\alpha_1}^{\Delta_0}T_{\alpha_0}^{\Delta_0}$, we have
$$\hat h^{sgn_{\alpha_0}(C^{\Delta'}_{\Delta_0})}_{\Delta_0,\Delta_1}\hat h^{sgn_{\alpha_1}(C^{\Delta'}_{\Delta_1})}_{\Delta_1,\Delta_2}\hat h^{sgn_{\alpha_2}(C^{\Delta'}_{\Delta_2})}_{\Delta_2,\Delta_3}\hat h^{sgn_{\alpha_3}(C^{\Delta'}_{\Delta_3})}_{\Delta_3,\Delta_4}\hat h^{sgn_{\alpha_4}(C^{\Delta'}_{\Delta_4})}_{\Delta_4,\Delta_5}\hat h^{sgn_{\alpha_4}(C^{\Delta'}_{\Delta_0})}_{\Delta_5,\Delta_0}=1.$$

Thus, (\ref{eq:inv1}) holds.

If there are no arrows between $\beta$ and $\alpha_0$, and no arrows between $\beta$ and $\alpha_1$ in $Q_{\Delta_0}$, then $\varepsilon_i=1$ for all $0\leq i\leq 5$, and $T^{\Delta_0}_\beta$ commutes with $T^{\Delta_0}_{\alpha_0}$ and $T^{\Delta_0}_{\alpha_1}$. By (\ref{eq:b11}), we have
$$\hat h_{\Delta_0,\Delta_5}\hat h_{\Delta_5,\Delta_4}\hat h_{\Delta_4,\Delta_3}\hat h_{\Delta_3,\Delta_2}\hat h_{\Delta_2,\Delta_1}\hat h_{\Delta_1,\Delta_0}=T_{\alpha_1}^{\Delta_0}T_{\alpha_0}^{\Delta_0}T_{\alpha_1}^{\Delta_0}T_{\alpha_0}^{\Delta_0}.$$
Thus, (\ref{eq:inv2}) holds.

If there is an arrow between $\alpha_0$ and $\beta$, but there are no arrows between $\beta$ and $\alpha_1$ in $Q_{\Delta_0}$, then $T_{\beta}^{\Delta_0}$ commutes with $T_{\alpha_1}^{\Delta_0}$.
We may assume that there is an arrow from $\alpha_0$ to $\beta$ since the other case can be proved similarly, then 
$\varepsilon_0=\varepsilon_1=\varepsilon_2=-,\varepsilon_3=\varepsilon_4=\varepsilon_5=+$. By (\ref{eq:b11}), we have
$$\hat h_{\Delta_0,\Delta_5}\hat h_{\Delta_5,\Delta_4}\hat h_{\Delta_4,\Delta_3}\hat h^{-1}_{\Delta_2,\Delta_3}\hat h^{-1}_{\Delta_1,\Delta_2}\hat h^{-1}_{\Delta_0,\Delta_1}=T_{\alpha_1}^{\Delta_0}.$$
Thus, (\ref{eq:inv2}) holds.

If there is an arrow between $\alpha_1$ and $\beta$, but there are no arrows between $\beta$ and $\alpha_0$ in $Q_{\Delta_0}$, then $T_{\beta}^{\Delta_0}$ commutes with $T_{\alpha_0}^{\Delta_0}$.
We may assume that there is an arrow from $\alpha_1$ to $\beta$ since the other case can be proved similarly, then 
$\varepsilon_1=\varepsilon_2=\varepsilon_3=-,\varepsilon_0=\varepsilon_4=\varepsilon_5=+$. By (\ref{eq:b11}), we have
$$\hat h_{\Delta_0,\Delta_5}\hat h_{\Delta_5,\Delta_4}\hat h^{-1}_{\Delta_3,\Delta_4}\hat h^{-1}_{\Delta_2,\Delta_3}\hat h^{-1}_{\Delta_1,\Delta_2}\hat h_{\Delta_1,\Delta_0}=T_{\alpha_0}^{\Delta_0}.$$
Thus, (\ref{eq:inv2}) holds.

If there is an arrow between $\beta$ and $\alpha_0$, and an arrow between $\beta$ and $\alpha_1$ in $Q_{\Delta_0}$, we may assume that there are arrows from $\alpha_0$ and $\alpha_1$ to $\beta$ in $Q_{\Delta_0}$ as the other case can be proved similarly. Then $\varepsilon_0=\varepsilon_1=\varepsilon_2=\varepsilon_3=-,\varepsilon_4=\varepsilon_5=+$. By (\ref{eq:b11}), we have
$$\hat h_{\Delta_0,\Delta_5}\hat h_{\Delta_5,\Delta_4}\hat h^{-1}_{\Delta_3,\Delta_4}\hat h^{-1}_{\Delta_2,\Delta_3}\hat h^{-1}_{\Delta_1,\Delta_2}\hat h^{-1}_{\Delta_0,\Delta_1}=1.$$
Thus, (\ref{eq:inv2}) holds.

The proof is complete.
\end{proof}

As one can see from the proof of Proposition \ref{prop:welldef}, we have the following.

\begin{corollary}\label{cor:relationa}
For $k\in \{4,5,6\}$ and distinct triangulations $\Delta_i, i=1,\ldots,k$ of $\Sigma$ such that $dist(\Delta_i,\Delta_{i+1\mod k})=1$ for $i=1,\ldots,k$ with $\Delta_2=\mu_{\alpha}(\Delta_1)$ and $\Delta_3=\mu_{\beta}(\Delta_2)$, we have
\begin{equation*}
\hat h_{\Delta_3,\Delta_2}\hat h_{\Delta_2,\Delta_1}=\hat h_{\Delta_3,\Delta_4}\cdots \hat h_{\Delta_{k-1},\Delta_k} \hat h_{\Delta_k,\Delta_1}
\end{equation*}
whenever $(\beta,\alpha)$ is not directed clockwise in $\Delta_1$.   
\end{corollary}

\begin{lemma}\label{lem:relationb} 
$(a)$ For any triangulation $\Delta$, for any non-self-folded arcs $\alpha,\beta\in \Delta$ such that $\alpha$ is non-self-folded in $\mu_\beta\Delta$, if $(\beta,\alpha)$ is not directed clockwise in $\Delta$, then we have
$$\hat h_{\mu_\beta\Delta,\Delta}\hat h_{\Delta,\mu_\alpha\Delta}\hat h_{\mu_\alpha\Delta,\Delta}=\hat h_{\mu_\beta\Delta,\mu_\alpha\mu_\beta\Delta}\hat h_{\mu_\alpha\mu_\beta\Delta,\mu_\beta\Delta}\hat h_{\mu_\beta\Delta,\Delta}.$$  

$(b)$ $\hat h_{\Delta,\mu_\alpha\Delta}\hat h_{\mu_\alpha\Delta,\Delta}\hat h_{\Delta,\mu_\beta\Delta}\hat h_{\mu_\beta\Delta,\Delta}=\hat h_{\Delta,\mu_\beta\Delta}\hat h_{\mu_\beta\Delta,\Delta}\hat h_{\Delta,\mu_\alpha\Delta}\hat h_{\mu_\alpha\Delta,\Delta}$
for any once punctured bigon $(\alpha_1,\alpha_2)$ in $\Delta$ such that $\alpha,\beta\in \Delta$ are the two diagonals connecting the puncture with $\beta\neq \alpha, \overline\alpha$.
\end{lemma}

\begin{proof}
As $(\beta,\alpha)$ is not directed clockwise in $\Delta$, there is no arrow from $\beta$ to $\alpha$ in $Q_\Delta$. Thus, $T_{\alpha}^{\mu_\beta\Delta}=\hat h_{\mu_\beta\Delta,\Delta}T_{\alpha}^{\Delta}\hat h^{-1}_{\mu_\beta\Delta,\Delta}$. Then $(a)$ follows by $\hat h_{\Delta,\mu_\alpha\Delta}\hat h_{\mu_\alpha\Delta,\Delta}=T^{\Delta}_{\alpha}, \hat h_{\mu_\beta\Delta,\mu_\alpha\mu_\beta\Delta}\hat h_{\mu_\alpha\mu_\beta\Delta,\mu_\beta\Delta}=T_{\alpha}^{\mu_\beta\Delta}$.

$(b)$ follows from the relation $T_\alpha^\Delta T_\beta^\Delta=T_\beta^\Delta T_\alpha^\Delta$.
\end{proof}

{\bf Proof of Theorem \ref{th:brgroup}.} On one hand, from Corollary \ref{cor:relationa}, Lemma \ref{lem:relationb} { and Theorem \ref{th:presentation of Tsurf}}, we see that $\hat\Gamma_{\Delta_0}$ is a quotient groupoid of ${\bf Tsurf}_\Sigma$ under $h_{\Delta',\Delta}\mapsto \hat h_{\Delta',\Delta}$. On the other hand, from the construction of $\hat \Gamma_{\Delta_0}$ at the begining of Section \ref{sec:proofofthmbrg}, we have $Aut_{{\hat \Gamma}_{\Delta_0}}(\Delta_0)=Br_{\Delta_0}$. As $Aut_{{\bf TSurf}_\Sigma}(\Delta_0)$ is a quotient of $Br_{\Delta_0}=Aut_{{\hat \Gamma}_{\Delta_0}}(\Delta_0)$, we have ${\bf Tsurf}_\Sigma$ is a quotient groupoid of $\Gamma_{\Delta_0}$ under $\hat h_{\Delta',\Delta_0}\mapsto h_{\Delta',\Delta_0}$ and $T_\alpha^{\Delta_0}\mapsto T_\alpha$. Therefore, we have $\Gamma_{\Delta_0}\cong {\bf Tsurf}_\Sigma$ under ${\bf Tsurf}_\Sigma$ under $h_{\Delta',\Delta}\mapsto \hat h_{\Delta',\Delta}$. It follows that $Aut_{{\bf TSurf}_\Sigma}(\Delta_0)\cong Br_{\Delta_0}$. The proof is complete.\endproof


\subsection{Proof of Theorem \ref{th:monomial mutation surfaces}}
\label{Proof of Theorem th:monomial mutation surfaces}

\begin{proof}
Theorem \ref{th:monomial mutation surfaces} follows by Theorem \ref{th:presentation of Tsurf}, and the following Lemmas \ref{lem1}, \ref{lem2}, \ref{lem21}, and \ref{lem3}.
\end{proof}

\begin{lemma}\label{lem1}
    For any ordinary triangulation $\Delta_1$ of $\Sigma$ with non-self-folded non-pending arcs $\alpha,\beta\in \Delta_1$ such that $\alpha$ and $\beta$ are not two sides in any triangle of $\Delta_1$, let $\Delta_2=\mu_{\alpha}(\Delta_1)$, $\Delta_3=\mu_{\beta}(\Delta_2)$ and $\Delta_4=\mu_{\alpha}(\Delta_3)$. Then 
$\mu_{\Delta_3,\Delta_2}\mu_{\Delta_2,\Delta_1}=\mu_{\Delta_3,\Delta_4}\mu_{\Delta_4,\Delta_1}.$ 
\end{lemma}

\begin{proof}
The result is immediate {by $\mu_\alpha$ and $\mu_\beta$ are commutes}, as $\alpha$ and $\beta$ are not two sides in any triangle of $\Delta_1$.
\end{proof}

\begin{lemma}\label{lem2}
     For the pentagon $\Sigma_5$, denote $\Delta_1=\{(1,3),(3,1),(1,4),(4,1)\}\cup \{\text{boundary arcs}\}$ and $\Delta_2=\mu_{(1,3)}(\Delta_1),\Delta_3=\mu_{(1,4)}(\Delta_2), \Delta_5=\mu_{(1,4)}(\Delta_1),\Delta_4=\mu_{(1,3)}(\Delta_5)$. Then we have

$(a)$ $\mu_{\Delta_3,\Delta_2}\mu_{\Delta_2,\Delta_1}=\mu_{\Delta_3,\Delta_4}\mu_{\Delta_4,\Delta_5}\mu_{\Delta_5,\Delta_1}.$

$(b)$ $\mu_{\Delta_2,\Delta_3}\mu_{\Delta_3,\Delta_2}\mu_{\Delta_2,\Delta_1}=\mu_{\Delta_2,\Delta_1}\mu_{\Delta_1,\Delta_5}\mu_{\Delta_5,\Delta_1}.$
    \end{lemma}
    
\begin{proof}
    By direct calculation, we have $\mu_{\Delta_3,\Delta_2}\mu_{\Delta_2,\Delta_1}(t_{13})=\mu_{\Delta_3,\Delta_2}(t_{12}t_{42}^{-1}t_{43})=t_{12}t_{42}^{-1}t_{43}$,
   
   $$\mu_{\Delta_3,\Delta_2}\mu_{\Delta_2,\Delta_1}(t_{14})=\mu_{\Delta_3,\Delta_2}(t_{14})=t_{12}t_{52}^{-1}t_{54},$$

\begin{equation*}
\begin{array}{rcl}
\mu_{\Delta_3,\Delta_4}\mu_{\Delta_4,\Delta_5}\mu_{\Delta_5,\Delta_1}(t_{13})&= &
\mu_{\Delta_3,\Delta_4}\mu_{\Delta_4,\Delta_5}(t_{13})\\
&=&
\mu_{\Delta_3,\Delta_4}(t_{12}t_{52}^{-1}t_{53})\\
&=& t_{12}t_{52}^{-1}t_{52}t_{42}^{-1}t_{43}=t_{12}t_{42}^{-1}t_{43},
\end{array}
\end{equation*}

\begin{equation*}
\begin{array}{rcl}
\mu_{\Delta_3,\Delta_4}\mu_{\Delta_4,\Delta_5}\mu_{\Delta_5,\Delta_1}(t_{14})&= &
\mu_{\Delta_3,\Delta_4}\mu_{\Delta_4,\Delta_5}(t_{13}t^{-1}_{53}t_{54})\\
&=&
\mu_{\Delta_3,\Delta_4}(t_{12}t_{52}^{-1}t_{53}t^{-1}_{53}t_{54})\\
&=& t_{12}t_{52}^{-1}t_{54},
\end{array}
\end{equation*}

Thus, we have $\mu_{\Delta_3,\Delta_2}\mu_{\Delta_2,\Delta_1}=\mu_{\Delta_3,\Delta_4}\mu_{\Delta_4,\Delta_5}\mu_{\Delta_5,\Delta_1}.$

$$\mu_{\Delta_2,\Delta_3}\mu_{\Delta_3,\Delta_2}\mu_{\Delta_2,\Delta_1}(t_{13})=\mu_{\Delta_2,\Delta_3}(t_{12}t_{42}^{-1}t_{43})=t_{12}t_{42}^{-1}t_{43},$$

$$\mu_{\Delta_2,\Delta_3}\mu_{\Delta_3,\Delta_2}\mu_{\Delta_2,\Delta_1}(t_{14})=\mu_{\Delta_2,\Delta_3}(t_{12}t_{52}^{-1}t_{54})=t_{12}(t_{51}t_{41}^{-1}t_{42})^{-1}t_{54}=t_{12}t_{42}^{-1}t_{45}t_{15}^{-1}t_{14},$$

$$\mu_{\Delta_2,\Delta_1}\mu_{\Delta_1,\Delta_5}\mu_{\Delta_5,\Delta_1}(t_{13})=
\mu_{\Delta_2,\Delta_1}(t_{13})=t_{12}t_{42}^{-1}t_{43},$$

\begin{equation*}
\begin{array}{rcl}
\mu_{\Delta_2,\Delta_1}\mu_{\Delta_1,\Delta_5}\mu_{\Delta_5,\Delta_1}(t_{14})&= &
\mu_{\Delta_2,\Delta_1}(t_{13}t^{-1}_{43}t_{45}t_{15}^{-1}t_{14})\\
&=&
t_{12}t_{42}^{-1}t_{43}t^{-1}_{43}t_{45}t_{15}^{-1}t_{14}\\
&=& t_{12}t_{42}^{-1}t_{45}t_{15}^{-1}t_{14}.
\end{array}
\end{equation*}

Thus, we have 
$\mu_{\Delta_2,\Delta_3}\mu_{\Delta_3,\Delta_2}\mu_{\Delta_2,\Delta_1}=\mu_{\Delta_2,\Delta_1}\mu_{\Delta_1,\Delta_5}\mu_{\Delta_5,\Delta_1}$.

The proof is complete.
\end{proof}

The following lemma can be proved similarly by calculation.

\begin{lemma}\label{lem21}
$(a)$  For the triangle $\Sigma$ with one special puncture, we label the boundary marked points clockwise by $1,2,3$, denote $\Delta_1=\{\ell_1, \overline{\ell}_1,(1,3)^+,\overline{(1,3)^+}\}\cup \{\text{boundary arcs}\}$ and $\Delta_2=\mu_{(1,3)^+}(\Delta_1),\Delta_3=\mu_{\ell_1}(\Delta_2),\Delta_6=\mu_{\ell_1}(\Delta_1), \Delta_5= \mu_{(1,3)^+}(\Delta_6), \Delta_4=\mu_{\ell_3}(\Delta_5)$, where $\ell_i$ is the special loop based at $i$ and $(1,3)^+$ is the internal arc connects $1$ and $3$. Then we have
$$\mu_{\Delta_3,\Delta_2}\mu_{\Delta_2,\Delta_1}=\mu_{\Delta_3,\Delta_4}\mu_{\Delta_4,\Delta_5}\mu_{\Delta_5,\Delta_6}\mu_{\Delta_6,\Delta_1}.$$

$(b)$  For the triangle $\Sigma$ with one $0$-puncture, we label the boundary marked points clockwise by $1,2,3$ and {the $0$-puncture} $0$, denote $\Delta_1=\{\ell_1, \overline{\ell}_1,(0,1),(1,0), (1,3)^+,\overline{(1,3)^+}\}\cup \{\text{boundary arcs}\}$ and $\Delta_2=\mu_{(1,3)^+}(\Delta_1),\Delta_3=\mu_{\ell_1}(\Delta_2),\Delta_6=\mu_{\ell_1}(\Delta_1), \Delta_5= \mu_{(1,3)^+}(\Delta_6), \Delta_4=\mu_{\ell_3}(\Delta_5)$, where $\ell_i$ is the loop based at $i$ and $(1,3)^+$ is the internal arc connects $1$ and $3$, and $(1,0), (0,1)$ are the pending arcs connects $0$ and $1$. Then we have
$$\mu_{\Delta_3,\Delta_2}\mu_{\Delta_2,\Delta_1}=\mu_{\Delta_3,\Delta_4}\mu_{\Delta_4,\Delta_5}\mu_{\Delta_5,\Delta_6}\mu_{\Delta_6,\Delta_1}.$$
    \end{lemma}

\begin{lemma}\label{lem3}
    With the notation of Lemma \ref{le:f-admissible triangulations}(b), {Let $\underline{\gamma} \in \underline{\Delta}$ be a non-self-folded arc that is neither a special loop nor a pending arc.} For any fixed ordering of $f^{-1}(\underline{\gamma}) = \{\gamma_1, \dots, \gamma_s\}$, the following diagram commutes.
\begin{equation}
\label{eq:vertical-horizontal diagram}
{\xymatrix{
  & \Delta \ar[d]_{\nu_{f,\Delta,\underline\Delta}} \ar[rr]^{ \mu_{\Delta',\Delta}}  &&  \Delta' \ar[d]
^{\nu_{f,\Delta',\underline\Delta'}}           \\
  & \underline \Delta \ar@{-}[rr]^{\mu^{ \varepsilon(f)}_{\mu_{\underline\gamma}\underline\Delta,\underline\Delta}}  && \underline\Delta'}}
  \end{equation}
where $\Delta'=\mu_{\gamma_s}\cdots \mu_{\gamma_2}\mu_{\gamma_1}(\Delta)$, $\mu_{\Delta',\Delta}=\mu_{\Delta',\mu_{\gamma_{s-1}}\cdots \mu_{\gamma_1}\Delta}\circ\cdots \circ \mu_{\mu_{\gamma_2}\mu_{\gamma_1}\Delta,\mu_{\gamma_1}\Delta}\circ \mu_{\mu_{\gamma_1}\Delta,\Delta}$ and

$\varepsilon(f):=\begin{cases}
1 & \text{if $f$ is orientation-preserving}\\
-1 & \text{if $f$ is orientation-reversing}.
\end{cases}$
\end{lemma}

\begin{proof}
For simplicity, we only consider the case that $\underline \gamma$ is not a loop around any pending arc. The case that $\underline \gamma$ is a loop around some pending arc can be proved similarly.

Assume that $\underline\gamma$ is a diagonal of the quadrilateral $(\underline\alpha_1,\underline\alpha_2,\underline\alpha_3,\underline\alpha_4)$ in $\underline\Delta$ such that $(\underline\alpha_1,\underline\alpha_2,\overline{ \underline\gamma})$ and $(\underline\gamma,\underline\alpha_3,\underline\alpha_4)$ are cyclic triangles. Denote by $\underline\gamma'$ the arc in $\underline\Delta'$ such that $(\underline\alpha_2,\underline\alpha_3,\underline\gamma')$ is a cyclic triangle.

We shall only prove the case that $f$ is orientation-preserving, the case that $f$ is orientation-reversing can be proved similarly. 

For any $\alpha\in \Delta$, {we first consider the case that $f(\alpha)$ is not an arc.} Then $f(\alpha)$ is a loop around some special puncture $o$. Denote by $\ell$ the special loop around $o$ in $\Delta$. As $\underline\gamma$ is assumed not a loop around any special puncture, we have $\ell\neq \underline\gamma$. Thus
$\mu_{\mu_{\underline\gamma}\underline\Delta,\underline\Delta}\nu_{f,\Delta,\underline\Delta}(t_\alpha)=\mu_{\mu_{\underline\gamma}\underline\Delta,\underline\Delta}(t_{\ell})=t_\ell$ and $\nu_{f,\Delta',\underline\Delta'}\mu_{\Delta',\Delta}(t_\alpha)=\nu_{f,\Delta',\underline\Delta'}(t_\alpha)=t_\ell$.

{We then consider the case that $f(\alpha)$ is an arc.}

Suppose that $\alpha$ is neither a self-folded arc nor a pending arc, nor a loop around a zero or special puncture. Assume that $\alpha$ is a diagonal of the quadrilateral $(\alpha_1, \alpha_2, \alpha_3, \alpha_4)$ in $\Delta$ such that $(\alpha_1, \alpha_2, \overline{\alpha})$ and $(\alpha, \alpha_3, \alpha_4)$ are cyclic triangles. Let $\alpha'$ be the arc in $\mu_{\alpha}(\Delta)$ such that $(\alpha_2, \alpha_3, \alpha')$ is a cyclic triangle. Then
\[\mu_{\mu_{\underline\gamma}\underline\Delta,\underline\Delta}\nu_{f,\Delta,\underline\Delta}(t_\alpha)=\mu_{\mu_{\underline\gamma}\underline\Delta,\underline\Delta}(t_{f(\alpha)})=\begin{cases}
    t_{f(\alpha)} & \text{ if $f(\alpha)\neq \underline\gamma$}\\
    t_{\underline\alpha_1} t_{\underline\gamma'}^{-1}t_{\underline\alpha_3} & \text{ if $f(\alpha)=\underline\gamma$.}
    \end{cases}
\]
\[\nu_{f,\Delta',\underline\Delta'}\mu_{\Delta',\Delta}(t_\alpha)=\begin{cases}
 \nu_{f,\Delta',\underline\Delta'}(t_\alpha) & \text{ if $f(\alpha)\neq \underline\gamma$} \\
 \nu_{f,\Delta',\underline\Delta'}(t_{\alpha_1} t_{\alpha'}^{-1}t_{\alpha_3}) & \text{ if $f(\alpha)=\underline\gamma$}
\end{cases}
=\begin{cases}
    t_{f(\alpha)} & \text{ if $f(\alpha)\neq \underline\gamma$}\\
    t_{\underline\alpha_1} t_{\underline\gamma'}^{-1}t_{\underline\alpha_3} & \text{ if $f(\alpha)=\underline\gamma$.}
    \end{cases}
    \]

Suppose that $\alpha$ is a self-folded arc, a pending arc, or a loop around a zero or special puncture. Then $f(\alpha)$ is also of this type, and hence $f(\alpha) \neq \underline{\gamma}$. Thus,
\[\mu_{\mu_{\underline\gamma}\underline\Delta,\underline\Delta}\nu_{f,\Delta,\underline\Delta}(t_\alpha)=\nu_{f,\Delta',\underline\Delta'}\mu_{\Delta',\Delta}(t_\alpha)=t_{f(\alpha)}.\]

Therefore, we have $\mu_{\mu_{\underline\gamma}\underline\Delta,\underline\Delta}\nu_{f,\Delta,\underline\Delta}=\nu_{f,\Delta',\underline\Delta'}\mu_{\Delta',\Delta}$.

The proof is complete.
\end{proof}

\subsection{Proofs of Theorems \ref{th:Udelta2} and \ref{thm:sectorgroup}.} \label{sec:proof of thm:sectorgroup}

\begin{proposition}\label{prop:quotient1} Let $\Sigma$ be a marked surface with $I_{p,0}=\emptyset$ and $\Delta$ be an ordinary triangulation of $\Sigma$. For any $i\in I_b\cup I_{p,1}$, fix a curve $\gamma_i\in \Delta$ with $t_{\gamma_i}\neq t_{\overline\gamma_i}$ and $s(\gamma_i)=i$ (all these curves are automatically distinct). Then the assignments $t_\gamma\mapsto u_{\overline{\gamma_{s(\gamma)}},\gamma}$ (e.g., $ t_{\gamma_i}\mapsto 1$) define a group homomorphism $\pi: \TT_\Delta\rightarrow \TT_\Delta$ which is a projection onto $\UU_\Delta$.
\end{proposition}

\begin{proof} 
First, we prove that $\pi$ is a homomorphism.

 (Monogon relations) For each special loop $\gamma$, we have $\pi(t_{\overline\gamma})=t^{-1}_{\gamma_{s(\overline\gamma)}}t_{\overline\gamma}=t^{-1}_{\gamma_{t(\gamma)}}x_{\gamma}=\pi(t_\gamma)$.

(Triangle relations) For each cyclic triangle $(\alpha_1,\alpha_2,\alpha_3)$ in $\Sigma$, we have $$\pi(t_{\alpha_1}t_{\overline\alpha_2}^{-1}t_{\alpha_3})=u_{\overline{\gamma_{s(\alpha_1)}},\alpha_1}(u_{\overline{\gamma_{s(\overline\alpha_2)}},\overline\alpha_2})^{-1}u_{\gamma_{s(\alpha_3)},\alpha_3}=t_{ \gamma_{s(\alpha_1)}}^{-1}t_{\alpha_1}t_{\overline\alpha_2}^{-1}t_{\alpha_3},$$
$$\pi(t_{\overline \alpha_3}t_{\alpha_2}^{-1}t_{\overline\alpha_3})=t_{\gamma_{s(\overline \alpha_3)}}^{-1}t_{\overline \alpha_3}t_{\alpha_2}^{-1}t_{\overline \alpha_1}.$$
Thus $\pi(t_{\alpha_1}t_{\overline\alpha_2}^{-1}t_{\alpha_3})=\pi(t_{\overline \alpha_3}t_{\alpha_2}^{-1}t_{\overline\alpha_1})$ follows by $s(\alpha_1)=s(\overline\alpha_3)$.

Therefore, we obtain a group homomorphism $\pi:\TT_\Delta\to \UU_\Delta$.

Next, show that $\pi^2=\pi$. Indeed, 
$$\pi^2(t_\gamma)=\pi(u_{\overline{\gamma_{s(\gamma)}},\gamma})=u_{\overline{\gamma_{s(\gamma)}},\gamma}$$
for any $\gamma$.

Finally, prove that the image of $\pi$ is $\UU_\Delta$. Indeed,
$$\pi(u_{\gamma,\gamma'})=\pi(t_{\overline\gamma}^{-1}t_{\gamma'})=u^{-1}_{\overline{\gamma_{s(\overline\gamma)}},\overline\gamma} u_{\overline{\gamma_{s(\gamma')}},\gamma'}=(t^{-1}_{\gamma_{s(\overline\gamma)}}t_{\overline\gamma})^{-1}(t^{-1}_{ \gamma_{s(\gamma')}}t_{\gamma'})=t_{\overline\gamma}^{-1}t_{\gamma'}=u_{\gamma,\gamma'}$$
for any $u_{\gamma,\gamma'}\in \UU_\Delta$. 

The proof is complete.
\end{proof}

The following follows immediately from Proposition \ref{prop:quotient1}.

\begin{corollary}\label{cor:relationBC}
Let $\Sigma\in {\bf Surf}$ with $I_{p,0}=\emptyset$, and let $\Delta$ be an ordinary triangulation of $\Sigma$. For each special puncture $p$, choose a special loop $\ell_p\in \Delta$. Then the sector subgroup $\UU_\Delta$ {is generated by the elements $u_{\overline{\gamma_{s(\gamma)}},\gamma}$ for $\gamma\in \Delta\setminus \{\ell_p \mid p\in I_{p,\geq 2}\}$, subject to the following relation for any cyclic triangle $(\alpha_1,\alpha_2,\alpha_3)$ in $\Delta$:}
\[
u_{\overline{\gamma_{s(\alpha_1)}},\alpha_1}(u_{\overline{\gamma_{s(\overline\alpha_2)}},\overline\alpha_2})^{-1}u_{\overline{\gamma_{s(\alpha_3)}},\alpha_3}=u_{\overline{\gamma_{s(\overline\alpha_3)}},\overline\alpha_3}(u_{\overline{\gamma_{s(\alpha_2)}},\alpha_2})^{-1}u_{\overline{\gamma_{s(\overline\alpha_1)}},\overline\alpha_1}.
\]
(If $\ell_p$ or $\overline\ell_p$ belongs to $\{\alpha_1,\alpha_2,\alpha_3\}$, replace $\ell_p$ by $\overline{\ell}_p$ in the relation above.)
\end{corollary}

\begin{proof}
    From Proposition \ref{prop:quotient1} and Definition, since $\Delta$ is an ordinary triangulation, we have $\UU_\Delta$ is generated by the elements $u_{\overline{\gamma_{s(\gamma)}},\gamma}$ for $\gamma\in \Delta$, subject to the relations:
    \begin{itemize}
        \item $u_{\overline{\gamma_{s(\gamma)}},\gamma}=u_{\overline{\gamma_{s(\gamma)}},\overline\gamma}$ for some special loop $\gamma\in \Delta$.
       \item $u_{\overline{\gamma_{s(\alpha_1)}},\alpha_1}(u_{\overline{\gamma_{s(\overline\alpha_2)}},\overline\alpha_2})^{-1}u_{\overline{\gamma_{s(\alpha_3)}},\alpha_3}=u_{\overline{\gamma_{s(\overline\alpha_3)}},\overline\alpha_3}(u_{\overline{\gamma_{s(\alpha_2)}},\alpha_2})^{-1}u_{\overline{\gamma_{s(\overline\alpha_1)}},\overline\alpha_1}$ for any cyclic triangle $(\alpha_1,\alpha_2,\alpha_3)$ in $\Delta$.
    \end{itemize}
    Then the result follows.
\end{proof}

{\bf Proof of Theorem \ref{th:Udelta2}}
Let $\Delta$ be an ordinary triangulation. For any marked point $i\in I_b\cup I_{p,1}$, from Remark \ref{rmk:generator}, we can choose an arc $\gamma_i\in \Delta$ such that $t_{\gamma_i}$ is a generator of the free or 1-relator  torsion free group $\TT_\Delta$ in Theorem \ref{thm:1-relator}. Then the result follows by Proposition \ref{prop:quotient1} and Theorem \ref{thm:1-relator}.  
\endproof

\begin{lemma}(\cite[Chapter 6]{Robinson})\label{lem:freeprodut}
Let $A$ be a free group of rank $m$ and $B$ be a free group of rank $n$. Let  $C$ be a group which contains both $A$ and $B$ as subgroups and is generated by $A$ and $B$. If $C$ is free of rank $m+n$ then $C=A*B$, the free product of $A$ and $B$.    
\end{lemma}


The following lemma follows by \cite[Theorem 68.7]{Mun}.

\begin{lemma}\label{lem:freeproduct1}
    Assume that $\widetilde\TT$ is a free group of rank $m$ with a basis $g_1,\cdots, g_m$, $\widetilde U$ is the subgroup generated by $g_1,\cdots,g_n$ and $F$ is the subgroup generated by $g_{n+1},\cdots, g_m$. Fix $a\in \widetilde U$, then $\widetilde T/\langle a\rangle=\widetilde U/\langle a\rangle * F$, the free product of $\widetilde U/\langle a\rangle$ and $F$.
\end{lemma}

{\bf Proof of Theorem \ref{thm:sectorgroup}.} Let $\Delta$ be an ordinary triangulation. For any marked point $i\in I_b\cup I_{p,1}$, by Remark \ref{rmk:generator}, we can choose an arc $\gamma_i\in \Delta$ such that $s(\gamma_i)=i$ and $t_{\gamma_i}$ is a generator of the free or 1-relator torsion-free group $\TT_\Delta$ described in Theorem \ref{thm:1-relator}. Let $F_{I_b\cup I_{p,1}}=\langle t_{\gamma_i}\mid i\in I_b\cup I_{p,1}\rangle$. Then $F_{I_b\cup I_{p,1}}$ is a free group of rank $|I_b\cup I_{p,1}|$. 

Let $a_1,\dots,a_m$ denote the generators of $\TT_\Delta$ from Theorem \ref{thm:1-relator}, arranged such that $\{a_{n+1},\dots, a_m\}=\{t_{\gamma_i}\mid i\in I_b\cup I_{p,1}\}$. Let $\pi: \TT_\Delta\to \UU_\Delta$ be the surjective map given in Proposition \ref{prop:quotient1}. We have $\pi(a_i)=1$ for $n+1\leq i\leq m$. Denoting $\bar a_i=\pi(a_i)$ for $1\leq i\leq n$, it follows that $\UU_\Delta$ is generated by $\bar a_1,\dots, \bar a_n$.

For any $\gamma\in \Delta$, we have the decomposition $t_\gamma=t_{\gamma_{s(\gamma)}}u_{\overline{\gamma_{s(\gamma)}},\gamma}$. Consequently, $\TT_\Delta$ is generated by $\UU_\Delta$ and $F_{I_b\cup I_{p,1}}$. By Theorem \ref{thm:1-relator}, $\TT_\Delta$ is either a free group or a 1-relator torsion-free group.

If $\TT_\Delta$ is a free group, then $\TT_\Delta=\UU_\Delta * F_{I_b\cup I_{p,1}}$ by Lemma \ref{lem:freeprodut}.

{If $\TT_\Delta$ is a 1-relator torsion-free group, assume that $\TT_\Delta$ admits the presentation $F\langle a_1,\dots, a_m\rangle /\langle a\rangle$ for some $a\in F\langle a_1,\dots, a_m\rangle$. By Remark \ref{rmk:generator}, the relator $a$ takes the form $a_{i_1}^{-1}a_{i_2}\cdots a_{i_{2k-1}}^{-1}a_{i_{2k}}$ for some $k\geq 1$ and indices $i_1,\dots,i_{2k}\in \{1,\dots, m\}$, satisfying $s(a_{i_j})=s(a_{i_{j+1}})$ for odd $j$ and $t(a_{i_j})=t(a_{i_{j+1}})$ for even $j$. A direct computation shows that the image of the relator is
\[ \pi(a)=\pi(a_{i_1})^{-1}\pi(a_{i_2})\cdots \pi(a_{i_{2k-1}})^{-1}\pi(a_{i_{2k}})\in F\langle \bar a_1,\dots, \bar a_n\rangle. \]
Thus, $\UU_\Delta \cong F\langle \bar a_1,\dots, \bar a_n\rangle /\langle \pi(a)\rangle$. By Lemma \ref{lem:freeproduct1}, we conclude that $\TT_\Delta=\UU_\Delta * F_{I_b\cup I_{p,1}}$.}

We now verify that relations (1), (2), and (3) constitute the defining relations of $\UU_\Delta$. It is straightforward to check that these relations hold; it remains to prove that they are sufficient.

For any $\gamma\in \Delta$, let $\gamma_1=\gamma_{s(\gamma)},\gamma_2,\dots,\gamma_n=\gamma$ be a sequence of arcs in $\Delta$ such that $s(\gamma_i)=s(\gamma)$ for all $i=1,\dots,n$, and $\gamma_i,\gamma_{i+1}$ form two sides of a triangle in $\Delta$ for all $i=1,\dots,n-1$ (replacing $\gamma$ by $\overline \gamma$ if $\gamma=\overline \ell_p$ for some special puncture $p$, if necessary). In terms of the generators in Theorem \ref{thm:sectorgroup}, we define
\[ u_{\overline{\gamma_{s(\gamma)}},\gamma}=u_{\overline{\gamma_1},\gamma_2}u_{\overline{\gamma_2},\gamma_3}\cdots u_{\overline{\gamma_{n-1}},\gamma_n}. \]
By construction, we have $u_{\overline{\gamma_{\ell_p}},\ell_p}=u_{\overline{\gamma_{\ell_p}},\overline{\ell_p}}$.

To prove that the relations in Theorem \ref{thm:sectorgroup} are defining, it suffices to show that they imply the relations in Corollary \ref{cor:relationBC}.

For any cyclic triangle $(\alpha_1,\alpha_2,\alpha_3)$ in $\Sigma$, the first and third relations in Theorem \ref{thm:sectorgroup} yield
\[ (u_{\overline{\gamma_{s(\overline\alpha_2)}},\overline\alpha_2})^{-1}u_{\overline{\gamma_{s(\alpha_3)}},\alpha_3}=u_{\alpha_2,\alpha_3},\quad u^{-1}_{\overline{\gamma_{s(\overline\alpha_1)}},\overline\alpha_1}u_{\overline{\gamma_{s(\alpha_2)}},\alpha_2}=u_{\alpha_1,\alpha_2},\quad u^{-1}_{\overline{\gamma_{s(\overline\alpha_3)}},\overline\alpha_3}u_{\overline{\gamma_{s(\alpha_1)}},\alpha_1}=u_{\alpha_3,\alpha_1}. \]
Consequently, by the second relation in Theorem \ref{thm:sectorgroup}, we obtain
\begin{equation*}
\begin{array}{rcl}
&& u_{\overline{\gamma_{s(\overline\alpha_2)}},\overline\alpha_2}^{-1}u_{\overline{\gamma_{s(\alpha_3)}},\alpha_3}u^{-1}_{\overline{\gamma_{s(\overline\alpha_1)}},\overline\alpha_1}u_{\overline{\gamma_{s(\alpha_2)}},\alpha_2}u^{-1}_{\overline{\gamma_{s(\overline\alpha_3)}},\overline\alpha_3}u_{\overline{\gamma_{s(\alpha_1)}},\alpha_1}\vspace{2.5pt}\\
&=& u_{\alpha_2,\alpha_3} u_{\alpha_1,\alpha_2} u_{\alpha_3,\alpha_1} \vspace{2.5pt}\\
&=& 1.
\end{array}
\end{equation*}
This completes the proof.
\endproof

\subsection{Proof of Theorem \ref{th:Br on Sigman}}

\label{subsec:Proof of Theorem th:Br_n 1}

We label the marked points of $\Sigma_n$ by $\{1,2,\dots,n\}$ in counterclockwise order. Let $\Delta=\{(1,i),(i,1)\mid i=3,\dots,n-1\}\cup \{\text{boundary arcs}\}$ be the star-like triangulation of $\Sigma_n$. By \cite[Theorem 3.26]{BR}, $\mathbb T_{\Delta}$ is a free group of rank $3n-4$ with basis 
\[ \{t_{ii^+}, t_{i^+i} \mid i=1,\dots, n-1\} \cup \{t_{j1} \mid j=3,\dots, n\}. \]
Denote $T_i:=T_{(1,i)}$ for $i=3,\dots, n-1$. 

To complete the proof, it suffices to show that the braid group $Br_{n-2}$ acts faithfully on $\mathbb U_{\Delta}$ via the assignment $\tau_i\mapsto T_{n-i}$. 

Let $H$ be the subgroup of $\mathbb T_{\Delta}$ generated by $\{t_{ii^+}, t_{i^+i} \mid i=1,\dots, n-1\}$. Then $H$ is a free subgroup of rank $2(n-1)$, and the braid group $Br_{\Delta}$ acts trivially on $H$.

Define $t_n=t_{n1}$ and $t_{n-1}=t_{n-1,n}$. Inductively, for $i \geq 3$, let $t_{i-1}=t_{i-1,i}t_{i+1,i}^{-1}t_{i+1}$. Thus, $t_i\in H$ for all $i\geq 2$. Denote $y_i=t_{i,1}^{-1}t_i$ for $i\geq 2$. Then $y_n=1$ and $y_i\in \UU_{\Delta}$ for all $i\geq 2$. For any $i$ with $3\leq i\leq n-1$, we have $T_i(y_j)=y_j$ for $j\neq i$, and
\[
T_i(y_i)=(t_{i,1}t_{i+1,1}^{-1}t_{i+1,i}t_{i-1,i}^{-1}t_{i-1,1})^{-1}t_i=y_{i-1}t_{i-1}^{-1}t_{i-1,i}t_{i+1,i}^{-1}t_{i+1}y_{i+1}^{-1}y_i=y_{i-1}y_{i+1}^{-1}y_i.
\]

Let $G$ be the subgroup of $\UU_\Delta$ generated by $y_2,y_3,\dots,y_{n-1}$. Then $G$ is invariant under the action of $Br_\Delta${, and $t_{j1}\in \langle G,H\rangle$ for $j=3,\dots,n$.} Since $\mathbb T_\Delta$ is a free group with basis $\{t_{ii^+}, t_{i^+i} \mid i=1,\dots,n-1\} \cup \{t_{j1} \mid j=3,\dots,n\}$, it follows that $G$ is a free subgroup of rank $n-2$. By Lemma \ref{lem:faithful}, $Br_{n-2}$ acts faithfully on $G$, and hence also faithfully on $\UU_\Delta$.

This completes the proof.
\endproof

\begin{lemma}\label{lem:faithful}
Let $G=\langle y_2,\cdots, y_{n-1}\rangle$ be a free group of rank $n-2$. Then the following actions 
$$\tau_{n-i}(y_j)=\begin{cases}
 y_{i-1}y^{-1}_{i+1}y_i & \text{if }j=i\\
 y_j & \text{otherwise},
\end{cases}$$
for all $3\leq i\leq n-1$ give a faithful action of $Br_{n-2}$ on $G$, where $y_n=1$ and $\tau_1,\cdots,\tau_{n-3}$ are the standard generators of $Br_{n-2}$. 
\end{lemma}

\begin{proof}
    Let $z_1=y_{n-2}^{-1}$ and inductively let $z_i=\tau_i(z_{i-1})$ for $i=2,\cdots,n-3$. Denote by $G'$ the subgroup of $G$ generated by $z_1,\cdots, z_{n-3}$. It is a free group of rank $n-3$.
    
    Then 
    $$z_2=\tau_{n-(n-2)}(y_{n-2}^{-1})=y_{n-2}^{-1}y_{n-1}y_{n-3}^{-1}=z_1y_{n-1}y_{n-3}^{-1},$$ 
    $$\tau_1(z_2)=\tau_1(z_1y_{n-1}y_{n-3}^{-1})=z_1y_{n-2}y_{n}^{-1}y_{n-1}y_{n-3}^{-1}=y_{n-1}y_{n-3}^{-1}=z_1^{-1}z_2.$$
    
For any $i\geq 1$, we have 
 $$z_{i+1}=z_1y_{n-1}y_{n-3}^{-1}y_{n-2}y_{n-4}^{-1}\cdots y_{n-i}y_{n-i-2}^{-1}=z_{i}y_{n-i}y_{n-i-2}^{-1},$$ 
 $$\tau_1(z_{i+1})=\tau_1(z_{i}y_{n-i}y_{n-i-2}^{-1})=z_1^{-1}z_iy_{n-i}y_{n-i-2}^{-1}=z_1^{-1}z_{i+1}.$$

For $i\geq 1$, we have $\tau_{i+1}(z_j)=z_j$ for $j<i$.

\begin{equation*}
\begin{aligned}
    \tau_{i+1}(z_{i}) &= \tau_{n-(n-i-1)}(z_1y_{n-1}y_{n-3}^{-1}y_{n-2}y_{n-4}^{-1}\cdots y_{n-{i+1}}y_{n-i-1}^{-1})  \\
     & =   z_1y_{n-1}y_{n-3}^{-1}y_{n-2}y_{n-4}^{-1}\cdots y_{n-{i+1}}\tau_{n-(n-i-1)}(y_{n-i-1}^{-1})\\
     & = z_{i+1}.
\end{aligned}
\end{equation*}

$$\tau_{i+1}(z_{i+1})=\tau_{n-(n-i-1)}(z_i y_{n-i}y_{n-i-2}^{-1}) =  z_{i+1}y_{n-i}y_{n-i-2}^{-1}= z_{i+1}z_i^{-1}z_{i+1}.$$

\begin{equation*}
\begin{aligned}
\tau_{i+1}(z_{i+2}) & =\tau_{n-(n-i-1)}(z_{i+1} y_{n-i-1}y_{n-i-3}^{-1})\\
&=    z_{i+1}y_{n-i}y_{n-i-2}^{-1}\tau_{n-(n-i-1)}(y_{n-i-1})y_{n-i-3}^{-1}\\
&= z_{i+2}.
\end{aligned}
\end{equation*}

$\tau_{i+1}(z_j)=\tau_{i+1}(z_{i+2} y_{n-i-2}y_{n-i-4}^{-1}\cdots y_{n-{j+1}}y_{n-j-1}^{-1})=z_j$ for any $j\geq i+2$.

In summary, $G'$ is invariant under $Br_{n-2}$ action and we have 
\[\begin{array}{ccl} \tau_1: 
\left\{\begin{array}{ll}
z_1\mapsto z_1, &\mbox{} \\
z_j\mapsto z_1^{-1}z_j, &\mbox{if $j\geq 2$}.
\end{array}\right.
\end{array}\]

For $2\leq i\leq n-3$, we have 
\[\begin{array}{ccl} \tau_i: 
\left\{\begin{array}{ll}
z_{i-1}\mapsto z_i, &\mbox{} \\
z_{i}\mapsto z_iz_{i-1}^{-1}z_i, &\mbox{} \\
z_j\mapsto z_j, &\mbox{if $j\neq i-1,i$}.
\end{array}\right.
\end{array}\]

By \cite[Theorem 3.2]{P}, the action on $G'$ is faithful. It follows that the action of $Br_{n-2}$ on $G$ is faithful. This completes the proof.
\end{proof}

\subsection{Proof of Theorem \ref{thm:inject}}\label{sec:proof of thm:inject}

Let $\Sigma_{n,1}$ be the once-punctured $n$-gon with puncture labeled $0$. We label the boundary marked points $\{1,2,\cdots,n\}$ of $\Sigma$ counterclockwise. 

We first show that the natural homomorphism $Br_n\to Br_{D_n}$ is injective. 

For $i\in \{1,2,\cdots,n\}$, denote by $(1,i)$ the simple curve connects $1$ and $i$ such that $0$ is on the left-hand side of $(1,i)$, denote by $(i,i^-), i\in \{1,2,\cdots,n\}$ the boundary arcs connecting $i$ and $i^-$. Denote $(i,1)=\overline{(1,i)}$.
We may let $\Delta=\{(0,1),(1,0),(1,i),(i,1)\mid i=1,\cdots, n-1\}\cup \{\text{boundary arcs}\}$. 

Let $t_n=t_{\overline{(1,n)}}, t_{n-1}=t_{\overline{(n,n-1)}}$ and inductively let $t_{i-1}=t_{\overline{(i,i-1)}}t_{i+1,i}^{-1}t_{i+1}$ for $i\geq 2$ and $t_0=t_2$. Denote $y_i=t_{\overline{(1,i)}}^{-1}t_i$ for $i\geq 1$ and $y_0=t_{(2,1)}^{-1}t_0$. Then $\langle y_0,y_1,\cdots,y_{n-1}\rangle$ is a free subgroup of $\TT_{\Delta}$ of rank $n$. For any $i$ with $1\leq i\leq n-1$, we have $T_{(1,i)}(y_j)=y_j$ for $j\neq i$ and
$$T_{(1,i)}(y_i)=y_{i-1}y_{i+1}^{-1}y_i.$$

 By Lemma \ref{lem:faithful}, the assignments $\tau_i\mapsto T_{(1,i)}$ give an injective homomorphism $Br_n\to \underline{Br}_\Delta$. From Theorem \ref{th:Br_n 1}(c), we see that $Br_{D_n}\to \underline {Br}_{\Delta}, \sigma_i\mapsto T_{(1,i)}, i=1,\cdots, n$ give a surjective homomorphism, where $\sigma_i,i=1,\cdots,n$ are the standard generators of the Artin braid group of type $D_n$. It is clear that the homomorphism $Br_{n}\to \underline{Br}_{\Delta}$ factors through the natural homomorphism $\iota:Br_{n}\to Br_{D_{n}}$. Therefore, $\iota$ is injective.

We then show that the natural homomorphism $Br_n\to Br_{\hat A_{n-1}}$ is injective.

Let $\Delta=\{(0,i),(i,0)\mid i=1,\cdots,n\}\cup \{\text{boundary arcs}\}$. Denote by $(i,i^+)$ the boundary arcs connecting $i$ and $i^+$. Let $t_{n+1}=t_{10}, t_{n}=t_{(n,1)}$, $t_{n-1}=t_{(n-1,n)}t_{\overline{(n,1)}}^{-1}t_{n+1}$ and inductively let $t_{i-1}=t_{(i-1,i)}t_{\overline{(i,i+1)}}^{-1}t_{i+1}$ for all $i$ with $n-2\geq i\geq 3$. Denote $y_i=t_{i,0}^{-1}t_i$ for $i$ with $1\leq i\leq n$. Denote $T_i=T_{(0,i)}$ for any $i\in \{2,\cdots,n\}$. As in the proof of Theorem \ref{th:Br on Sigman}, we have $G':=\langle y_1, y_2,\cdots, y_n \rangle$ is a free subgroup of $\mathbb T_\Delta$ of rank $n$ and for any $i\in\{2,\cdots, n\}$ we have
$$T_{i}(y_j)=\begin{cases}
 y_{i-1}y^{-1}_{i+1}y_i & \text{if }j=i\\
 y_j & \text{otherwise}.
\end{cases}$$

By Lemma \ref{lem:faithful}, we have $Br_{n}\cong \langle T_i\mid i=2,\cdots, n\rangle\subset \underline{Br}_{\Delta}$. Therefore, the homomorphism $Br_{n}\to Br_{\Delta}, \tau_i\mapsto {T_{i+1}}$ is injective, where $\tau_i,i=1,\cdots, n-1$ are the standard generators of $Br_\Delta$. From Theorem \ref{th:brgroup}, we see that $Br_{\hat A_{n-1}}\to \underline{Br}_{\Delta}, \sigma_i\mapsto {T_{i+1}}, i=1,\cdots, n$ give a surjective homomorphism, where $\sigma_i,i=1,\cdots,n$ are the standard generators of the Artin braid group of type $\hat A_{n-1}$. It is clear that the homomorphism $Br_{n}\to \underline{Br}_{\Delta}$ factors through the natural homomorphism $\iota:Br_{n}\to Br_{\hat A_{n-1}}$. Therefore, $\iota$ is injective. 

The proof is complete.
\endproof

\subsection{Proof of Theorem \ref{th:Br on Sigman1}}
\label{subsec:proof of Theorem clusterbraidgrouptypeB}

It suffices to prove that $Br_{\Delta}$ acts faithfully on $\mathbb U_{\Delta}$.

Let $\Sigma$ be an $n$-gon with one special puncture labeled $0$. We label the boundary marked points $\{1,2,\cdots,n\}$ of $\Sigma$ counterclockwise. For $i\in \{1,2,\cdots,n\}$, denote by $(1,i)$ the simple curve connects $1$ and $i$ such that $0$ is in the left hand side of $(1,i)$, denote by $(i,i^-), i\in \{1,2,\cdots,n\}$ the boundary arcs connecting $i$ and $i^-$. Denote $(i,1)=\overline{(1,i)}$.
We may let $\Delta=\{(1,i),(i,1)\mid i=1,\cdots, n-1\}\cup \{\text{boundary arcs}\}$. Denote $T_i=T_{(1,i)}$.

Let $G$ be the subgroup of $\mathbb T_{\Delta_1}$ generated by $t_{(i,i^-)}, t_{\overline{(i,i^-)}}, i=1,\cdots, n$. It is a free subgroup of rank $2n$. We have $Br_{\Delta}$ acts trivially on $G$.

Let $D_n=t_{\overline{(1,n)}}, D_{n-1}=t_{\overline{(n,n-1)}}$ and inductively let $D_{i-1}=t_{\overline{i,i-1}}t_{i+1,i}^{-1}D_{i+1}$ for $i\geq 2$. Let $D_0=D_2$. Thus $D_i\in G$ for any $i\geq 0$. Denote $y_i=t_{\overline{(1,i)}}^{-1}D_i$ for $i\geq 1$ and $y_0=t_{(2,1)}^{-1}D_0$. Then $y_n=1$ and $y_i\in \UU_{\Delta}$ for any $i\geq 0$. For any $i$ with $2\leq i\leq n-1$, we have $T_i(y_j)=y_j$ for $j\neq i$ and
$$T_i(y_i)=(t_{\overline{(1,i)}}t_{\overline{(1,i+1)}}^{-1}t_{(i+1,i)}t_{\overline{(i,i-1)}}^{-1}t_{\overline{(1,i-1)}})^{-1}D_i=y_{i-1}D_{i-1}^{-1}t_{\overline{(i,i-1)}}t_{(i+1,i)}^{-1}D_{i+1}y_{i+1}^{-1}y_i=y_{i-1}y_{i+1}^{-1}y_i,$$
$$T_1(y_1)=(t_{(1,1)}t_{\overline{(1,2)}}^{-1}t_{(2,1)}t_{\overline{(1,2)}}^{-1}t_{(2,1)})^{-1}D_1=
y_0D_0^{-1}
D_2y_2^{-1}
y_0D_0^{-1}
D_2y_2^{-1}
y_1=y_{0}y_{2}^{-1}y_0y_2^{-1}y_1.$$

Let $G'$ be the subgroup of $\UU_\Delta$ generated by $y_0,y_1,\cdots,y_{n-1}$. Then $G'$ is invariant under the action of $Br_\Delta$ and a free subgroup of rank $n$. 

By \cite[Proposition 5.1]{Cri}, we have $Br_{C_{n-1}}\to Br_{n}, \sigma_i\mapsto \tau_i$ for $i=1,\cdots,n-2$ and $\sigma_{n-1}\mapsto \tau_{n-1}^2$ give an injective group homomorphism, where $\sigma_i$ (resp. $\tau_i$), $i=1,\cdots, n-1$ are the standard generators of $Br_{C_{n-1}}$ (resp. $Br_n$). Then by Lemma \ref{lem:faithful}, $Br_{\Delta}$ acts faithfully on $G'$ with and thus also faithfully on $\UU_\Delta$ via $\sigma_i\mapsto T_{n-i}$. 

The proof is complete.
\endproof

\subsection{Proof of Theorem \ref{th:symmetries of n-gon}} \label{sec:symmetries of n-gon}

{Part (a) follows immediately from the fact that any inner automorphism is homogeneous of degree $0$.}

(b) We consider only the case when $\Delta=\Delta_0$. For any $i=3,4,\cdots, n-1$, {since no two of the arcs $(1,i)$, $(2n+1,2n+i)$, and $(n+1,n+i)$ form two edges of any triangle,} we have $T_{1,i},T_{2n+1,2n+i},T_{n+1,n+i}$ are pairwise commutative, denote $\sigma_i=T_{1,i}T_{2n+1,2n+i}T_{n+1,n+i}$ and
$\sigma_{n+1}=T_{1,n+1}T_{2n+1,1}T_{n+1,2n+1}T_{1,n+1}.$

For any $3\leq i\leq j\leq n+1$, denote 
$\sigma_{[i,j]}=\sigma_i\sigma_{i+1}\cdots \sigma_j.$  Let $\tau=(\sigma_{[3,n+1]})^{n-1}.$ Then $T_i \sigma_{[j,k]}=\sigma_{[j,k]}T_i$ for any $i$ and $j,k$ with $i<j-1$ or $i>k+1$. Therefore, for any $3\leq i\leq n-1$ we have
$$\begin{array}{rcl}
\sigma_{[3,n+1]} T_{1,i} &=& (\sigma_3\sigma_4\cdots \sigma_{n+1}) T_{1,i}\\
&=& \sigma_{[3,i-1]}\sigma_i\sigma_{i+1} T_{1,i} \sigma_{[i+2,n+1]}\\
&=& \sigma_{[3,i-1]} T_{1,i}T_{2n+1,2n+i}T_{n+1,n+i} T_{1,i+1}T_{2n+1,2n+i+1}T_{n+1,n+i+1}T_{1,i}\sigma_{[i+2,n+1]}
\\
&=& \sigma_{[3,i-1]} T_{2n+1,2n+i}T_{n+1,n+i}(T_{1,i}T_{1,i+1}T_{1,i})T_{2n+1,2n+i+1}T_{n+1,n+i+1}\sigma_{[i+2,n+1]}
\\
&=& \sigma_{[3,i-1]} T_{2n+1,2n+i}T_{n+1,n+i}(T_{1,i+1}T_{1,i}T_{1,i+1})T_{2n+1,2n+i+1}T_{n+1,n+i+1}\sigma_{[i+2,n+1]}
\\
&=& \sigma_{[3,i-1]} T_{1,i+1}T_{2n+1,2n+i}T_{n+1,n+i}T_{1,i}\sigma_{i+1}\sigma_{[i+2,n+1]}
\\
&=& \sigma_{[3,i-1]} T_{1,i+1}\sigma_i\sigma_{i+1}\sigma_{[i+2,n+1]}
\\
&=& T_{1,i+1}\sigma_{[3,n+1]}.
\end{array}$$

By symmetry, we have $\sigma_{[3,n+1]} T_{n+1,n+i}=T_{n+1,n+i+1}\sigma_{[3,n+1]}$.

We have 
$$\begin{array}{rcl}
& &\sigma^2_{[3,n+1]} T_{1,n}\\
&=& \sigma_{[3,n+1]}(\sigma_3\sigma_4\cdots \sigma_{n+1}) T_{1,n}\\
&=& \sigma_{[3,n+1]}\sigma_{[3,n-1]}\sigma_n\sigma_{n+1} T_{1,n}\\
&=& \sigma_{[3,n+1]}\sigma_{[3,n-1]} T_{1,n}T_{2n+1,3n}T_{n+1,2n} T_{2n+1,1}T_{n+1,2n+1}T_{1,n+1}T_{2n+1,1}T_{1,n}
\\
&=& \sigma_{[3,n+1]}\sigma_{[3,n-1]} T_{2n+1,3n}T_{n+1,2n} T_{2n+1,1}T_{n+1,2n+1}T_{1,n}T_{1,n+1}T_{1,n}T_{2n+1,1}
\\
&=&\sigma_{[3,n+1]} \sigma_{[3,n-1]} T_{2n+1,3n}T_{n+1,2n} T_{2n+1,1}T_{n+1,2n+1}T_{1,n+1}T_{1,n}T_{1,n+1}T_{2n+1,1}
\\
&=& \sigma_{[3,n+1]}T_{2n+1,3n}T_{n+1,2n} T_{2n+1,1}T_{n+1,2n+1}T_{1,n+1}\sigma_{[3,n-1]}T_{1,n}T_{1,n+1}T_{2n+1,1}
\\
&=& \sigma_{[3,n]}T_{1,n+1}T_{2n+1,1}T_{n+1,2n+1}T_{1,n+1}T_{2n+1,3n}T_{n+1,2n} T_{2n+1,1}T_{n+1,2n+1}T_{1,n+1}\\
&& \sigma_{[3,n-1]}T_{1,n}T_{1,n+1}T_{2n+1,1}\\
&=& \sigma_{[3,i-1]} T_{n+1,i+3}\sigma_i\sigma_{i+1}\sigma_{[i+2,n+1]}
\\
&=& T_{n+1,n+3}\sigma^2_{[3,n+1]}.
\end{array}$$

Thus, $\sigma^2_{[3,n+1]}T_{1,n}=T_{n+1,n+3}\sigma^2_{[3,n+1]}$ and $\sigma^{n-1}_{[3,n+1]}T_{1,n+1}=T_{n+1,2n+1}\sigma^{n-1}_{[3,n+1]}$.

Therefore, for any $3\leq i\leq n-1$
$$
\begin{array}{rcl}
\tau T_{1,i }&=&(\sigma_{[3,n+1]})^{n-1} T_{1,i}\\
&=& (\sigma_{[3,n+1]})^{i-1} T_{1,n} (\sigma_{[3,n+1]})^{n-i}
\\
&=& (\sigma_{[3,n+1]})^{i-3} T_{n+1,n+3}(\sigma_{[3,n+1]})^{n-i+2}
\\
&=& T_{n+1,n+i}(\sigma_{[3,n+1]})^{n-1}\\
&=& T_{n+1,n+i} \tau
\end{array}$$
$$
\begin{array}{rcl}
\tau T_{1,n}&=&(\sigma_{[3,n+1]})^{n-1} T_{1,n}\\
&=& (\sigma_{[3,n+1]})^{n-3} T_{n+1,n+3} (\sigma_{[3,n+1]})^{2}
\\
&=& T_{n+1,2n}(\sigma_{[3,n+1]})^{n-1}\\
&=& T_{n+1,2n}\tau.
\end{array}$$
$$
\tau T_{1,n+1}=T_{n+1,2n+1}\tau.$$

It follows that $\phi(T_\gamma)=\tau^{-1}T_\gamma \tau$ for any $\gamma\in \Delta$.

As $\phi^3=id$, we see that $\tau^3\in C(Br_{\Delta})$, the center of $Br_\Delta$. {Since $C(Br_\Delta)\cong C(Br_{3k-2})$ is a cyclic group (see, e.g., \cite[Theorem 4.2]{Gon}) and $Br_{3k-2}$ is a graded group in which each standard generator has degree $1$, a comparison of degrees shows that $\tau^3=(\tau_1\tau_2\cdots\tau_{3k-3})^{3k-2}$ generates $C(Br_\Delta)\cong C(Br_{3k-2})$.}

$(c)$ By the relation $R9$, we have
$T_{i}T=TT_{n+2-i}$ for all $i$. If $n$ is odd, we have $T_i T^{\frac{n-1}{2}}=T^{\frac{n-1}{2}}T_{i+1}$ for all $i$. Thus, $\phi(T_{i})=\tau^{-1} T_{i} \tau$ with $\tau=T^{\frac{n-1}{2}}$. 

The proof is complete.
\endproof

\subsection{Proof of Proposition \ref{prop:BtoA}}\label{sec:proof of prop:BtoA}

The following lemma can be proved by direct calculation.

\begin{lemma}\label{lem:2kgon}
Consider the triangulations $\Delta_0=\{(1,2i+1),(2i+1,1),(2i-1,2i+1),(2i+1,2i-1)\mid i=1,\cdots, k-1\}\cup\{\text{boundary arcs}\}$ and $\Delta'_0=\{(2,2i+2),(2i+2,2),(2i,2i+2),(2i+2,2i)\mid i=1,\cdots, k-1\}\cup\{\text{boundary arcs}\}$ of the $2k$-gon $\Sigma_{2k}$. Denote $\rho=T^{\Delta_0}_{35}T^{\Delta_0}_{57}\cdots T^{\Delta_0}_{2k-3,2k-1}$, $\phi=T^{\Delta_0}_{13}T^{\Delta_0}_{15}\cdots T^{\Delta_0}_{1,2k-1}$ and $\tau=\rho\phi\rho$. Then we have

$(a)$ $h_{\Delta_0,\Delta'_0}h_{\Delta'_0,\Delta_0}(t_{1,2i+1})=\tau(t_{1,2i+1})=t_{12}t_{32}^{-1}t_{31}t_{2i+1,1}^{-1}t_{2i+1,2i+2}t_{2i+3,2i+2}^{-1}t_{2i+3,2i+1}$ {(indices taken modulo $2k$)} for any $i$ with $1\leq i\leq k-2$.

$(b)$
$h_{\Delta_0,\Delta'_0}h_{\Delta'_0,\Delta_0}(t_{2i-1,2i+1})=\tau(t_{2i-1,2i+1})=t_{2i-1,2i}t_{2i+1,2i}^{-1}t_{2i+1,2i+2}t_{2i+3,2i+2}^{-1}t_{2i+3,2i+1}$ ({(indices taken modulo $2k$)}) for any $i$ with $1\leq i\leq k-2$.

In particular, we have $h_{\Delta_0,\Delta'_0}h_{\Delta'_0,\Delta_0}=\tau$.
\end{lemma}

We now provide a proof of Proposition \ref{prop:BtoA}. Let $\Delta'=\{(sn+1,sn+i), (sn+i,sn+1), ((t-1)n,tn), (tn,(t-1)n), (n,tn), (tn,n)\mid 0\leq s\leq k-1, 2\leq i\leq n, 2\leq t\leq k\}\cup \{\text{boundary arcs}\}$.

Consider the natural embedding of $\Sigma_{2k}\hookrightarrow \Sigma_{kn}$ via $(1,2,\cdots, 2k)\mapsto (1,n,1+n,2n,1+2n,\cdots, 1+(k-1)n,kn)$, by Lemma \ref{lem:2kgon}, we have $h_{\Delta,\Delta'}h_{\Delta',\Delta}=\tau_{n+1}$. Thus, $\tau_{n+1}\in Br_\Delta^{f_\sigma}$.

Denote $\Delta_i=\mu_{(1,i)}\mu_{n+1,n+i}\cdots \mu_{(k-1)n+1,(k-1)n+i}\Delta$ for any $i=3,\cdots,n$. By direct calculation, we have $\tau_i=h_{\Delta,\Delta_i}h_{\Delta_i,\Delta}$. Thus, $\tau_i\in Br_\Delta^{f_\sigma}$.

From Theorem \ref{th:brgroup}, we see that $\tau_i\tau_{i+1}\tau_i=\tau_{i+1}\tau_i\tau_{i+1}$ for $i$ with $3\leq i\leq n-1$, $\tau_i\tau_j=\tau_j\tau_i$ if $|i-j|\neq 1$.

For any $\ell$, by Lemma \ref{lem:2kgon}, we have

\begin{equation*}
\begin{aligned}
&  \quad\; \tau_n \tau_{n+1}\tau_n \tau_{n+1}(t_{1,\ell n+1})\\
& = \tau_n \tau_{n+1}\tau_n(t_{1n}t_{n+1,n}^{-1}t_{n+1,1}t_{\ell n+1,1}^{-1}t_{\ell n+1,(\ell+1)n}t_{(\ell+1)n+1,(\ell+1)n}^{-1}t_{(\ell+1)n+1,\ell n+1})
\\
&= \tau_n \tau_{n+1}  (t_{1,n-1}t_{n,n-1}^{-1}t_{n,1}
t_{\ell n+1,1}^{-1}
t_{\ell n+1,(\ell+1)n-1}
t_{(\ell+1)n,(\ell+1)n-1}^{-1}
t_{(\ell+1)n,\ell n+1}
)  \\
&= \tau_n (t_{1,n-1}t_{n,n-1}^{-1}t_{n,n+1}t_{1,n+1}^{-1}t_{1,\ell n+1}
t_{(\ell+1)n+1,\ell n+1}^{-1}
t_{(\ell+1)n+1,(\ell+1)n}
t_{(\ell+1)n-1,(\ell+1)n}^{-1}
t_{(\ell+1)n-1,\ell n+1})\\
&=t_{1,n-1}t_{n,n-1}^{-1}t_{n,n+1}t_{1,n+1}^{-1}t_{1,\ell n+1}
t_{(\ell+1)n+1,\ell n+1}^{-1}
t_{(\ell+1)n+1,(\ell+1)n}
t_{(\ell+1)n-1,(\ell+1)n}^{-1}
t_{(\ell+1)n-1,\ell n+1}\\
&=\tau_{n+1}\tau_n \tau_{n+1}\tau_n (t_{1,\ell n+1}).
\end{aligned}
\end{equation*}

\begin{equation*}
\begin{aligned}
& \;\;\;\;\;\tau_n \tau_{n+1}\tau_n \tau_{n+1}(t_{\ell n+1,(\ell+1)n+1})\\
& =\tau_n \tau_{n+1}\tau_n(t_{\ell n+1,(\ell+1)n}t_{(\ell+1)n+1,(\ell+1)n}^{-1}
t_{(\ell+1)n+1,(\ell+2)n+1}t_{(\ell+2)n,(\ell+2)n+1}^{-1}t_{(\ell+2)n,(\ell+1)n+1})
\\
&= \tau_n \tau_{n+1}  (t_{\ell n+1,(\ell+1)n-1}t_{(\ell+1)n,(\ell+1)n-1}^{-1}t_{(\ell+1)n,\ell n+1}t_{(\ell+1)n+1,\ell n+1}^{-1}
t_{(\ell+1)n+1,(\ell+2)n-1}\\
& \;\;\;\;\cdot\; t_{(\ell+2)n,(\ell+2)n-1}^{-1}t_{(\ell+2)n,(\ell+1)n+1}
)  \\
&= \tau_n (t_{\ell n+1,(\ell+1)n-1}t_{(\ell+1)n,(\ell+1)n-1}^{-1}
t_{(\ell+1)n,(\ell+1)n+1}t_{(\ell+2)n+1,(\ell+1)n+1}^{-1}t_{(\ell+2)n+1,(\ell+2)n}\\
& \;\;\;\;\cdot\; 
t^{-1}_{(\ell+2)n-1,(\ell+2)n}t_{(\ell+2)n-1,(\ell+1)n+1})\\
&= t_{\ell n+1,(\ell+1)n-1}t_{(\ell+1)n,(\ell+1)n-1}^{-1}
t_{(\ell+1)n,(\ell+1)n+1}t_{(\ell+2)n+1,(\ell+1)n+1}^{-1}t_{(\ell+2)n+1,(\ell+2)n}t^{-1}_{(\ell+2)n-1,(\ell+2)n}\\
& \;\;\;\;\cdot\; 
t_{(\ell+2)n-1,(\ell+1)n+1}\\
&=\tau_{n+1}\tau_n \tau_{n+1}\tau_n (t_{\ell n+1,(\ell+1)n+1}).
\end{aligned}
\end{equation*}

\begin{equation*}
\begin{aligned}
& \;\;\;\;\;\tau_{n+1} \tau_n \tau_{n+1}\tau_n (t_{1,n})\\
& =\tau_{n+1}\tau_n T(t_{1,n-1}t_{n,n-1}^{-1}t_{n,n+1}t_{1,n+1}^{-1}t_{1,n})
\\
&=\tau_{n+1}\tau_n (t_{1,n-1}t_{n,n-1}^{-1}t_{n,n+1}t_{2n,n+1}^{-1}t_{2n,2n+1}t_{n+1,2n+1}^{-1}t_{n+1,n}) \\
&= \tau_{n+1}(t_{1,n-1}t_{n,n-1}^{-1}t_{n,n+1}t_{2n,n+1}^{-1}t_{2n,2n-1}t_{n+1,2n-1}^{-1}t_{n+1,n})\\
&=t_{1,n-1}t_{n,n-1}^{-1}t_{n,n+1}t_{2n,n+1}^{-1}t_{2n,2n-1}t_{n+1,2n-1}^{-1}t_{n+1,n}\\
&=\tau_n \tau_{n+1}\tau_n \tau_{n+1}(t_{1,n}).
\end{aligned}
\end{equation*}

Similarly, we have $\tau_{n+1} \tau_n \tau_{n+1}\tau_n (t_{\ell n+1,(\ell+1)n})=\tau_n \tau_{n+1}\tau_n \tau_{n+1}(t_{\ell n+1,(\ell+1)n})$ for all $\ell$ with $1\leq \ell \leq k-1$.

Therefore, $\tau_{n+1}\tau_n \tau_{n+1}\tau_n(t_\gamma)=\tau_n \tau_{n+1}\tau_n \tau_{n+1}(t_\gamma)$ for all $\gamma\in \Delta$. Thus, by Theorem \ref{th:Br on Sigman}, we have $\tau_{n+1}\tau_n \tau_{n+1}\tau_n=\tau_n\tau_{n+1}\tau_n\tau_{n+1}$.

Thus, $\langle \tau_3,\cdots, \tau_n, \tau_{n+1}\rangle$ is isomorphic to a quotient group of $Br_{C_{n-1}}$. Under the surjective map $f_\sigma:\TT_{\Delta}\to \TT_{f_\sigma(\Delta)}$, we see that $\tau_3,\cdots,\tau_n,\tau_{n+1}$ act on $\TT_{f_\sigma(\Delta)}$ via $\tau_i\mapsto T_{f(1,i+2)}, \tau_{n+1}\to T_{f(1,n+1)}$. {Thus, by Theorem \ref{th:Br on Sigman1} (b) (c), we have 
\[\langle \tau_3,\cdots, \tau_n, \tau_{n+1}\rangle\cong Br_{C_{n-1}}\cong Br_{\Sigma_{kn}/\sigma}\].}

The proof is complete.
\endproof

\subsection{Proofs of Theorem \ref{th:functoriality nc-surface} and Proposition \ref{prop:symmetries and orbifolds}}
\label{subsec:Proof of Theorem functoriality nc-surface}

The following is immediate.

\begin{lemma}\label{lem:lift} Let
$\mathcal L,\mathcal A$ be semifirs with $\mathcal A=\mathcal L\langle x^{-1}\mid x\in S\rangle$ is a localization of $\mathcal L$ and $Frac(\mathcal L)=Frac(\mathcal A)$. Let $\mathcal F'$ be a skew-field and $\varphi:\mathcal L\to \mathcal F'$ be a ring homomorphism such that $\varphi(x)\neq 0$ for any $x\in S$. Then $\varphi$ can be extended to a ring homomorphism $\varphi: \mathcal A\to \mathcal F'$.
\end{lemma}

\begin{lemma}\label{lem:map}
Let $\underline\Sigma$ be a monogon with a special puncture $\underline p$, and let $\ell$ denote the special loop in $\underline\Sigma$. 

$(a)$ Let $\Delta$ be the star-like triangle of $\Sigma_{|\underline p|}$ at the marked point $1$. Then the assignments 
\[{x_{1i}, x_{i1}\mapsto \frac{\sin((i-1)\pi/|\underline p|)}{\sin(\pi/|\underline p|)}x_{\ell}}\] 
define a $\kk _{\underline\Sigma}$-algebra homomorphism 
$$\kk _{\underline\Sigma}[x^{\pm 1}_{\gamma}\mid \gamma\in \Delta] \mapsto Frac(\mathcal A_{\underline\Sigma}).$$

$(b)$ Let $\Sigma$ be an $n$-gon with a special puncture $p$ and let $\Delta$ be the star-like triangle at the marked point $1$, explicitly given by $\Delta=\{(1,i),\overline{(1,i)}\mid i=1,2,\cdots, n\}\cup \{\text{boundary arcs}\}$, where $(1,i)$ denotes the arc connecting $1$ and $i$ such that the special puncture $p$ is on the right. If $n|p|=|\underline p|$, then the assignments $$x_{\gamma}\mapsto \begin{cases}
    x_\ell, & \text{if $\gamma$ is a boundary arc,}\\\vspace{1mm}
   {\frac{\sin(n\pi/|\underline p|)}{\sin(\pi/|\underline p|)}}x_{\ell}, & \text{if $\gamma=(1,1)$,}\\\vspace{1mm} 
   {\frac{\sin((i-1)\pi/|\underline p|)}{\sin(\pi/|\underline p|)}}x_{\ell}, & \text{otherwise,}
\end{cases}$$ 
define a $\kk _{\underline\Sigma}$-algebra homomorphism 
$$\kk _{\underline \Sigma}[x^{\pm 1}_{\gamma}\mid \gamma\in \Delta] \to Frac(\mathcal A_{\underline\Sigma}).$$
\end{lemma}

\begin{proof}
    It follows by direct calculation.
\end{proof}

{\bf Proof of Theorem \ref{th:functoriality nc-surface}}.

(a) follows immediately by the relations in Definition \ref{def:ASigma}.

(b)
By Lemma \ref{le:f-admissible triangulations}, there exists a triangulation $\underline \Delta$ of $f(\Sigma)$ that can be lifted to a triangulation $\Delta$ of $\Sigma$, i.e., $f(\Delta)=\underline \Delta$. For each special loop $\underline\gamma$ in $\underline\Delta$, the preimage $f^{-1}(\underline\gamma)$ is either a polygon or a polygon with one special puncture. Restricting $\Delta$ to $f^{-1}(\underline\gamma)$, We obtain a triangulation of $f^{-1}(\underline\gamma)$. We may assume that it is of the form in Lemma \ref{lem:map}.

Now define a map $\hat f_*$ on $x_\gamma,\gamma\in \Delta$ as follows:

$\bullet$ If $f(\gamma)$ is an arc, then set $\hat f_*(x_\gamma)=x_{f(\gamma)}$. 

$\bullet$ If $f(\gamma)$ is not an arc, then $\gamma$ is an arc inside the $n$-gon $f^{-1}(\underline \gamma)=(\gamma_1,\cdots,\gamma_n)$ for some special loop $\underline\gamma\in \underline\Delta$ encloses a special puncture $\underline p$. {We may assume that $t(\gamma)=t(\gamma_1)$.}

In case $f^{-1}(\underline \gamma)$ encloses no special puncture, {suppose $(\gamma_1,\gamma_2\cdots,\gamma_{k-1},\gamma)$ is a $k$-gon for some $k\leq n$.} Then define 
$$\hat f_*(x_\gamma)=\frac{\sin((k-1)\pi/|\underline p|)}{\sin(\pi/|\underline p|)}x_{\underline\gamma}.$$

In case $f^{-1}(\underline \gamma)$ encloses a special puncture $p$, we may assume that $(\gamma_1,\gamma_2\cdots,\gamma_{k-1},\gamma)$ is a $k$-gon for some $2\leq k\leq n+1$ (let $k=2$ if $\gamma=\gamma_n$). Then define $$\hat f_*(x_\gamma)=\frac{\sin((k-1)\pi/|\underline p|)}{\sin(\pi/|\underline p|)}x_{\underline\gamma}.$$

By Lemma \ref{lem:map}, the assignments define a $\kk _{\Sigma'}$-algebra homomorphism 
$$\hat f_*:\kk _{\Sigma'}[x^{\pm 1}_\gamma\mid \gamma\in \Delta]\to \kk _{\Sigma'}[x^{\pm 1}_{\gamma'}\mid \gamma'\in \underline \Delta]\hookrightarrow Frac(\Delta')=Frac(\mathcal A_{\Sigma'}).$$

According to Theorem \ref{thm:laurent}, we have $\hat f_*(x_\beta)\neq 0$ for any $f$-admissible curve $\beta$ in $\Sigma$. Therefore, by Lemma \ref{lem:lift}, $\hat f_*$ extends to a $\kk _{\Sigma'}$-algebra homomorphism  
$$\hat f_*:\kk _{\Sigma'}\otimes_{\kk _{\Sigma}}{\mathcal A}_\Sigma^f\to {\mathcal Frac}({\mathcal A}_{\Sigma'}).$$ 

Moreover, it is clear that the image of ${\mathcal A}_\Sigma^f$ in $\hat f_*$ is in ${\mathcal A}_{\Sigma'}$.

The proof is complete.
\endproof

{\bf Proof of Proposition \ref{prop:symmetries and orbifolds}} From the proof of Theorem \ref{th:functoriality nc-surface}, there exist a triangle $\Delta$ and $\Delta'$ of $\Sigma$ and $\Sigma'$, respectively, and a $\kk _{\Sigma'}$-algebra homomorphism 
$$\hat f_*:\kk _{\Sigma'}[x^{\pm 1}_\gamma\mid \gamma\in \Delta]\to \kk _{\Sigma'}[x^{\pm 1}_{\gamma'}\mid \gamma'\in \underline \Delta].$$ 

As $f:\Sigma\to \Sigma'=\Sigma/\Gamma$ is the quotient map, we have $\hat f_*$ is surjective. By Lemma \ref{lem:map}, we see that $f(x_{\sigma(\gamma)})=f(x_\gamma)$ for all $\gamma\in \Delta$ and $\sigma\in \Gamma$, and $Ker~\hat f_*$ is generated by the following elements:

$\bullet$ $x_{\gamma}-x_{\overline \gamma}$ for all arcs $\gamma\in \Delta$ such that $f(\gamma)$ is a special loop enclosing a special puncture $p$ such that $|p|\neq |f(p)|$;

$\bullet$ $x_{\gamma_k}-2\cos({\frac{k\cdot |p|}{|f(p)|}}\pi)x_\gamma$ for all pairs $(\gamma,\gamma_k)$ in $\Delta$ such that $f(\gamma)$ is a special loop enclosing a special puncture $p$ such that $|p|\neq |f(p)|$, and $f(\gamma_k)$ is a closed curve with $k$ self-intersection points and enclosing the same special puncture as $f(\gamma)$.

Now consider the extended $\kk _{\Sigma'}$-algebra homomorphism  
$$\hat f_*:\kk _{\Sigma'}\otimes_{\kk _{\Sigma}}{\mathcal A}_\Sigma^f\to {\mathcal Frac}({\mathcal A}_{\Sigma'}).$$ 
One can see that $\hat f_*$ does not depend on the choice of the $f$-compatible pair $(\Delta,\Delta')$ and $\hat f_*(x_{\sigma(\gamma)})=\hat f_*(x_\gamma)$ for any curve $\gamma$ and $\sigma\in \Gamma$. Thus, we obtain a natural $\kk _{\Sigma'}$-algebra homomorphism 
$$\hat f_*:\kk _{\Sigma'}\otimes_{\kk _{\Sigma}}({\mathcal A}_\Sigma^f)_\Gamma\to {\mathcal A}_{\Sigma'},$$
whose kernel is generated by the following elements:

$\bullet$ $x_{\gamma}-x_{\overline \gamma}$ for all arcs $\gamma$ such that $f(\gamma)$ is a special loop enclosing a special puncture $p$ such that $|p|\neq |f(p)|$;

$\bullet$ $x_{\gamma_k}-2\cos({\frac{k\cdot |p|}{|f(p)|}}\pi)x_\gamma$ for all pairs $(\gamma,\gamma_k)$ such that $f(\gamma)$ is a special loop enclosing a special puncture $p$ such that $|p|\neq |f(p)|$, and $f(\gamma_k)$ is a closed curve with $k$ self-intersection points and enclosing the same special puncture as $f(\gamma)$.

As every curve in $\Sigma$ can be lifted to a curve in $\Sigma$, $\hat f_*:\kk _{\Sigma'}\otimes_{\kk _{\Sigma}}({\mathcal A}_\Sigma^f)_\Gamma\to {\mathcal A}_{\Sigma'}$ is surjective. 

This completes the proof.
\endproof

\section{Commutative and quantum cluster structures and their symmetries}

\subsection{Ordinary and quantum seeds}
\label{subsec:classical monomial mutations and braid}

Fix $n\leq m\in \mathbb Z_{>0}$, given any seed of geometric type ${\bf S}=({\bf x},\tilde B)$ with $\tilde B\in Mat_{m\times n}(\mathbb Z)$, we denote $G_{\bf S}=\mathbb Z^m$. 

Denote by ${\mathcal A}$ the cluster algebra of ${\bf S}$ and by ${\mathcal A}'$ its localization by all cluster variables.

The celebrated Laurent Phenomenon asserts a (canonical) embedding ${\bf j}_{\bf S}:{\mathcal A}\hookrightarrow \kk{\mathbb Z}^m=\kk [x_1^{\pm 1},\ldots,x_m^{\pm 1}]$ for any seed ${\bf S}$ (here we view elements of $e$ of ${\mathbb Z}^m$ as Laurent monomials $x^e$). This, in turn, defines the opposite embedding 
$$\iota_{\bf S}:{\mathbb Z}^m\hookrightarrow {\mathcal A}'$$
which is our ``noncommutative" cluster.

Thus, the Laurent Phenomenon asserts that {any cluster variable is a Laurent polynomial of any cluster}. Thus, for any polynomial (not Laurent) $x\in \kk G_{\bf S'}=\kk \mathbb Z^m$, its image $\iota_{\bf S'}(x)$ is in the image of $\iota_{\bf S}$, { where ${\bf S'}$ is a seed mutation equivalent to ${\bf S}$.}

The following is well-known, see, e.g., \cite[Corollary 6.3]{FZ4}. 

\begin{theorem} 
\label{th:leading term}
For any two mutation-equivalent (ordinary or quantum) seeds ${\bf S}$ and ${\bf S}'$, there exists a unique isomorphism $\mu_{{\bf S}',{\bf S}}$ of $\mathbb{Z}^m$ such that the $k$-th cluster variable $x'_k=\iota_{{\bf S}'}(x^{e_k})$ of ${\bf S}'$ admits the expansion
\[x'_k = \iota_{\bf S}\bigl(x^{\mu_{{\bf S}',{\bf S}}(e_k)}\bigr) + \text{lower terms},\]
{where the lower terms are sums of Laurent monomials $x^{\bf v}$ with $\mu_{{\bf S}',{\bf S}}(e_k)-{\bf v}\in \sum_{i=1}^n\mathbb Z_{\geq 0} b_i$, and $b_i$ denote the columns of $\widetilde B$. In particular, $\mu_{{\bf S}',{\bf S}}(e_k)$ is precisely the extended ${\bf g}$-vector of $x'_k$ with respect to the initial seed ${\bf S}$.} More generally, for any ${\bf m}\in \mathbb{Z}^m$,
\[x'^{\bf m} = \iota_{\bf S}\bigl(x^{\mu_{{\bf S}',{\bf S}}({\bf m})}\bigr) + \text{lower terms},\]
where the lower terms satisfy the analogous condition $\mu_{{\bf S}',{\bf S}}({\bf m})-{\bf v}\in \sum_{i=1}^n\mathbb Z_{\geq 0} b_i$.
\end{theorem}

In particular, for any $k=1,\ldots,n$, we have
\[\mu_{\mu_k({\bf S}),{\bf S}}(e_j)=
\begin{cases}
    {e_j} & \mbox{ if $j\neq k$},\\
    -e_k+[b_k]_+ & \mbox{ if $j=k$},
\end{cases}
: G_{\bf S}\to G_{\bf S'}
\]
for any $j=1,\ldots, m$, where $b_k$ is the $k$-th column of $\tilde B$. 

Let $\Gamma$ denote the groupoid whose objects {are all seeds mutation-equivalent to a fixed initial seed ${\bf S}$}, and whose morphisms are generated by the monomial mutation maps $\mu_{{\bf S}',{\bf S}''}\colon \mathbb{Z}^m \to \mathbb{Z}^m$ and their inverses. 

Following \cite[Section 2.2]{SSVZ}, define {\it transvection} $T_k=T_{k,{\bf S}}\in Br_{\bf S}$ to be 
$\mu_{{\bf S},\mu_k{\bf S}}\circ \mu_{\mu_k{\bf S},{\bf S}}: G_{\bf S}\to G_{{\bf S}}$, to be precisely, $T_k(e_j)=e_j+\delta_{kj} b_k$ for any $j=1,\ldots,m$. 

Let $Br_{\bf S}=\langle T_{k,{\bf S}}\mid k=1,2,\cdots,n \rangle \subset Aut_\Gamma({\bf S})$. By definition, it is a subgroup of $Aut(G_{\bf S})\cong GL_m(\Z)$. 

The following is immediate.

\begin{lemma} The assignments $g\mapsto \mu_{{\bf S}',{\bf S}}g \mu_{{\bf S}',{\bf S}}^{-1}$ defines an isomorphism $Aut_\Gamma({\bf S})\simeq Aut_\Gamma({\bf S}')$.
\end{lemma}

\begin{proposition} \label{prop:iso2} 
We have $Br_{\bf S}\cong Br_{\mu_i\bf S}$ for any $i=1,2,\cdots,n$. 
\end{proposition}

 \begin{proof} 
 By calculation, we have
\[\mu_{\mu_i({\bf S}),{\bf S}}^{-1}T_{k,\mu_i{\bf S}}\mu_{\mu_i({\bf S}),{\bf S}}=\begin{cases}
 T_{k,{\bf S}}, & \text{ if } b_{ik}\geq 0,\\
 T_{i,{\bf S}}^{-1}T_{k,{\bf S}}T_{i,{\bf S}}, & \text{ if }b_{ik}<0.
\end{cases}\]
 The result follows.
 \end{proof}

In other words, group $Br_{\bf S}$ depends only on cluster algebra ${\mathcal A}={\mathcal A}({\bf S})$, denote it by $Br_{\mathcal A}$.  We refer to $Br_{\mathcal A}$ as cluster braid group of ${\mathcal A}$. {Note that $Br_{\bf S}$ is also independent of the coefficients of $\mathcal{A}$.}

 

We expect that $Br_{\bf S}\cong Aut_\Gamma({\bf S})$.

\begin{proposition} \label{prop-braid1}
The following relations

$\bullet$
$\underbrace{T_iT_jT_i\cdots}_m =\underbrace{T_jT_iT_j\cdots}_m$, where 
$m=\begin{cases} 
2 & \text{if $b_{ji}=b_{ij}=0$}\\
3 & \text{if $|b_{ji}b_{ij}|=1$}\\
4 & \text{if $|b_{ji}b_{ij}|=2$}\\
6 & \text{if $|b_{ji}b_{ij}|=3$}\\
\end{cases}$

hold in {$Br_{\bf S}$}.  
\end{proposition}

\begin{proof}
    Follows by direct calculation, as in Theorem \ref{th:faithful rank 2}.
\end{proof}

Similarly, recall that a quantum seed ${\bf S}_q$
is a triple $({\bf X},\Lambda,\tilde B)$, where ${\bf X}$ is the quantum cluster $\{X_1,\ldots,X_m\}$  subject to relations in the ambient quantum torus group $G_{{\bf X},\Lambda}$ with the presentation
$$X_iX_j=q^{\lambda_{ij}}X_jX_i \ ,$$
where $q^{1/2}$ is the generator of the center of $G_{{\bf X},\Lambda}$ and  $\Lambda=(\lambda_{ij})$ is a
 skew-symmetric matrix compatible with $\tilde B$, i.e., $\Lambda \tilde B=\begin{pmatrix} -{\bf d}\\
{\bf 0}
\end{pmatrix}$, where ${\bf d}=diag(d_1,\ldots,d_n)$ and all $d_i\in \Z_{>0}$.

\begin{lemma} For each $i = 1, \dots, n$, the assignment
\[
X_j \longmapsto \begin{cases}
X^{-e_i + b_i} & \text{if } j = i, \\
X^{e_j} & \text{if } j \neq i,
\end{cases}
\qquad (j = 1, \dots, m)
\]
defines a unique automorphism $T_i$ of the quantum torus $G_{{\bf X}, \Lambda}$ that commutes with the anti-involution $\overline{(\cdot)}$.
\end{lemma}

Denote by $Br_{{\bf S}_q}$ the subgroup of $Aut(G_{{\bf X},\Lambda})$ generated by $T_1,\ldots,T_n$.

\begin{proposition}\label{prop:iso}
We have
 $Br_{{\bf S}_q}\cong Br_{{\bf S}}$.
\end{proposition}

\begin{proof}
It follows from Lemma \ref{lem:iso}.
\end{proof}

 \begin{lemma}\label{lem:iso} Let $Br_q$ be an automorphism groups of $G_{\bf X}$ commuting with the anti-involution $\overline {\cdot}$ of $G_{\bf X}$. Then the  specialization  $q\mapsto 1$ defines an injective homomorphism $Br_q\hookrightarrow GL_m(\Z)$.
     
 \end{lemma}

 \begin{proof}
    Assume that $\sigma\in Br_{q}$ belongs to the kernel. Then for any $i=1,\cdots, m$, we have $\sigma(X^{e_i})=q^{a_i}X^{e_i}$ for some $a_i\in \frac{1}{2}\mathbb Z$. As $\sigma$ commute with the anti-involution, we see that $\sigma(X^{e_i})$ is bar-invariant, it follows that {$a_i=0$} for any $i$. Therefore $\sigma$ is the identity in $Br_q$. Consequently, $Br_q\hookrightarrow GL_m(\Z)$ is injective.
 \end{proof}

The following result follows immediately from Propositions \ref{prop:iso2} and \ref{prop:iso}.

\begin{theorem} $Br_{{\bf S}_q}\cong Br_{{\bf S}'_q}$ for any mutation-equivalent quantum seeds ${\bf S}_q$ and ${\bf S}'_q$.   
\end{theorem}

In other words, the group $Br_{{\bf S}_q}$ depends only on the quantum cluster algebra ${\mathcal A}_q={\mathcal A}({\bf S}_q)$ and denote it by $Br_{{\mathcal A}_q}$.  We refer to $Br_{\mathcal A_q}$ as the cluster braid group of ${\mathcal A}_q$.

\subsection{Abelianization and $q$-abelianization of noncommutative surfaces}

\label{subsec:Abelianization and $q$-abelianization of Noncommutative surfaces}

The following is immediate.

\begin{lemma} 
\label{le:abelianization of a surface}
The quotient of the abelianization ${\mathcal A}_\Sigma^{\mathrm{ab}}$ by the relations $x_{\overline{\gamma}} = x_\gamma$ for all $\gamma$ is a localization of the ordinary cluster algebra $\mathcal{A}(\Sigma)$ {at all cluster variables.}
\end{lemma}

\begin{lemma}\label{lem:finiteorder}
In the notation of Section \ref{subsec:rank 2 braid}, denote the image of $T_1,T_2$ under the homomorphism $Br_\Delta\to (Br_\Delta)^{ab}$ by $T_1^{ab}$ and $T_2^{ab}$, respectively. Then $T_1^{ab}T_2^{ab}$ has finite order whenever $r_1r_2\in \{1,2,3\}$.
\end{lemma}

\begin{proof}
The result follows from the fact that the characteristic polynomial for $T_1^{ab}T_2^{ab}$ is $\lambda^2+(r_1r_2-2)\lambda +1$, which divides $\lambda^{12}-1$. 

The proof is complete.
\end{proof}

\begin{conjecture} The homomorphism $Br_\Delta\to (Br_\Delta)^{ab}$ is never injective.    
\end{conjecture}

\begin{example}
For the commutative cluster algebra from the once-punctured torus, we have
$$T_1^{ab}(T_2^{ab}T_3^{ab})^2=(T_2^{ab}T_3^{ab})^2T_1^{ab}.$$
$$(T_3^{ab}T_2^{ab})T_1^{ab}(T_3^{ab}T_2^{ab})^{-1}T_1^{ab}=T_1^{ab}(T_3^{ab}T_2^{ab})T_1^{ab}(T_3^{ab}T_2^{ab})^{-1}.$$
Thus $Br_{\Delta}^{ab}$ is not free, but $Br_\Delta$ is free by Proposition \ref{pro:free}.
\end{example}

\medskip

For the rest of this section, we always assume that $I_{p,0}(\Sigma)\cup I_{p,1}(\Sigma)=\emptyset$ and $\mathcal A_q(\Sigma)$ is a (generalized) quantum cluster algebra from $\Sigma$ with boundary coefficients. The readers are referred to \cite{BCDX} for the definition of (generalized) quantum cluster algebra. For each triangulation $\Delta$, denote by $(X^\Delta,B^\Delta,\Lambda^\Delta)$ the associated quantum seed. We also write $\Lambda^\Delta$ as $\Lambda$ if there is no case of confusion.

\begin{definition}\label{qcd}
Let $\Delta$ be a triangulation. A map $v:\Delta \rightarrow \mathbb Q$ is called a \emph{quantum cluster data on $\Delta$} if it satisfies
\begin{enumerate}[$(1)$]
  \item $v(\gamma)=-v(\overline{\gamma})$;
  \item
$v(\gamma_{1})+v(\gamma_{2})+v(\gamma_{3})=\frac{1}{2}\left(\Lambda(\gamma_1,\gamma_2)+\Lambda(\gamma_2,\gamma_3)+\Lambda(\gamma_3,\gamma_1)\right)$ for any cyclic triangle $(\gamma_{1},\gamma_{2},\gamma_{3})$ in $\Delta$;
\item $v(\gamma)=0$ for any special loop $\gamma$ in $\Delta$.
\end{enumerate}

\end{definition}

Given a non-boundary arc $\alpha\in \Delta$, denote $\Delta'=\mu_\alpha(\Delta)$. Throughout this section, assume $\alpha'\in \Delta'\setminus \Delta$, $(\alpha_1,\alpha,\overline\alpha_4)$ and $(\alpha,\alpha_3,\overline\alpha_2)$ are cyclic triangles in $\Delta$, and $(\alpha_1,\alpha_2,\overline{\alpha'})$ is a cyclic triangle in $\Delta'$, see Figure \ref{Fig:1}.  

\centerline{\begin{tikzpicture}
\draw[->,line width=1pt] (1.5,0) -- (2.25,0);
\draw[-,line width=1pt] (2.25,0) -- (3,0);
\draw[->,line width=1pt] (1.5,-1.5) -- (2.25,-1.5);
\draw[-,line width=1pt] (2.25,-1.5) -- (3,-1.5);
\draw[->,line width=1pt] (1.5,0) -- (1.5,-0.75);
\draw[-,line width=1pt] (1.5,-0.75) -- (1.5,-1.5);
\draw[->,line width=1pt] (3,0) -- (3,-0.75);
\draw[-,line width=1pt] (3,-0.75) -- (3,-1.5);
\draw[->,line width=1pt] (1.5,-1.5) -- (2.0625, -0.9375);
\draw[dashed,->,line width=1pt] (1.5,0) -- (2.0625,-0.5625);
\draw[dashed,line width=1pt] (2.0625,-0.5625) -- (3,-1.5);
\draw[-,line width=1pt] (3,0) -- (2.0625, -0.9375);
\node[above] at (2.25,0) {$\alpha_{4}$};
\node[left] at (1.5,-0.75) {$\alpha_{1}$};
\node[below] at (2.25,-1.5) {$\alpha_{2}$};
\node[right] at (3,-0.75) {$\alpha_{3}$};
\node[right] at (2.35,-1) {${\alpha}'$};
\node[right] at (1.48,-1) {$\alpha$};
\end{tikzpicture}}

\centerline{{\rm Figure 7.2}} \label{Fig:1}

\begin{lemma}\label{equal}
If $v$ is a quantum cluster data on $\Delta$, then $$\begin{array}{rcl}
& &
v(\alpha_{1})+v({\alpha})+v(\alpha_{3})+\frac{1}{2}(\Lambda(\alpha_1,\alpha_3)-\Lambda(\alpha_1,\alpha)-\Lambda(\alpha,\alpha_3))\vspace{2.5pt}\\
&=&
v(\alpha_{4})+v(\overline{\alpha})+v(\alpha_{2})+\frac{1}{2}(\Lambda(\alpha_4,\alpha_2)-\Lambda(\alpha_4,\alpha)-\Lambda(\alpha,\alpha_2)).
\end{array}
$$
\end{lemma}

\begin{proof}

As $v$ is a quantum cluster data, we have $$\begin{array}{rcl}
& &
v(\alpha_{1})+v({\alpha})+v(\alpha_{3})-v(\alpha_{4})-v(\overline{\alpha})-v(\alpha_{2})\vspace{2.5pt}\\
&=&
\left(v(\alpha_{1})+v({\alpha})+v(\overline\alpha_{4})\right)+\left(v(\alpha_{3})+v(\overline\alpha_{2})+v({\alpha})\right)\vspace{2.5pt}\\
&=&
\frac{1}{2}\left(\Lambda(\alpha_{1},\alpha)+\Lambda(\alpha,\alpha_{4})+\Lambda(\alpha_{4},\alpha_{1})+\Lambda(\alpha_3,\alpha_2)+\Lambda(\alpha_2,\alpha)+\Lambda(\alpha,\alpha_3)\right).
\end{array}$$

Thus, the required equality is equivalent to
$$
\begin{array}{rcl}
& &
\Lambda(\alpha_{1},\alpha)+\Lambda(\alpha,\alpha_{4})+\Lambda(\alpha_{4},\alpha_{1})+\Lambda(\alpha_3,\alpha_2)+\Lambda(\alpha_2,\alpha)+\Lambda(\alpha,\alpha_3)\vspace{2.5pt}\\
&+&
\left(\Lambda(\alpha_1,\alpha_3)-\Lambda(\alpha_1,\alpha)-\Lambda(\alpha,\alpha_3)\right)-\left(\Lambda(\alpha_4,\alpha_2)-\Lambda(\alpha_4,\alpha)-\Lambda(\alpha,\alpha_2)\right)\vspace{2.5pt}\\
&=& \Lambda(\alpha_4,\alpha_1)+\Lambda(\alpha_3,\alpha_2)+\Lambda(\alpha_1,\alpha_3)+\Lambda(\alpha_2,\alpha_4)=0.
\end{array}
$$

Because of $(B^\Delta,\Lambda^\Delta)$ is compatible and $\alpha_1,\alpha_2\neq \alpha$, we obtain
$$\Lambda(\alpha_4,\alpha_1)+\Lambda(\alpha_2,\alpha_1)-\Lambda(\alpha_3,\alpha_1)=0,\;\; \Lambda(\alpha_1,\alpha_2)+\Lambda(\alpha_3,\alpha_2)-\Lambda(\alpha_4,\alpha_2)=0.$$

Therefore, take the addition of the above two equations, we have $$\Lambda(\alpha_4,\alpha_1)+\Lambda(\alpha_3,\alpha_2)+\Lambda(\alpha_1,\alpha_3)+\Lambda(\alpha_2,\alpha_4)=0.$$
The result follows.
\end{proof}

\medskip

\begin{proposition}\label{muta}
Let $v$ be a quantum cluster data on $\Delta$. The following assignments define a quantum cluster data on $\Delta'$
$$v'({\gamma})=\left\{\begin{array}{lll}\hspace{-3pt}  v({\gamma}), & {\rm if} \gamma\in \Delta\cap \Delta';\vspace{2.5pt}\\
\hspace{-3pt} v(\alpha_{1})+v({\alpha})+v(\alpha_{3})+\frac{1}{2}(\Lambda(\alpha_1,\alpha_3)-\Lambda(\alpha_1,\alpha)-\Lambda(\alpha,\alpha_3)), & {\rm if} \gamma=\alpha';\vspace{2.5pt}\\
-v'(\overline\alpha'), & {\rm if} \gamma=\overline\alpha'.
\end{array}\right.$$
\end{proposition}

\begin{proof}

Condition (1) of Definition \ref{qcd} is immediately satisfied for $v'$. For condition (2), it suffices to prove that condition (2) holds for cyclic triangles $(\alpha_{1}, \overline\alpha_{2}, \overline\alpha')$ and $(\alpha_{4}, \alpha_{3}, \overline\alpha')$. We shall only prove that for the triangle $(\alpha_{4}, \alpha_{3}, \overline\alpha')$ since the other case can be proved similarly.

As $(B^\Delta,\Lambda^\Delta)$ is compatible, we have $\Lambda(\alpha_2,\alpha_4)=\Lambda(\alpha_1,\alpha_4)+\Lambda(\alpha_3,\alpha_4)$ and $\Lambda(\alpha_3,\alpha')=\Lambda(\alpha_3, \alpha_1)-\Lambda(\alpha_3,\alpha)$, $\Lambda(\alpha',\alpha_4)=\Lambda(\alpha_2, \alpha_4)-\Lambda(\alpha,\alpha_4)$. Therefore, by the construction of $v'$, we have
$$
\begin{array}{rcl}
& &
v'(\alpha_{4})+v'(\alpha_{3})+v'(\overline\alpha')\vspace{2.5pt}\\
&=& v(\alpha_{4})+v(\alpha_{3})-v(\alpha_{1})-v({\alpha})-v(\alpha_{3})
-\frac{1}{2}(\Lambda(\alpha_1,\alpha_3)-\Lambda(\alpha_1,\alpha)-\Lambda(\alpha,\alpha_3))\vspace{2.5pt}\\
&=&
v(\alpha_{4})+v(\overline{\alpha})+v(\overline\alpha_{1})-\frac{1}{2}(\Lambda(\alpha_1,\alpha_3)-\Lambda(\alpha_1,\alpha)-\Lambda(\alpha,\alpha_3))\vspace{2.5pt}\\
&=&
\frac{1}{2}(\Lambda(\alpha_4,\alpha)+\Lambda(\alpha,\alpha_1)+\Lambda(\alpha_1,\alpha_4))+\frac{1}{2}(\Lambda(\alpha_1,\alpha)+\Lambda(\alpha,\alpha_3)-\Lambda(\alpha_1,\alpha_3))\vspace{2.5pt}\\
&=&
\frac{1}{2}(\Lambda(\alpha_4,\alpha)+\Lambda(\alpha_1,\alpha_4)+\Lambda(\alpha,\alpha_3)+\Lambda(\alpha_3,\alpha_1))\vspace{2.5pt}\\
&=&
\frac{1}{2}(-\Lambda(\alpha,\alpha_4)+\Lambda(\alpha_2,\alpha_4)-\Lambda(\alpha_3,\alpha_4)-\Lambda(\alpha_3,\alpha)+\Lambda(\alpha_3,\alpha_1))\vspace{2.5pt}\\
&=&
\frac{1}{2}(\Lambda(\alpha',\alpha_4)+\Lambda(\alpha_4,\alpha_3)+\Lambda(\alpha_3,\alpha')).
\end{array}
$$

For condition $(3)$, if $\alpha$ is not a special loop, then any special loop $\gamma$ in $\Delta'$ is a special loop in $\Delta$ and thus $v'(\gamma)=v(\gamma)=0$. If $\alpha$ is a special loop, then $\alpha',\overline{\alpha'}$ are the special loops in $\Delta'$ but not in $\Delta$. Assume that $\alpha$ is in the bigon $(\gamma_1,\gamma_2)$ with $s(\alpha)=s(\gamma_1)$, then we have $$v'(\alpha')=v(\gamma_{2})+v({\alpha})+v(\overline\gamma_{2})+\frac{1}{2}(\Lambda(\gamma_2,\gamma_2)-\Lambda(\gamma_2,\alpha)-\Lambda(\alpha,\gamma_2))=0.$$ 

Therefore, the result follows.
\end{proof}

We denote $\mu_\alpha v=v'$ and call it the mutation of $v$ at $\alpha$.

\begin{lemma}\label{involution}

In the previous notation, mutation of the quantum cluster data is an involution, that is, $\mu_{\alpha'}\mu_{\alpha}(v)=v$.

\end{lemma}

\begin{proof}
It suffices to show $\mu_{\alpha'}\mu_{\alpha}(v)(\alpha)=v(\alpha)$. By calculation, we have
$$
\begin{array}{rcl}
\mu_{\alpha'}\mu_{\alpha}(v)(\alpha)
\hspace{-8pt} &=& \hspace{-8pt} v'(\overline\alpha_{1})+v'({\alpha})+v'(\overline\alpha_{3})+\frac{1}{2}(\Lambda(\alpha_1,\alpha_3)-\Lambda(\alpha_1,\alpha')-\Lambda(\alpha',\alpha_3))\vspace{2.5pt}\\
\hspace{-8pt} &=& \hspace{-8pt}
-v(\alpha_{1})-v(\alpha_{3})+\frac{1}{2}\Lambda(\alpha_1,\alpha_3)\vspace{2.5pt}\\
\hspace{-8pt} & & \hspace{-8pt} +
v(\alpha_{1})+v({\alpha})+v(\alpha_{3})+\frac{1}{2}(\Lambda(\alpha_1,\alpha_3)-\Lambda(\alpha_1,\alpha)-\Lambda(\alpha,\alpha_3))\vspace{2.5pt}\\
\hspace{-8pt} & & \hspace{-8pt} +
\frac{1}{2}(-\Lambda(\alpha_1,\alpha_3)+\Lambda(\alpha_1,\alpha)-\Lambda(\alpha_1,\alpha_3)+\Lambda(\alpha,\alpha_3))\vspace{2.5pt}\\
\hspace{-8pt} &=& \hspace{-8pt}
v(\alpha).
\end{array}
$$
The result follows.
\end{proof}

\begin{proposition}\label{exist1}
For any triangulation $\Delta$ there exists at least one quantum cluster data.
\end{proposition}

\begin{proof}
{For any triangle $\delta \in \Delta$, taking conditions (1) and (3) into account,} we see that condition (2) of Definition \ref{qcd} yields a linear equation in two or three variables. Let $s$ be the number of triangles in $\Delta$. Thus, the existence of quantum cluster data on $\Delta$ is equivalent to the solvability of the linear system $AX=b$ determined by the triangles in $\Delta$. It therefore suffices to show that the matrix $A$ has full rank $s$.

Suppose, to the contrary, that $A$ does not have full rank. After permuting the rows of $A$, we may assume that the first $t$ rows $r_1, \dots, r_t$ form a minimally linearly dependent set; that is, they are linearly dependent, but any proper subset is linearly independent. Let $\delta_i$ be the triangle corresponding to the row $r_i$ for $1 \leq i \leq t$. By the minimality assumption, for each $i$, every non-special arc of $\delta_i$ must appear as an arc of some other triangle $\delta_j$ with $j \neq i$. Consequently, the subsurface $\bigcup_{i=1}^t \delta_i$ of $\Sigma$ is a closed surface. This contradicts the assumption that $I_{p,0}(\Sigma) \cup I_{p,1}(\Sigma) = \emptyset$.

The proof is complete.
\end{proof}

\medskip

We now define the quantum cluster data for a surface.

\begin{definition}\label{quantumclusterdata}

A map $v: \{\text{arcs in }\Sigma\}\rightarrow \mathbb Q$ is called a \emph{quantum cluster data on $\Sigma$} if it satisfies

\begin{enumerate}[$(1)$]
  \item $v({\gamma})=-v(\overline{\gamma})$;
  \item
  $v(\gamma_{1})+v(\gamma_{2})+v(\gamma_{3})=\frac{1}{2}\left(\Lambda(\gamma_1,\gamma_2)+\Lambda(\gamma_2,\gamma_3)+\Lambda(\gamma_3,\gamma_1)\right)$ for each cyclic triangle $(\gamma_{1},\gamma_2,\gamma_{3})$ in $\Sigma$;
  \item $v(\gamma)=0$ for any special loop $\gamma$ in $\Delta$.
\end{enumerate}

\end{definition}

\medskip

Let $\beta_1,\beta_2\in \Delta$. Assume that $|b^\Delta_{12}|=1$. Then $\mu_1\mu_2\mu_1\mu_2\mu_1(\Delta)=\Delta$, see \cite[Section 9.4]{FST}.

\begin{lemma}\label{5-cycle}
With the previous notation. Let $v$ be a quantum cluster data on $\Delta$. If $|b^{\Delta}_{12}|=1$ for some $\beta_1,\beta_2\in \Delta$, then $\mu_1\mu_2\mu_1\mu_2\mu_1(v)=v$.
\end{lemma}
\begin{proof}
We assume that $\beta_1$ and $\beta_2$ are diagonals of the pentagon $\Sigma_5$ in $\Sigma$. For clarity of notation, we also assume $\beta_1=(1,3)$ and $\beta_2=(1,4)$, the diagonal connecting $1$ with $3$ and $1$ with $4$, respectively. 

We shall only prove that $\mu_1\mu_2\mu_1\mu_2\mu_1v(13)=v(13)$, $\mu_1\mu_2\mu_1\mu_2\mu_1v(14)=v(14)$ can be proved in a similar way.

$$
\begin{array}{rcl}
\mu_1\mu_2\mu_1\mu_2\mu_1v(13)
\hspace{-8pt} &=& \hspace{-8pt}\mu_2\mu_1\mu_2\mu_1v(13)\vspace{2.5pt}\\
\hspace{-8pt} &=& \hspace{-8pt}
v(12)+v(25)+v(53)+\frac{1}{2}(\Lambda(12,35)-\Lambda(12,25)-\Lambda(25,35))\vspace{2.5pt}\\
\hspace{-8pt} &=& \hspace{-8pt}
v(12)+v(25)+v(52)+v(24)+v(43)\vspace{2.5pt}\\
              & &\hspace{-8pt} + \frac{1}{2}(\Lambda(25,34)-\Lambda(25,24)-\Lambda(24,34))\vspace{2.5pt}\\
              & &\hspace{-8pt} + \frac{1}{2}(\Lambda(12,35)-\Lambda(12,25)-\Lambda(25,35))\vspace{2.5pt}\\
\hspace{-8pt} &=& \hspace{-8pt}
v(12)+v(43)+v(21)+v(13)+v(34)\vspace{2.5pt}\\
              & &\hspace{-8pt} + \frac{1}{2}(\Lambda(12,34)-\Lambda(12,13)-\Lambda(13,34))\vspace{2.5pt}\\
              & &\hspace{-8pt} + \frac{1}{2}(\Lambda(25,34)-\Lambda(25,24)-\Lambda(24,34))\vspace{2.5pt}\\
              & &\hspace{-8pt} + \frac{1}{2}(\Lambda(12,35)-\Lambda(12,25)-\Lambda(25,35)).
\end{array}
$$

As $(B^\Delta,\Lambda^\Delta)$ is compatible, we have $\Lambda(12,34)-\Lambda(13,34)-\Lambda(24,34)=0$, $-\Lambda(12,13)+\Lambda(12,35)-\Lambda(12,25)=0$ and $\Lambda(25,34)-\Lambda(25,24)-\Lambda(25,35)=0$. It follows that $\mu_1\mu_2\mu_1\mu_2\mu_1v(13)=v(13)$. Our result follows.
\end{proof}

The following theorem together with Proposition \ref{exist1} implies an existence of quantum cluster data on $\Sigma$.

\begin{theorem}\label{existence}
 Let $\Delta$ be a triangulation of $\Sigma$. If $v$ is a quantum cluster data on $\Delta$, then $v$ can be uniquely extended to a quantum cluster data $v$ on $\Sigma$ via the mutations of quantum cluster data.
\end{theorem}

\begin{proof}
For any arc $\gamma$, it can be obtained from $\Delta$ via different sequences of flips. It suffices to prove that the values assigned to $\gamma$ are independent of the mutation sequence at $v$. By Lemma \ref{involution}, this is equivalent to showing that $\mu_{\beta_s}\cdots \mu_{\beta_1}(v)=v$ for any sequence of flips $\mu_{\beta_1}, \dots, \mu_{\beta_s}$ such that $\mu_{\beta_s}\cdots \mu_{\beta_1}(\Delta)=\Delta$.

Consider the flip graph of triangulations of $\Sigma$. By Theorem \ref{thm:56cycles}, the cycles in this graph are generated by cycles of length 4, 5, and 6. Note that a cycle of length 6 exists only if $\Sigma$ contains a zero or special puncture.

In the case of length 4, since the mutation of quantum cluster data is an involution, we have $\mu_i\mu_j\mu_i\mu_j(v)=v$. The length 5 case follows from Lemma \ref{5-cycle}. In particular, the result holds for all $\Sigma$ without $0$ or special punctures.

{For any cycle of length 6, it can be lifted to a cycle of length $9$ in the exchange graph of the hexagon $\Sigma_6$ (see Figure \ref{Fig:96}). Thus, the length $6$ case follows.}

The proof is complete.
\end{proof}

\begin{figure}[ht]
\includegraphics{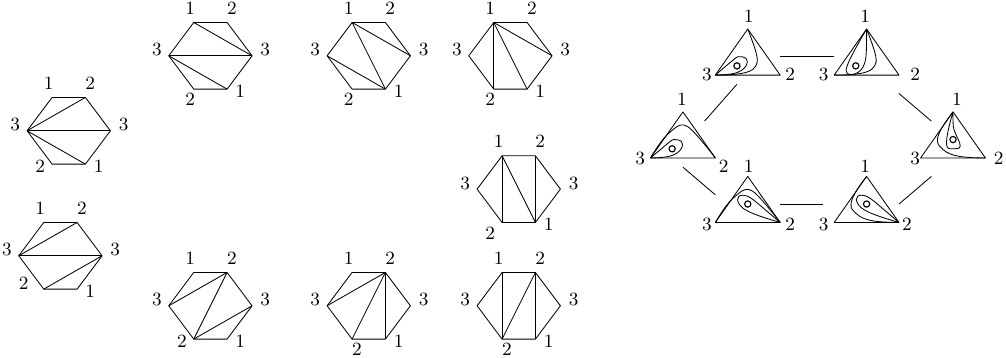}
\caption{Length $6$ cycle coverd by length $9$ cycle}
\label{Fig:96}
\end{figure}

\medskip

\begin{corollary}\label{flip}

Let $v$ be a quantum cluster data on $\Sigma$. Then for any quadrilateral in $\Sigma$, as shown in Figure \ref{Fig:1}, we have
$$
\begin{array}{rcl}
v(\alpha')
\hspace{-3pt}&=&\hspace{-3pt}
v(\alpha_{1})+v({\alpha})+v(\alpha_{3})+\frac{1}{2}(\Lambda(\alpha_1,\alpha_3)-\Lambda(\alpha_1,\alpha)-\Lambda(\alpha,\alpha_3))\vspace{2.5pt}\\
\hspace{-3pt}&=&\hspace{-3pt}
v(\alpha_{\hspace{-2pt}4})+v(\overline{\alpha})+v(\alpha_{2})+\frac{1}{2}(\Lambda(\alpha_4,\alpha_2)-\Lambda(\alpha_4,\alpha)-\Lambda(\alpha,\alpha_2)). \end{array}
$$

\end{corollary}

\begin{proof}

Let $\Delta$ be a triangulation of $\Sigma$ so that $\alpha_1,\alpha_2,\alpha_3,\alpha_4, \alpha\in \Delta$. Restricting $v$ to $\Delta$, we obtain a quantum cluster data $v|_{\Delta}$ on $\Delta$. Clearly, $v$ is an extension of $v|_{\Delta}$. According to Theorem \ref{existence}, $v|_{\Delta}$ can be uniquely extended to a quantum cluster data on $\Sigma$ via mutations, thus is $v$. Then the result is followed by Lemma \ref{equal}.
\end{proof}

\medskip

\begin{theorem}\label{surj}

Let $v$ be a quantum cluster data on $\Sigma$. Then $$\pi: \kk_\Sigma(q)\otimes_{\kk_\Sigma} \mathcal A_{\Sigma}\rightarrow \kk_\Sigma(q)\otimes_{\mathbb Q[q^{\pm \frac{1}{2}}]} \mathcal A_q(\Sigma),\;\; x_{{\gamma}}\to q^{v({\gamma})}X_{\gamma}$$ gives a surjective $\mathbb Q[q^{\pm 1}]$-algebra homomorphism. Moreover, for any $x\in \mathcal A_{\Sigma}$, $$\pi(\overline x)= \overline{\pi(x)}.$$
\end{theorem}

\begin{proof}

For any triangle $(\gamma_{1},\gamma_{2},\gamma_{3})$ in $\Sigma$, as $v$ is a quantum cluster data on $\Sigma$, $v(\gamma_{1})+v(\gamma_{2})+v(\gamma_{3})=\frac{1}{2}\left(\Lambda(\gamma_1,\gamma_2)+\Lambda(\gamma_2,\gamma_3)+\Lambda(\gamma_3,\gamma_1)\right)$. Thus,
$$q^{v(\gamma_{1})}X_{\gamma_1} q^{-v(\overline\gamma_{2})}X^{-1}_{\overline\gamma_2}q^{v(\overline\gamma_{3})}X_{\gamma_3}=q^{v(\overline\gamma_{3})}X_{\overline\gamma_3}q^{-v(\gamma_{2})}X^{-1}_{\gamma_2}q^{v(\overline\gamma_{1})}X_{\overline\gamma_1},$$ that is,
$$\pi(x_{\gamma_{1}}x^{-1}_{\overline\gamma_{2}} x_{\gamma_{3}})=\pi(x_{\overline\gamma_{1}}x^{-1}_{\gamma_{2}} x_{\overline\gamma_{3}}).$$

For any quadrilateral in $\Sigma$, as shown in Figure \ref{Fig:1}, if $\alpha$ is not a special loop, by Corollary \ref{flip}, we have $$
\begin{array}{rcl}
v(\alpha')
\hspace{-3pt}&=&\hspace{-3pt}
v(\alpha_{1})+v({\alpha})+v(\alpha_{3})+\frac{1}{2}(\Lambda(\alpha_1,\alpha_3)-\Lambda(\alpha_1,\alpha)-\Lambda(\alpha,\alpha_3))\vspace{2.5pt}\\
\hspace{-3pt}&=&\hspace{-3pt}
v(\alpha_{4})+v(\overline{\alpha})+v(\alpha_{2})+\frac{1}{2}(\Lambda(\alpha_4,\alpha_2)-\Lambda(\alpha_4,\alpha)-\Lambda(\alpha,\alpha_2)). \end{array}
$$

Thus, we have 
$$\pi(x_{\alpha'})=\pi(x_{\alpha_{1}}x^{-1}_{\overline\alpha}x_{\alpha_{3}})+\pi(x_{\alpha_{4}}x^{-1}_{\alpha}x_{\alpha_{2}}).$$

For any bigon $(\alpha_1,\alpha_2)$ around a special puncture $p$, assume that $\alpha$ is the loop around $p$ such that $(\alpha_1,\alpha_2,\alpha)$ is a triangle and $\alpha'$ is the loop around $p$ such that $(\alpha',\alpha_2,\alpha_1)$ is a triangle, then $v(\alpha)=v(\alpha')=0$ and $v(\alpha_1)+v(\alpha_2)=\frac{1}{2}\Lambda(\alpha_2,\alpha_1)$. 

Therefore, we have
$$
\begin{array}{rcl}
\pi(x_{\alpha'}) = X_{\alpha'}
\hspace{-3pt}&=&\hspace{-3pt}
X_{\overline\alpha_{1}}X^{-1}_{\alpha}X_{ \alpha_{ 1}}+2\cos(\frac{\pi}{|p|})q^{-\frac{1}{2}\Lambda(\alpha_1,\alpha_2)}X_{\overline\alpha_{ 1}}
X^{-1}_{\alpha}X_{\overline\alpha_2}+X_{\alpha_{ 2}}X^{-1}_{\alpha}X_{\overline\alpha_{2}}\vspace{2.5pt}\\
\hspace{-3pt}&=&\hspace{-3pt}
\pi(x_{\overline\alpha_{1}}x^{-1}_{\alpha}x_{ \alpha_{ 1}})+2\cos(\frac{\pi}{|p|})\pi(x_{\overline\alpha_{ 1}}
x^{-1}_{\alpha}x_{\overline\alpha_2})+\pi(x_{\alpha_{ 2}}x^{-1}_{\alpha}x_{\overline\alpha_{2}}). \end{array}
$$

Therefore, $x_{{\gamma}}\to q^{v({\gamma})}X_{\gamma}$ define an algebra homomorphism $\pi$. Moreover, as $\mathcal A_q(\Sigma)$ is generated by cluster variables $X_{\gamma}$, it follows that $\pi$ is surjective.

As $v(\overline\gamma)=-v(\gamma)$ and $\overline{x_{\gamma}}=x_{\overline\gamma}$, $\pi(\overline {x_{\gamma}})= \overline{\pi(x_{\gamma})}$. Since the bar involutions on $\mathcal A_q(\Sigma)$ and $\mathcal A_\Sigma$ are algebra anti-homomorphisms, $\pi(\overline x)= \overline{\pi(x)}$ for all $x\in \mathcal A_\Sigma$.

The proof is complete.
\end{proof}

As an application of Theorem \ref{surj}, we give a new expansion formula for quantum cluster variables of $\mathcal A_q$ and prove the positivity.

\begin{corollary}\label{cor:expansion21}
Let $v$ be a quantum cluster data on $\Sigma$. Let $\Delta$ be a triangulation and $\gamma$ be an arc in $\Sigma$. Then
\begin{equation}\label{eq:expan}
\textstyle X_{\gamma}=q^{-v(\gamma)}\sum_{\vec\gamma\in Adm(\gamma, \Delta)} q^{v(\vec\gamma)}c_{\vec \gamma} X(\vec\gamma),
\end{equation}
where $v(\vec\gamma)=\sum v(\gamma_i)$ and $X(\vec\gamma)=X_{\gamma_1}X_{\gamma_2}^{-1}X_{\gamma_3}\cdots$ for any $\vec\gamma=(\gamma_1,\gamma_2,\gamma_3,\cdots)$.
In particular, the positivity conjecture holds for all quantum (generalized) cluster algebras from noncommutative surfaces which have neither $0$-punctures  nor ordinary punctures.
\end{corollary}

\begin{proof}
{The equality \eqref{eq:expan} follows immediately by Theorem \ref{thm:laurent} and Theorem \ref{surj}. Since $|o|\geq 2$ for any special puncture $o$, by \eqref{eq:cgamma} and \eqref{eq:cgammav}, we have $c_{\vec \gamma}\geq 0$. It follows that the coefficients in the expansion \eqref{eq:expan} are non-negative. The positivity follows.}
\end{proof}

\section{Appendix: Groupoids and their symmetries}\label{sec:Appendix A}

Let $\Gamma$ be a groupoid and $\underline \Gamma$ be a directed sub(multi)graph of $\Gamma$ such that $\underline \Gamma$ generates $\Gamma$.



For any object $i$ of $\Gamma$ denote by $\underline{Aut}_\Gamma(i)$ the subgroup of $Aut_\Gamma(i)$  generated by $h h'$ with $h,h'\in \underline \Gamma$, $s(h)=t(h')=i$, $t(h)=s(h')$ (we will sometimes refer to $\underline{Aut}_\Gamma(i)$ as the two-cycle group of automorphisms of $i$). {The following theorem states that $Aut_\Gamma(i)$ is generated by two-cycles under certain conditions.}

\begin{theorem}
\label{th:two-cycle generation}
{Let $\Gamma$ be a groupoid and let $\underline \Gamma$ be a directed subgraph of $\Gamma$ that generates $\Gamma$. Assume that for any edge $h$ in $\underline \Gamma$, its inverse $h^{-1}$ is also an edge in $\underline \Gamma$. Furthermore, assume that $\underline{\Gamma}$ contains no cycles of length $1$ and satisfies the following two conditions:}
\begin{enumerate}
    \item[$(C1)$] For every simple cycle $c$ in $\underline{\Gamma}$, $c=g_{2m}\cdots g_1$ in $\Gamma$ for some arrows $g_1,\ldots,g_{2m}$ in $\underline{\Gamma}$ such that $s(g_k)=t(g_{2m-k+1}),
    t(g_k)=s(g_{2m-k+1})$ for all $k=1,\ldots,m$.

    \item[$(C2)$] For every arrow $f:i\to j$ in $\underline{\Gamma}$, we have $f^{-1}\circ \underline{Aut}_\Gamma(j)\circ f \subseteq \underline{Aut}_\Gamma(i).$
\end{enumerate}
Then $Aut_\Gamma(i)=\underline{Aut}_\Gamma(i)$ for every object $i$ of $\Gamma$.
\end{theorem}

\begin{proof}
We first show that every simple cycle in $\underline{\Gamma}$ belongs to $\underline{Aut}_\Gamma(i)$, where $i$ is its base vertex.

By $(C1)$, it suffices to show that $g_{2m}\cdots g_1\in \underline{Aut}_\Gamma(s(g_1))$ whenever $s(g_k)=t(g_{2m-k+1}),
t(g_k)=s(g_{2m-k+1})$ for all $k=1,\ldots,m$. We proceed by induction on $m$.

The case $m=0$ is trivial. Suppose that $m>0$. By the induction hypothesis, we have $g_{2m-1}\cdots g_2 \in \underline{Aut}_\Gamma(t(g_1)).$

Since $g_{2m}g_1$ is a $2$-cycle based at
$s(g_1)$, it follows from $(C2)$ that
\[g_{2m}\cdots g_1
=(g_{2m}g_1)
\left[g_1^{-1}(g_{2m-1}\cdots g_2)g_1\right]\in \underline{Aut}_\Gamma(s(g_1)).\]
Thus every simple cycle in $\underline{\Gamma}$ belongs to the
corresponding subgroup $\underline{Aut}_\Gamma(i)$.

We now prove that $Aut_\Gamma(i)\subseteq \underline{Aut}_\Gamma(i)$ by induction on the length $n$ of a cycle representing an element of
$Aut_\Gamma(i)$.

The case $n=0$ is trivial. Suppose that $n>0$, and let $f=f_n\cdots f_1$ be a cycle based at $i$.

Choose a closed subpath $c=f_b\cdots f_{a+1}, 0\leq a<b\leq n,$ of minimal positive length. Since $\underline{\Gamma}$ has no cycles
of length $1$, the path $c$ is either a
$2$-cycle or a simple cycle. In either case, by the preceding argument,
we have
$c\in \underline{Aut}_\Gamma(s(c)).$

Set
$u=f_a\cdots f_1,v=f_n\cdots f_{b+1}$. Then $vu$ is a closed path based at $i$ of length strictly less than $n$. By the induction hypothesis,
$vu\in\underline{Aut}_\Gamma(i).$ Moreover, applying $(C2)$ successively
along the arrows of $u$, we obtain $u^{-1}cu\in\underline{Aut}_\Gamma(i).$
Consequently, $f=vcu=(vu)(u^{-1}cu) \in\underline{Aut}_\Gamma(i).$

Therefore, $Aut_\Gamma(i)\subseteq \underline{Aut}_\Gamma(i)$. Since the reverse inclusion holds by definition, we conclude that
$Aut_\Gamma(i)=\underline{Aut}_\Gamma(i).$
\end{proof}

The following is immediate.

\begin{lemma} 
\label{le:groupoid representation}
For any {groupoid} ${\mathcal C}$, the assignments $i\mapsto Aut_{\mathcal C}(i)$ define a functor $Aut:{\mathcal C}\to {\bf Grp}'$, whose objects are groups and whose morphisms are group isomorphisms.
\end{lemma}

\begin{lemma} Given a small category ${\mathcal C}$ and  a group $\Gamma\subset Aut({\mathcal C})$, the ${\mathcal D}:={\mathcal C}/\Gamma$ is a well-defined quotient category.

In particular, $Aut_{\mathcal D}(\Gamma\cdot c)=(Aut_{\mathcal C}(c))^{Stab_\Gamma(c)}$ for any object $c$ of ${\mathcal C}$.
\end{lemma}

Let ${\mathcal C}$ and ${\mathcal D}$ be isomorphic small categories and $F_0$ be an isomorphism ${\mathcal C}\simeq {\mathcal D}$. Define a category ${\mathcal C}\# {\mathcal D}$ which contains ${\mathcal C}$ and ${\mathcal D}$ as subcategories, $Ob({\mathcal C}\# {\mathcal D})=Ob({\mathcal C})\sqcup Ob({\mathcal D})$ and morphisms of 
${\mathcal C}\# {\mathcal D}$ are compositions of morphisms of  ${\mathcal C}$ and ${\mathcal D}$ with the invertible morphisms $a_i:i\mapsto F_0(i)$ and their inverses $a_{F_0(i)}:=a_i^{-1}$ subject  to
$$ f  a_{s(f)}=a_{t(f)} f$$
for any morphisms $f$ in ${\mathcal C}$.

The following is immediate.

\begin{lemma} There is a unique (involutive)  automorphism $F$ of ${\mathcal C}\#{\mathcal D}$ such that $F|_{\mathcal C}=F_0$, $F|_{\mathcal D}=F_0^{-1}$, and $F(a_i)=a_i^{-1}$ for any object $i$ of ${\mathcal C}$. Moreover, 
the assignment $i\mapsto  a_i$ is a natural transformation from the identity functor to  $F$.

\end{lemma}

This construction generalizes to the direct product of ${\mathcal B}\times {\mathcal C}$ of any categories   ${\mathcal B}$ and ${\mathcal C}$ (see e.g., \cite[Section II.3, page 36]{Mac}). Namely, $Ob({\mathcal B}\times {\mathcal C}):=Ob({\mathcal B})\times Ob({\mathcal C})$ and $Hom_{{\mathcal B}\times {\mathcal C}}((b,c),(b',c'))=Hom_{\mathcal B} (b,b')\times Hom_{\mathcal B} (c,c')$ for any object $b,b'$ of ${\mathcal B}$ and $c,c'$ of ${\mathcal C}$ with the natural composition law
$$(\varphi,\psi)(\varphi',\psi)=(\varphi \varphi',\psi,\psi')$$
whenever $\varphi \varphi'$ is defined in ${\mathcal B}$ and $\psi \psi'$ is defined in ${\mathcal C}$.

In particular, 
$(\varphi,\psi) =(\varphi,Id_{t(\psi)})(Id_{s(\varphi)},\psi)=(Id_{t(\varphi)},\psi)(\varphi,Id_{s(\psi)})$.

The following is immediate. 

\begin{lemma} For any endofunctors 
 $F_{\mathcal B}$ of ${\mathcal B}$ and   
 $F_{\mathcal C}$ of ${\mathcal C}$  one has
 
(a) The assignments $(b,c)\mapsto (F_{\mathcal B}(b),F_{\mathcal C}(c))$, $(b,c)\in Ob({\mathcal B}\times {\mathcal C})$
define a unique endofunctor $F_{\mathcal B}\times F_{\mathcal C}$ of  ${\mathcal B}\times {\mathcal C}$. 

 (b) For any natural transformations $\tau_{\mathcal B}:Id_{\mathcal B}\to F_{\mathcal B}$ and $\tau_{\mathcal C}:Id_{\mathcal C}\to F_{\mathcal C}$ the assignments $(b,c)\mapsto (\tau_{\mathcal B}(b),{\tau_{\mathcal C}(c)})$ define a natural transformation $\tau_{\mathcal B}\times \tau_{\mathcal C}: Id_{{\mathcal B}\times {\mathcal C}}\to F_{\mathcal B}\times F_{\mathcal C}$.
\end{lemma}

Then define the quotient category ${\mathcal C}/G$ whose object set is $Ob({\mathcal C})/G$ the set of orbits and whose $Hom$ set is the composition closure of the equivalence relation $f\equiv f'$ for morphisms $f:a\to b$ and $f':a'\to b'$ of ${\mathcal C}$ iff $f'=g(f)$ for some $g\in G$ (e.g., $a'=g(a)$, $b'=g(a)$).

\begin{lemma} ${\mathcal C}/G$ is a well-defined category.
    
\end{lemma}

We will also use the following fact. Let  ${\bf Grp}$ denote the category of groups where morphisms are group homomorphisms. 

The following is immediate

\begin{lemma}
\label{le:unique up to conjugation homomorphism}  
Let $\mathcal C$ be a category and $F:\mathcal C\to {\bf Grp}$ be a functor.

 (a) Let $\Gamma$ be a connected subgroupoid of $\mathcal{C}$.
 Then there exist:

$\bullet$ a group $G(\Gamma)$, unique up to up to inner automorphisms, which is isomorphic to $F(i)$ for all $i\in \mathrm{Ob}(\Gamma)$;

$\bullet$ a group $A(\Gamma)$, unique up to inner automorphisms, which is isomorphic to $Aut_{\Gamma}(i)$ for all $i\in Ob(\Gamma)$;

$\bullet$ an action of $A(\Gamma)$ on $G(\Gamma)$, 
such that for any object $i\in \Gamma$, the following diagram commutes (up to up to inner automorphisms of $G(\Gamma)$ and inner automorphisms $A(\Gamma)$):
\[
\xymatrix{
A(\Gamma)\times G(\Gamma) \ar[r] \ar[d] & G(\Gamma) \ar[d] \\
Aut_{\Gamma}(i)\times F(i) \ar[r] & F(i)
.}
\]
In particular the kernel of the corresponding homomorphism $\rho:A(\Gamma)\to Aut(G(\Gamma))$ is well-defined.

(b) Let $\Gamma,\Gamma'$ be connected subgroupoids of $\mathcal C$. Let $f: \Gamma\to \Gamma'$ be a functor and let $\tau$ be a natural transformation from the identical inclusion $\Gamma\hookrightarrow \mathcal C$ to the composition of $f$ with the identical inclusion $\Gamma'\hookrightarrow \mathcal C$
Then there exist group homomorphisms $A(\tau):A(\Gamma)\to A(\Gamma')$ and $G(\tau): G(\Gamma)\to G(\Gamma')$ such that the following diagram commutes:
    \[
    \xymatrix{
    A(\Gamma)\times G(\Gamma) \ar[r] \ar[d]_{A(\tau)\times G(\tau)} & G(\Gamma) \ar[d]_{G(\tau)} \\
    A(\Gamma')\times G(\Gamma') \ar[r] & G(\Gamma')
    }
    \]



%





\end{lemma}

\end{document}